\documentclass[12pt]{amsart}
\usepackage{geometry}                % See geometry.pdf to learn the layout options. There are lots.
\usepackage{hyperref}                % ... or a4paper or a5paper or ... 
\usepackage{graphicx}
\usepackage{color}
\usepackage[most]{tcolorbox}

\tcbset{colback=yellow!10!white, colframe=red!50!black, 
        highlight math style= {enhanced, %<-- needed for the ?remember? options
            colframe=red,colback=red!10!white,boxsep=0pt}
        }

\usepackage{amssymb}
\usepackage{mathrsfs}
\usepackage{bm,marginnote,relsize}
\usepackage{amscd}
\usepackage{epstopdf}
\newtheorem{thm}{Theorem}[section]
\newtheorem*{thm*}{Theorem}
\newtheorem{prop}[thm]{Proposition}
\newtheorem*{prop*}{Proposition}
\newtheorem{lem}[thm]{Lemma}
\newtheorem*{lem*}{Lemma}

\newcommand{\wh}{\widehat}
\newcommand{\wt}{\widetilde}
\theoremstyle{definition}
\newtheorem{definition}[thm]{Definition}
\newtheorem*{definition*}{Definition}
\newtheorem{example}[thm]{Example}
\theoremstyle{remark}

\newtheorem{remark}[thm]{Remark}

\title[Polyfolds and SFT]{Lectures on POLYFOLDS and Symplectic Field Theory}
\author{Joel W.  Fish and Helmut Hofer}
\newcommand{\rup}{\rotatebox[origin=c]{180}{{$\Upsilon$}}}
\begin{document}
\maketitle
\vspace{5cm}
 \begin{figure}[h]
\begin{center}
\includegraphics[width=7.5cm]{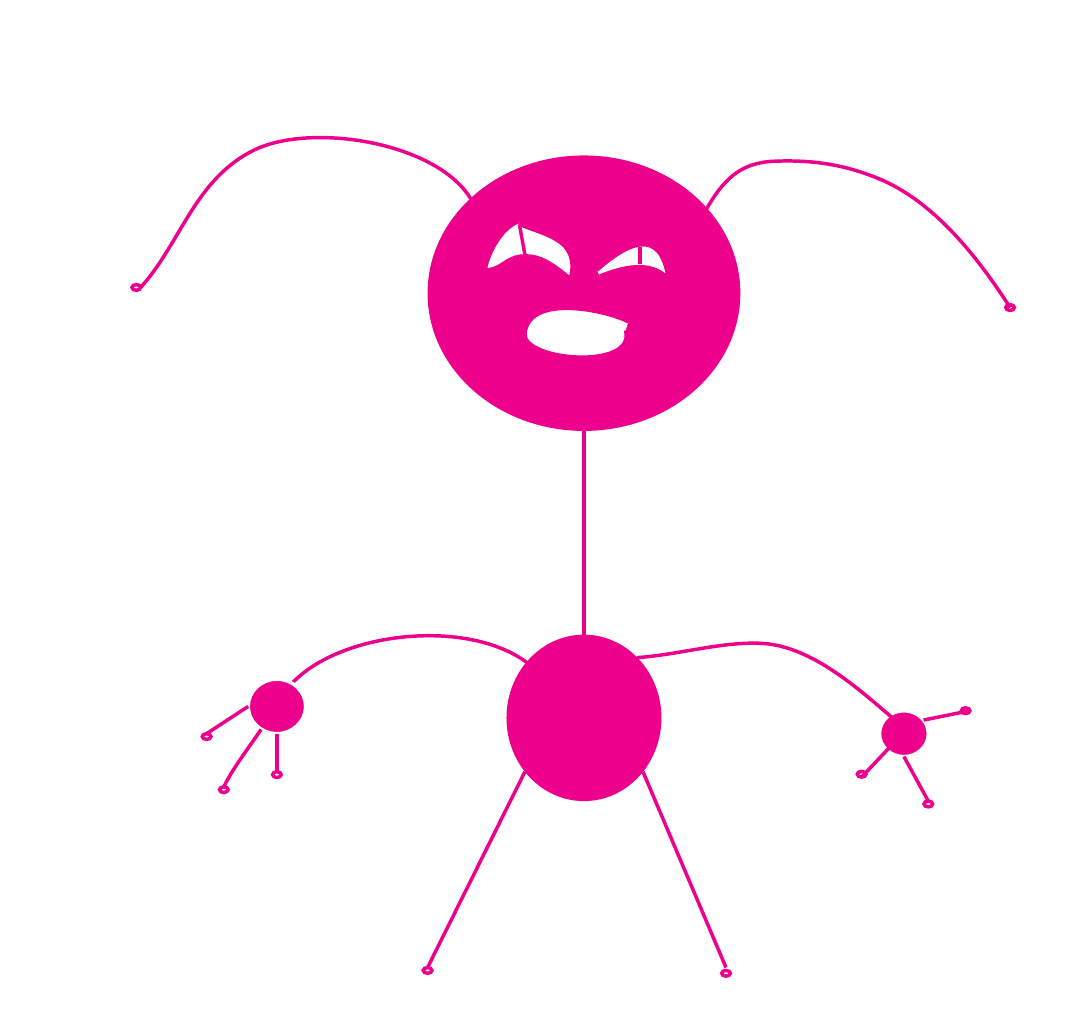}
\end{center}
\end{figure}
\newpage

\tableofcontents
\pagebreak

\part*{Introduction}
 \section*{Background}
 The following material has been compiled for the \textbf{SFT 9 Workshop} in Augsburg, Germany, taking place Monday, 27 August 2018 - Friday, 31 August 2018. The necessary background material 
 will be provided in  a pre-course, Saturday, 25 August 2018 - Sunday, 26 August 2018. For those not able to participate in the pre-course we shall give references to the assumed material.
 Particularly we mention the primer \cite{FH-primer}, the local-local constructions, \cite{FH-local-local}, and the appendices in the present text.

 This present  text describes a polyfold approach to the construction of symplectic field theory (SFT), introduced in \cite{EGH},  based on the technology developed 
 in \cite{HWZ2017} and \cite{FH-book}. The present draft is a shortened version of the upcoming lecture note \cite{FH-Lecturenote} which will provide more details and will, among other things, construct SFT in detail.  In the current text we end by giving the perturbation and transversality 
 results needed. A reader familiar with \cite{EGH,BM,Wendl} should be able to work out the orientations and, using the fact that the pull-back of differential forms by the evaluation maps are sc-differential forms, integrate them over the moduli spaces, see \cite{HWZ2017}.
The material presented here then leads to a construction of a {Hamiltonian} $\bm{H}_{\mathfrak{p}}$ in the context of a given closed manifold
$(Q,\lambda)$ equipped with a non-degenerate contact form associated to a careful choice of perturbation $\mathfrak{p}$, and  how these are related for different choices of the perturbations.  The transversality methods described here, cover also the 
cases which we have not explicitly studied in the lecture series, namely the behavior under symplectic cobordisms, their compositions, as well as the change 
of complex multiplication $J$.

Polyfold theory by itself has nothing to do with symplectic geometry. It is rather a mixture of a generalized differential geometry, a generalization of classical nonlinear Fredholm theory,
and some category theory. This theory has been developed by Hofer, Wysocki and Zehnder in a series of papers \cite{HWZ2,HWZ3,HWZ3.5}. Subsequently the theory was further generalized and a comprehensive treatment 
is contained in \cite{HWZ2017}, which is a reference monograph on which further developments rest. 
In \cite{HWZ5} the theory was applied to derive  Gromov-Witten theory, an  application, which was more a proof of concept.  In \cite{HWZ5} the  authors did not try to utilize some 
of the important features of the polyfold theory, namely that one can use the language to construct powerful methods and tools for the construction of concrete polyfolds. These ideas 
have been developed in the upcoming monograph \cite{FH-book}. The available references \cite{FH-primer,FH-local-local}  contain  part one and two of \cite{FH-book}. Again \cite{FH-book} is primarily not a book about symplectic geometry, but it explains the ideas  using examples arising from analytical problems in symplectic geometry. One of the ideas put forward and carried out is the development of a \textbf{modular theory} which allows to \textbf{recycle} constructions. 
This is an extremely important point in applications given the fact that the constructions get longer and longer 
and are becoming increasingly difficult to check.
The modular features also include
a pre-Fredhom theory, which in applications considerably simplifies the task of showing that a sc-smooth section is sc-Fredholm.  We shall not discuss this important point in this lecture, but it will feature prominently in \cite{FH-book}.

In these lectures we shall show, how a very complicated nonlinear problem arising in symplectic geometry can be dealt with quite efficiently. Using the results in \cite{FH-book,HWZ2017}
takes care of the basic  analytical problems, so that one can concentrate without distraction on the transversality theory, the study of moduli spaces and ultimately the construction of SFT. In particular,  the Nonlinear Functional Analysis/ Nonlinear Analysis in 
\cite{HWZ2017}/\cite{FH-book} can be used  in applications to address problems by general principles rather than ad hoc methods: this will shorten proofs and adds to the clarity. In the current text 
we shall particularly describe the \textbf{big picture} and 
 the \textbf{modular aspects} of the theory. 

There are other methods which have been developed to study moduli spaces in symplectic geometry. They all  originate in the underlying idea of finite-dimensional reductions 
put forward in the original work \cite{FO} by Fukaya and Ono.  We particularly mention \cite{FOOO,FOOO1}, \cite{McW1,McW2,McW3}, and \cite{FOOO2013,FOOO2015I,FOOO2017II}. These ideas have been used (sometimes in a modified way) to deal with problems 
in symplectic geometry, see the work \cite{FOOO2010,FOOO2010b},
\cite{Pardon1,Pardon2}, \cite{BH}, and very recently
\cite{IS}.

In some sense the theory of Kuranshi structures and Polyfold theory are equivalent as studied in detail 
in  \cite{Yang1,Yang2,Yang3,Yang4}. This means that in many cases ideas in one theory can be transported to the other.

We also would like to mention that the polyfold theory goes well with geometric perturbation methods,  as in \cite{HuN,HuN2,Wendl2}. This means that one is able
to stay within the framework of geometric perturbations as long as these methods work before switching to more general perturbations. 
SFT is an example which shows that the perturbation theory has to use special geometric features and this will be discussed 
later on. The thrust of the current lectures 
lies on the transversality and perturbation theory. 

We shall refer to \cite{HWZ2017} as the  \textbf{Polyfold Book} and to \cite{FH-book}  as the \textbf{Construction Book}.
The Polyfold book is about `developing'
    an abstract non-linear functional analysis, and the Construction Book
    is about `using' that non-linear functional analysis to build specific
    (concrete) function spaces, in a modular way, which are very useful for
    moduli problems in Symplectic Field Theory.

\subsection*{Remarks by the Second Author}
Finally two important remarks. The first is  about multisections, which are important in both approaches and originate in \cite{FO}. It is a 
non-trivial matter to extend multisections defined on a boundary with corners and the difficulties are described  and dealt with in the polyfold framework  in \cite{HWZ2017}, and in a form more useful in inductive constructions in  \cite{HS}. 
In the work by Fukaya et al,  \cite{FOOO1}, Part II, it is said without comment that they extend (see after Lemma 7.2.121), which one might consider a gap, since a proof
is considerably involved. Even in later works this point is not addressed until the paper \cite{FOOO2017II}, where 
a method is described which is only applicable to certain type of multisections (Fukaya refers to this method as the \textbf{Collar Method}).  This method  used in an application requires one to show 
that one can work with these special sections.  In a recent  email exchange of Fukaya with the second author (henceforth  HH),\cite{Fuk},  
Fukaya describes a different way to deal with the issue, which
    turns out to address the issues in a way similar to the approach used
     in \cite{HWZ2017} and \cite{HS}, albeit in a somewhat different language.
Dusa McDuff also has a method, and HH almost expects that it is is also related. 

 In \cite{IS} Ishikawa introduces a somewhat more structured version of multisections, called \textbf{grouped multisections}, which is geared towards his constructions. The version given in \cite{IS} v1 has an incomplete list of requirements as HH pointed out to the author, but this particular problem already has been rectified in v2.
  Since the lectures, besides explaining many aspects of the polyfold 
theory also   deal with the transversality issues of SFT, a word she be provided about \cite{IS} by HH.
It seems at first glance that the issues which have to be addressed are indeed discussed. Of course, the details 
have to be carefully checked.  Although the approach to SFT in \cite{IS} and the approach given here are quite different there is a common thread given by the issues to be addressed.  
We shall point out these issues and the reader might find it helpful when reading \cite{IS}.\\

\section*{The Lectures and Their Prerequisites}
We have added some useful material as appendices and give some reference to other used  material.
For the initial lectures we list some of required  background material. Moreover, the lecture not contains somewhat more material 
than can be delivered in lecture. Usually the material is needed in a later, but not the next, lecture. We expect the reader to familiar with the additional material when it is needed.\\

\noindent$\bullet$ Lecture 1 requires the knowledge from the Primer on Polyfolds, see \cite{FH-primer}, pages 8-23. Comprehensive background material about polyfold structures on groupoidal categories can be found in \cite{HWZ2017}.\\

\noindent{$\bullet$} Lecture 2 requires   some of the material  on DM-Theory, see \cite{HWZ5}, pages 11-28,  \cite{Hofer} and Appendix \ref{Sec16}.
The cited references take material from the upcoming \cite{HWZ-DM}, which develops DM-theory in such a way that it fits with the polyfold theory. \\

\noindent$\bullet$ Lecture 3 requires some basic facts on  Stable Hamiltonian Structures, see \cite{CV1} or Appendix \ref{APP11}.\\

\noindent$\bullet$ Lecture 4 requires some basic facts on  M-polyfolds, see \cite{FH-primer}. A more comprehensive background is given in 
\cite{HWZ2017}.\\

\noindent$\bullet$ Lecture 5 requires some basic facts on  M-polyfolds and covers material from  \cite{FH-primer}.\\

\noindent$\bullet$ Lecture 6, 7 and 8 require some basic facts on M-polyfolds and the imprinting method introduced in Lecture 5.
For background material see \cite{FH-primer}. A detailed discussion is contained in \cite{FH-book}.

\newpage
\part{Big Picture}
We begin by describing the general set-up for constructing the moduli spaces of SFT.

\part*{Lecture 1}
\section{Smooth Structures on Certain Categories}
We describe a categorical framework which we shall apply to the category of stable maps as they occur in symplectic field theory. Here a stable map is not necessarily pseudoholomorphic. The original idea
of a stable pseudoholomorphic map or curve evolved from the work of Gromov, \cite{G}, and  was formulated by Kontsevich, \cite {Kon}. The appropriate notion for symplectic field theory was given in 
\cite{EGH} and its compactness properties were studied in \cite{BEHWZ}. 

We start with some abstract notions which are not too difficult. 
\subsection{GCT's}
The notions introduced in this and the following subsections are treated in detail in \cite{HWZ2017}.
We shall consider  groupoidal categories ${\mathcal C}$ which are defined as follows.
\begin{definition}
A {\bf groupoidal category} ${\mathcal C}$ is a category with the following properties.
\begin{itemize}
\item[(1)]  Every morphism is an isomorphism.
\item[(2)]  Between any two objects there are only finitely many morphisms.
\item[(3)] The orbit space $|{\mathcal C}|$, i.e. the collection of isomorphism classes, is a set.
\end{itemize}
\qed
\end{definition}
Further, the objects we shall consider are what we call GCT's      (GC=groupoidal category,  T=Topology)
\begin{definition}
A {\bf GCT} is a pair $({\mathcal C},{\mathcal T})$ where ${\mathcal C}$ is a groupoidal category and ${\mathcal T}$ a metrizable topology on the orbit space 
$|{\mathcal C}|$.
\qed
\end{definition}
Given a GCT $({\mathcal C},{\mathcal T})$ we can talk about objects in the category whose isomorphism classes are close to a given isomorphism class.
Using polyfold theory we can introduce the notion of  some kind of smooth structure on a GCT.

\subsection{Uniformizers}
Denote by  $\mathsf{M}$ the category whose objects are  M-polyfolds and the morphisms are the sc-smooth maps. By $\mathsf{M}_{\text{tame}}$ we denote the full subcategory of tame M-polyfolds.  Let $O$ be a M-polyfold and $G$ a finite  group acting on $O$ by sc-diffeomorphisms written as 
$$
G\times O\rightarrow O: (g,o)\rightarrow g\ast o.
$$
Associated to $(O,G)$ we have the {\bf translation groupoid} which is the category denoted by $G\ltimes O$. Its objects are the points $o\in O$ and the morphisms are the pairs $(g,o)$ viewed as morphisms
$$
o\xrightarrow{(g,o)} g\ast o.
$$
The following generalizes some ideas which also occur when studying orbifolds.
\begin{definition}
Given a GCT $({\mathcal C},{\mathcal T})$ a {\bf local uniformizer} at the object $\alpha$ with automorphism group $G$ is a covariant functor $\Psi:G\ltimes O\rightarrow {\mathcal C}$ with the following properties:
\begin{itemize}
\item[(1)] $O$ is a M-polyfold and $G$ acts by sc-diffeomorphisms and $G\ltimes O$ is the associated translation groupoid.
\item[(2)] The functor $\Psi$ is injective on objects and $\alpha$ lies in the image, i.e. there is a unique $\bar{o}\in O$ with $\Psi(\bar{o})=\alpha$.
\item[(3)] The functor $\Psi$ is full and faithful.
\item[(4)] Passing to orbit spaces $|\Psi|:|G\ltimes O|\rightarrow |{\mathcal C}|$ is a homeomorphism onto an open neighborhood of $|\alpha|$. Here $|\alpha|$ denotes the isomorphism class of the object $\alpha$.
\end{itemize}
A {\bf tame uniformizer} is a uniformizer where $O$ is a tame M-polyfold.
\qed
\end{definition}
The first crucial definition is that of a local uniformizer construction. 
\begin{definition}
A {\bf local uniformizer construction} is a functor $F:{\mathcal C}\rightarrow \text{SET}$ such that 
$F(\alpha)$ is a set of uniformizers at $\alpha$. We call $F$ a {\bf tame local unifomizer construction} provided for every $\alpha$ the set $F(\alpha)$ consists of tame uniformizers.
\qed
\end{definition}

\subsection{Compatibility}
Assume we are given $(({\mathcal C},{\mathcal T}),F)$, where $({\mathcal C},{\mathcal T})$ is a GCT and $F:{\mathcal C}\rightarrow \text{SET}$ a local uniformizer construction.
Compatibility of the various $\Psi$ requires an additional construction $\bm{M}$.
Given $\Psi\in F(\alpha)$ and $\Psi'\in F(\alpha')$ consider
$$
G\ltimes O\xrightarrow{\Psi} {\mathcal C} \xleftarrow{\Psi'} G'\ltimes O'
$$
The {\bf associated weak fibered product} is the set 
$$
\bm{M}(\Psi,\Psi')=\{(o,\Phi,o')\ |\ o\in O,\ o'\in O',\ \Phi\in \text{mor}_{\mathcal C}(\Psi(o),\Psi'(o'))\}.
$$
\begin{definition}
We have the following structural maps for $\Psi,\Psi',\Psi''$ coming from $F$.
\begin{itemize}
\item[(1)] $ O\xleftarrow {s} \bm{M}(\Psi,\Psi')\xrightarrow{t}O'$, {\bf source map} and {\bf target map}, i.e. $s(o,\Phi,o')=o$, $t(o.\Phi,o')=o'$.
\item[(2)] $u:O\rightarrow \bm{M}(\Psi,\Psi):o\rightarrow (o,1_{\Psi(o)},o)$, {\bf unit map}.
\item[(3)] $\iota:\bm{M}(\Psi,\Psi')\rightarrow \bm{M}(\Psi',\Psi):(o,\Phi,o')\rightarrow(o',\Phi^{-1},o)$, {\bf inversion maps}.
\item[(4)] $m:\bm{M}(\Psi',\Psi''){_{s}\times_t} \bm{M}(\Psi,\Psi')\rightarrow \bm{M}(\Psi,\Psi'')$, {\bf multiplication map},
defined by 
$$
m((o',\Phi',o''),(o,\Phi,o'))=(o,\Phi'\circ\Phi,o'').
$$
\end{itemize}
We call $\bm{M}(\Psi,\Psi')$ a \textbf{transition set}.
\qed
\end{definition}
Now comes the second important construction given $(({\mathcal C},{\mathcal T}), F)$.
\begin{definition}
A {\bf transition construction} $\bm{M}$ for $(({\mathcal C},{\mathcal T}),F)$ is a construction which  equips every  $\bm{M}(\Psi,\Psi')$  with a M-polyfold structure so that $s$ and $t$ are local sc-diffeo\-morphisms and all other structure maps are sc-smooth. Then $(({\mathcal C},{\mathcal T}),F,\bm{M})$ is called a {\bf polyfold structure}  on the GCT $({\mathcal C},{\mathcal T})$. If $F$ is a tame local uniformizer construction 
we call $(({\mathcal C},{\mathcal T}),F,\bm{M})$ a {\bf tame polyfold structure}.
\qed
\end{definition}

\subsection{Bundle Structure}
Denote by $\text{\textbf{Ban}}$ the category of Banach spaces.
Assume we are given a GCT ${\mathcal C}$ and a functor $\mu:{\mathcal C }\rightarrow \text{\bf Ban}$ and assume 
we can define a new GCT ${\mathcal E}_\mu$, where the objects are $(\alpha,e)$, $e\in \mu(\alpha)$ and the morphisms 
are of the form $(\Phi,e):(s(\Phi),e)\rightarrow (t(\Phi),\mu(\Phi)(e))$. Then we obtain
$$
P:{\mathcal E}_\mu\rightarrow {\mathcal C}.
$$
We can introduce the notion of a strong Bundle uniformizers which is given by a diagram
$$
\begin{CD}
G\ltimes K @> \bar{\Psi} >>{\mathcal E}_\mu\\
@VVV  @V P VV\\
G\ltimes O @> \Psi>>   {\mathcal C}
\end{CD}
$$
where $\bar{\Psi}$ is fiberwise a linear bijection.
All the previous discussions can be extended.  For example a {\bf  transition construction}
produces strong bundle structure on 
$$
\bm{M}(\bar{\Psi},\bar{\Psi}')\rightarrow \bm{M}(\Psi,\Psi').
$$
Structure maps fit into diagrams like this 
$$
\begin{CD}
K @< s << \bm{M}(\bar{\Psi},\bar{\Psi}') @> t >> K'\\
@V p VV     @V  \bm{P} VV  @V p' VV\\
O  @< s << \bm{M}(\Psi,\Psi') @> t >>  O
\end{CD}
$$
The idea of a polyfold structure on ${\mathcal C}$  can be generalized as follows.
Assume  we are given a polyfold construction
for the GCT $({\mathcal C},{\mathcal T})$, say
$$
(({\mathcal C},{\mathcal T}), F,\bm{M}),
$$
and a functor $\mu:{\mathcal C}\rightarrow \text{\bf Ban}$ defining  ${\mathcal E}_{\mu}$. We assume that the latter has been equipped with a metrizable topology 
so that it is turned into a GCT as well, for which the map $|P|:|{\mathcal E}_\mu|\rightarrow |{\mathcal C}|$ is continuous.
Then a strong bundle uniformizer construction is a functor $\bar{F}:{\mathcal C}\rightarrow \text{SET}$ associating to $\alpha$ a set
$\bar{F}(\alpha)$ of strong bundle uniformizers $\bar{\Psi}:G\ltimes K\rightarrow {\mathcal E}$ with the obvious properties.
Similarly we have a notion of a transition construction.

Having the strong bundle structure for $P:{\mathcal E}_\mu\rightarrow {\mathcal C}$ one can introduce new smooth objects.
\begin{definition}
Consider a  section functor $A$ of $P$, i.e. a functor $A:{\mathcal C}\rightarrow {\mathcal E}_\mu$ with $P\circ A=\text{Id}$.
We can distinguish several classes of such functors.
\begin{itemize}
\item[(1)]  $A$ is a \textbf{sc-smooth section functor} provided for $\bar{\Psi}\in \bar{F}(\alpha)$ the local representation 
$\bar{\Psi}^{-1}\circ A\circ \Psi$ is sc-smooth.
\item[(2)]  $A$ is a \textbf{sc$^+$-smooth section functor} provided for $\bar{\Psi}\in \bar{F}(\alpha)$ the local representation 
$\bar{\Psi}^{-1}\circ A\circ \Psi$ is sc$^+$-smooth.
\item[(2)] $A$ is a \textbf{sc-Fredholm functor} provided for $\bar{\Psi}\in \bar{F}(\alpha)$ the local representation 
$\bar{\Psi}^{-1}\circ A\circ \Psi$ is sc-Fredholm.
\end{itemize}
\qed
\end{definition}
The take away from this first lecture is that for groupoidal categories ${\mathcal C}$ which have an orbit space equipped with a metrizable topology there is a notion of a smooth structure, i. e.
a polyfold structure and there is a strong bundle version as well
$$
\boxed{
\begin{array}{cc}
\begin{CD}
K @< s << \bm{M}(\bar{\Psi},\bar{\Psi}') @> t >> K'\\
@V p VV     @V  \bm{P} VV  @V p' VV\\
O  @< s << \bm{M}(\Psi,\Psi') @> t >>  O
\end{CD}
&\phantom{XXXXXXX} \begin{CD}
\bm{M}(\bar{\Psi},\bar{\Psi}')\\
@VVV\\
\bm{M}(\Psi,\Psi')
\end{CD}
\end{array}
}
$$
 Our goal is to apply this to the category of stable maps ${\mathcal S}$, the bundle of $(0,1)$-forms along them ${\mathcal E}_{J}\rightarrow {\mathcal S}$, and $\bar{\partial}_{\widetilde{J}}$.
 
Given the GCT $({\mathcal C},{\mathcal T})$ every local uniformizer $\Psi$ has a \textbf{footprint} which is the open subset $|\Psi(O)|$ of $|{\mathcal C}|$. 
We can take a set $\bm{\Psi}$ of uniformizers such that the footprints cover $|{\mathcal C}|$. We can take the disjoint union of all domains defining a M-polyfold $X$ and one can 
use the $\bm{M}(\Psi,\Psi')$ to defined as a disjoint union a M-polyfold $\bm{X}$. Then $X\equiv (X,\bm{X})$ is an ep-groupoid and there is natural functor 
$\beta_{\bm{\Psi}}:X_{\bm{\Psi}}\rightarrow {\mathcal C}$ which is an equivalence of categories. In some sense $X_{\bm{\Psi}}$ is a sc-smooth model of ${\mathcal C}$. 
For two different choices there is a precise relationship between the associated $X_{\bm{\Psi}}$. For certain constructions these smooth models are useful, see \cite{HWZ2017}.

\newpage

\part*{Lecture 2}
\section{The Category of Stable Riemann Surfaces}
This example deals with a version of the Deligne-Mumford theory, \cite{DM}. For a geometric approach to the DM-theory see \cite{RS}.
For the use of the DM-theory within the later constructions it is important 
that we can use different gluing profiles, see \cite{HWZ5}. We assume the reader to be familiar with some Riemann surface theory, see \cite{HWZ5,Hofer} and Appendix \ref{Sec16}.
\subsection{Basic Notions and Concepts}
First we introduce the notion of a compact nodal Riemann surface with (unordered) marked points.
\begin{definition}
Consider a tuple $(S,j,M,D)$, where $(S,j)$ is a compact Riemann surface 
possibly having different connected components, $M$ is a finite subset of so-called {\bf marked points},  and $D$ is a finite collection of unordered pairs $\{x,y\}$, called {\bf nodal pairs}, where $x,y\in S$.
We abbreviate the union $\bigcup_{\{x,y\}\in D} \{x,y\}$ by $|D|$. We assume that the data satisfies the following additional properties.
\begin{itemize}
\item  For $\{x,y\}\in D$ we have that $x\neq y$.
\item If $\{x,y\}\cap \{x',y'\}\neq \emptyset$ then $\{x,y\}=\{x',y'\}$.
\item $M\cap |D| =\emptyset$
\end{itemize}
We call $(S,j,MD)$ a {\bf nodal  Riemann surface with marked points} 
(We allow, of course, the possibility $D=\emptyset$ and $M=\emptyset$.). We say that $(S,j,M,D)$ is \textbf{stable} if for every connected component $C$ of $S$ the genus $g_C$ of $C$ 
satisfies the inequality $2\cdot g_C +\sharp (C\cap (M\cup |D|))\geq 3$.  We generally assume that the set of marked points is un-ordered.
\qed
\end{definition}
 \begin{figure}[h]
\begin{center}
\includegraphics[width=6.5cm]{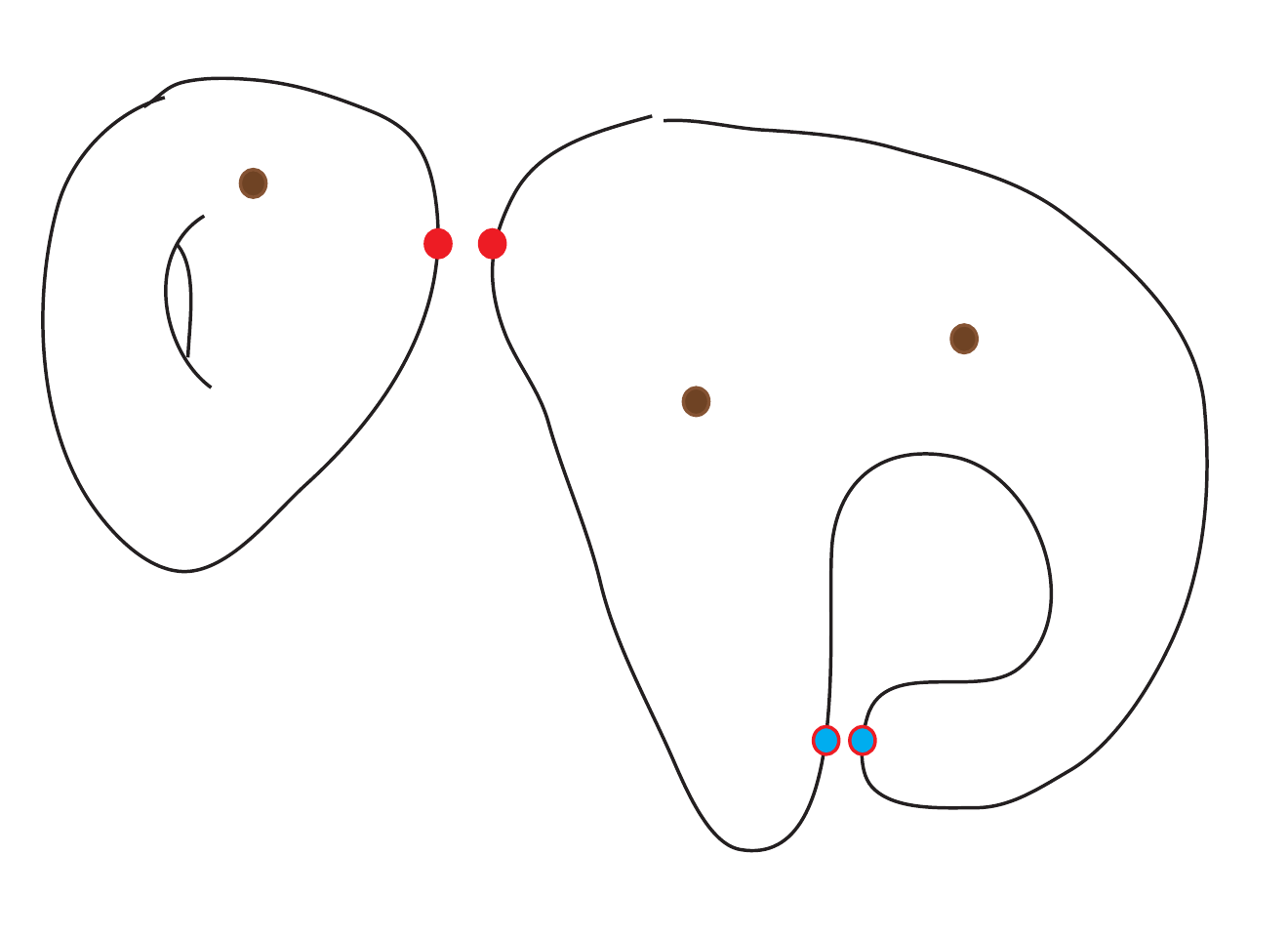}
\end{center}
\caption{A stable Riemann surface with two nodal pairs and three marked points.}\label{FIG 1}
\end{figure}
Given $\tau=(S,j,M,D)$ and $\tau'=(S',j',M,,D')$ we define a \textbf{morphism}
$\Phi:\tau\rightarrow \tau'$ as a tuple $(\tau,\phi,\tau')$ where $\phi:(S,j)\rightarrow (s',j')$
is a biholomorphic map having the following properties.
\begin{itemize}
\item $\phi(M)=M'$.
\item $\phi_\ast (D)=D'$, where $\phi_\ast D=\{\{\phi(x),\phi(y)\}\ |\ \{x,y\}\in D\}.$
\end{itemize}
Finally we are able to define the category of stable Riemann surfaces.
\begin{definition}
The {\bf category ${\mathcal R}$ of stable Riemann surfaces} has as objects 
the nodal stable Riemann surfaces with marked points and as morphisms 
the previously introduced $\Phi$.\qed
\end{definition}
We observe that every morphism is in fact an isomorphism.
It requires more work to show that between two objects 
there are at most finitely morphisms.    In fact, the orbit space $|{\mathcal R}|$
even has a natural metrizable topology. Hence  
\begin{thm}
The category ${\mathcal R}$ admits a natural metrizable  topology ${\mathcal T}$
and $({\mathcal R},{\mathcal T})$ is a GCT. For this topology every connected component of $|{\mathcal R}|$ is compact.
\qed
\end{thm}
This can be proved with the results presented in \cite{Hummel}. A proof will be given in \cite{HWZ-DM} and there is also a discussion in 
\cite{H2014,HWZ5}.
\subsection{Good Deformations}
A version of Deligne-Mumford theory can be obtained in the framework of uniformizers 
as follows, where we 
 need the following definition.
\begin{definition}
Let $\tau=(S,j,M,D)$ be a stable Riemann surface with automorphism group $G$.  A {\bf small disk structure} $\bm{D}$
for $\tau$ consists of a choice of a smooth closed disk-like neighborhood $D_x$ for every $x\in |D|$ 
such that 
\begin{itemize}
\item $D_x\cap (|D|\cup M)=\{x\}$.
\item $ D_x\cap D_y= \emptyset$ for $x\neq y$.
\item $\bigcup_{x\in |D|} D_x $ is invariant under $G$.
\end{itemize}
\qed
\end{definition}
 \begin{figure}[h]
\begin{center}
\includegraphics[width=14.5cm]{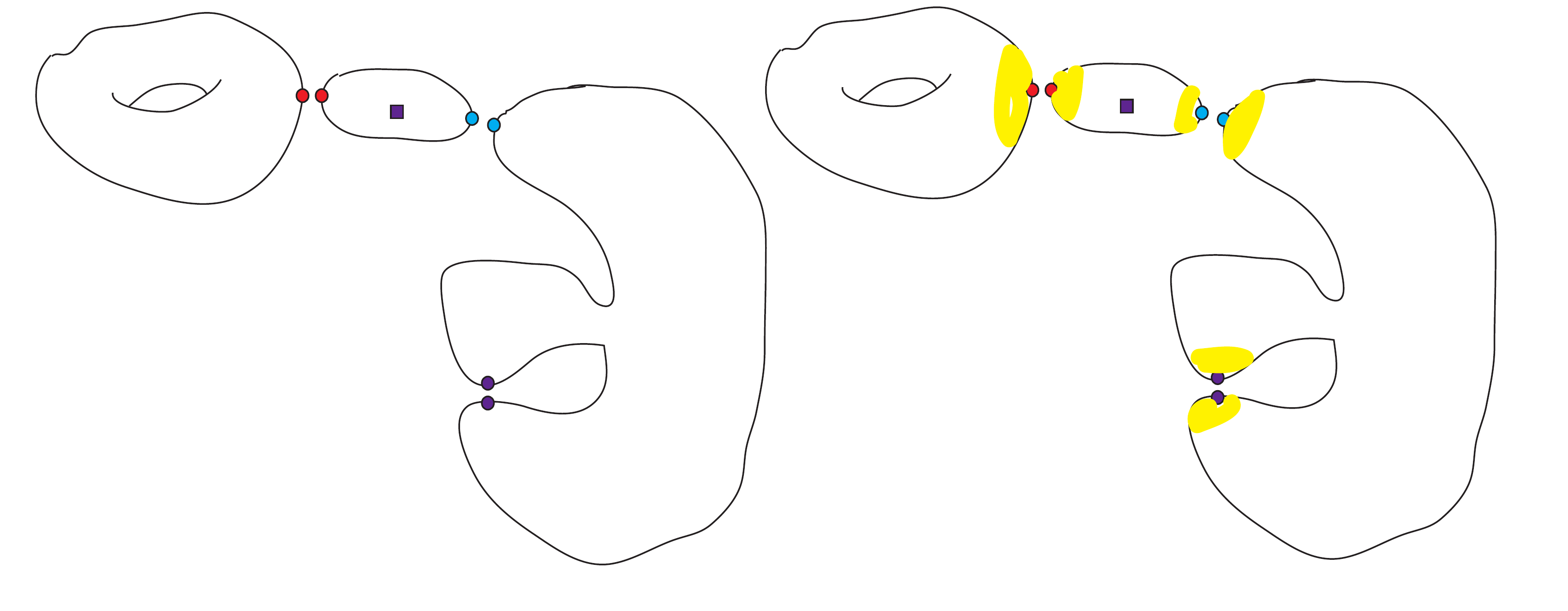}
\end{center}
\caption{A stable Riemann surface and adding a small disk structure.}\label{FIG 01}
\end{figure}

We need to introduce the complex vector space  $H^1(\tau)$ obtained as follows.
Recall that the Cauchy-Riemann operator can be defined on the complex vector space $\Gamma_0(\tau)$ consisting of smooth section $u$ of $TS\rightarrow S$ which vanish at the nodal points 
and the points in $M$, so that $\bar{\partial}(u)$ is a  $TS$-valued $(0,1)$-form, i.e. $\bar{\partial}(u)(z):T_zS\rightarrow T_zS$ is complex anti-linear for every $z$. Denote by $\Omega^{0,1}(\tau)$
the complex vector space of such $(0,1)$-forms. Then $\bar{\partial}:\Gamma_0(\tau)\rightarrow \Omega^{0,1}(\tau)$ is a complex linear injective operator with finite dimensional cokernel.
We define $H^1(\tau) =\Omega^{0,1}(\tau)/ \bar{\partial}(\Gamma_0(\tau))$. This is a complex vector space of dimension $3g_a-3+\sharp M-\sharp D$, where $g_a$ is the arithmetic genus defined by $g_a=1+\sharp D +\sum_{C}[g(C)-1]$.
\begin{definition}
Let $\tau$ be an object in ${\mathcal R}$ with automorphism group $G$ and assume that a small disk structure ${\bm{D}}$ has been fixed. A {\bf good deformation}\index{good deformation} $\mathfrak{j}$  for $\tau$ with given ${\bm{D}}$ consists of a $G$-invariant open neighborhood
$V$ of $0\in H^1(\tau)$ and a smooth family $V\ni v\rightarrow j(v)$ of almost complex structures on $S$
so that the following holds.
\begin{itemize}
\item[(1)] $ j(0)=j$.
\item[(2)] $j(v)=j$ on all $D_x$ for $x\in |D|$ and $v\in V$.
\item[(3)] $[Dj(v)]:H^1(\tau)\rightarrow H^1(\tau_v)$ is a complex linear isomorphism for every $v\in V$.
\item[(4)] For every $g\in G$ and $v\in V$ the map $g:\tau_v\rightarrow \tau_{g\ast v}$ is biholomorphic.
\end{itemize}
\qed
\end{definition}
It is a known fact that for given $\tau$ and small disk structure a good deformation always exists.
\subsection{Uniformizers for ${\mathcal R}$}
Fix a good deformation $\mathfrak{j}$ for $\tau$ and ${\bm{D}}$. Then we can define for a natural gluing parameter $\mathfrak{a}\in {\mathbb B}_{\tau}$ and $v\in V$ the objects $\tau_{v,\mathfrak{a}}$  by
$$
\tau_{v,\mathfrak{a}} =(S_\mathfrak{a},j(v)_\mathfrak{a},M_\mathfrak{a},D_\mathfrak{a}).
$$
On $V\times {\mathbb B}_\tau$ we have the natural action of $G$ by diffeomorphisms
defining us the translation groupoid $G\ltimes (V\times {\mathbb B}_\tau)$. We also obtain a functor
\begin{eqnarray}\label{DM-ppsi}
\Psi:G\ltimes (V\times {\mathbb B}_\tau)\rightarrow {\mathcal R}
\end{eqnarray}
which on objects maps 
$$
(v,\mathfrak{a})\rightarrow \tau_{v,\mathfrak{a}}
$$  
and on  morphisms
 $$
 (g,(v,\mathfrak{a}))\rightarrow (\tau_{v,\mathfrak{a}},g_\mathfrak{a},\tau_{g\ast v,g\ast\mathfrak{a}}).
 $$
The following theorem is well-known and a consequence of standard DM-theory and holds for every gluing profile.

\begin{thm}\label{DM_MAIN}
The orbit space $|{\mathcal R}|$ has a natural metrizable topology for which every connected component is compact. 
Moreover, the following holds, which also characterizes the topology.
Let $\tau$, ${\bm{D}}$, and $\mathfrak{j}$ as described above and let $ \Psi$ be the associated functor defined in {\em (\ref{DM-ppsi})}.
Then there exists a $G$-invariant open neighborhood $O$ of $(0,0)$ in $H^1(\tau)\times {\mathbb B}_\tau$
such that the following holds.
\begin{itemize}
\item[(1)] $\Psi:G\ltimes O\rightarrow {\mathcal R}$ is fully faithful and injective on objects, and $\Psi(0,0)=\tau$.
\item[(2)] The map $|\Psi|: |G\ltimes O|\rightarrow |{\mathcal R}|$ induced between orbit spaces
defines a homeomorphism onto an open neighborhood of $|\tau|$.
\item[(3)] For every $(v,\mathfrak{a})$ the Kodaira-Spencer differential associated to $\tau_{v,\mathfrak{a}}$
is an isomorphism.
\end{itemize}
\qed
\end{thm}
We note that on objects $\Psi(0,0)=\tau$ and on morphisms $\Psi(g,(0,0))=(\tau,g,\tau)$. This is the point where 
the current topic connects with the topic discussed in Lecture 1.
As described in the  polyfold theory for groupoidal category,  a method to explore the richness
of such a category is the construction of uniformizers. The  functors $\Psi$ we just introduced give a  natural uniformizer construction for ${\mathcal R}$. 
We just mention  there are also natural constructions of a similar  kind for the category of stable maps in Gromov-Witten theory, \cite{H2014},  and, which is the topic of this text, also for SFT.
In view of Theorem \ref{DM_MAIN} we can make the following definition.
\begin{definition}
Let $\tau$ be an object in ${\mathcal R}$ with automorphism group $G$.
We say that 
$$
\Psi:G\ltimes O\rightarrow {\mathcal R}
$$
 is a {\bf good uniformizer} \index{good uniformizer} associated to $\tau$ if
it is obtained, after a choice of small disk structure ${\bm{D}}$ and deformation $\mathfrak{j}$, 
as a restriction satisfying the properties (1)--(3). Given $\tau$ there is a set of choices we can make, resulting 
in a set of associated uniformizers denoted by $F(\tau)$. 
\qed
\end{definition}
If $\Phi:\tau\rightarrow \tau'$ is a morphism it is clear that there is a 1-1 correspondence between the 
choices for $\tau$ and $\tau'$, respectively. Hence we obtain a natural bijection $F(\Phi):F(\tau)\rightarrow F(\tau')$. This defines a functor
$$
F_{\text{DM}}:{\mathcal R}\rightarrow \text{SET}
$$
associating to an object $\tau$ a set of good uniformizers having $\tau$ in the image.
As was already said the above construction works for every gluing profile. However, if we would like 
to have a smooth transition between two uniformizers we have to be careful in the choice of the gluing profile.
In fact the gluing profile $r\rightarrow -\frac{1}{2\pi}\cdot \ln(r)$ gives a construction which is equivalent to the classical DM-theory. However, this gluing profile is not as useful  for SFT. If we use the exponential gluing profile $\varphi: r\rightarrow e^{\frac{1}{r}}-e$ we loose the holomorphicity, but otherwise the associated DM-theory has the same features and the methods extend to SFT. These details are described in \cite{HWZ-DM}.
The following is a very important result.
\begin{thm}
Let $\varphi$ be the exponential gluing profile.  Associated to $({\mathcal R},{\mathcal T},F^{\varphi}_{\text{DM}})$ there is a transition construction $\bm{M}^{\varphi}_{\text{DM}}$, where each $\bm{M}^{\varphi}_{\text{DM}}(\Psi,\Psi')$ is equipped
with a smooth manifold structure.
\qed
\end{thm}

\subsection{More on DM-Theory}
We  shall describe some results in the DM-context which we shall use later. When we discuss in Lecture 3 buildings of stable maps we have to consider their underlying domains.
What we shall obtain is the following data. For each $i\in \{0,...,k\}$ a tuple $\sigma_i:=(\Gamma^-_i,S_i,j_i,M_i,D_i,\Gamma^+_i)$ and for $i\in \{1,...,k\}$ a bijection 
$b_i:\Gamma^+_{i-1}\rightarrow \Gamma^-_i$. The points in $M_i$ are called marked points, the points in $\Gamma^\pm_i$ positive and negative punctures.
In general the $\sigma_i$ are not stable. They come, however, with decompositions $S_i=S^{tc}_i\sqcup S^{ntc}_i$, where 
 $S^{nt}_i$ is a finite union of simply connected closed Riemann surfaces  containing  no point from $M_i$ and containing  precisely one point from $\Gamma^-_i$ and $\Gamma^+_i$.
 However, not every connected component having the latter property belongs to $S^{tc}$. In addition we are given a finite group $G$ acting by biholomorphic maps
 preserving the floors, i.e. $(\phi^g_0,...,\phi^g_k)$ such that for $z\in \Gamma_{i-1}^+$ we have that $b_i\circ \phi_{i-1}(z)=\phi_i\circ b_i(z)$.
In addition $\phi_i(S_i^{tc})=S^{tc}_i$ and $\phi_i(S_i^{ntc})=S^{ntc}_i$.  Given this data we shall add so-called {\bf stabilization points}. As we shall see later on they have to be picked carefully
and are also associated to other data, transversal constraints, which will be introduced later.  The set of stabilization points $\Xi$  is assumed to be invariant under the $G$-action.
After adding the stabilization points it is assumed that each domain component together with all the special points on it becomes stable.
Consider $\bar{\sigma}$ which consists of the Riemann surface $(S,j)$, the marked points $\bar{M}=\Gamma^-_0\sqcup\Gamma^+_k\sqcup M\sqcup\Xi$, where $\Xi$ are the stabilization points. 
We denote by $\bar{D}$ the union of the $D_i$ and all $\{(z,b_i(z)\}$, where $i\in \{1,...,k\}$ and $z\in \Gamma^+_{i-1}$. Then $\bar{\sigma}=(S,j,\bar{M},\bar{D})$ is a stable Riemann surface
with automorphism group $G^\ast$ containing $G$. We shall refer to $\bar{\sigma}$ as the associated \textbf{DM-Data}.

\begin{figure}[h]
\begin{center}
\includegraphics[width=7.0cm]{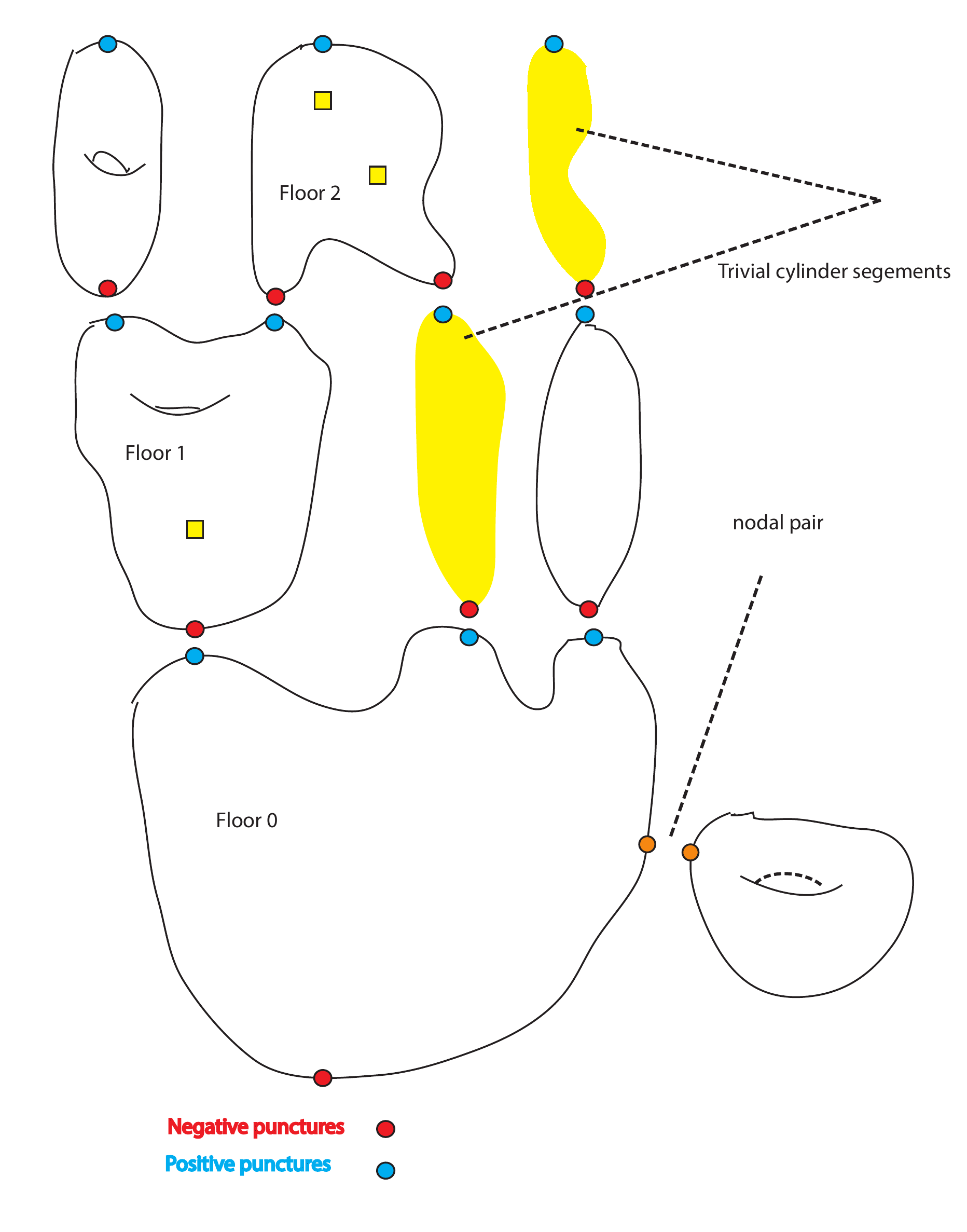} \includegraphics[width=7.0cm]{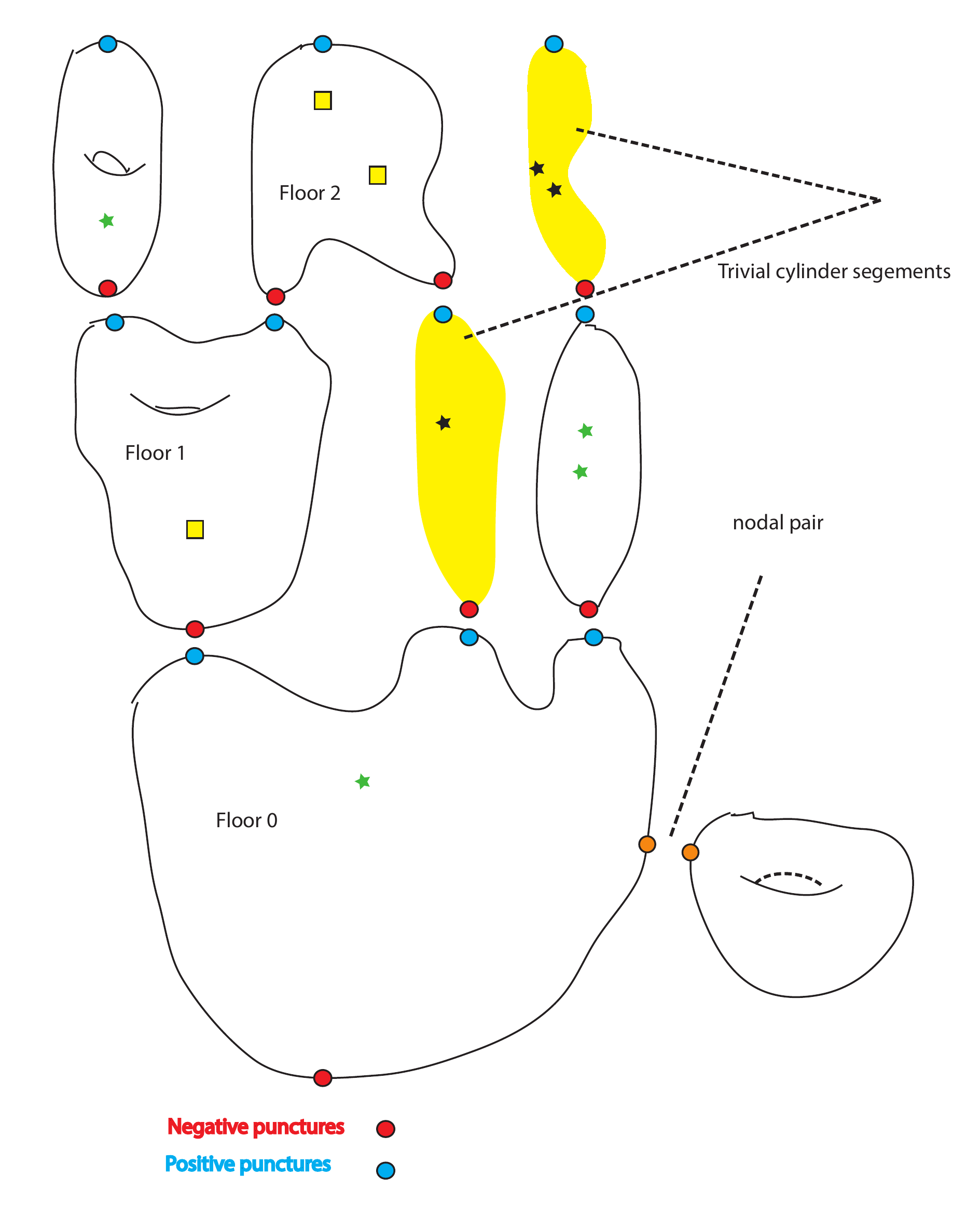}
\end{center}
\caption{A  Riemann surface building with three floors with and without stabilization points added.}\label{FIG 00123}
\end{figure}

There is  a natural action of $G$ on $H^1(\bar{\sigma})$. This finite-dimensional complex vector space
splits naturally according to the partition of $S$ into
$$
S=(S_0^{tc}\cup S_0^{ntc})\cup...\cup (S_{k}^{tc}\cup S_{k}^{ntc})
$$
and the action by $G$ preserves this splitting.  This is generally not the case for the $G^\ast$-action.
We obtain the 
natural identification
\begin{eqnarray}\label{DM-COMP}
&H^1(\bar{\sigma})\equiv H^1(\bar{\sigma}_0^{tc})\oplus H^1(\bar{\sigma}_0^{ntc})\oplus..\oplus H^1(\bar{\sigma}_{k}^{tc})\oplus H^1(\bar{\sigma}_{k}^{ntc}):&\\
&v\equiv(v_0^{tc},v_0^{ntc},..., v_{k}^{tc},v_{k}^{ntc}).&\nonumber
\end{eqnarray}
Here $\bar{\sigma}_i^{tc}$ is $S_i^{tc}$ equipped with the special points coming from $\bar{M}$ and $|\bar{D}|$, and similarly for
$\sigma_{i}^{ntc}$. We have a natural action of $G^\ast$ on $H^1(\bar{\sigma}_0^{tc})\oplus H^1(\bar{\sigma}_0^{ntc})\oplus..\oplus H^1(\bar{\sigma}_{k}^{tc})\oplus H^1(\bar{\sigma}_{k}^{ntc})$ in view of (\ref{DM-COMP}) whose restriction to $G$ has a diagonal form with respect to the indicated decomposition.
\begin{definition}\label{GOODDM}
We assume we started with $\sigma=(\sigma_0,b_1,...,b_k,\sigma_k)$, were $\sigma_i$ and $b_i$ are as described before equipped
with the action of a finite group $G$ by biholomorphic maps preserving the floors $\sigma_i$ and the other data.
After fixing a $G$-invariant stabilization set $\Xi$ we denote the associated DM-data by $\bar{\sigma}$. Assume we have picked  the small disk structure $\bm{D}$ so that the union of disks is invariant under $G^\ast$, where $G^\ast$ and $G$  act in the way
as just described, a {\bf good $\tau$-adapted 
deformation}\index{good $\alpha$-adapted deformation} of $j$ consists of an  open $G^\ast$-invariant neighborhood ${\mathcal V}$ of $0\in
H^1(\bar{\sigma}_0^{tc})\oplus H^1(\bar{\sigma}_0^{ntc})\oplus..\oplus H^1(\bar{\sigma}_{k}^{tc})\oplus H^1(\bar{\sigma}_{k}^{ntc})$
and a family $\mathfrak{j}:{\mathcal V}\ni v\rightarrow j(v)$  of almost complex structures on $\bar{\sigma}$ so that the following properties hold.
\begin{itemize}
\item[(1)] For $v\in {\mathcal V}$ it holds that $j(v)|S_{i}^{tc}=j_{S_{i}^{tc}}(v_i^{tc})$ and
$j(v)|S_{i}^{ntc}=j_{S_{i}^{ntc}}(v_i^{ntc})$, where $v$ decomposes into the direct sum of the $v_i^{tc}$ and $v_i^{ntc}$.
\item[(2)]  For all $g\in G^\ast$ the map $g:(S,j(v),\bar{M}, \bar{D})\rightarrow (S,j(g\ast v),\bar{M},\bar{D})$
is biholomorphic.
\item[(3)] $j(0)=j$ and  for $v\in {\mathcal V}$ it holds $j(v)=j$ on the disks
of the  small disk structure $\bm{D}$.
\item[(4)] For every $v\in {\mathcal V}$ the Kodaira-Spencer differential $[Dj(v)]:H^1(\bar{\sigma})\rightarrow H^1(\bar{\sigma}_v)$
is a complex linear isomorphism, where 
$$
\bar{\sigma}_v=(S,j(v),\bar{M},\bar{D}).
$$
 Moreover the map
${\mathcal V}\ni v\rightarrow j(v)$ is injective. 
\item[(5)] The map $v\rightarrow \bar{\sigma}_v$ defines a good uniformizer for the DM-spaces of fixed nodal type.
More precisely, denoting by $\tau$ the nodal type of $\bar{\sigma}$
$$
 \psi: G^\ast\ltimes {\mathcal V}\rightarrow {\mathcal R}_{\tau},
 $$
 where on objects $\psi(v)=\bar{\sigma}_v$ and on morphisms 
 $$
 \psi(g,v)=(\sigma_v,g,\bar{\sigma}_{g\ast v}),
 $$
  is a good uniformizer.
Here  ${\mathcal R}_{\tau}$ has  its orbit space
equipped with the topology coming from the one on the orbit space of ${\mathcal R}$. 
\end{itemize}
\qed
\end{definition}
The {\bf nodal type}\index{nodal type} of a stable Riemann surface in ${\mathcal R}$ is obtained as follows.
Starting  with $\bar{\sigma}=(S,j,\bar{M},\bar{D})$ we first produce a decorated graph where we take a vertex $v=v(C)$ for every 
domain component $C$ of $S$ which we label with its genus $g(v):=g(C)$ and the number of marked points $m(v):=m(C)$ on $C$.
For every nodal pair $\{x,y\}\in\bar{D}$ we have that $x\in C$ and $y\in C'$ (where $C'=C$ is a possibility), and we draw
an edge connecting the vertices associated to $C$ and $C'$. The stability of $\bar{\sigma}$ is equivalent to the statement 
that for every vertex $v$ the {\bf edge number}\index{edge number} $e(v)$ satisfies the inequality
$$
2g(v)+m(v)+e(v)\geq 3.
$$
There is an obvious notion of isomorphism for two such decorated graphs. By definition the nodal type $\tau(\bar{\sigma})$ 
\index{$\tau(\bar{\sigma})$} of $\bar{\sigma}$  is the isomorphism class
of the associated decorated graph.
We need the following result.
\begin{thm}
Given the DM-data $\bar{\sigma}$ associated  and the auxiliary structure,
there exists for given small disk structure $\bm{D}$ as previously described 
a good deformation $\mathfrak{j}$.
\end{thm}
\subsection{Further Concepts}\label{covering}
Finally we would like to introduce the category $\wt{\mathcal R}$.
Its objects are closed stable Riemann surfaces $\wt{\tau}$ with an ordered set of marked points, i.e. $\wt{M}= (m_1,m_2,...,m_{\ell})$. 
Then such a stable Riemann surface takes the form $\wt{\tau}=(S,j,\wt{M},D)$. A morphism $\wt{\Phi}:\wt{\tau}\rightarrow \wt{\tau}'$ is a tuple
$(\wt{\tau},\phi,\wt{\tau}')$ where $\phi:(S,j)\rightarrow (S',j')$ is a biholomorphic map satisfying $\phi_{\ast}D=D'$ and $\phi_\ast\wt{M}=\wt{M}'$.
Here $\phi_{\ast}\wt{M}=(\phi(m_1),...,\phi(m_{\ell}))$. We have a forgetful functor 
$$
\mathsf{f}: \wt{\mathcal R}\rightarrow {\mathcal R}
$$
which on objects maps $\wt{\tau}=(S,j,\wt{M},D)$ to $\tau=(S,j,M,D)$ where $M$ is the set of un-ordered marked points underlying $\wt{M}$.
The morphism $\wt{\Phi}=(\wt{\tau},\phi,\wt{\tau}')$ is mapped to $\Phi=(\tau,\phi,\tau')$. The forgetful functors is an example of a proper covering functor.
and they will occur quite frequently.  Below we give the general definition.
\begin{definition}
 Let $\pi: \wt{\mathcal C}\rightarrow {\mathcal C}$ a functor between groupoidal categories.
 We say $\pi$ is a \textbf{proper covering functor}
provided the following holds.
\begin{itemize}
\item[(1)] $\pi$ is surjective on objects and finite to one.
\item[(2)] $\text{mor}(\wt{\mathcal C})\xrightarrow{\langle\pi,s\rangle} {\mathcal C}{_{s}\times_\pi} \text{obj}(\wt{\mathcal C}): \phi\rightarrow (\pi(\phi),s(\phi))
$ is a bijection.
\end{itemize}
\qed
\end{definition}
We note that $\mathsf{f}: \wt{\mathcal R}\rightarrow {\mathcal R}$ is a proper covering functor. 
Using the uniformizers from the construction $(F,\bm{M})$ we can make the following construction. Let $\tau=(S,j,M,D)$. An ordering of $M$
is a bijection $\mathsf{o}:I_k:=\{1,...,k\}\rightarrow M$, where $k=\sharp M$. We shall write $M^{\mathsf{o}}$ for the tuple $(\mathsf{o}(1),...,\mathsf{o}(k))$.
We observe that $G$ acts on $\text{Bij}(I_k,M)$ by $g\ast\mathsf{o} = g\circ \mathsf{o}$. 
Consider  $\Psi\in F_{DM}(\tau)$,  say $\Psi:G\ltimes O\rightarrow {\mathcal R}$,
defined by $(v,\mathfrak{a})\rightarrow \tau_{v,\mathfrak{a}}=(S_{\mathfrak{a}},j(v)_{\mathfrak{a}},M_{\mathfrak{a}},D_{\mathfrak{a}})$.
We consider the smooth manifold $\text{Bij}(I_k,M)\times O$ with the action of $G$ by 
$$
g\ast(\mathsf{o},(v,\mathfrak{a}))= (g\circ\mathsf{o},g\ast(v,\mathfrak{a})).
$$
We obtain the translation groupoid $G\ltimes (\text{Bij}(I_k,M)\times O)$ and the forgetful functor 
$$
G\ltimes (\text{Bij}(I_k,M)\times O)\rightarrow G\ltimes O
$$
This functor is a proper covering functor which has the additional properties that between objects it is a surjective local diffeomorphism
and the obvious map 
$$
\text{mor}(G\ltimes (\text{Bij}(I_k,M)\times O)\rightarrow \text{mor}(G\ltimes O){_{s}\times_{\mathsf{f}}}\text{obj}(G\ltimes (\text{Bij}(I_k,M)\times O))
$$
is a diffeomorphism. There exists a commutative diagram
$$
\begin{CD}
G\ltimes (\text{Bij}(I_k,M)\times O)@>\wt{\Psi}>> \wt{\mathcal R}\\
@V\mathsf{f}VV  @VVV\\
G\ltimes O@>\Psi>> {\mathcal R}
\end{CD}
$$
where $\wt{\Psi}(\mathsf{o},(v,\mathfrak{a}))=(S_{\mathfrak{a}},j(v)_{\mathfrak{a}},M^{\mathsf{o}}_{\mathfrak{a}},D_{\mathfrak{a}})$ on objects 
and on morphisms 
$$
\wt{\Psi}(g,\mathsf{o},(v,\mathfrak{a}))= (\wt{\Psi}(\mathsf{o},(v,\mathfrak{a})), g_{\mathfrak{a}},\wt{\Psi}(g\circ \mathsf{o},g\ast (v,\mathfrak{a})))
$$
This is an example for a uniformizer construction for a proper covering. This also works for the transition construction.
\newpage

\part*{Lecture 3}

\section{The Example of Stable Maps in SFT}
The reader is assumed to be familiar with the basic concepts about `Stable Hamiltonian Structures', see \cite{CV1} or Appendix \ref{APP11}.
We start with $(Q,\lambda,\omega)$, a non-degenerate stable Hamiltonian structure $(\lambda,\omega)$  on the closed odd-dimensional manifold $Q$.
We obtain the \textbf{Reeb vectorfield} $R$ defined by $\lambda(R)\equiv 1$ and $d\lambda(R,\ .)\equiv 0$.  Associated to $R$ we have the set of periodic orbits 
${\mathcal P}={\mathcal P}(Q,\lambda,\omega)$ and define ${\mathcal P}^\ast=\{\emptyset\}\cup {\mathcal P}$. Given an admissible complex
multiplication $J$  for $\xi=\text{ker}(\lambda)$, there exists a spectral gap function $\bar{\delta}_J:{\mathcal P}^\ast \rightarrow (0,2\pi]$, which comes 
from a self-adjoint operator associated to $J$ and $R$, which plays a role in the Fredholm theory.
We fix an admissible  weight function $\delta_0$ define on ${\mathcal P}^\ast$ such that $0<\delta_0<\bar{\delta}_J$.

\subsection{Stable Maps of Height  One and Regularity $(3,\delta_0)$}
A \textbf{stable map of height one}  is a tuple $\alpha=(\Gamma^-,S,j,D,M,[\widetilde{u}],\Gamma^+)$. Here $(S,j)$ is a closed Riemann surface,
$D$ a set of nodal pairs, $M$ unordered marked points,  $\Gamma^\pm$ positive and negative punctures, respectively.
By  $[\widetilde{u}]$ we denote an  equivalence class of maps $\widetilde{u}:S\setminus(\Gamma^+\cup\Gamma^-)\rightarrow {\mathbb R}\times Q$
of class $(3,\delta_0)$. Two maps are {\bf equivalent} provided they differ by a constant ${\mathbb R}$-shift.
At every positive puncture $z\in \Gamma^+$ the map $\widetilde{u}$ is asymptotic in the $(3,\delta_0)$ sense to a periodic orbit $\gamma_z$,
and similarly at $z\in \Gamma^-$. The map $\widetilde{u}$ is continuous over nodal pairs.  
A {\bf trivial cylinder component} is a connected component $C$ of $S$ with exactly one positive and negative puncture asymptotic to the same orbit and no marked points and nodal points. Moreover it is homotopic, while being asymptotic to the periodic orbits, to a standard cylinder parameterization. 
In addition the following stability condition holds.
\begin{definition}
We say $\alpha=  (\Gamma^-,S,j,D,M,[\widetilde{u}],\Gamma^+)$ is {\bf stable} provided   $S$ has at least one connected component $C$, which is not a trivial cylinder component. This component has  at least one of the following properties, where 
we write $\widetilde{u}=(a,u)$.
\begin{itemize}
\item[(1)] $\int_{\dot{C}}u^\ast\omega>0$, $\dot{C}=C\setminus(\Gamma^+\cup\Gamma^-)$.
\item[(2)] $2g(C) + \sharp (C\cap (M\cup |D|))\geq 3$.
\end{itemize}
\qed
\end{definition}
 \begin{figure}[h]
\begin{center}
\includegraphics[width=6.5cm]{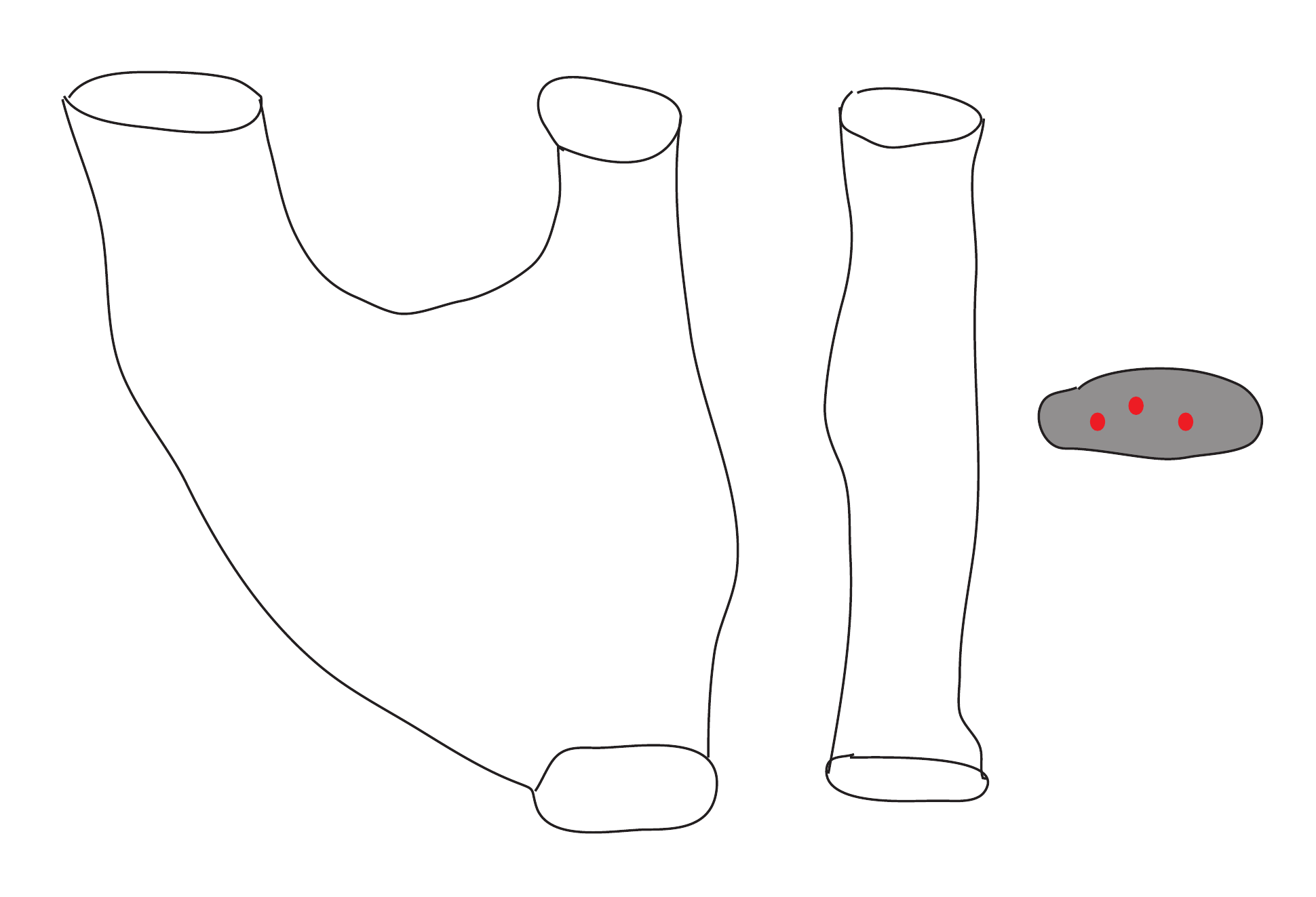}
\end{center}
\caption{A stable building of height one. One trivial cylinder component.}\label{FIG 122}
\end{figure}
\subsection{Stable Buildings}
Given a stable map of height $1$ we can take at positive punctures the real projectification $P_{\mathbb R}(T_z,j)$ 
and at  negative punctures we  take $P_{\mathbb R}(T_z,-j)$.  We obtain $ \widehat{\Gamma}^\pm\rightarrow \Gamma^\pm$  which is  a principle $S^1$-bundle naturally. The stable buldings of height $k+1$ have the following form
$$
\alpha=(\alpha_0,\widehat{b}_1,\alpha_1,...,\widehat{b}_k,\alpha_k)
$$
Here each $\alpha_i$ is a stable map of height $1$ and $\widehat{b}_i$ fits into the commutative diagram
$$
\begin{CD}
\widehat{\Gamma}_{i-1}^+@>\widehat{b}_i>> \widehat{\Gamma}^-_i\\
@VVV @VVV\\
\Gamma^+_{i-1} @>b_i>>  \Gamma^-_{i}
\end{CD}
$$
for $i\in \{1,...,k\}$.
The map $b_i$ is a bijection and the periodic orbits associated to $z\in\Gamma^+_{i-1}$ and $b_i(z)$ coincide.
Further the asymptotic limits are $\widehat{b}_i(z)$ matching. With $\alpha=(\alpha_0,\wh{b}_1,...,\wh{b}_k,\alpha_k)$
and $\alpha_i=(\Gamma^-_i,S_i,j_i,M_i,[\wt{u}_i],\Gamma^+_i)$ we sometimes collect all the $[\wt{u}_i]$ as $[\wt{u}]$
and the interface nodal pairs as $\Theta$ consisting of all $\wh{b}_i(z)$ for $z\in \Gamma^+_{i-1}$ and $i\in \{1,...,k\}$.
With other words $[\wt{u}]=([\wt{u}_0],\wh{b}_1,..,\wh{b}_k,[\wt{u}_k])$.
Alternatively we can identify $\wh{b}_i(z)$
with a decorated node $[\wh{z},\wh{z}']$, which is an equivalence class of oriented real lines in the tangent space $T_zS_{i-1}$ and $T_{b_i(z)}S_{i}$.
Moreover, we put $S=S_0\sqcup..\sqcup S_k$,
$D=D_0\sqcup..\sqcup D_k$, and $M=M_0\sqcup..\sqcup M_k$,
Then we can write $\alpha$ as
$$
\alpha=(\Gamma^-_0,S,j,M,D,\Theta,[\wt{u}],\Gamma^+_k),
$$
which sometimes is the more convenient notation.
\subsection{The Category of Stable Maps}
The category of stable maps of class $(3,\delta_0)$ denoted by ${\mathcal S}={\mathcal S}^{3,\delta_0}(Q,\lambda,\omega)$
has the stable buildings of arbitrary height as objects. A morphism $\Phi:\alpha\rightarrow \alpha'$ is given by 
$(\alpha,\phi,\alpha')$, where $\phi:(\Gamma^-_0,S,j,M,D,\Theta,\Gamma^+_k)\rightarrow (\Gamma'^-_0,S',j',M',\Theta',\Gamma'^+_k)$
such that $[\widetilde{u}'_i\circ\phi_i]=[\widetilde{u}_i]$.
An already nontrivial result is given by the following theorem.
\begin{thm}
${\mathcal S}$ is a groupoidal category and its orbit space has a natural metrizable topology ${\mathcal T}$ and 
$({\mathcal S},{\mathcal T})$ is a GCT.
\qed
\end{thm}
Given $J:\xi\rightarrow \xi$ we have the functor 
$$
\mu_J:{\mathcal S}\rightarrow \text{\bf Ban},
$$
which  associates to $\alpha$ the Hilbert space of $(T({\mathbb R}\times Q),\widetilde{J})$-valued $(0,1)$-forms  of class $(2,\delta_0)$ along the underlying  stable map.
We obtain a new groupoidal category ${\mathcal E}_J$, where the objects are pairs $(\alpha,e)$, $e\in \mu_J(\alpha)$, and morphisms
are $(\Phi,e)$, with $\Phi$ a morphism in ${\mathcal S}$ satisfying  $s(\Phi)=\alpha$, and where we view the morphism as 
$$
(\Phi,e):(\alpha,e)\rightarrow (t(\Phi),\mu(\Phi)(e)).
$$
It turns out that the orbit space of ${\mathcal E}_J$ has a natural metrizable topology, so that $({\mathcal E}_J,{\mathcal T})$ is a GCT.
We obtain 
$$
P:{\mathcal E}_J\rightarrow {\mathcal S}
$$
and can define  the section functor $\bar{\partial}_{\widetilde{J}}$ by 
$$
\bar{\partial}_{\widetilde{J}}(\alpha)=\left(\alpha,\frac{1}{2}[T\tilde{u}+\widetilde{J}\circ T\widetilde{u}\circ j]\right).
$$
\begin{definition}
The \textbf{(coarse) moduli space} $\overline{\mathcal M}_J$  associated to $\bar{\partial}_{\widetilde{J}}$ is the orbit space 
of the full subcategory of $\widetilde{J}$-holomorphic objects, i.e. those $\alpha$ for which 
$$
\bar{\partial}_{\widetilde{J}}(\alpha)=(\alpha,0).
$$
The \textbf{moduli category} is the full subcategory associated to $\wt{J}$-holomorphic objects.
\qed
\end{definition}
The following holds.
\begin{thm}
After fixing a suitable gluing profile $\varphi$ and a strictly increasing sequence $\delta$ of admissible weight functions,
 $P$ has naturally the structure of a strong bundle over the tame polyfold ${\mathcal S}$ and 
$\bar{\partial}_{\widetilde{J}}$ is Fredholm.  The orbit space $|\bar{\partial}_{\widetilde{J}}^{-1}(0)|$ intersected with every connected component
of $|{\mathcal S}|$ is compact.
\qed
\end{thm}
The theorem means that after fixing $\delta$ and $\varphi$ we have a natural uniformizer construction  $\bar{F}:{\mathcal S}\rightarrow \text{SET}$  producing tame strong bundle uniformizers 
$$
\begin{CD}
G\ltimes K @>\bar{\Psi}>>  {\mathcal E}_J\\
@V p VV @V P VV\\
G\ltimes O   @>\Psi>>   {\mathcal S}
\end{CD}
$$
and a transition construction which produces tame strong bundles
$$
\bm{M}(\bar{\Psi},\bar{\Psi}')\rightarrow \bm{M}(\Psi,\Psi').
$$
The sc-Fredholm property of $\bar{\partial}_{\widetilde{J}}$ means that the local representatives of the latter with respect to
a $\bar{\Psi}$
$$
\bar{\partial}_{\wt{J},\bar{\Psi}} := \bar{\Psi}^{-1}\circ\bar{\partial}_{\wt{J}}\circ\Psi:  O\rightarrow K
$$
are sc-Fredholm.  
We have a grading functor $d:{\mathcal S}\rightarrow {\mathbb N}=\{0,1,...\}$ associating to $\alpha=(\alpha_0,\widehat{b}_1,..,\widehat{b}_k,\alpha_k)$
its top floor number which is $k$, i.e. the number of floors minus one. Having the tame  polyfold structure we pick for $\alpha$ a uniformizer
$\Psi\in F(\alpha)$ and $\bar{o}\in O$ with $\Psi(\bar{o})=\alpha$. Then we can take $d_O(\bar{o})$ which does not depend on the choice of $\Psi$
and the polyfold structure has the property that $d(\alpha)=d_O(\bar{o})$. This map descends 
to the orbit space $Z=|{\mathcal S}|$, giving us the \textbf{degeneracy index} $d_Z:Z\rightarrow {\mathbb N}$.
\begin{prop}
Every $z\in Z$ has an open neighborhood $U(z)$ such that $d_Z|U(z)\leq d_Z(z)$.
\qed
\end{prop}

\begin{remark}
We can also consider $\wt{\mathcal S}$ which consists of stable buildings where the marked points are ordered, the negative punctures on floor $0$ are ordered
and the positive punctures on the top floor are numbered as well. Then we have a forgetful functor $\wt{\mathcal S}\rightarrow {\mathcal S}$ which is also
a proper covering functor, similarly as in the case of stable Riemann surfaces.
\qed
\end{remark}

\subsection{Basic Topological Structure}
The orbit space $Z:=|{\mathcal S}|$ is metrizable and we can consider $\pi_0(Z)$. Given $a\in \pi_0(Z)$
we can represent it by a building of height $1$ with no nodal points, say $\alpha=(\Gamma^-,S,j,M,\emptyset,[\widetilde{u}],\Gamma^+)$.
The compact space $S$ might  have different components and can be written as $S= S^{ntc}\sqcup S^{tc}$ splitting it into trivial cylinder components
and the rest. We can consider $\pi_0(S)$ and $\pi^{\text{ntriv}}_0(S): = \pi_0(S^{ntc})$. We make the following rough classification.
\begin{definition}
We say
\begin{itemize}
\item[$\bullet$]  $a\in \pi_0(Z)$ is a  \textbf{parent} provided $\sharp\pi_0(S)=1$.
\item[$\bullet$] $a\in \pi_a(Z)$ is a \textbf{descendent} provided $\sharp\pi^{\text{ntriv}}_0(S)=1$ and $\sharp\pi_0(S)>1$.
\item[$\bullet$] $a\in \pi_a(Z)$ is a \textbf{disjoint union } provided $\sharp\pi^{\text{ntriv}}_0(S)\geq 2$. 
\end{itemize}
Within the disjoint union we can also distinguish \textbf{parent classes}  $a$ which are characterized by 
$2\leq  \sharp\pi^{\text{ntriv}}_0(S)=\sharp\pi_0(S)$ and their \textbf{descendents} defined by $2\leq  \sharp\pi^{\text{ntriv}}_0(S)<\sharp\pi_0(S)$.
\qed
\end{definition}
 \begin{figure}[h]
\begin{center}
\includegraphics[width=9.5cm]{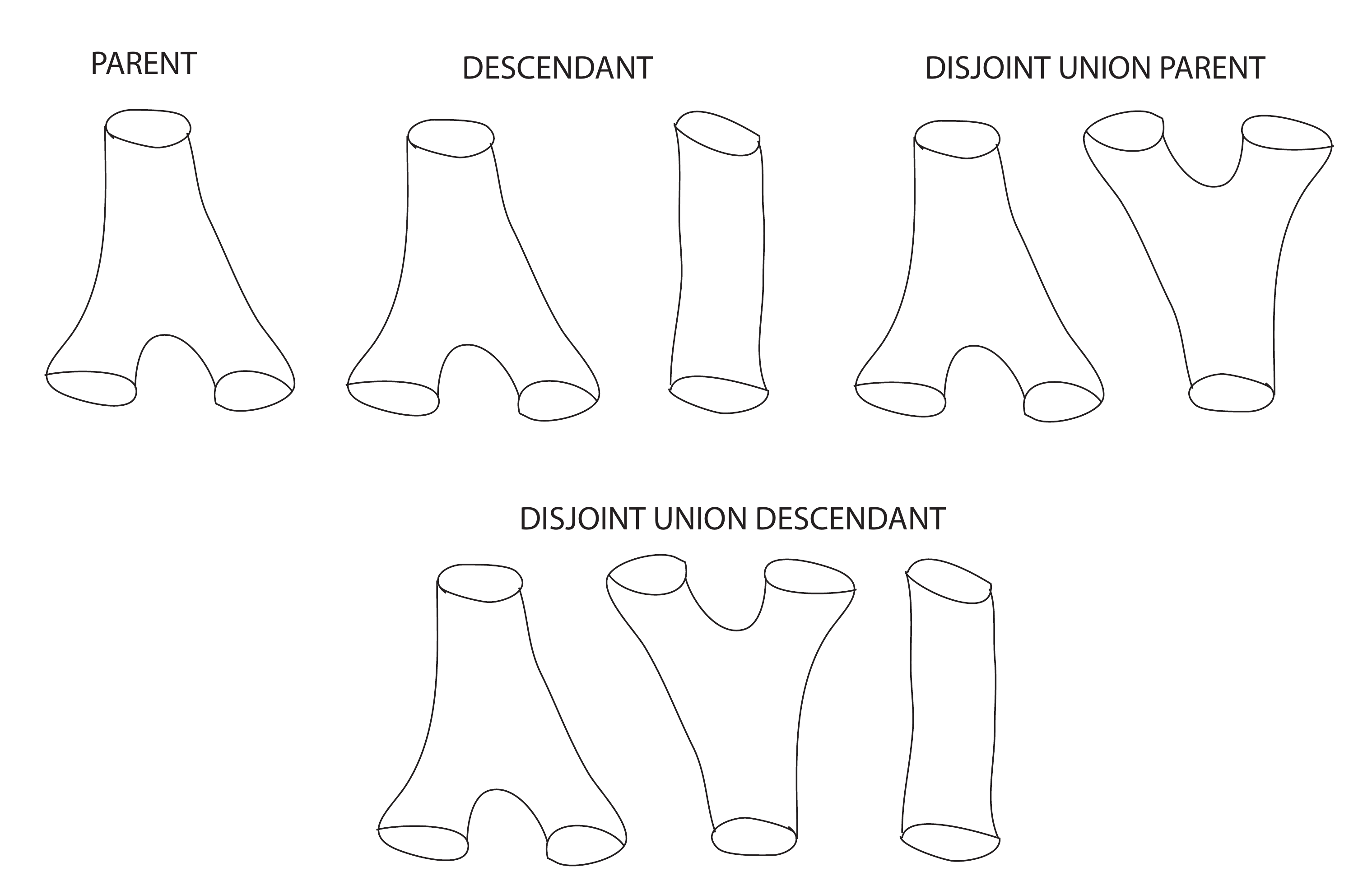}
\end{center}
\caption{Classification of classes in $\pi_0(Z)$.}\label{FIG 1122}
\end{figure}
Hence we have a splitting 
$$
\pi_0(Z) =\pi_0^p(Z)\sqcup \pi_0^d(Z)\sqcup \pi_0^u(Z)
$$
and a further splitting 
$$
\pi_0^u(Z)=\pi^{up}_0(Z)\sqcup \pi^{ud}_0(Z).
$$
These distinctions are important for the perturbation theory and its inductive treatment. 
\begin{definition}
Given $a\in \pi_0(Z)$ we denote by ${\mathcal S}_a$ the full subcategory of ${\mathcal S}$
associated to the objects $\alpha$ with $|\alpha|\in a$.
\qed
\end{definition}
A very important concept is that of a face.
\begin{definition}
A \textbf{face} of $Z$ is the closure $\theta$ of a connected component $\dot{\theta}$ in $\{z\in Z\ |\ d_Z(z)=1\}$.
\qed
\end{definition}
More about this later.
\begin{definition}
Let $a\in \pi_0(Z)$ and define $\text{face}_a$ to be the collection of faces contained in $a$.
\qed
\end{definition}
\begin{definition}
For a face $\theta$ we denote by ${\mathcal S}_\theta$ the full subcategory associated to objects
with isomorphism class in $\theta$.
\qed
\end{definition}
An important property is the following, which we shall refer to as the polyfold structure on ${\mathcal S}$ being \textbf{face-structured}. 
\begin{thm}
Every object $\alpha$ belongs to precisely $d(\alpha)$-many different ${\mathcal S}_\theta$, i.e.
there exist precisely $d(\alpha)$ many $\theta$ with $|\alpha|\in\theta$.
\qed
\end{thm}
A face $a$ determines $(a',a'')$ as follows. Take a representative $\alpha$, $|\alpha|\in a$,  with $d(\alpha)=1$
written as $\alpha=(\alpha',\widehat{b},\alpha'')$. Then $(a',a'')$ is determined by $|\alpha'|\in a'$ and $|{\alpha}''|\in a''$.
Given $a\in \pi_0^{up}(Z)$ take a representative $\alpha$ with $d(\alpha)=0$.  Then the components 
$\mathfrak{c}\in \pi^{\text{ntriv}}_0(\bar{S})$ determine $\alpha_{\mathfrak{c}}$ which determine $a_{\mathfrak{c}}$.
\subsection{Grading and Organization}
For the inductive procedures of perturbation theory later on we need a grading which is adjusted to the problem at hand.  In case of a stable Hamiltonian structure the organizational issues are somewhat more complicated 
than in the case of a contact structure. A organizational scheme for the the more general stable Hamiltonian case is based on the use of of the following map $a\rightarrow d^J_a$.
\begin{definition}
Define $d^J:\pi_0(Z)\rightarrow \{-1,0,1,2,..\}:a\rightarrow d^J_a$ by 
$$
d^J(a)=\left[
\begin{array}{cc}
-1 &  \text{if} \ a\cap \overline{\mathcal M}_J=\emptyset.\\
\max\{d_Z(z)\ |\ z\in a\cap \overline{\mathcal M}_J\}&\text{if}\ a\cap \overline{\mathcal M}_J\neq \emptyset.
\end{array}\right.
$$
\qed
\end{definition}
This organizes the elements in $\pi_0(Z)$. The following is very important.
\begin{thm}
\begin{itemize}
\item[(1)] If  $\bar{a}\in \pi^p_0(Z)$ and $a\in \pi^d_0(Z)$ is a descendent then $d_{\bar{a}}^J=d^J_a$.
\item[(2)] If $\theta\in \text{face}_a$ for some $a\in \pi_0(Z)$ and $\theta$ comes from $(a',a'')$, then 
$d^J_a\geq 1+ d^J_{a'} +d^J_{a'}$.
\item[(3)] For $a\in \pi_0^{up}(Z)$ we have the equality  $(d^J_a +1)=\sum_{\mathfrak{c}\in \pi_0(\bar{S})}  (d^J_{a_{\mathfrak{c}}}+1)$.
\item[(4)] If $\bar{a}\in \pi_0^{up}$ and $a\in \pi_0^{ud}(Z)$ is a descendent then $d^J_{\bar{a}}=d^J_a$.
\end{itemize}
\qed
\end{thm}
The induction usually goes in the following order.
\begin{eqnarray*}
\begin{array} {cccccccc}
& \pi^{p,(-1)}_0(Z)   &\Longrightarrow&  \pi^{d,(-1)}_0(Z) &\Longrightarrow&   \pi^{up,(-1)}_0(Z) &\Longrightarrow&   \pi^{ud,(-1)}_0(Z) \\
\Longrightarrow & \pi^{p,(0)}_0(Z)   &\Longrightarrow&  \pi^{d,(0)}_0(Z) &\Longrightarrow&   \pi^{up,(0)}_0(Z) &\Longrightarrow&   \pi^{ud,(0)}_0(Z) \\
\Longrightarrow&...&...&...&...\\
\Longrightarrow&\pi^{p,(\ell)}_0(Z)   &\Longrightarrow&  \pi^{d,(\ell)}_0(Z) &\Longrightarrow&   \pi^{up,(\ell)}_0(Z) &\Longrightarrow&   \pi^{ud,(\ell)}_0(Z)\\
\Longrightarrow&...&...&...&...
\end{array}
\end{eqnarray*}
As we shall see the only freedom to pick data is for the $a\in \pi^{p,\ell}_0(Z)$ subject to some boundary compatibility, i.e. faces,
which have data determined by the previous steps,  i.e. associated to $\pi_0^{\leq \ell-1}(Z)$. The amount 
of construction at each step using $d^J$ is more involved compared to another possibility which arises
dealing with a contact form.  In case $(Q,\lambda,d\lambda)$ is a manifold with a non-degenerate contact form 
one can take another organizing princple, which has in some sense more steps, which, however is less demanding on 
the level of constructing the multi-sections.  Both lead to the same result for contact manifolds, but the latter cannot be 
used for stable Hamiltonian structures directly. The problem is that it might happen that $d_Z:Z\rightarrow [0,\infty]$
is unbounded on many $a\in \pi_0(Z)$. When we describe the inductive scheme we shall restrict ourselves
to the case of contact forms. In the expanded version of this lecture note we shall describe the more general case
as well.

\begin{definition}
The subset $\wh{Z}$ of $Z$ consists of all $|\alpha|$, so that all its trivial cylinder parts are $\wt{J}$-holomorphic. 
\qed
\end{definition}
We note that the definition does not depend on the choice of $J$. We observe that the interiors of 
$\wh{Z}$ and $Z$ are the same. 
We shall introduce later on a topology $\wh{\mathcal T}$ on $\wh{Z}$, called the strong topology, which is finer than ${\mathcal T}|\wh{Z}$
and which will play a important role in the perturbation and compactness theory and we shall explain why it is important later on. 
We note  at this point that $\overline{\mathcal M}_J\subset \wh{Z}$. The perturbation will be done in such a way that 
the modified coarse moduli space still belongs to $\wh{Z}$. 
That we  introduce this topology has to do with some subtleties which we explain now.
 \begin{figure}[h]
\begin{center}
\includegraphics[width=15.5cm]{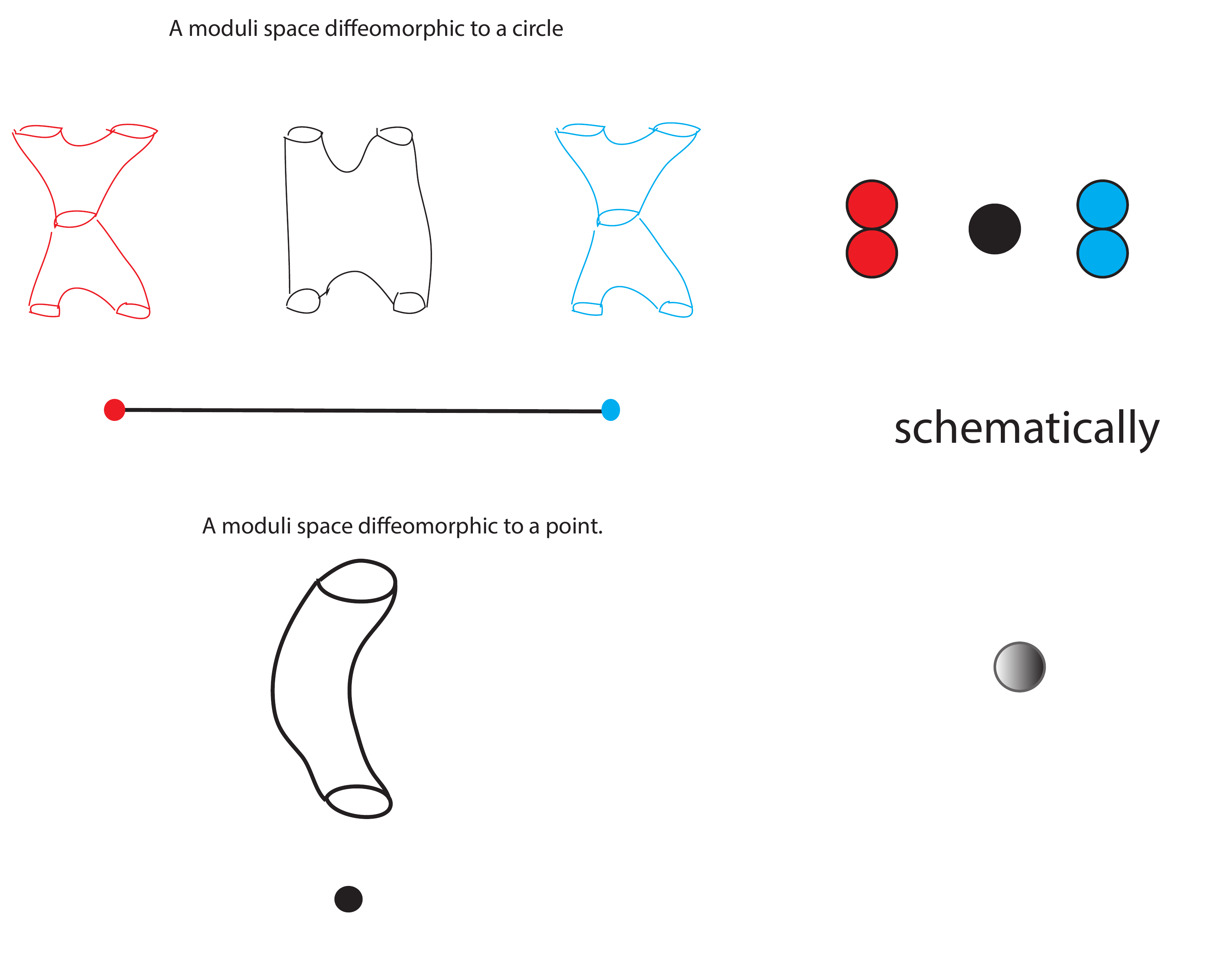}
\end{center}
\caption{A moduli space of a union class, part 1.}\label{FIG 11122}
\end{figure}

The Figure \ref{FIG 11122} gives two moduli spaces, one diffeomorphic to a closed interval, the other to a point.
Assume that one belongs to $a'$ and the other to $a''$. Denote the union class by $a$ which will have a two-dimensional moduli space
and is more complicated than one might think. It is illustrated in the following Figure \ref{FIG 111122}.
 \begin{figure}[h]
\begin{center}
\includegraphics[width=16.5cm]{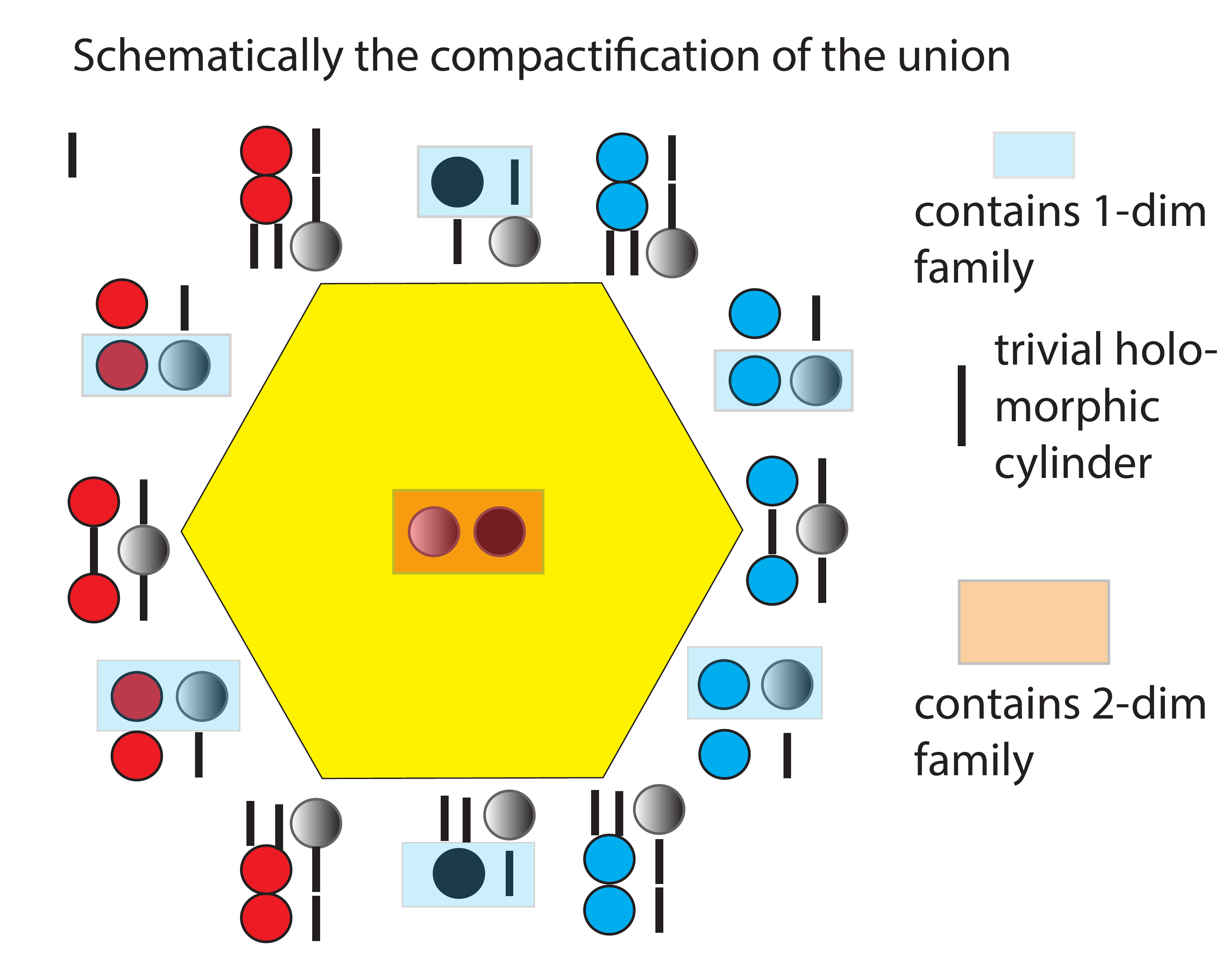}
\end{center}
\caption{A moduli space of a union class, part 2.}\label{FIG 111122}
\end{figure}

A problem arising in the inductive perturbation theory is that the perturbations up to level $\ell$ determine 
for a disjoint union class on level $\ell+1$ the perturbation at the boundary as well as the perturbation in the interior.  The basic 
question is,  if they fit smoothly together? The answer is that it depends on the sets on which the perturbations were carried out. This is the reason for introducing $\wh{\mathcal T}$ on $\wh{Z}$.

\newpage

\part*{Lecture 4}

\section{Structures on the Category of Stable Maps}
\subsection{Background}
Assume we have are given a GCT ${\mathcal C} $ equipped with a polyfold structure and an associated strong polyfold bundle
$P:{\mathcal E}_{\mu}\rightarrow {\mathcal C}$.
This allows to talk about interesting 
smooth objects.  We have seen that we can immediately distinguish three classes of sc-smooth section functors:\\

\textbf{sc-smooth, sc$^+$-smooth, sc-Fredholm}.\\

\noindent It also allows to distinguish interesting sc-smooth saturated full subcategories.  \textbf{Saturated} means that if an object is isomorphic to 
an object in the subcategory it also belongs to it.  Such subcategories have the form ${\mathcal C}_U$, where 
$U\subset |{\mathcal C}|$, and are generated by the objects with isomorphism class in $U$.  Interesting examples
are in the stable maps case:
\begin{itemize}
\item[-]  ${ \mathcal S}_a,\ a\in \pi_0(Z)$
\item[-]  ${\mathcal S}_{\theta},\ \theta\in \text{face}_a$
\end{itemize}
Recall that a \textbf{submanifold} $M$ of a M-polyfold $O$  is a subset such that every point $m\in M$ has an open neighborhood
$U(m)$ and a sc$^+$-smooth retraction $r:U\rightarrow U$ with $r(U)=M\cap U$.   Recall that sc$^+$-smooth means 
that $r:U\rightarrow U^1$ is defined and sc-smooth. $M$ has an induced M-polyfold structure which on the subset of points with $d_M(m)=0$
is the same as a classical smooth manifold structure. If $M$ is tame it is a classical smooth manifold with boundary with corners.

 Consider a functor $\Theta:{\mathcal C}\rightarrow [0,1]\cap {\mathbb Q}^+$.
It defines a full subcategory $\text{supp}(\Theta)$ associated to all objects $c$ with $\Theta(c)>0$. The subcategory
is saturated in the sense that all isomorphic objects belong to it. Each object $c$ in this subcategory carries the weight $\Theta(c)\in (0,1]$.
We shall call $\Theta$ a full weighted subcategory. It defines a subset $|\text{supp}(\Theta)|\subset |{\mathcal C}|$.
\begin{definition}
We say $\Theta$ is of \textbf{manifold-type} provided the following holds.
\begin{itemize}
\item[(1)] $\Theta$   only takes values in $\{0,1\}$
\item[(2)]  Each element with positive weight 
has trivial isotropy.
\item[(3)] For every object $c$ with positive weight pick a uniformizer $\Psi\in F(c)$ with $\Psi:G\ltimes O\rightarrow {\mathcal C}$.
Then there exist a submanifold $M$ of $O$ containing $\bar{o}$ and an open neighborhood $U=U(\bar{o})$ such that 
$\Theta\circ\Psi|U: U\rightarrow [0,1]\cap {\mathbb Q}$ satisfies
\begin{itemize}
\item[$\bullet$] $\Theta\circ \Psi (o)=1$ for $o\in U\cap M$ and $0$ otherwise.
\end{itemize}
\end{itemize}
\qed
\end{definition}
\begin{thm}
If $\Theta$ is of manifold type  $|\text{supp}(\Theta)|$ has in a natural way the structure of a smooth manifold. 
\end{thm}
\begin{definition}
 $\Theta$ is of \textbf{orbifold type} provided the following holds.
\begin{itemize}
\item[(1)] $\Theta $ only takes values in $\{0,1\}$. 
\item[(2)] For every object $c$ with positive weight pick a uniformizer $\Psi\in F(c)$ with $\Psi:G\ltimes O\rightarrow {\mathcal C}$.
Then there exist a submanifold $M$ of $O$ containing $\bar{o}$ and an open $G_c$-invariant neighborhood $U=U(\bar{o})$ such that 
$\Theta\circ \Psi (o)=1$ for $o\in U\cap M$ and $0$ otherwise.
\end{itemize}
\qed
\end{definition}
\begin{thm}
If $\Theta$ is of orbifold type  the orbit space of $\text{supp}(\Theta)$ has in a natural way the structure of a smooth orbifold.
\qed
\end{thm}
The most general class is given in the  following definition.
\begin{definition}
$\Theta$ is of \textbf{branched weighted orbifold type} provided  
for every object $c$ with positive weight the following holds. For a uniformizer $\Psi\in F(c)$ with $\Psi:G\ltimes O\rightarrow {\mathcal C}$ there exist finitely many  submanifolds $M_i$ of $O$ containing $\bar{o}$ and an open $G_c$-invariant neighborhood $U=U(\bar{o})$ such that 
$$
\Theta\circ \Psi (o)=\frac{1}{\sharp I}\cdot \sharp \{i\in I\ |\ o\in M_i\}
$$
 for $o\in U\cap M$ and $0$ otherwise.  
\qed
\end{definition}
One can show that the orbit space of $\text{supp}(\Theta)$ has in a natural way the structure of  what is called a weighted branched orbifold.
There is a whole theory, \cite{HWZ2017}, about orientations,  about sc-smooth differential forms on ${\mathcal C}$, a de Rham complex 
and a Stokes formula. There are related ideas in \cite{MCDUFF}.
\subsection{Preliminaries for a Perturbation Theory of $\bar{\partial}_{\widetilde{J}}$}
Consider the functor $\Lambda_0:{\mathcal E}\rightarrow {\mathbb Q}^+$ (the latter only has the identities as morphisms)
defined by $\Lambda_0(\alpha,0)=1$ and $\Lambda_0(\alpha,e)=0$ otherwise. Then $\Lambda_0\circ\bar{\partial}_{\widetilde{J}}(\alpha)>0$
means that $\bar{\partial}_{\widetilde{J}}(\alpha)=(\alpha,0)$, i.e. $\alpha$ is a $\widetilde{J}$-holomorphic object. 
In this case
$$
|\text{supp}(\Lambda_0\circ\bar{\partial}_{\widetilde{J}})| =\overline{\mathcal M}_J.
$$
For a uniformizer $\bar{\Psi}:G\ltimes K\rightarrow {\mathcal E}$ consider $\Lambda_0\circ\bar{\Psi}: G\ltimes K\rightarrow {\mathbb Q}^+$.
For the tame strong bundle $p:K\rightarrow O$ let $s$ be the zero-section which is a sc$^+$-section, i.e. in particular sc-smooth.
Then 
$$
\Lambda_0\circ\bar{\Psi}(k) = \sharp \{s\ |\ s\circ p(k)=k\}
$$
Next we observe that 
$$
\Lambda_0\circ\bar{\partial}_{\widetilde{J}}\circ \Psi =(\Lambda_0\circ\bar{\Psi})\circ \bar{\partial}_{\bar{\Psi}}.
$$
The set of points $o\in O$ for which this is positive consists precisely of all $o$ with $s(o)=\bar{\partial}_{\bar{\Psi}}(o)$.
This a polyfold Fredholm problem and if we have tranversality it defines a smooth submanifold $M$ of $O$. Then the functor 
$\Theta_0= \Lambda_0\circ\bar{\partial}_{\widetilde{J}}$ over $O$ can be written as $\Theta_0\circ\Psi$
and it takes value $1$ if and only if it lies on the submanifold.  In this case $\Theta_0$ is a functor which is manifold like 
over the full subcategory ${\mathcal S}_{\Psi(O)}$.
Perturbation theory and transversality theory in the presence of symmetries is always difficulty and we shall use 
as Fukaya-Ono, \cite{FO},  the idea of multisections.  In functorial terms we need to introduce fractional objects with the assumption that 
any underlying object can only occur once.  This is accomplished by functors 
$$
\Theta: {\mathcal S} \rightarrow [0,1]\cap {\mathbb Q}^+.
$$
These come together with the functors
$$
\Lambda:{\mathcal E}\rightarrow [0,1]\cap {\mathbb Q}^+,
$$
which we shall call \textbf{partition of unity},
having the property that $\sum_{e\in \mu(\alpha)}\Lambda(\alpha,e)=1$.  Note that for 
$\Lambda$ as above the functor $\Theta=\Lambda\circ \bar{\partial}_{\widetilde{J}}$ has the before-mentioned
properties. As we shall see there is a large world of sc-smooth $\Lambda$ in general position to $\bar{\partial}_{\widetilde{J}}$
such that $\Theta$ is sc-smooth and of branched orbifold type and $a\cap \text{supp}(\Theta)$ is for every $a$ compact.
It is, of course, absolutely crucial that the perturbations $\Lambda$ are compatible with the algebraic structures 
on ${\mathcal S}$ responsible for the ability to encode certain data in the SFT-formalism.

\subsection{Structures}
Denote by $\mathfrak{P}$ the category whose objects are maps  $m:I\rightarrow {\mathcal P}$, where $I$ is a finite set.
A morphism $b:m\rightarrow m'$ is a bijection $b:I\rightarrow I'$ such that $m'\circ b=m$.
$$
\textbf{Concatenation Structure}\ \ \ \ \ \ \ \mathfrak{P}\xleftarrow{\text{ev}^-} {\mathcal S} \xrightarrow{\text{ev}^+} \mathfrak{P}
$$
For example $\text{ev}^+(\alpha)$ is the map $\Gamma^+_k\rightarrow {\mathcal P}$, which associates to a puncture the periodic orbit.
Similarly $\text{ev}^-(\alpha):\Gamma_0^-\rightarrow {\mathcal P}$.
We now can build a new category using the  weak fibered product
$$
 {\mathcal S}\times_{\mathfrak{P}} {\mathcal S} 
$$
Observe that if $\alpha',\alpha''$ are objects in ${\mathcal S}$ and there exists a morphism 
$b:\text{ev}^+(\alpha')\rightarrow \text{ev}^-(\alpha'')$ we obtain the object $(\alpha',b,\alpha'')$ in the weak fibered product.
Next we observe that this object is related to a finite number of objects in $\partial{\mathcal S}$. Namely
there are finitely many lifts $\widehat{b}$ such that 
$(\alpha',\widehat{b},\alpha'')$ is an object in ${\mathcal S}$ and since $d(\alpha',\wh{b},\alpha'')\geq 1$, it belongs to the boundary
$\partial{\mathcal S}$ of the category ${\mathcal S}$.  There is also an obvious associativity when considering these type of lifts,
or the other way round when forgetting decorations
\begin{eqnarray*}
&(\alpha',\wh{b},\alpha'',\wh{b}',\alpha''')\rightarrow ((\alpha',\wh{b},\alpha''),b',\alpha'')\rightarrow (\alpha',b,\alpha'',b',\alpha'')&\\
&(\alpha',\wh{b},\alpha'',\wh{b}',\alpha''')\rightarrow (\alpha',{b},(\alpha'',\wh{b}',\alpha''))\rightarrow (\alpha',b,\alpha'',b',\alpha'')&
\end{eqnarray*}
In the perturbation theory we would like not(!) to perturb over trivial cylinder segments. There the Cauchy-Riemann operator
is automatically transversal and in addition the $\wt{J}$-holomorphic trivial cylinders play an important role  as we shall see.
Let us define the functor
$$
\Pi:{\mathcal E}\rightarrow {\mathcal E}
$$
 by setting $e$ equal to zero on trivial cylinder components.
This functor covers the identity. \\

\noindent\textbf{Warning:} The functor is not sc-smooth in the sense that 
$\bar{\Psi}^{-1}\circ\Pi\circ\bar{\Psi}$ is a strong bundle map. However, it has some  `sc-smoothness properties' in the sense  that there are many (local)  sc$^+$-sections 
which satisfy $\Pi_{\bar{\Psi}}\circ s=s$  over ${\mathcal S}_{\wh{Z}}$.   Using sc$^+$-multisection functors build on such special sc$^+$-sections will turn out to be  enough 
 to deal with all occurring 
transversality issues around $\bar{\partial}_{\widetilde{J}}$ . We shall introduce 
this class  of functors $\Lambda:{\mathcal E}_{\wh{U}}\rightarrow {\mathbb Q}^+\cap [0,1]$, $\wh{U}\in \wh{\mathcal T}$,  later on.
Hence, measured in this way, $\Pi$ has some sc-smoothness properties by defining a sufficiently large class of 
sc$^+$-smooth sections.  There is indeed a philosophical point here: there are many types of natural functors, which show very specific `smoothness properties' in the weak sense similarly to the property of $\Pi$ just discussed.\\

Having $\Pi$ we impose on  a partition of unity
$\Lambda:{\mathcal E}\rightarrow [0,1]\cap {\mathbb Q}^+$, the following first requirement
\begin{eqnarray*}
\Lambda(\alpha,e) =(\Lambda_0\circ (\text{Id}-\Pi)(\alpha,e))\cdot \Lambda(\alpha,e).
\end{eqnarray*}
At this point our condition requires a partition of unity functor $\Lambda$ to have the property that 
if it  takes positive values on $(\alpha,e)$ it must hold true that $e$ vanishes on trivial cylinder segments.
There is an other quite obvious requirement relating to the concatenation which we shall use to upgrade the requirement.  With
$\alpha=(\alpha_0,\widehat{b}_1,...,\widehat{b}_k,\alpha_k)$ and the obvious meaning of $e_i$
\begin{eqnarray*}
\Lambda(\alpha,e) =(\Lambda_0\circ (\text{Id}-\Pi)(\alpha,e))\cdot \prod_{i=0}^{d(\alpha)} \Lambda(\alpha_i,e_i).
\end{eqnarray*}
Another crucial condition is connected with the occurrence of disjoint unions.  Let $\alpha$ be a stable map with $d(\alpha)=0$. 
and denote by $\bar{S}$ the space obtained by identifying $x\equiv y$ for a nodal disk pair $\{x,y\}$.
Then $\bar{S}$ decomposes as  $\bar{S}= S^{tc} \sqcup \bar{S}^{ntc}$ and each connected component of $ \bar{S}^{ntc}$ has an associated stable map.
$$
\textbf{Disjoint Union Structure}\ \ \alpha\ \  \stackrel{\text{forgetful}}{\Longrightarrow} \ \  \{ \alpha_{\mathfrak{c}}\ |\ \mathfrak{c}\in \pi_0^{\text{ntriv}}(\bar{S})\}.
$$
The  $\alpha_{\mathfrak{c}}$ are  $\widetilde{J}$-holomorphic if the $\alpha$ are. This leads to the following requirement which encapsulates SFT,
and which has to hold over a suitable full category ${\mathcal S}_{\wh{U}}$, where $\wh{U}$ is a so called strong neighborhood of $\overline{\mathcal M}_J$ 
in $\wh{Z}$, which in addition satisfies certain compatibility conditions with the structures at play.
\\
\begin{eqnarray}\label{Good}
\resizebox{0.81\hsize}{!}{\boxed{
\Lambda(\alpha,e) =(\Lambda_0\circ (\text{Id}-\Pi)(\alpha,e))\cdot \prod_{i=0}^{d(\alpha)} \prod_{\mathfrak{c}\in \pi_0^{\text{ntriv}}(\bar{S}_i)}\Lambda(\alpha_{i,\mathfrak{c}},e_{i,\mathfrak{c}})\ \ \text{for}\ \ |\alpha|\in \wh{U}}}
\end{eqnarray}
This formula relates the value of $\Lambda(\alpha,e)$ to the value on smaller pieces provided $|\alpha|$ belongs to 
some set $\wh{U}$. The inductive content of this formula is obvious. It is also clear that in the inductive construction 
of $\Lambda$ with respect to some organization of $\pi_0(Z)$, the sets $\wh{U}\cap a$ have to be constructed inductively too.

The goal of these lectures is to show that there are enough $\Lambda$ which are sc$^+$-smooth, i.e. 
sc$^+$-multisections,  which are so small that, in some sense to be made precise,
the orbit space of the support of $\Theta:=\Lambda\circ \bar{\partial}_{\widetilde{J}}$ has a compact intersection with every $a\in \pi_0(Z)$,
and the local problems are all in general position, so that $\Theta:{\mathcal S}\rightarrow [0,1]\cap {\mathbb Q}^+$
is a weighted a branched orbifold in general position to the boundary of ${\mathcal S}$.

Once this is achieved 
we can study orientation questions for $\Theta=\Lambda\circ \bar{\partial}_{\widetilde{J}}$. One also needs to study the relationship 
between different $\Theta$ obtained this way. Of course, if we start turning the outline into a reality we need to know 
more about the polyfold structure for ${\mathcal S}$ and the strong bundle structure for $P:{\mathcal E}\rightarrow {\mathcal S}$, f.e.
we need to understand enough about the smoothness properties of $\Pi$ to describe the sc$^+$-sections compatible with $\Pi$.

Before we go into the necessary constructions we shall describe a result which will be an outcome. For this 
we need to present some additional ideas.  Given a face $\theta$ consider the saturated full subcategory ${\mathcal S}_{\theta}$.\
One can show that given $\alpha$ in ${\mathcal S}_{\theta}$ and $\Psi\in F(\alpha)$, say $\Psi:G\ltimes O\rightarrow {\mathcal S}$,
 the preimage of ${\mathcal S}_{\theta}$ is a tame sub-M-polyfold $X_O\subset O$ invariant under $G$.
Without going more into details we obtain the diagram
$$
\begin{array}{cc}
\begin{CD}
G\ltimes O @>\Psi>>  {\mathcal S}\\
@AAA     @A \text{incl} AA\\
G\ltimes X^{\theta}_O @>\Psi|{X^{\theta}_O} >> {\mathcal S}_{\theta}
\end{CD}
&\ \ \ \ \  \text{tame sub-M-polyfold}\ \ 
\bm{M}(\Psi|X^{\theta}_O,\Psi'|X^{\theta}_{O'})\subset \bm{M}(\Psi,\Psi')
\end{array}
$$
This can be interpreted as ${\mathcal S}_{\theta}$ being a subpolyfold, and that it inherits a polyfold structure as well.  Given 
$\Theta:{\mathcal S}\rightarrow {\mathbb Q}^+\cap [0,1]$ of branched weighted orbifold type, we can formulate what it means 
that it is tame and in general position to the boundary of ${\mathcal S}$ which is the union of all ${\mathcal S}_{\theta}$, see \cite{HWZ2017}.
\begin{thm}
There are arbitrarily small sc$^+$-multisection functors  $\Lambda$ satisfying the \textbf{\em Requirement} over a suitable strong neighborhood $ \wh{U}$ of $\overline{\mathcal M}_J$, in general position to $\bar{\partial}_{\widetilde{J}}$ so that 
$\Theta:=\Lambda\circ \bar{\partial}_{\widetilde{J}}:{\mathcal S}\rightarrow {\mathbb Q}^+\cap [0,1]$ is of weighted branched orbifold type, in general position 
to $\partial{\mathcal S}$, and tame.  Further for every $a\in \pi_0(Z)$ the intersection $a\cap |\text{supp}(\Theta)|$ is compact.
\qed
\end{thm}
What means arbitrarily small has to be made precise.  We shall introduce a quantitative criterion.
If we have $\Lambda$ and $\Lambda'$ obtained by making different choices the other result needed is to understand generic enough homotopies between them. This will be addressed later and as we shall see this point is subtle
due to ``transversality to the diagonal problems'' arising from the fact that one needs to consider certain types
of fibered products. We have now described many of the results we are aiming at and will start with the 
construction of the polyfold and strong bundle structure.

\subsection{Additional Structures}
There are some additional structural features which are important, see \cite{HWZ2017}
for background material. 
Starting with ${\mathcal S}$ we can construct a category $\wt{\mathcal S}$ 
with a surjective forgetful functor $\pi: \wt{\mathcal S}\rightarrow {\mathcal S}$ and we do this as follows.
First of all we mark every periodic orbit$([\gamma],T,k)$  in ${\mathcal P}$  by taking a specific representative $(\gamma,T,k)$.
This choice specifies the distinguished point $\gamma(0)\in \gamma(S^1)$, referred to as \textbf{marker}.
The objects of $\wt{\mathcal S}$ are constructed from objects $\alpha$ in ${\mathcal S}$ by adding 
some additional features. Start with $\alpha=(\alpha_0,\wh{b}_1,..,\wh{b}_k,\alpha_k)$.
First of all we order the marked points in $M$, i.e. replace the set $M$ by  $(m_1,...,m_{\ell})$ by an ordered list.
Then we number the positive punctures on the top floor and we also number the negative punctures on the bottom floor.
Finally we add a decoration to each puncture $z$ in $\Gamma_0^-\cup\Gamma^+_k$, i.e. pick an oriented real line in $T_zS$ such that the directional limit 
hits the asymptotic periodic orbit at the previously fixed marker. Note that the number of possible decorations for a bottom or top punctures is the covering number of the associated periodic orbit.
 In the future, when we talk about a \textbf{decorated puncture} 
we mean the negative ones on floor $0$ and the the positive ones on the top floor. Recall that the interface nodal pairs are decorated. Further we shall consider matching punctures coming from the floor interfaces and we shall
refer to them as \textbf{decorated nodal pairs}.

A morphism $\wt{\Phi}:\wt{\alpha}\rightarrow \wt{\alpha}'$ has the form $\wt{\Phi}=(\wt{\alpha},\phi,\wt{\alpha}')$ where $\phi$ is a biholomorphic 
respecting the data, in particular the numberings of the top punctures, the bottom punctures as well as of the marked points.
Moreover $T\phi$ is assumed to map decoration to decoration. The forgetful functor 
$$
\pi:\wt{\mathcal S}\rightarrow {\mathcal S}
$$
is finite to one.  Let $\alpha$ be an object in ${\mathcal S}$ and denote by $\ell^\pm$ the number of positive and negative punctures on the top and bottom floor, respectively.
Let $m$ be the number of marked points and denote by $k^+_1,...,k^+_{\ell^\pm}$ the covering numbers of the asymptotic periodic orbits. 
Then the preimage of $\alpha$ contains $(m! )\cdot( \ell^-!)\cdot (\ell^+!)\cdot (k^-_1!)\cdot..\cdot (k^-_{\ell^-}!)\cdot (k^+_1!)\cdot..\cdot (k^+_{\ell^+}!)$ many elements.
The functor $\pi$ is an example of a proper covering functor. It is surjective on objects and finite to one and further the following property holds.
The map
$$
\text{mor}({\wt{\mathcal S}})\rightarrow \text{mor}({\mathcal S}){_{s}\times_\pi} \text{obj}(\wt{\mathcal S}): \wt{\Phi} \rightarrow (\Phi, s(\wt{\Phi}))
$$
is a bijection, here $\Phi=\pi(\wt{\Phi})$.  Let us note that on trivial cylinder segments the two asymptotic markers 
are not correlated. For example if we have a $\wt{J}$-holomorphic cylinder of covering number $k$
there are $k^2$ different constellations, however, there are $k$-many isomorphism classes. 
Using $(F,\bm{M})$ associated
to ${\mathcal S}$ one can construct for every $\Psi\in F(\alpha)$, $\alpha$  in ${\mathcal S}$, 
an ep-groupoid $X_{\Psi}$ with a proper covering functor fitting into the commutative diagram
$$
\begin{CD}
X_{\Psi} @>>>  \wt{\mathcal S}\\
@V\bm{\pi}VV  @V \pi VV\\
G\ltimes O @>\Psi >>  {\mathcal S}
\end{CD}
$$ 
The top arrow is injective on objects and fully faithful.
The functor $\bm{\pi}$ is a proper covering functor, namely it is a surjective local sc-diffeomorphism and 
$$
\bm{X_{\Psi}}\xrightarrow{\langle \bm{\pi},s\rangle} \text{mor}(G\ltimes O){_{s}\times_{\bm{\pi}}} X_{\Psi}
$$
is a sc-diffeomorphism.

A polyfold has a tangent bundle. In our case this applies to ${\mathcal S}$ and $\wt{\mathcal S}$.  A uniformizer $\Psi: G\ltimes O\rightarrow {\mathcal S}$ defines a bundle unifomizer
$T\Psi: G\ltimes TO\rightarrow T{\mathcal S}$ fitting into the commutative diagram
$$
\begin{CD}
G\ltimes TO @>T\Psi >> T {\mathcal S}\\
@V \text{Id}\ltimes p VV @V VV\\
G\ltimes O^1@>\Psi^1>>   {\mathcal S}^1
\end{CD}
$$
This allows to define sc-differential forms on ${\mathcal S}$ and in fact the De Rham complex.
The same holds for $\wt{\mathcal S}$.   Details are in \cite{HWZ2017}.
Working with $\wt{\mathcal S}$ we have the evaluation at marked points functor
$$
\wt{\mathcal S}\rightarrow Q^{\times} :=\{\ast\}\sqcup Q\sqcup (Q\times Q)\sqcup (Q\times Q\times Q)...
$$
The pull-back of a differential form on $Q$ etc is a sc-differential form on $\wt{\mathcal S}$.
There is also a sc-smooth forgetful functor in the DM-space associated to ordered marked points
$$
\text{forget}:\wt{\mathcal S}\rightarrow \wt{\mathcal R}
$$
Of course, it has to be made precise what sc-smoothness means. 
These structures are important for the construction of SFT.
\newpage
\part{Local-Local Constructions}
A stable map is defined on a Riemann surface and as we have seen in Part 1
we are interested  in defining something like a smooth neighborhood of this object. Since the object class
of a groupoidal category is not a set in general a notion of a neighborhood literally does not make sense
and has to be substituted by an appropriate concept, namely that of a uniformizer. In some 
sense a uniformizer construction is  the \textbf{local theory}.
However, an object can be cut into smaller pieces since the Riemann surface can be `chopped up'. We see
that the object is a fibered product of smaller pieces. Hence we might expect that the local pictures 
is obtained as a fibered product of a local theory around the  local pieces. The latter is what we refer
to as the \textbf{local-local theory} and as we have just discussed  gluing the local-local pieces we obtain the local theory. Of course, there is still the question how the uniformizers interact (recall the notion of transition construction),
and this can be considered as the \textbf{global theory}.
%\newpage
\part*{Lecture 5}
\section{Tools for Local-Local Constructions}
We begin with a method which is useful constructing new M-polyfolds  from old ones. Comprehensive references are \cite{FH-primer,FH-book}.
\subsection{The Imprinting-Method}
The imprinting-method based on the following theorem is very powerful.
\begin{thm}
Assume that $X$ is a M-polyfold and $Y$ a set and $\oplus: X\rightarrow Y$ a surjective map with the following properties.
\begin{itemize}
\item[(1)]   The quotient topology ${\mathcal T}_{\oplus}$ on $Y$ is metrizable.
\item[(2)]   For every $y\in Y$ there exists $U(y)\in {\mathcal T}_{\oplus}$ and $H:U(y)\rightarrow X$ such that 
\begin{itemize}
\item[(a)] $\oplus\circ H=Id_{U(y)}$.
\item[(b)] $H\circ\oplus:\oplus^{-1}(U(y))\rightarrow X$ is sc-smooth.
\end{itemize}
\end{itemize}
Then $Y$ has a unique M-polyfold structure characterized by the following properties.
\begin{itemize}
\item[(i)]  $\oplus:X\rightarrow Y$ is sc-smooth.
\item[(ii)] Every $H:U(y)\rightarrow X$ which has the properties as in (2) is sc-smooth.
\end{itemize}
\qed
\end{thm}
The theorem motivates the following definition.
\begin{definition}
A M-polyfold structure on a set $Y$ defined by the imprinting-method is given by a diagram $\oplus:X\rightarrow Y$
which has the properties (1) and (2) stated  in the theorem.
\qed
\end{definition}
\subsection{The Gluing Example}
We illustrate this by an example.
Take a compact nodal Riemann  surface $(S,j,D)$,  perhaps with smooth boundary, and consider the sc-Hilbert space $E$  of maps $u:S\rightarrow {\mathbb R}^N$ 
consisting of maps of class $H^{3,\delta_0}_c$, which means away from nodal points of class $H^3_{loc}$ and at nodes
exponentially asymptotic with decay-rate $\delta_0$ to the matching nodal value.  Level $m$ is defined by regularity $(m+3,\delta_m)$.
We  fix disks around the nodal points and  use the exponential gluing profile $\varphi(r)=e^{\frac{1}{r}}-e$.
We  have a smooth manifold ${\mathbb B}$  of natural gluing parameters and for every such $\mathfrak{a}$ we can consider $(S_{\mathfrak{a}},j_{\mathfrak{a}},D_{\mathfrak{a}})$.
Define $Y$ to be the union of all maps of class $(3,\delta_0)$ defined on the various $(S_{\mathfrak{a}},j_{\mathfrak{a}},D_{\mathfrak{a}})$ and introduce
$$
\oplus: {\mathbb B}\times E\rightarrow Y:\oplus(\mathfrak{a},u)=w,
$$
where $w$ is obtained as follows. 
We shall use a smooth cut-off function $\beta:{\mathbb R}\rightarrow [0,1]$ satisfying $\beta(s)=1$ for $s\leq -1$, $\beta'(s)<0$ for $s\in (-1,1)$ and $\beta(s)+\beta(-s)=1$.
If $z\in S_{\mathfrak{a}}$ belongs to the core region, which can be identified naturally with a subset of $S_{\mathfrak{a}}$ we define
$\oplus(\mathfrak{a},u)(z)=u(z)$.   
On the glued necks we define $\oplus(\mathfrak{a},u)$ as follows, where for convenience we shall use standard coordinates.\\

As a model for a disk pair we take ${\mathbb R}^+\times S^1 \coprod {\mathbb R}^-\times S^1$.
Given a gluing parameter $|a|<1/4$ we define $Z_0$ to be the above space and for nonzero $a$ 
we set $R=\varphi(|a|)$ and $a=|a|\cdot e^{2\pi i\theta}$ and define 
$Z_a$ to consist of all $\{(s,t),(s',t')\}$, $(s,t)\in [0,R]\times S^1$, $(s',t')\in [-R,0]\times S^1$ and $s=s'+R$, $t=t'+\theta$.
Now the gluing of $(u^+,u^-)$ over $Z_a$ is defined by 
$$
\oplus_a(u^+,u^-)(\{(s,t),(s',t')\}) =\beta(|s|-R/2)\cdot u^+(s,t) + \beta(|s'|-R/2)\cdot u^-(s',t')
$$
Let us write $Y_N$ if the target is ${\mathbb R}^N$. The construction is functorial in the sense that 
for a smooth map $f:{\mathbb R}^N\rightarrow {\mathbb R}^M$ the map
$$
Y_N\rightarrow Y_M: u\rightarrow f\circ u
$$
is sc-smooth. This cane be checked using the results in \cite{HWZ8.7}.   Hence $Y$ can be viewed as  functor which associates to ${\mathbb R}^N$ the M-polyfold $Y_N$ and to a smooth map 
$f:{\mathbb R}^N\rightarrow {\mathbb R}^M$ a sc-smooth map.  
The fact that $Y$ is such functor implies quite easily that it has an extension where the image is any smooth manifold $Q$, see \cite{FH-primer} for a proof.
We summarize the discussion as follows.  Denote by ${\mathcal N}$ the category whose objects are the various ${\mathbb R}^N$ and 
the morphisms are the smooth maps.  We denote by $\mathsf{M}$ the category of M-polyfolds and by $\mathsf{M}_{tame}$
the full subcategory of tame M-polyfolds.
\begin{thm}
The construction which associates to ${\mathbb R}^N$ the M-polyfold $Y_N$ and to a smooth map $f:{\mathbb R}^N\rightarrow {\mathbb R}^M$
the map $Y(f): Y_N\rightarrow Y_M: u\rightarrow f\circ u$ is a covariant functor into $\mathsf{M}_{tame}$.
\qed
\end{thm}
That such functors have natural extensions to the category of smooth manifolds with smooth maps between them is not difficult to
see and we shall give the argument later on.
\subsection{More Properties of Imprinting}
Imprinting has some naturality properties.

\begin{thm}\label{THMP1.4}
Assume that $X$ is a M-polyfold and $Y$ and $Z$ are sets and the maps  $\oplus_1$ and $\oplus_2$ in the diagram
$$
X\xrightarrow{\oplus_1} Y\xrightarrow{\oplus_2} Z
$$
are surjective. Define $\oplus:X\rightarrow Z$ by $\oplus =\oplus_2\circ\oplus_1$.  Assume further that $\oplus_1:X\rightarrow Y$ is an imprinting
and assume that $Y$ is equipped with the associated M-polyfold structure. Then the following two statements are equivalent.
\begin{itemize}
\item[(1)] $\oplus:X\rightarrow Z$ is an imprinting.
\item[(2)] $\oplus_2:Y\rightarrow Z$ is an imprinting.
\end{itemize}
Moreover the induced M-polyfold structures (and topology) on $Z$ by both constructions coincide.
\qed
\end{thm}

Here is another result which is  very important for our approach. 
\begin{thm}\label{CORTY2.8}
Assume that $\oplus:X\rightarrow Y$ is a M-polyfold construction by the imprinting-method and $Y'$ a subset of $Y$.
If $X':=\oplus^{-1}(Y')$ is a sub-M-polyfold of $X$ then $\oplus':X'\rightarrow Y'$, where $\oplus'=\oplus|X'$, is an imprinting.
The associated M-polyfold construction on $Y'$ defines the topology induced from $Y$. Moreover $Y'$ is a sub-M-polyfold of $Y$ 
and the induced  M-polyfold from $Y$ and the $\oplus'$-structure coincide.
\qed
\end{thm}
The imprinting method  is well-behaved with certain trivial procedures.
\begin{thm}[Product]\label{product12.3}
Let $\oplus:X\rightarrow Y$ and $\oplus':X'\rightarrow Y'$ be two M-polyfold constructions by imprinting.
Then $\oplus\times\oplus':X\times X'\rightarrow Y\times Y'$ is a M-polyfold construction by imprinting.
The induced M-polyfold structure on $Y\times Y'$ is the product structure. In particular the quotient topology ${\mathcal T}_{\oplus\times \oplus'}$
on $Y\times Y'$ is the product topology ${\mathcal T}_{\oplus}\times {\mathcal T}_{\oplus'}$, which is the topology having as basis the products of open sets.
\qed
\end{thm}
Here is another result.
\begin{thm}[Disjoint Union]\label{sum12.3}
The disjoin union $\oplus\sqcup\oplus'$ \index{$\oplus\sqcup\oplus'$} of two imprintings  is an imprinting.
\qed
\end{thm}

\subsection{Operations}\label{qsec3.1}
Using the previously established results we can take the {\bf product} $\oplus\times\oplus'$ and the {\bf disjoint union} $\oplus\sqcup\oplus'$ of two given
imprintings. Given $\oplus:X\rightarrow Y$ and an injective map $\phi:Y'\rightarrow Y$ with the property that $X':=\oplus^{-1}(\phi(Y'))$
is a sub-M-polyfold of $X$ we can consider the commutative diagram
$$
\begin{CD}
X @>\oplus >>  Y\\
@A\text{incl} AA   @A\phi AA\\
X' @> \phi^{-1}\circ \oplus|X' >>   Y'
\end{CD}
$$
In this case with $\oplus':=\phi^{-1}\circ \oplus|X'$ we see by a previous discussion that $\oplus':X'\rightarrow Y'$ is an imprinting.
\begin{definition}
Given an imprinting $\oplus:X\rightarrow Y$ and an injective map between sets $\phi:Y'\rightarrow Y$, we say that
$\phi$ is {\bf admissible} provided 
$$
X':=\oplus^{-1}(\phi(Y))
$$
 is a sub-M-polyfold of $X$. In this case we define the {\bf pull-back}
$\phi^\ast\oplus $ by 
$$
\phi^{-1}\circ\oplus|(\oplus^{-1}(\phi(Y'))): \oplus^{-1}(\phi(Y'))\rightarrow Y'.
$$
\qed
\end{definition}
The following is an easy exercise.
\begin{lem}
Assume that $\oplus:X\rightarrow Y$ is an imprinting and $\phi:Y'\rightarrow Y$ is admissible so that $\oplus'=\phi^{\ast}\oplus$\index{$\oplus'=\phi^{\ast}\oplus$} is defined.
Suppose further $\psi:Y''\rightarrow Y'$ is admissible for $\oplus'$ defining $\psi^\ast\oplus'$.  Then $\phi\circ\psi$ is admissible for 
$\oplus$ and naturally $(\phi\circ\psi)^\ast\oplus = \psi^\ast(\phi^\ast\oplus)$.
\qed
\end{lem}
These are some basic operations which one can carry out to stay within the scope of imprinting constructions.
The playing field can be vastly extended by adding what we call restriction maps. This is done in the next subsection.

\subsection{Restrictions}  \label{qsec3.2}
We start with a definition adding an additional piece of structure to the imprinting method.
\begin{definition}
An {\bf imprinting with restriction} is a pair $(\oplus,\bm{p})$, \index{$(\oplus,\bm{p})$} where
 $\oplus:X\rightarrow Y$ is a M-polyfold construction by imprinting, and $\bm{p}$ is a finite family of maps
 $p_i:Y\rightarrow A_i$, $i\in I$, where the $A_i$ are M-polyfolds and the compositions  $p_i\circ\oplus:X\rightarrow A_i$ are sc-smooth.
\qed
\end{definition}
The following definition will be crucial for fibered product constructions.
\begin{definition}
Assume that $(\oplus,\bm{p})$ and $(\oplus',\bm{p}')$ are imprintings  with restrictions, and $i_0\in I$ and $i_0'\in I'$ are given
so  that $A_{i_0}=A_{i_0'}'$. We say that $(\oplus,\bm{p})$ and $(\oplus',\bm{p}')$ are $(i_0,i_0')$-{\bf plumbable} provided the subset
$$
X{_{i_0}\times_{i_0'}} X':=\{(x,x')\in X\times X'\ |\ p_{i_0}\circ\oplus(x)=p_{i_0}'\circ\oplus'(x')\}
$$
of $X\times X'$ is a sub-M-polyfold.  
\qed
\end{definition}
Define  $Y{_{i_0}\times_{i_0'}} Y'=\{(y,y')\in Y\times Y'\ |\ p_{i_0}(y)=p_{i_0}'(y')\}$  and denote by 
$$
\phi: Y{_{i_0}\times_{i_0'}} Y'\rightarrow Y\times Y'
$$
the inclusion map.  Take the product imprinting  $\oplus\times\oplus'$ and observe that
$$
(\oplus\times\oplus')^{-1}(\phi( Y{_{i_0}\times_{i_0'}} Y')) = X{_{i_0}\times_{i_0'}} X'
$$
which by assumption is a sub-M-polyfold.   Define $\bm{p}'' = \bm{p}{_{i_0}\sqcup_{i'_0}}\bm{p}'$ by 
\begin{eqnarray}
p_{j}'' =\left[ \begin{array}{cc}
p_j\circ\pi_1&\ \text{for}\ j\in I\setminus\{i_0\}\\
p_{j}'\circ\pi_2& \ \text{for}\ j'\in I'\setminus\{i_0'\}
\end{array}\right.,
\end{eqnarray}
where the $\pi_i$ are the projections from the fibered product onto the first and second factor, respectively.
\begin{definition}
If $(\oplus,\bm{p})$ and $(\oplus',\bm{p}')$ are $(i_0,i_0')$-{plumbable}\index{plumbable} we define the imprinting with restriction $(\oplus,\bm{p}){_{i_0}\times_{i_0'}}(\oplus',\bm{p}')$
where
$$
\oplus{_{i_0}\times_{i_0'}}\oplus' : =\phi^{\ast}(\oplus\times\oplus')
$$
and call it the $(i_0,i_0')$-{\bf plumbing}\index{plumbing} of $(\oplus,\bm{p})$ and $(\oplus',\bm{p}')$.
\qed
\end{definition}
For imprinting  constructions with restrictions $(\oplus,\bm{p})$ and $(\oplus',\bm{p}')$ we can define 
 first the disjoint union $\oplus\sqcup \oplus'$.  For a pair $(i,i')\in I\times I'$ we define 
$p''_{(i,i')}: Y\sqcup Y'$ by $p''_{(i,i')}(y)=p_i(y)$ for $y\in Y$ and $p''_{(i,i')}(y')=p_{i'}'(y')$. Then 
we call $(\oplus \sqcup\oplus',\bm{p}'')$ the disjoint union of $(\oplus,\bm{p})$ and $(\oplus',\bm{p}')$ and write
it as $(\oplus,\bm{p})\sqcup(\oplus',\bm{p}')$. The following is obvious.
\begin{thm}[Disjoin Union]
The disjoint union of two imprinting construc\-tions with restrictions is an imprinting construction with restrictions.
\qed
\end{thm}
\subsection{Submersive Imprinting Constructions with Restrictions}
We also need some extensions.
\begin{definition}
Consider a triple $(\oplus,\bm{p},f)$, where $(\oplus,\bm{p})$ 
is an imprinting  construction with restrictions,
say $\oplus :X\rightarrow Y$, $p_i:Y\rightarrow A_i$ for $i\in I$ and 
$f:Y\rightarrow Z$ a surjective map onto a M-polyfold such that $f\circ\oplus :X\rightarrow Z$ is submersive.
We shall say that $(\oplus,\bm{p},f)$ is a \textbf{submersive imprinting  with restrictions}\index{submersive imprinting with restrictions} provided with $\bar{f}=f\circ\oplus$, $\bar{p}_i=p_i\circ \oplus$ the following additional compatibility condition holds.
For every $x_0\in X$ with $z_0=\bar{f}(x_0)$ there exists an open neighborhood $W$ of $(x_0,z_0)$  in $X\times Z$
and a sc-smooth retraction $\rho:W\rightarrow W$ of the form $\rho(x,z)=(\bar{\rho}(x,z),z)$ with $\rho(W)=W\cap \text{gr}(\bar{f})$
and
$$
\bar{p}_i\circ \bar{\rho}(x,z) = \bar{p_i}(x),\ \ i\in I,\ (x,z)\in W.
$$
\qed
\end{definition}
 \begin{figure}[h]
\begin{center}
\boxed{\includegraphics[width=8.5cm]{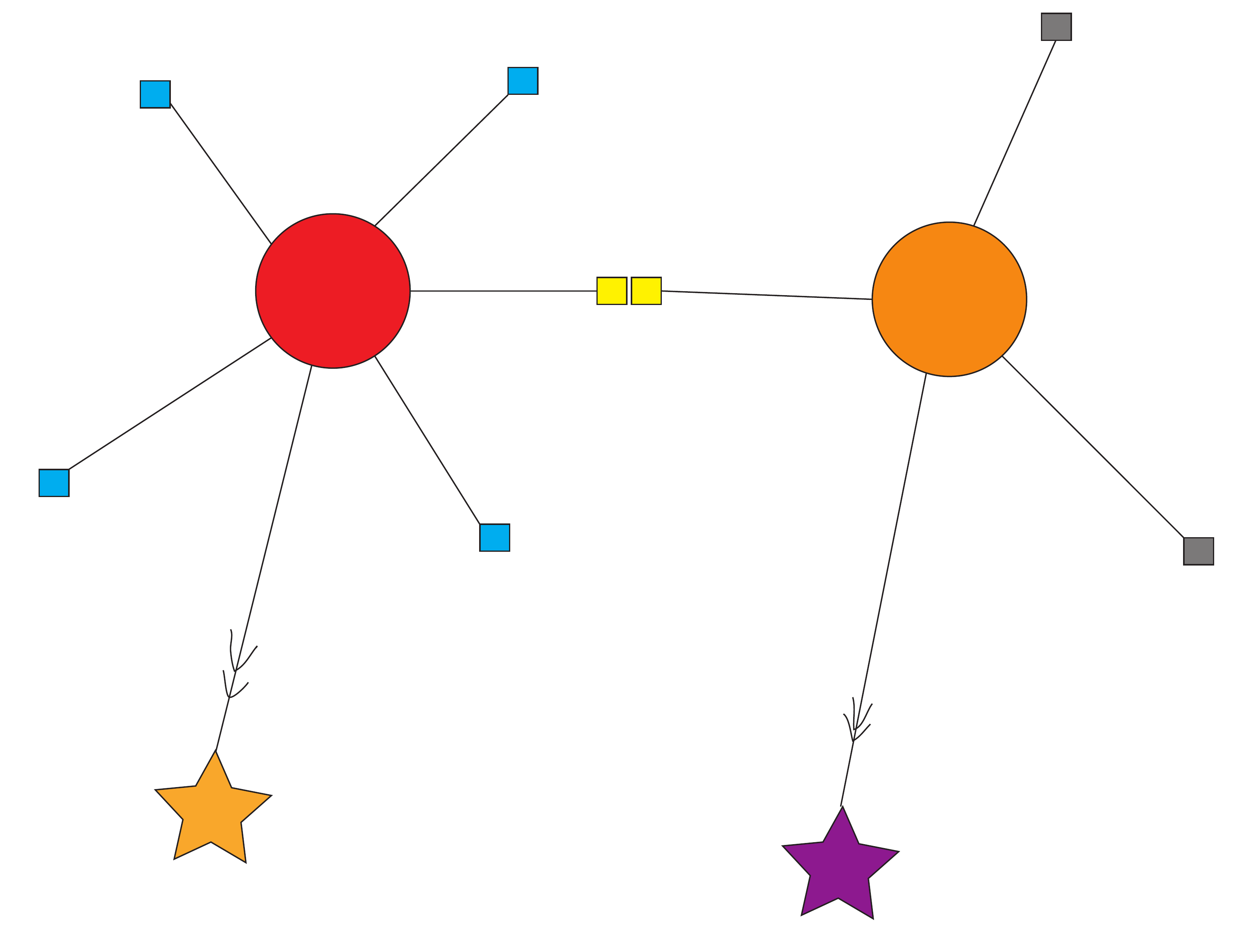}}
\end{center}
\caption{A Plumbing}\label{LEGO}
\end{figure}

\begin{definition}
Given two submersive imprinting constructions with restrictions 
 $(\oplus,\bm{p},f)$ and $(\oplus',\bm{p}',f')$ and $i_0\in I$, $i_0'\in I'$ the $(i_0,i_0')$-fibered product is defined by 
 $$
 (\oplus,\bm{p},f){_{i_0}\times_{i_0'}}(\oplus',\bm{p}',f'):=(\oplus{_{i_0}\times_{i_0'}}\oplus', \bm{p}{_{i_0}\sqcup_{i_0'}}\bm{p}',f{_{i_0}\times_{i_0'}}f').
 $$
It is again a submersive imprinting construction with restrictions.
\qed
\end{definition}
The following result will be particularly useful.
\begin{thm}\label{THMXX-2-12}
Assume that $X$ is a tame M-polyfold and $V$ a smooth finite dimensional manifold with boundary with corners. Suppose that $p:X\rightarrow V$ is a sc-smooth submersive map and  assume that the equality $d_V(p(x))=d_X(x)$ holds for all $x\in X$
and $\oplus:X\rightarrow Y$ is an imprinting and $p':Y\rightarrow V$ a surjective map fitting into the commutative diagram
$$
\begin{CD}
X  @> \oplus >> Y\\
@V p VV   @V p' VV\\
V @= V
\end{CD}
$$
Then the induced M-polyfold structure on $Y$ is tame, $p'$ is sc-smooth and submersive,  and 
$d_V(p'(y))=d_Y(y)$ for all $y\in Y$.
\end{thm}

\newpage
\part*{Lecture 6}
\section{Concrete Local-Local Constructions and Operations}
We shall describe the local-local constructions we are going to need for the later discussed polyfold constructions.
The basic input is the gluing construction described 
in the previous Lecture 5. 
\subsection{Nodal}
Given a nodal disk pair $(D_x\cup D_y,\{x,y\})$ and  weight sequence $\delta: 0<\delta_0<\delta_1<...$ 
we define the sc-Hilbert space $E_{\mathcal D}^{\delta}$ of maps of class $(3,\delta_0)$ which are continuous over the node,  and the set $X^{\delta_0}_{{\mathcal D},\varphi}({\mathbb R}^N)$
by
$$
X^{\delta_0}_{{\mathcal D},\varphi}({\mathbb R}^N) = E^{\delta_0}_{\mathcal D}\sqcup \left( \bigcup_{0<|a|<1/4} H^3(Z_a,{\mathbb R}^N)\right).
$$
We fix smooth compact  concentric annuli $A_x\subset D_x$ and $A_y\subset D_y$ of the same modulus and can associate 
to them the sc-Hilbert spaces $H^3(A_x,{\mathbb R}^N)$ and $H^3(A_y,{\mathbb R}^N)$. We assume that $A_x$ corresponds under 
holomorphic polar coordinates to a subset contained in $[0,20]\times S^1$.  We use the previously introduced
$$
\oplus : {\mathbb B}_{\mathcal D}\times E^{\delta}_{\mathcal D}\rightarrow X^{\delta_0}_{{\mathcal D},\varphi}.
$$
We also have the extraction of gluing parameter
$p_{{\mathbb B}_{\mathcal D}}:X^{\delta_0}_{{\mathcal D},\varphi}\rightarrow {\mathbb B}_{\mathcal D}$. Then 
we also obtain the restriction map
$$
H^3(A_x,{\mathbb R}^N)\xleftarrow{p_x} X^{\delta_0}_{{\mathcal D},\varphi}({\mathbb R}^N)\xrightarrow{p_y} H^3(A_y,{\mathbb R}^N)
$$
\begin{thm}
The tuple $(\oplus,\{p_x,p_y\},p_{{\mathbb B}_{\mathcal D}})$ is a submersive imprinting  with restrictions.
\qed
\end{thm}
We shall denote by $X^{\delta}_{{\mathcal D},\varphi}({\mathbb R}^N)$ the set $X^{\delta_0}_{{\mathcal D},\varphi}({\mathbb R}^N)$ equipped with the M-polyfold structure coming from $\oplus$.  Denote by ${\mathcal N}$ the category whose objects are the real vector spaces ${\mathbb R}^N$ and the morphisms are the smooth maps $f:{\mathbb R}^N\rightarrow {\mathbb R}^M$.  We associate to $f$ the map $X^{\delta}_{{\mathcal D},\varphi}(f): X^{\delta}_{{\mathcal D},\varphi}({\mathbb R}^N)\rightarrow X^{\delta}_{{\mathcal D},\varphi}({\mathbb R}^M)$ defined by
$u\rightarrow f\circ u$. 
\begin{thm}
The construction $X^{\delta}_{{\mathcal D},\varphi}$ defines a functor ${\mathcal N}\rightarrow \mathsf{M}_{\text{tame}}$.
\qed
\end{thm}
We shall call a functor like $X^{\delta}_{{\mathcal D},\varphi}$ a {\bf construction functor}. Their importance becomes clear in the next subsection.
\subsection{Extension of Construction Functors}

The following definition applies to the classical and nodal case. A somewhat modified version would apply to the periodic orbit case.
Versions of the theorem below for various situations all rely on the same idea.
\begin{definition}\label{def:poly construction}
An {\bf M-polyfold construction}\index{M-polyfold construction} over ${\mathcal N}$ consists of a covariant
  functor $X$, which, to each ${\mathbb R}^N$, associates an
  M-polyfold, $X(N)$, and to each morphism $f:{\mathbb R}^N\rightarrow {\mathbb R}^L$ it associates
  an sc-smooth map $X(f):X(N)\rightarrow X(L)$.
Moreover we require that the M-polyfolds come with an additional
  structure, namely that for each object ${\mathbb R}^N$ we have a map which associates
  to a point $u\in X(N)$ a subset ${\rm im}(u)\subset {\mathbb R}^N$,
  which we call the image of $u$.
The following is assumed to hold:
  \begin{itemize}
    \item[(1)] 
      Given an open subset $U$ of ${\mathbb R}^N$ the subset of $X(N)$
      consisting of all $u$ with ${\rm im}(u)\subset U$ is open
      in $X(N)$.
    \item[(2)] 
      We have ${\rm im}(X(f)(u))=f({\rm im}(u))$.
    \item[(3)] 
      If $f,g:N\rightarrow L$ are morphisms and $u\in X(N)$, then 
      \begin{equation*}                                                   %% EQN
	f\big|_{{\rm im}(u)}=g	\big|_{{\rm
	im}(u)}\qquad\text{implies}\qquad X(f)(u)=X(g)(u)
	\end{equation*}
    \end{itemize}
    \qed
\end{definition}
Denote by ${\mathcal M}$ the category of smooth manifolds with the  smooth maps as morphisms.
The proof of the following theorem is given in \cite{FH-primer}.
\begin{thm}\label{EXTEND}
The functor $X:{\mathcal N}\rightarrow \mathsf{M} $ from a M-polyfold construction over  ${\mathcal  N}$ has a natural extension
$X:{\mathcal M}\rightarrow \mathsf{M}$,  which
  associates to a manifold $M$ in ${\mathcal M}$ a  M-polyfold $X(M)$ and to
  a smooth map $f:M\rightarrow M'$ an sc-smooth map between the M-polyfolds
  \begin{equation*}                                                       %% EQN
    X(f): X(M)\rightarrow X(M').
    \end{equation*}
Further we have a natural sc-diffeomorphism $X(N)\rightarrow X({\mathbb
  R}^N)$.
  \qed
\end{thm}

\subsection{Periodic Orbit}
Next we consider the periodic orbit local-local construction.  We assume we are given a weighted periodic orbit $\bm{\bar{\gamma}}$ in ${\mathbb R}^N$.
We define the ssc-Hilbert manifold  associate to an ordered nodal disk pair ${\mathcal D}=(D_x\sqcup D_y,(x,y))$ denoted by 
$Z_{\mathcal D}({\mathbb R}\times {\mathbb R}^N,\bm{\bar{\gamma}})$. It consists of tuples $\widetilde{u}=(\widetilde{u}^x,[\widehat{x},\widehat{y}],\widetilde{u}^y)$,
where $\widetilde{u}^x$ is a map on the punctured $D_x\setminus\{x\}$ and similarly for $y$. Moreover the following holds.  Pick
a representative $(\widehat{x},\widehat{y})$  and take  holomorphic polar coordinates $\sigma_{\widehat{x}}^+$ and $\sigma^-_{\widehat{y}}$.  Then for a suitable 
$\gamma\in [\gamma]$.
$$
\widetilde{u}^x\circ \sigma_{\widehat{x}}^+(s,t) = (Ts+c^x, \gamma(kt))+\widetilde{r}^x(s,t)\ \ \text{and}\ \ \widetilde{u}^y\circ\sigma_{\widehat{y}}^-(s',t')=(Ts'+c^y,\gamma(kt'))+\widetilde{r}^y(s',t').
$$
Here $\widetilde{r}^x$ and $\widetilde{r}^y$ are in $H^{3,\delta_0}$, where $\delta_0$ is an exponential weight.  We define level $m$ in our space as such tuples
where the latter is of regularity $(3+m,\delta_m)$.
\begin{thm}
$Z_{\mathcal D}({\mathbb R}\times {\mathbb R}^N,\bm{\bar{\gamma}})$ is a ssc-Hilbert manifold.
\qed
\end{thm}
Define the ssc-manifold with boundary  $\mathfrak{Z}=[0,1)\times Z_{\mathcal D}({\mathbb R}\times {\mathbb R}^N,\bm{\bar{\gamma}})$
and take the open neighborhood ${\mathcal V}$ of $\partial\mathfrak{Z}$ defined as the collection of all $(r,\widetilde{u})$, where either 
$r=0$ or if $r\in (0,1)$ it holds that 
\begin{eqnarray}
&\varphi(r) +c^y-c^x>0&\\
&\varphi^{-1}\left( \frac{1}{T}\cdot \left( \varphi(r)+c^y-c^x\right)\right)\in (0,1/4).&\nonumber
\end{eqnarray}
Define the set 
$$
Y^{3,\delta_0}_{{\mathcal D},\varphi}= \left(\{0\}\times Z_{\mathcal D}({\mathbb R}\times {\mathbb R}^N,\bm{\bar{\gamma}})\right)\coprod \left((0,1)\times \left(\coprod_{0<|a|<1/4} H^3(Z_a,{\mathbb R}\times {\mathbb R}^N)\right)\right).
$$
We define $\bm{\bar{\oplus}}:{\mathcal V}\rightarrow Y^{3,\delta_0}_{{\mathcal D},\varphi}$ by 
$$
\bm{\bar{\oplus}}(0,(\widetilde{u}^x,[\widehat{x},\widehat{y}],\widetilde{u}^y))= (0,(\widetilde{u}^x,[\widehat{x},\widehat{y}],\widetilde{u}^y))
$$
and 
$$
\bm{\bar{\oplus}}(r,(\widetilde{u}^x,[\widehat{x},\widehat{y}],\widetilde{u}^y))=(r,\oplus_a(\widetilde{u}^x,(\varphi(r)\ast \widetilde{u}^y)),
$$
where $a=|a|\cdot [\widehat{x},\widehat{y}]$ and $T\cdot \varphi(|a|)=\varphi(r)+c^y-c^x.$  We can also fix concentric annuli around the boundaries, say $A_x$ and $A_y$ and define
restriction maps, where we particularly note the form of $p_y$.
$$
p_x(r,\widetilde{w}) = \widetilde{w}|A_x\ \ \text{and}\ \ p_y(r,\widetilde{w})= ((-\varphi(r))\ast \widetilde{w})|A_y.
$$
We also have the map $p_{[0,1)}:Y^{3,\delta_0}_{{\mathcal D},\varphi}\rightarrow [0,1)$.
\begin{thm}
The tuple $(\bm{\bar{\oplus}},\{p_x,p_y\},p_{[0,1)})$ is a submersive imprinting with restrictions.
\qed
\end{thm}
Denote by $Y^{3,\delta}_{{\mathcal D},\varphi}({\mathbb R}\times {\mathbb R}^N,\bm{{\gamma}})$ the M-polyfold just constructed.
We can build a category whose objects are pairs $({\mathbb R}^N,\bm{\bar{\gamma}})$ and the morphisms are smooth maps
$f: ({\mathbb R}^N,\bm{\bar{\gamma}})\rightarrow ({\mathbb R}^M,\bm{{\gamma}}')$ where it is assumed $([\gamma'],T',k')=([f\circ\gamma],T,k)$.
Then we can associate to $(r,\widetilde{u})\in Y^{3,\delta}_{{\mathcal D},\varphi}({\mathbb R}\times {\mathbb R}^N,\bm{{\gamma}})$ the element
$(r,(\text{Id}\times f)\circ \widetilde{u})$.   It turns out that this defines a construction functor and using the idea from Theorem \ref{EXTEND} also this
functor has an extension to the category, where the objects are pairs $(Q,([\gamma],T,k))$. That means we consider 
maps into ${\mathbb R}\times Q$ being exponentially asymptotic to a suitable $\wt{J}$-holomorphic parameterization 
of the $k$-fold covered cylinder ${\mathbb R}\times \gamma(S^1)$.
\subsection{Classical}
In this case we consider a compact Riemann surface with smooth boundary $(\Sigma,j)$ and define the ssc-manifold $H^3(\Sigma,{\mathbb R}\times {\mathbb R}^N)$ where level $m$ is regularity $3+m$.  We can take mutually disjoint boundary annuli, say $A_i$, $i\in I$.
We have the restriction maps $p_i: H^3(\Sigma,{\mathbb R}\times {\mathbb R}^N)\rightarrow H^3(A_i)$ and the tautological imprinting 
with $\oplus=Id$.

\subsection{Using Operations}
Assume that $(S,j,D)$ is a compact nodal Riemann surface with smooth boundary and a finite group $G$ acts on $(S,j,D)$ by biholomorphic maps.
We can take a small disk structure $\bm{D}$ so that the associated union of disks is invariant. For each $D_x$ take a concentric compact boundary annulus 
$A_x$ with smooth boundary so that the union of all $A_x$ is invariant under $G$. For every $\{x,y\}\in D$ we have ${\mathcal D}_{\{x,y\}}=(D_x\sqcup D_y,\{x,y\})$
and the associated submersive $\oplus$-construction with restrictions. 
$$
\begin{CD}
@.  {\mathbb B}_{{\mathcal D}_{\{x,y\}}}\times E^{\delta}_{{\mathcal D}_{\{x,y\}}} @.\\
@.   @V \oplus VV@.\\
H^3(A_x,{\mathbb R}^N)@< p_x << X^{\delta_0}_{{\mathcal D},\varphi}({\mathbb R}^N) @> p_y>> H^3(A_y,{\mathbb R}^N)\\
@.   @Vp_{{\mathbb B}_{{\mathcal D}_{\{x,y\}}}}VV @.\\
@. {\mathbb B}_{{\mathcal D}_{\{x,y\}}} @.
\end{CD}
$$
Defining $A=\bigcup_{x\in |D|} A_x$, and  ${\mathbb B}_D$, $E^\delta$  and $X^\delta({\mathbb R}^N)$  as the obvious product
 we obtain 
$$
\begin{CD}
{\mathbb B}_D\times E^\delta @>\oplus >> X^\delta({\mathbb R}^N)@> p >> H^3(A,{\mathbb R}^N)\\
@.   @V p_{{\mathbb B}_D} VV   @.\\
@.    {\mathbb B}_D  @.
\end{CD}
$$
We represent this Lego-type piece by the following picture. It represents a part of the constructions needed to deal with the 
floors of a stable map.
 \begin{figure}[h]
\begin{center}
\boxed{\includegraphics[width=8.5cm]{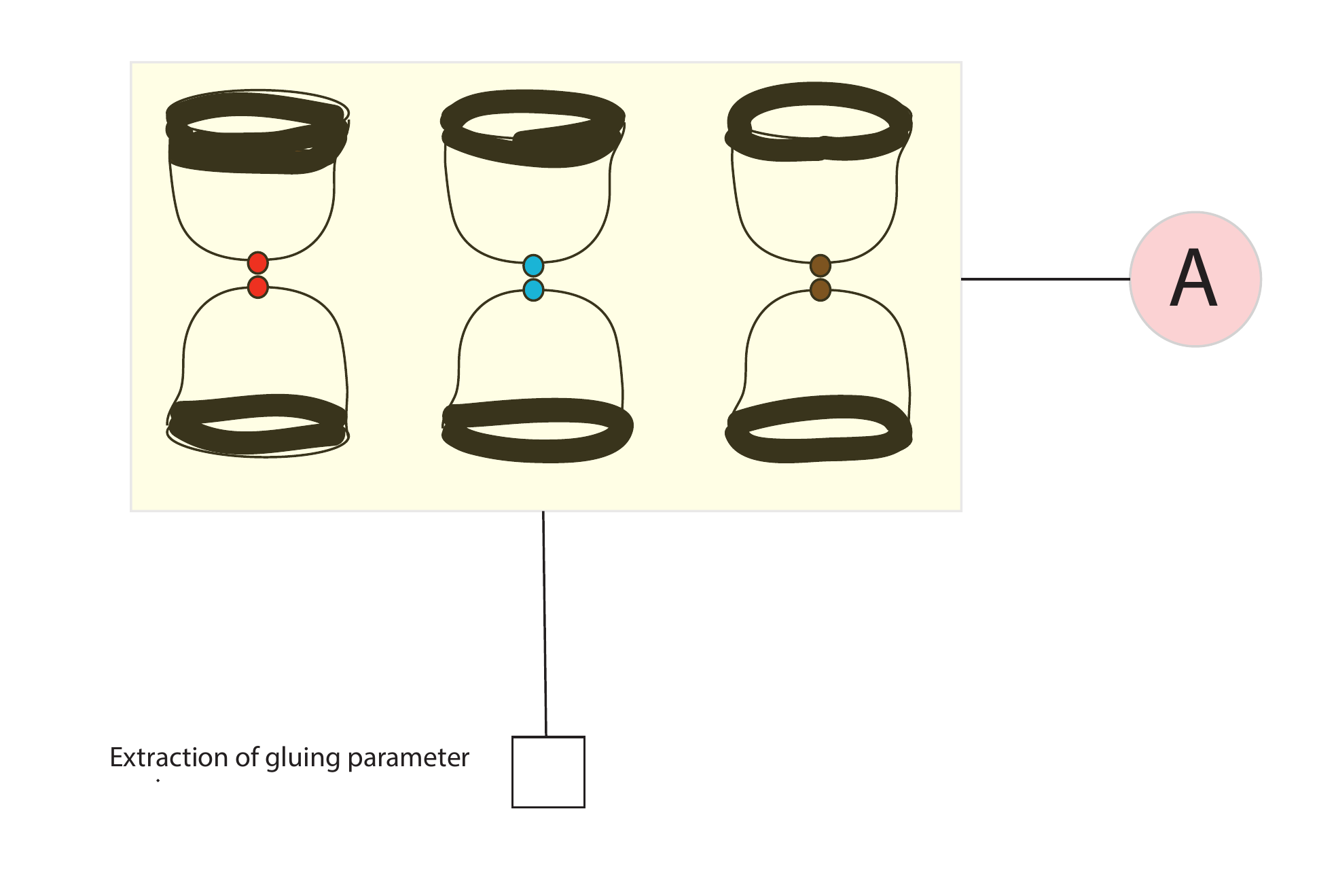}}
\end{center}
\caption{A Lego piece.}\label{FIG 2}
\end{figure}
\newpage
From $(S,j,D)$ and $\bm{D}$ we construct $(R,j)$  as follows. We remove the union of the $D_x$ but glue in the $A_x$.
Then $(R,j)$ has besides it original boundary coming from $S$ also the additional boundary with annuli whose union is $A$.
This gives a diagram as depicted in Figure \ref{FIG 3}.
 \begin{figure}[h]
\begin{center}
\includegraphics[width=7.5cm]{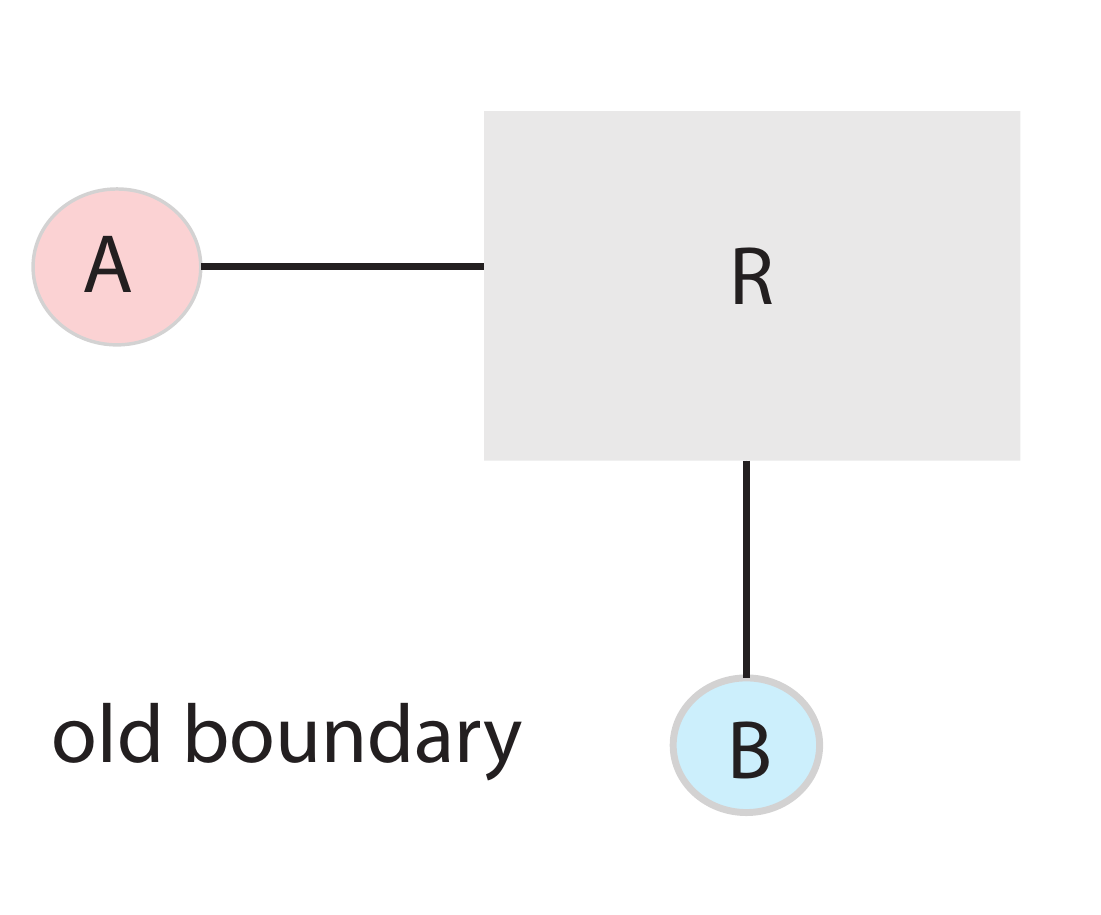}
\end{center}
\caption{Another Lego piece.}\label{FIG 3}
\end{figure}

\noindent The fibered product with respect to the $A$-part gives a M-polyfold construction functor for $(S,j,D)$ with $\bm{D}$, which considers the union of all maps 
of class $(3,\delta_0)$ with domains being the glued $S$. We have a submersion which is the extraction of the gluing parameter. Moreover, if we can partition the union 
of boundary components in unions which are $G$-invariant we have several restrictions. Assume we have a partition of the old boundary into an upper part and a lower part.
In this case we obtain a submersive imprinting with restrictions.  This is represent by the following Figure \ref{FIG 4}.
 \begin{figure}[h]
\begin{center}
\boxed{\includegraphics[width=9.5cm]{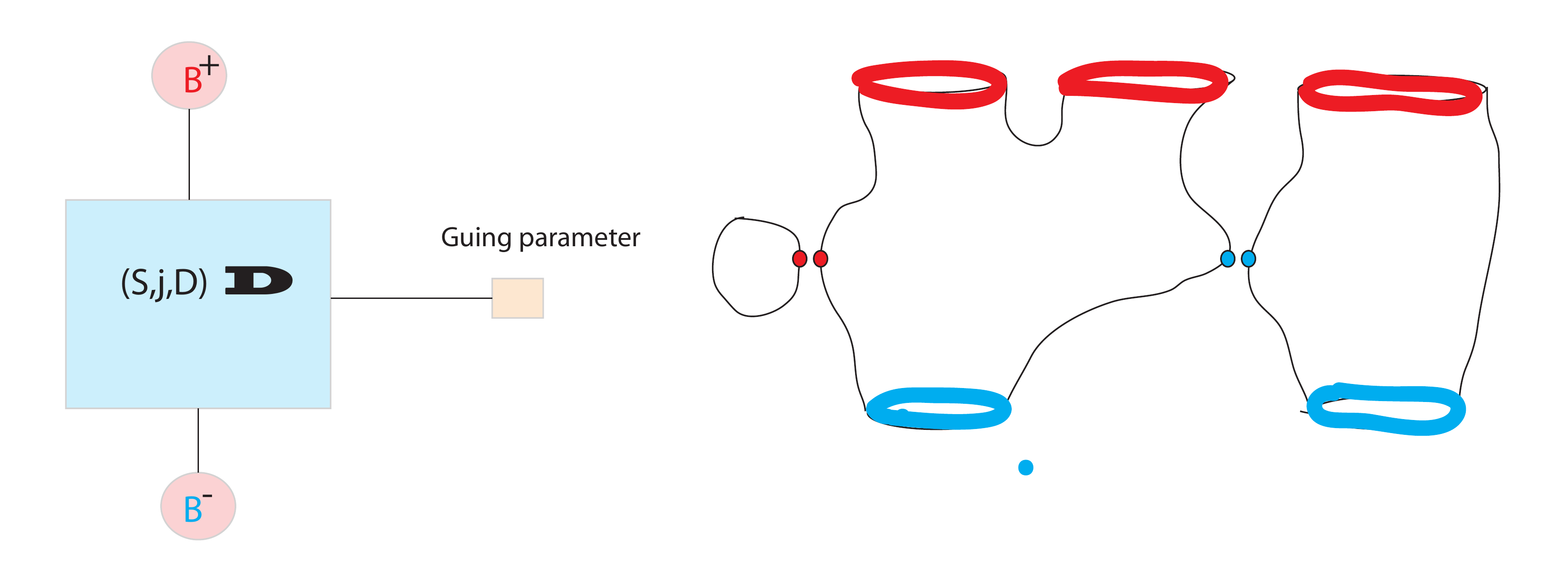}}
\end{center}
\caption{A bigger piece from a fibered product construction.}\label{FIG 4}
\end{figure}
\newpage
We note that the new construction is also a construction functor and we have extensions to manifolds.  In short we have used a fibered product construction to obtain 
from several construction functors a new one. It is based on a compact nodal Riemann surface which has a partition of its boundary into an upper part and a lower part
and a finite group acting by biholomorphic maps preserving the decomposition of the boundary. In addition one has fixed a small disk structure $\bm{D}$ and considers 
the associated glued surfaces. On these we consider maps of some Sobolev class with image in ${\mathbb R}^N$.  This is equipped with a M-polyfold structure 
via the imprinting method. Since everything behaves well with respect to smooth maps $f:{\mathbb R}^N\rightarrow {\mathbb R}^M$ we obtain a construction functor.

Associated to the periodic orbit situation we can make the following construction.  Assume we have several periodic orbits 
$\bm{\bar{\gamma}}_i$, $i=1,..,k$. We assume that the images of the various $\gamma_i$ are disjoint.
We obtain the submersive $Y_{{\mathcal D}_{\{x,y\}},\varphi}({\mathbb R}\times {\mathbb R}^N,\bm{\bar{\gamma}}_i)\rightarrow [0,1)$ and we take the product
and pull-back by the diagonal map $\Delta:[0,1)\rightarrow [0,1)^k$. 
 \begin{figure}[htp]
\begin{center}
\boxed{\includegraphics[width=9.5cm]{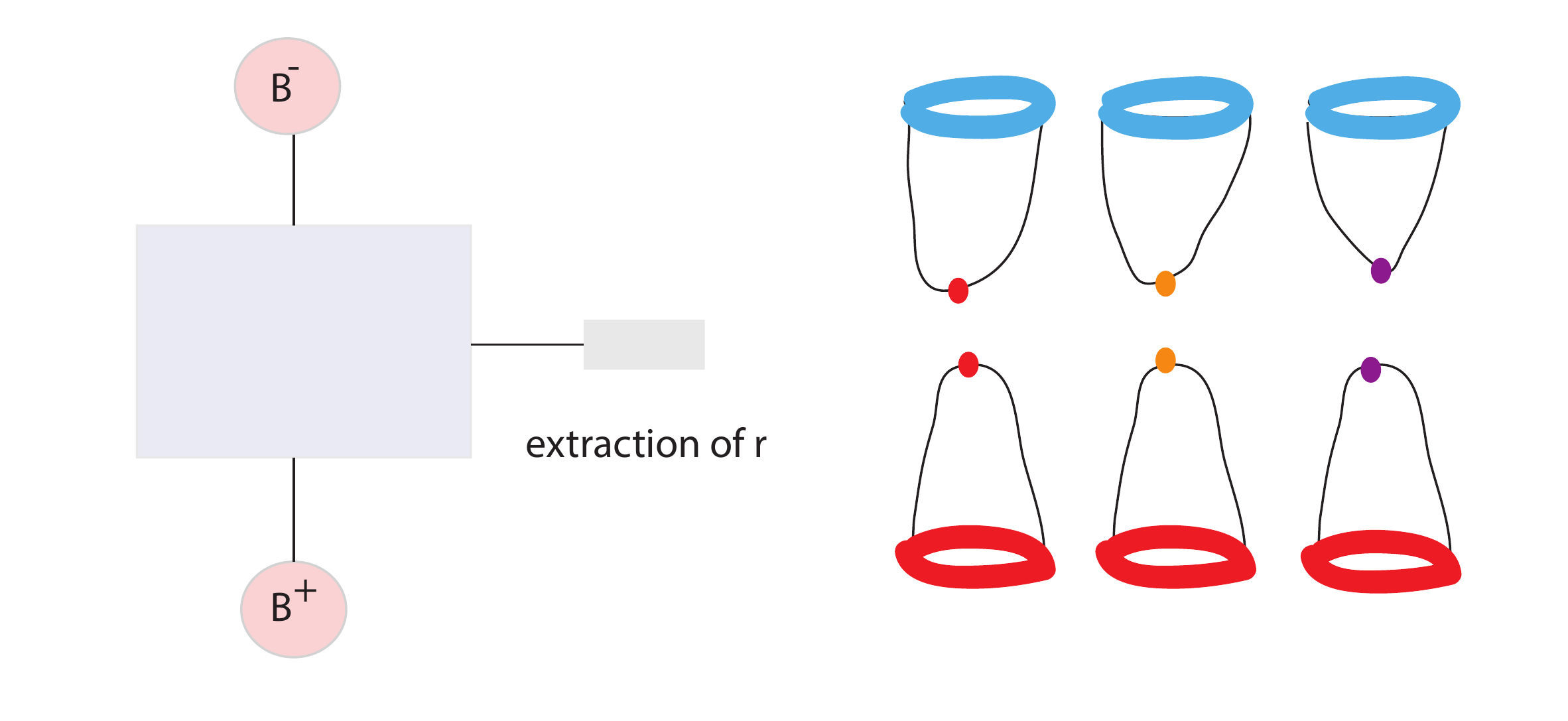}}
\end{center}
\caption{A bigger piece from a fibered product construction and pull-back.}\label{FIG 5}
\end{figure}

The two boxed displays represent construction functors with restrictions and some submersive extractions of data. We shall use them for the constructions 
related to SFT.
\newpage

\part{From Local-Local to Local}
%\newpage
We show how to obtain from the local-local theory the local theory by fibered product constructions.
\part*{Lecture 7}
\section{Constructions Associated to Stable Maps}
We shall use the results obtained so far to carry out constructions associated to stable maps.
\subsection{Data Preparation}
We start with a stable map $\alpha=(\alpha_0,\widehat{b}_1,...,\widehat{b}_k,\alpha_k)$, where each $\alpha_i$ is written 
as 
$$
\alpha_i=(\Gamma^-_i,S_i,j_i,M_i,D_i,[\widetilde{u}_i],\Gamma^+_i).
$$
 \begin{figure}[htp]
\begin{center}
\includegraphics[width=14.5cm]{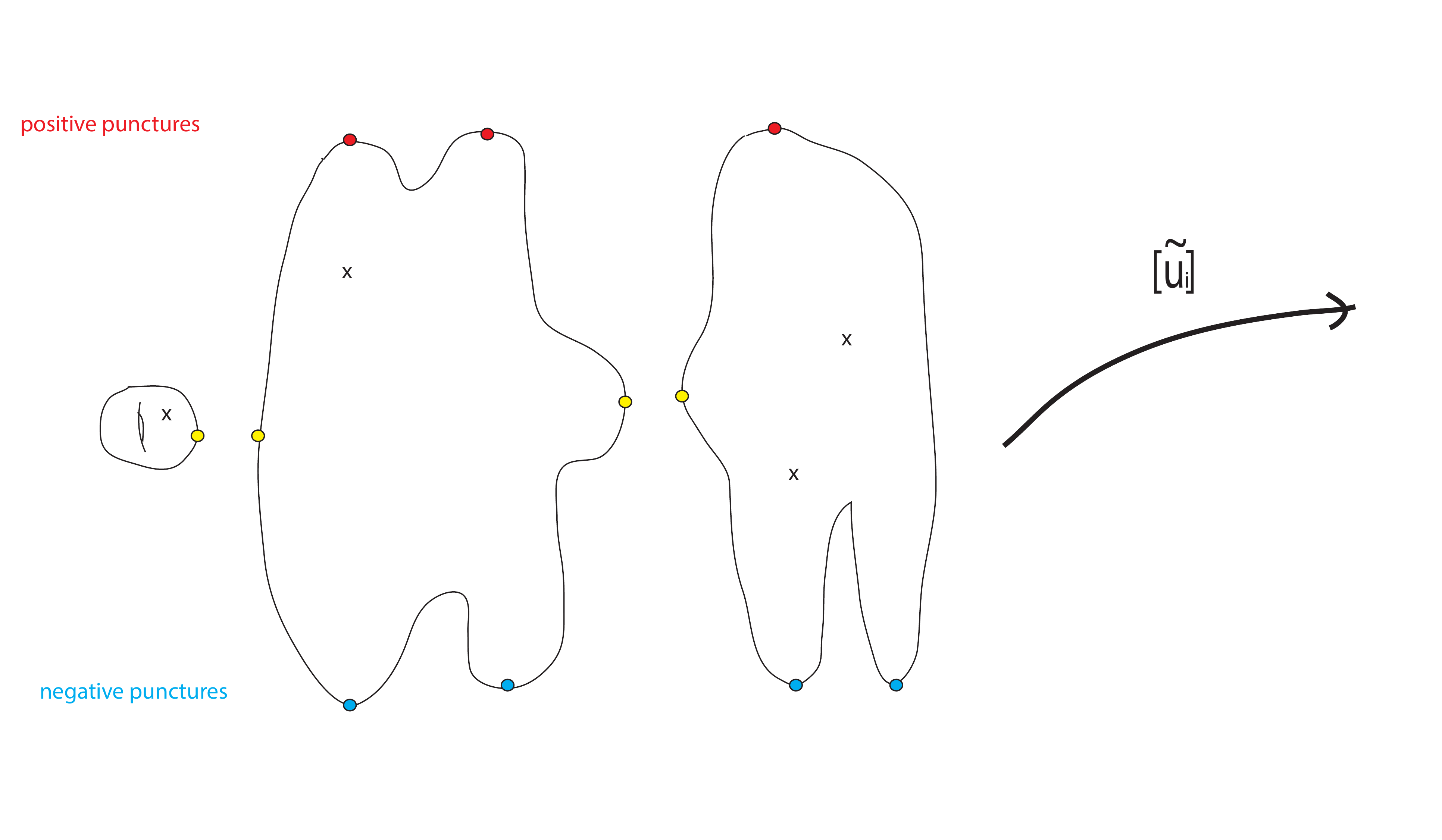}
\end{center}
\caption{A floor of a stable building. A representative of $[\widetilde{u}_i]$ has image in ${\mathbb R}\times Q$.}\label{FIG 6}
\end{figure}
The stable map $\alpha$ comes with a finite automorphism group $G$ which preserves the floor structure. A representative $\widetilde{u}_i$
of $[ \widetilde{u}_i]$ has its image in ${\mathbb R}\times Q$. 
We extract $\sigma=(\sigma_0,b_1,...,b_k,\sigma_k)$ with $\sigma_i=(\Gamma^-_i,S_i,j_i,D_i,\Gamma^+_i)$ essentially ignoring marked points, but we remember where they were.
 \begin{figure}[htp]
\begin{center}
\includegraphics[width=6.5cm]{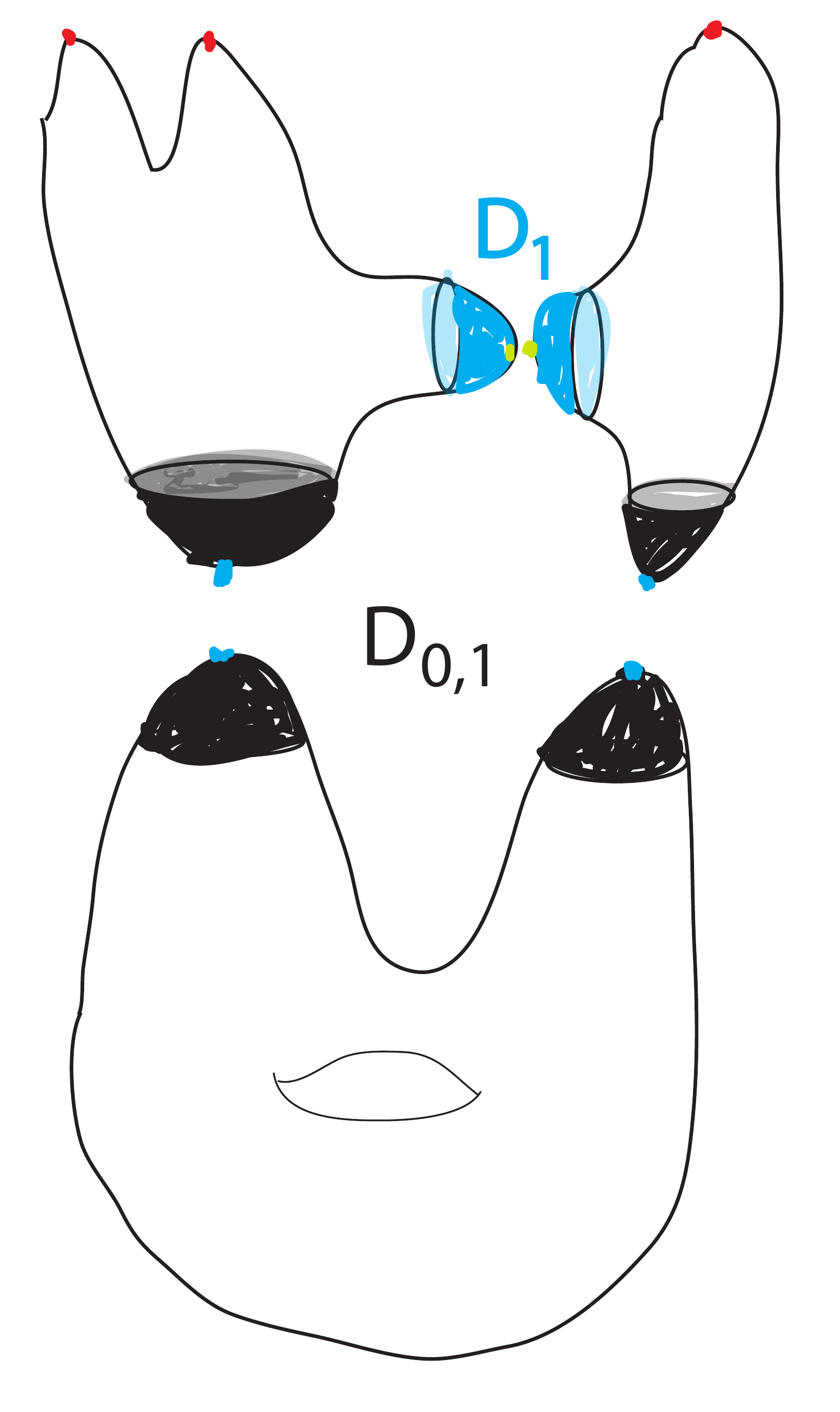}
\end{center}
\caption{The domain of a stable building with two floors and the fixed disks.}
\end{figure}
We have the group $G$ acting on $\sigma_i$ by biholomorphic maps.  We fix a finite $G$-invariant subset $\rup_i$, called an \textbf{anchor set}, which is disjoint from the points in $|D_i| $, $M_i$, 
and $\Gamma_i=\Gamma_i^-\cup\Gamma^+_i$.  We fix closed disks with smooth boundaries around the nodal points in $|D_i|$ for $i\in \{0,...,k\}$,
 so that the union is invariant under
$G$. We denote the collection by $\bm{D}_i$. We fix disks around $\Gamma^+_i$ giving $\bm{D}_i^+$ for $i\in \{0,...,k-1\}$,  and around $\Gamma^-_i$, for 
$i\in \{1,...,k\}$,  giving $\bm{D}^-_i$.
We pick all these  disks small enough so that they are mutually disjoint, and are also disjoint from the points in $\rup_i$ and $M_i$, $\Gamma^-_0$ and $\Gamma^+_k$.
We take compact annuli in the disks 
associated to $\Gamma^\pm_i$ and denote their union by $A^\pm_i$.  Removing $\bm{D}^\pm_i$ from $S_i$ and adding $\bm{A}^\pm_i$ we obtain
\begin{eqnarray}\label{EQNX1}
(A_i^-, R_i,j_i,\bm{D}_i,\rup_i,A^+_i)\ \ \text{for}\ i\in \{1,...,k-1\}
\end{eqnarray}
and
\begin{eqnarray}
(\Gamma^-_0, R_0,j_0,\bm{D}_0,\rup_0,A^+_0)\ \ \text{and}\ \  (A^-_k,R_k,j_k,\bm{D}_k,\rup_k,\Gamma^+_k).
\end{eqnarray}
We also obtain $\bm{D}_{i-1,i}$ for $i\in \{1,...,k\}$ where $\bm{D}_{i-1,i}$ is the union of the ordered disk pairs
$$
(D_z\cup D_{b_i(z)},(z,b_i(z)),\ z\in\Gamma^+_{i-1}
$$
with associated
\begin{eqnarray}\label{EQNP1}
(A_{i-1}^+,\bm{D}_{i-1,i},A^-_i).
\end{eqnarray}
What we are doing is to consider data associated to the floors as (\ref{EQNX1}) and data for the interfaces, which we shall introduce soon.
 \begin{figure}[htp]
\begin{center}
\includegraphics[width=14.5cm]{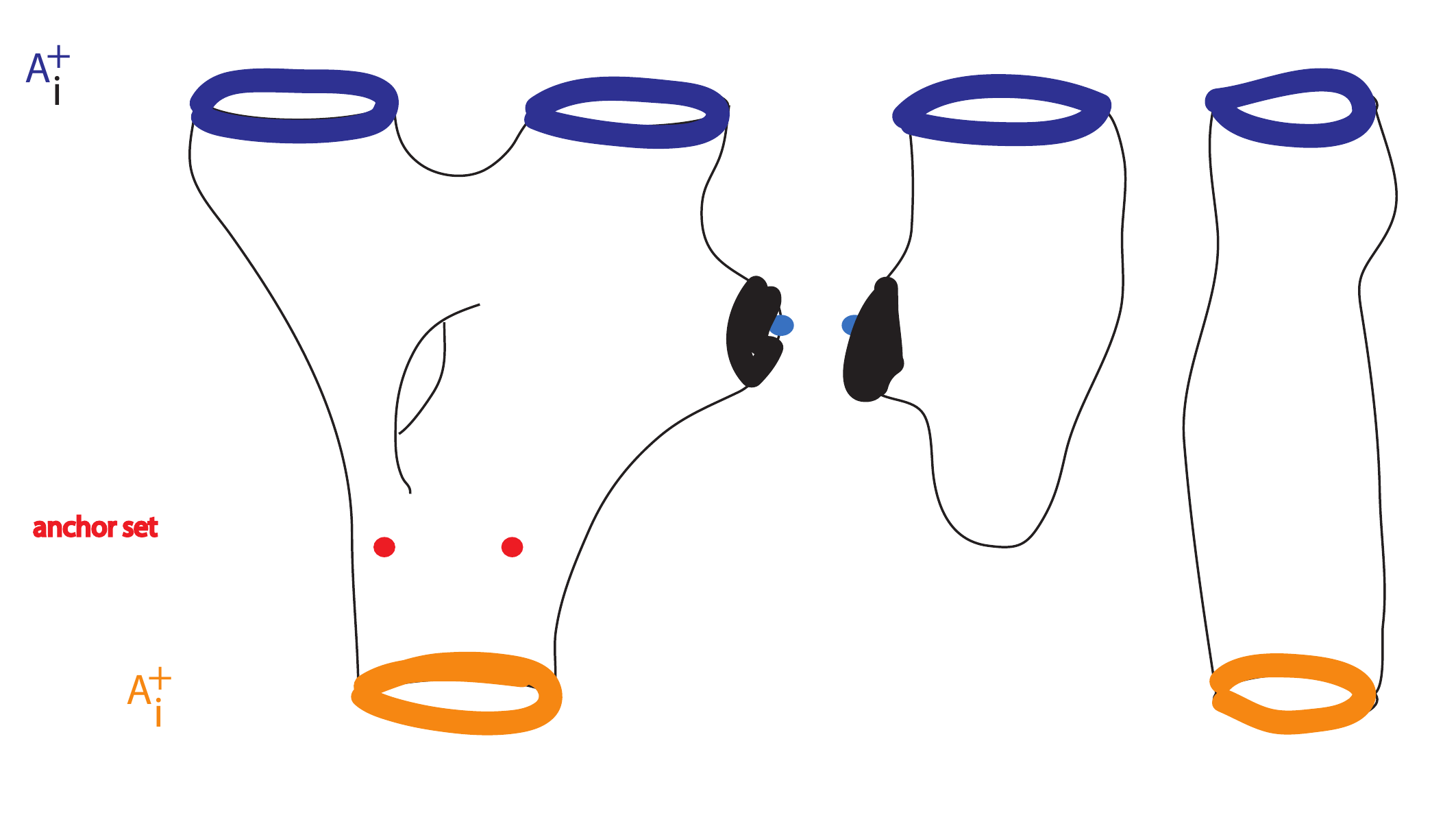}
\end{center}
\caption{This shows the situation for a floor. The interfaces are being dealt with separately.}\label{Fig12-3}
\end{figure}
\subsection{M-Polyfolds for Floors and Interfaces}
We shall first consider the situation on the floors and then on the interfaces. The following are all imprinting constructions previously discussed.
\subsubsection{Floors} Associated to what is represented in Figure \ref{Fig12-3}   we have a construction functor which associates to ${\mathbb R}\times {\mathbb R}^N$ maps on the glued surfaces as discussed earlier. We take the 
sub-M-polyfold, which is a global sc-smooth retraction, consisting of the elements for which the anchor average with respect to $\rup_i$ vanishes.
$$
\text{av}_{\tiny\rup_i}(\widetilde{u}) := \frac{1}{\sharp {\tiny\rup}_i}\cdot \sum_{z\in {\tiny\rup}_i} a_i(z) =0.
$$
Let us denote the functor associated to the $i$-th floor by $\mathsf{F}_i$. There is a submersive map extracting the gluing parameter for the nodal points from the $i$-th floor
$$
p_{D_i}: \mathsf{F}_i\rightarrow {\mathbb B}_{D_i}
$$
and there are the restrictions associated to the lower and upper union of annuli $A^\pm_i$ giving us
$$
\mathsf{A}^-_i:=H^3(A_i^-,{\mathbb R}\times {\mathbb R}^N)\xleftarrow{\text{rest}^-_i} \mathsf{F}_i \xrightarrow{\text{rest}^+_i} H^3(A^+_i,{\mathbb R}\times {\mathbb R}^N)=:\mathsf{A}^+_i
$$
Here are the {\bf submersive $i$-th floor-Lego with restrictions}, where for the middle diagram we have that  $i\in \{1,....,k-1\}$
$$
\begin{array}{ccc}
\boxed{
\begin{CD}
@.\mathsf{A}^+_0\\
@.   @AA\text{rest}^+_0 A\\
Z_0@>\oplus>>\mathsf{F}_0@> p_{D_0} >>   {\mathbb B}_{D_0}\\
@.\\
@.
\end{CD}
}
&
\boxed{
\begin{CD}
@.\mathsf{A}^+_i \\
@.@AA\text{rest}^+_i A\\
Z_i@>\oplus>>\mathsf{F}_i@> p_{D_i} >>   {\mathbb B}_{D_i}\\
@.@VV\text{rest}^-_i V\\
@. \mathsf{A}^-_i
\end{CD}
}
&
\boxed{
\begin{CD}
@.@.\\
@.@.\\
Z_k@>\oplus>>\mathsf{F}_k@> p_{D_k} >>   {\mathbb B}_{D_k}\\
@.@VV\text{rest}^-_k V\\
@. \mathsf{A}^-_k
\end{CD}
}
\end{array}
$$
\subsubsection{Interfaces}
Associated to  $(A_{i-1}^+,\bm{D}_{i-1,i},A^-_i)$ for $i\in \{1,...,k\}$ we have the submersive 
$$
\boxed{
\begin{CD}
@.\mathsf{A}_i^-\\
@.@AA\text{rest}_i^- A\\
Z_{i-1,i}@>\oplus>>\mathsf{F}_{i-1,i}@> p_{[0,1)}>>  [0,1)\\
@.@VV\text{rest}_{i-1}^+ V\\
@.\mathsf{A}_{i-1}^+
\end{CD}
}
$$
Recall that the restriction $\text{rest}^-_i$  is obtained by first restricting the map  to the union of annuli $A^-_i$ and then using the ${\mathbb R}$-action 
shifting by  $-\varphi(r_i)$ if $r_i\neq 0$. We can line up all these construction functors and use fibered product and obtain a new submersive construction functor
$\mathsf{F}_{\bm{\sigma},{\tiny\rup},\varphi}\rightarrow [0,1)^k\times {\mathbb B}_D$ which is a M-polyfold construction for a given ${\mathbb R}\times {\mathbb R}^N$ together with a weighted periodic orbit assignment.

\subsection{Assembling the Pieces}
We line up the diagrams constructed in Lecture 7 and take a fibered product. 
$$
\resizebox{0.53\hsize}{!}{
\boxed{
\begin{CD}
Z_k@>\oplus>>\mathsf{F}_k@> p_{D_k} >>   {\mathbb B}_{D_k}\\
@.@VV\text{rest}^-_k V\\
@.\mathsf{A}^-_k\\
@.  ......\\
@.\mathsf{A}_i^-\\
@.@AA\text{rest}_i^- A\\
Z_{i-1,i}@>\oplus>>\mathsf{F}_{i-1,i}@> p_{[0,1)}>>  [0,1)\\
@.@VV\text{rest}_{i-1}^+ V\\
@.\mathsf{A}_{i-1}^+\\
@.  ....\\
@.\mathsf{A}_1^-\\
@. @AA\text{rest}_1^- A\\
Z_{0,1}@>\oplus>>\mathsf{F}_{0,1}@> p_{[0,1)}>>  [0,1)\\
@.@VV\text{rest}_{0}^+ V\\
@.\mathsf{A}_{0}^+\\
@.@AA\text{rest}^+_0 A\\
Z_0@>\oplus>>\mathsf{F}_0@> p_{D_0} >>   {\mathbb B}_{D_0}
\end{CD}  
}
}
$$
The associated fibered product diagram has the form $Z\xrightarrow{\oplus} X$ and is an imprinting construction.
We shall explain the latter in the following.  Rather than working with $X$ we introduce another set which is in natural bijection 
to $X$ and whose elements  are more intuitive and will be used in the uniformizer construction.
More precisely we shall construct an open subset of a ssc-manifold ${\mathcal O}$
and a set $Z^3_{\bm{\sigma},{\tiny\rup},\varphi}({\mathbb R}\times {\mathbb R}^N,\bar{\digamma})$ fitting into the following diagram
\begin{eqnarray}\label{refdia}
\begin{CD}
Z @>\oplus>>  X\\
@VVV    @VVV\\
{\mathbb B}_D\times{\mathcal O} @>\bar{\oplus}>> Z^3_{\bm{\sigma},{\tiny\rup},\varphi}({\mathbb R}\times {\mathbb R}^N,\bar{\digamma}).
\end{CD}
\end{eqnarray}
where the first vertical arrow is a ssc-diffeomorphism,  and the second one a bijection. Hence $\bar{\oplus}$ is an imprinting.
The diagram ${\mathbb B}_D\times {\mathcal O} \rightarrow Z^3_{\bm{\sigma},{\tiny\rup},\varphi}({\mathbb R}\times {\mathbb R}^N,\bar{\digamma})$ is  more useful
for the further constructions as we already said. We carry this out in the next lecture. We shall refer to it as the work-horse M-polyfold.

\newpage
\part*{Lecture 8}
\section{The Workhorse M-polyfold}
We shall construct the lower part of the diagram (\ref{refdia}). We assume that we started with a stable map
$\alpha=(\alpha_0,\wh{b}_1,...,\wh{b}_k,\alpha_k)$ and extracted the domain data $\sigma$ to carry out the previously described constructions.
\subsection{Two  ssc-Manifolds}
Given $\sigma=(\sigma_0,b_1,...,b_k,\sigma_k)$, where 
$$
\sigma_i=(\Gamma^-_i,S_i,j_i,D_i,\Gamma^+_i)
$$
 we assume that we have an assignment $\bar{\digamma}$ 
which associates to a puncture a weighted periodic orbit. Further we assume that the assignment is compatible with $b_i$.
Finally we assume that we are given an action of a finite group $G$ preserving the floor structure and acting by biholomorphic maps
and preserving the other structure which was given. We fix an anchor set $\rup=\rup_0\sqcup....\sqcup\rup_k$ which is 
invariant under $G$ and make the following definition.
\begin{definition}
The  ssc-manifold $Z^3_{\sigma,{\tiny\rup}}({\mathbb R}\times {\mathbb R}^N,\bar{\digamma})$
consists of tuples
$$
\wt{u}:=(\wt{u}_0,\wh{b}_1,....,\wh{b}_k,\wt{u}_k),
$$
 where $\wt{u}_i$ is of class $(3,\delta_0)$ asymptotic to the weighted
periodic orbits  prescribed by $\bar{\digamma}$ so that 
the data across interfaces is $\wh{b}_i$ matching and the anchor averages vanish.
\qed
\end{definition}

Next we consider the open subset ${\mathcal O}$ of $[0,1)^k\times Z^3_{\sigma,{\tiny\rup}}({\mathbb R}\times {\mathbb R}^N,\bar{\digamma})$ defined as follows.
\begin{definition}\label{DEFX8.2}
The open subset ${\mathcal O}$
consists of all tuples $(r_1,...,r_k,\wt{u})$ where $\wt{u}=(\wt{u}_0,\wh{b}_1,...,\wh{b}_k,\wt{u}_k)$ so that the following holds.
Either $r_i=0$ or if $r_i\in (0,1)$ we have for $z\in \Gamma^+_{i-1}$
\begin{eqnarray}
&\varphi(r_i)-c^z(\wt{u})+c^{b_i(z)}(\wt{u})>0&\\
& \varphi^{-1}\left( \frac{1}{T_z}\cdot \left( \varphi(r_i)-c^z(\wt{u})+c^{b_i(z)}(\wt{u})\right)\right)\in (0,1/4).&\nonumber
\end{eqnarray}
Here $c^z(\wt{u})$ are the asymptotic constants which we obtain introducing holomorphic polar coordinates on the disks of the 
small disk structure. Further $T_z$ is the period associated to the periodic orbit associated to $z$ and $b_i(z)$, respectively.
\qed
\end{definition}
Then ${\mathbb B}_D\times {\mathcal O} $ is a ssc-manfold which will be important for the further constructions.

\subsection{Definition of the Basic Space}
We fix for every nodal pair in $D_i$ and for every interface pair $(z,b_i(z))$, where $z\in \Gamma^+_{i-1}$ 
a compact disk pair with smooth boundary so that the union is invariant under the group action.
Hence we obtain a small disk structure $\bm{D}$.  A gluing parameter $\widetilde{\mathfrak{a}}$ is a map associating to $\{x,y\}\in D_i$ an element $a_{\{x,y\}}\in {\mathbb B}_{\{x,y\}}$ and to $(z,z_i)\in D_{i-1,1}$
an element $a_{(z,z')}\in {\mathbb B}_{\{z,z'\}}$.  We shall write $\mathfrak{a}_i$ for the restriction to $D_i$ and $\mathfrak{a}_{i-1,i}$ for the restriction to $D_{i-1,i}$.  We assume that the disks do not contain points in $\rup$ and $M$.
We denote the smooth manifold of gluing parameters by ${\mathbb B}$. 
\begin{definition}
A gluing parameter $\widetilde{\mathfrak{a}}$ for $\sigma$ is {\bf admissible} provided for every $i\in \{1,...,k\}$ it holds that 
either all $a_{(z,z')}=0$ or all are nonzero for $(z,z')\in D_{i-1,i}$. If for some $i$ we have that $\mathfrak{a}_{i-1,i}\equiv 0$ then we say 
we a {\bf nontrivial $(i-1,i)$-interface}. The subset of ${\mathbb B}$ consisting of  admissible gluing parameter is denoted by ${\mathbb B}_{\text{ad}}$.\qed
\end{definition}
 \begin{figure}[htp]
\begin{center}
\includegraphics[width=7.0cm]{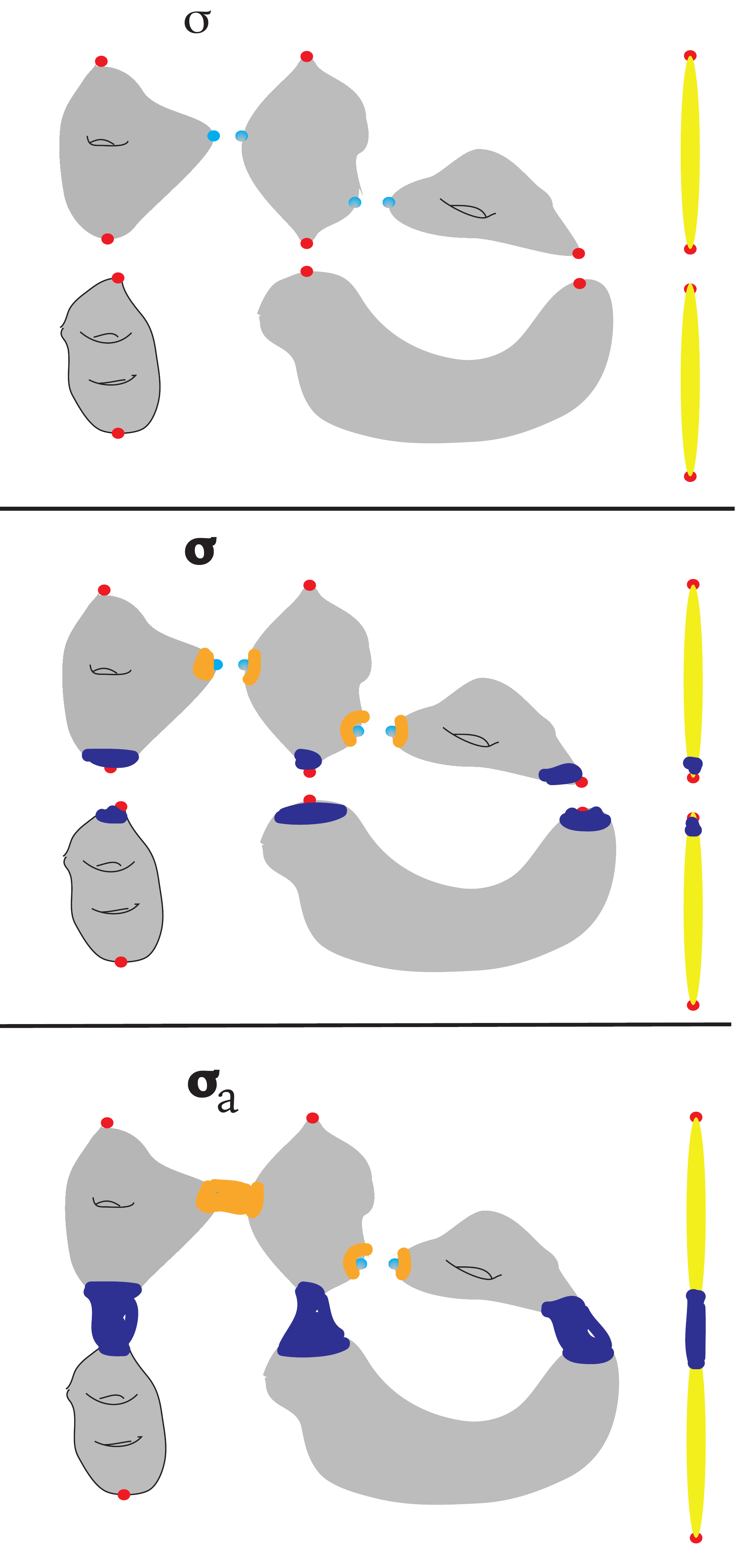}
\end{center}
\caption{$\sigma$, the enhanced $\bm{\sigma}$, and a glued enhanced $\bm{\sigma}_{\mathfrak{a}.}$}\label{FIG 7}
\end{figure}
With the small disk structure we obtain the enhanced $\bm{\sigma}=(\bm{\sigma}_0,b_1,...,b_k,\bm{\sigma}_k)$, where 
$\bm{\sigma}_i=(\Gamma^-_i,S_i,j_i,\bm{D}_i,\Gamma^+_i)$.
Given $\widetilde{\mathfrak{a}}$ we have the associated interface sequence. For this define $\ell=\ell(\widetilde{\mathfrak{a}})$ to be the number
of elements $i\in \{1,...,k\}$ such that $\mathfrak{a}_{i-1,i}\equiv 0$. We can list these elements as $0<i_1<....<i_{\ell}\leq k$ and set $i_0=0$ and $i_{\ell+1}=k+1$.
Hence we obtain the map
$$
\widetilde{\mathfrak{a}}\rightarrow (i_1(\widetilde{\mathfrak{a}}),...,i_{\ell(\widetilde{\mathfrak{a}})}(\widetilde{\mathfrak{a}})).
$$
We define $\sigma_{\widetilde{a}}^e$ for $e=0,...,\ell$ by 
$$
\sigma_{\widetilde{a}}^e=(\Gamma^-_{i_e},S^e_{\widetilde{\mathfrak{a}}},j^e_{\widetilde{\mathfrak{a}}},D^e_{\widetilde{\mathfrak{a}}},\Gamma^+_{i_{e+1}-1}).
$$
Here we glue for $e\in \{0,...,\ell\}$ the surfaces $S_{i_e},...,S_{i_{e+1}-1}$ at their nodes and at their trivial interfaces $(i_e,i_e+1),...,(i_{e+1}-2,i_{e+1}-1)$.
The nodal pairs $D^e_{\widetilde{\mathfrak{a}}}$ consist of those elements in $D_{i_e}\sqcup..\sqcup D_{i_{e+1}-1}$ which have vanishing gluing parameters.
We obtain 
$$
\sigma_{\widetilde{a}}= (\sigma_{\widetilde{\mathfrak{a}}}^0,b_{i_1(\widetilde{\mathfrak{a}})},...,b_{i_{\ell}(\widetilde{\mathfrak{a}})},\sigma^{\ell}_{\widetilde{\mathfrak{a}}}).
$$
The disks of the small disk structure $\bm{D}$ define a small disk structure for $\sigma_{\widetilde{\mathfrak{a}}}$ denoted by $\bm{D}_{\widetilde{\mathfrak{a}}}$ and we obtain the enhanced 
$$
\bm{\sigma}_{\widetilde{\mathfrak{a}}}=(\bm{\sigma}^0_{\widetilde{\mathfrak{a}}},\bm{D}_{i_1-1,i_1},..,\bm{D}_{i_{\ell}-1,i_{\ell}},\bm{\sigma}^{\ell}_{\widetilde{\mathfrak{a}}}).
$$
Given the original periodic orbit assignment $\bar{\digamma}$  the restriction to 
$$
\Gamma_{\widetilde{\mathfrak{a}}} =\Gamma_0^-\cup\Gamma^+_{i_1-1}\cup\Gamma^-_{i_1}.....\cup\Gamma_{i_{\ell}}^-\cup\Gamma^+_{k}
$$
is denoted by $\bar{\digamma}_{\widetilde{\mathfrak{a}}}$.  Hence we obtain  the collection
$$
\boxed{{{\{
(\bm{\sigma_{\widetilde{\mathfrak{a}}}},\bar{\digamma}_{\widetilde{\mathfrak{a}}})\ |\ \widetilde{\mathfrak{a}}\ \text{admissible gluing parameter} \}}}}
$$
We also have the original anchor set $\rup=\rup_0\sqcup...\sqcup \rup_k$.  Given an admissible gluing parameter $\widetilde{\mathfrak{a}}$
we have the associated anchor set $\rup_{\widetilde{\mathfrak{a}}}$ associated to $ \bm{\sigma_{\widetilde{\mathfrak{a}}}}$ defined by 
$$
\boxed{{{\mathlarger{
\rup_{\widetilde{\mathfrak{a}}}=\rup_{i_0}\sqcup \rup_{i_1}\sqcup...\sqcup \rup_{i_{\ell}}}}}}
$$
We shall also use the remnants of the omitted $\rup_i$ which can be naturally identified with subsets of the glued $ \bm{\sigma_{\widetilde{\mathfrak{a}}}}$.
\begin{figure}[htp]
\begin{center}
\includegraphics[width=7.0cm]{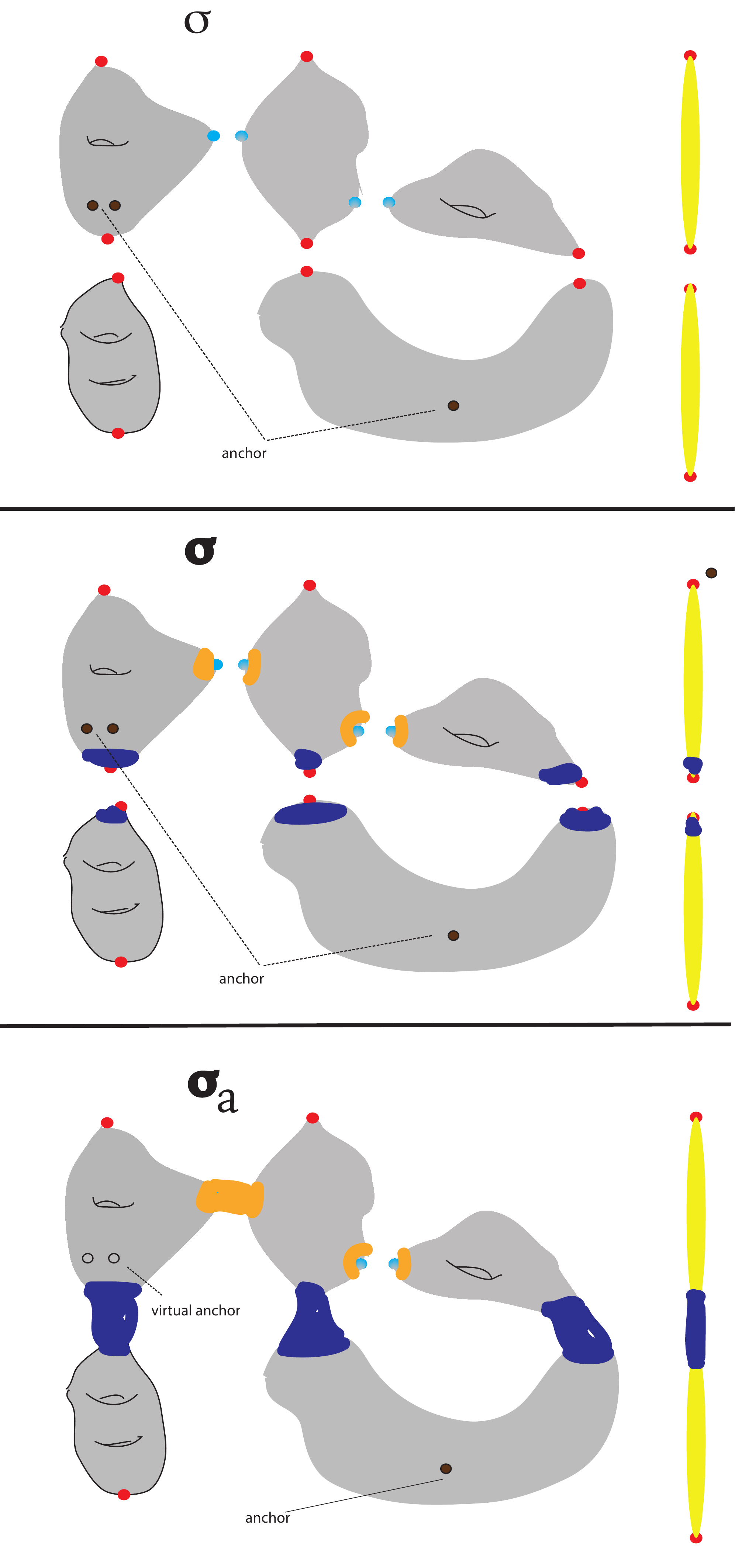}
\end{center}
\caption{Anchor and virtual anchor.}\label{FIG 8}
\end{figure}
For every admissible $\widetilde{\mathfrak{a}}$ we have $\sigma_{\widetilde{\mathfrak{a}}}$ and we consider 
maps on the underlying (punctured) domain of class $(3,\delta_0)$, which are asymptotic to the periodic orbits prescribed by 
$\bar{\digamma}_{\widetilde{\mathfrak{a}}}$. Moreover for the interfaces we have matching data $\wh{b}_{i_e}$. In addition 
we require that the anchor averages vanish for the $\rup_{i_e}$ for $e=1,...\ell$ and the following relationship with virtual anchor values.
Writing the data as $\wt{w}=(\wt{w}_0,\wh{b}_{i_1},...,\wh{b}_{i_{\ell}},\wt{w}_{\ell})$, we require
\begin{itemize}
\item[(1)] $\text{av}_{{\tiny\rup}_{i_e}}(\widetilde{w})=0$ for $e\in \{0,...,\ell\}$.
\item[(2)] For every $e\in \{0,...,\ell\}:$ $\text{av}_{{\tiny\rup}_{i-1}}(\widetilde{w})< \text{av}_{{\tiny\rup}_{i}}(\widetilde{w})$ for $i_e<i<i_{e+1}$.
\end{itemize}
Finally we define the set
$Z^{3,\delta_0}_{\bm{\sigma},{\tiny\rup},\varphi}({\mathbb R}\times {\mathbb R}^N,\digamma)$ to consist of all the tuples of maps just described.
\subsection{The Imprinting}
Recall the ssc-manifold ${\mathbb B}_{D}\times {\mathcal O}$.  Given  $(\mathfrak{a},(r_1,...,r_k,\wt{u}))$ we first construct an admissible gluing parameter $\wt{\mathfrak{a}}$ in ${\mathbb B}_{ad}$ as follows.  
\begin{definition}\label{DEFX8.3.1}
We set $\wt{\mathfrak{a}}|D=\mathfrak{a}$.
For $i\in \{0,...,k\}$ we define for $z\in\Gamma^+_{i-1}$ the gluing parameter $a_{(z,b_i(z))}\cdot \wh{b}(z)$ as follows.  If $r_i=0$ we put
$a_{(z,b_i(z))}=0$ and otherwise we define it by 
$$
T\cdot \varphi(|a_{(z,b_i(z))}|) = \varphi(r_i) + c^{b_i(z)}(\wt{u}) - c^z(\wt{u}).
$$
\qed
\end{definition}
Hence, given $(\mathfrak{a},(r_1,...,r_k,\wt{u}))$ we obtain the total (admissible) gluing parameter $\wt{\mathfrak{a}}$ and an admissible $\sigma_{\wt{\mathfrak{a}}}$.
The imprinting is defined as follows.  Given $(\mathfrak{a},(r_1,...,r_k,\wt{u}))$ we construct $\wt{\mathfrak{a}}$ and consider its sequence 
of nontrivial interface indices  $0=i_0<i_1<...< i_{\ell}<i_{\ell+1}=k+1$.  We take the data $\wt{u}=(\wt{u}_0,\wh{b}_1,....\wh{b}_k,\wt{u}_k)$
and produce the shifted $\wt{u}^\ast$ as follows.  For $i\in \{0,1,...,k\}$ we define with $i_e\leq i <i_{e+1}$, $e\in \{0,...,\ell\}$
$$
\wt{u}^\ast_i = (\varphi(r_{i_e+1})+..+\varphi(r_i))\ast \wt{u}_i.
$$
Then we define $\wt{w}_e$ by ordinary gluing 
$$
\wt{w}_e:= \oplus_{\wt{\mathfrak{a}}_e}(\wt{u}^{\ast}_{i_e},...,\wt{u}^{\ast}_i)
$$
and finally
$$
\wt{w}=(\wt{w}_0,\wh{b}_{i_1},...,\wh{b}_{i_{\ell}},\wt{w}_{i_{\ell}}),
$$
which is an element in $Z^{3,\delta_0}_{\bm{\sigma},{\tiny\rup},\varphi}({\mathbb R}\times {\mathbb R}^N,\digamma)$.
\begin{thm}\label{thmk}
$\bar{\oplus}:{\mathbb B}_D\times {\mathcal O}\rightarrow Z^{3,\delta_0}_{\bm{\sigma},{\tiny\rup},\varphi}({\mathbb R}\times {\mathbb R}^N,\digamma)$ is an imprinting and defines als a construction functor. 
\end{thm}
As a consequence we can define $Z^{3,\delta}_{\bm{\sigma},{\tiny\rup},\varphi}({\mathbb R}\times Q,\digamma)$.
Note that Theorem \ref{thmk} is concerned with the bottom horizontal line in (\ref{refdia}) and we leave it to the reader to fill in the 
vertical arrows.
\subsection{Transversal Constraints}
Relevant for us is a suitable sub-M-polyfold associated to picking a stabilization set and suitable transversal constraints.
Recall that we started with a stable map $\alpha$. 
An element  $\wt{w}$ of   the M-polyfold  $Z^3_{\bm{\sigma},{\tiny\rup},\varphi}({\mathbb R}\times Q,{\bar{\digamma}})$ 
 takes the form
$$
\wt{w}=(\wt{w}_0,\wh{b}_{i_1},...,\wh{b}_{i_\ell},\wt{w}_e)
$$
where the   $\wt{w}_e$  are defined on a possibly glued surface $\sigma_{\wt{\mathfrak{a}}}$,
where $\wt{\mathfrak{a}}$ is an sc-smooth function of $\wt{w}$. For unglued interfaces the transition and relationship between
the asymptotic periodic orbits is given by the $\wh{b}_{i_e}$, where $i_1<...<i_{\ell}$ is the sequence of nontrivial interface 
indices. We recall that for such an element the anchor values
$\text{av}_{{\tiny\rup}_i}(\wt{w})$ are defined for $i\in \{0,...,k\}$.  Here the sets $\rup_i$ are naturally identified 
as subsets of the underlying glued surface associated to $\sigma_{\wt{\mathfrak{a}}}$. 
For $\wt{w}$ with nontrivial interface sequence $0<i_1<..< i_{\ell}<k$ it follows by definition of the 
M-polyfold $Z^3_{\bm{\sigma},{\tiny\rup},\varphi}({\mathbb R}\times Q,{\bar{\digamma}})$ that 
$$
\text{av}_{\tiny\rup_{i_e}}(\wt{w})=0\ \ \text{for}\ e\in \{0,...,\ell\}.
$$
 Let $\Xi$ be a finite subset of 
the underlying domain of $\sigma$, which does not belong to the disks of the small disk structure, and is also disjoint 
from the  anchor points and the points in $\Gamma_0^-\sqcup\Gamma^+_k$, the points in $M$ and nodal points.
We assume also  that $\Xi$ is invariant under $G$. Since $G$ preserves floors the decomposition 
$$
\Xi=\Xi_0\cup...\cup\Xi_k,
$$
where $\Xi_i$ consists of the points on the $i$-th floor, is being preserved.
Given $z\in \Xi_i$ we denote by $[z]$ its $G$-orbit. For every $[z]$ we consider two possible cases
of associating to it a co-dimension two constraint. In the first case we consider 
a submanifold $H_{[z]}$ of $Q$ of co-dimension $2$ and define 
\begin{eqnarray}
\wt{H}_{[z]}={\mathbb R}\times H_{[z]},
\end{eqnarray}
which we  shall call a ${\mathbb R}$-invariant constraint. In the second case we take a submanifold of $Q$ of co-dimension $1$ and define
\begin{eqnarray}
\wt{H}_{[z]}:=\{\bar{a}_{[z]}\}\times H_{[z]}.
\end{eqnarray}
 After fixing constraints as described above we obtain a map which associates to $z\in\Xi$ 
the submanifold $\wt{H}_{[z]}$ of codimension two in ${\mathbb R}\times Q$.  It is important that this map factors
through the orbits of $\Xi$. We  shall abbreviate the assignment by ${\mathcal H}$ \index{${\mathcal H}$} and call it a set of {\bf transversal constraints}
\index{transversal constraints}.
\begin{definition}\label{DEFGH2.28}
With $\sigma$, $\bm{D}$, and $\rup$, let ${\mathcal H}$ be a set of transversal constraints.
The subset $Z^3_{\bm{\sigma},{\tiny\rup},{\mathcal H},\varphi}({\mathbb R}\times Q,\bar{\digamma})$
of $Z^3_{\bm{\sigma},{\tiny\rup},\varphi}({\mathbb R}\times Q,\bar{\digamma})$ consists of all
$\wt{w}$  such that the following holds for every $i\in \{0,...,k\}$.
\begin{itemize}
\item[(1)]  $(-\text{av}_{\tiny\rup_i}(\wt{w}))\ast \wt{w}(z)\in \wt{H}_{[z]}$ for $z\in\Xi_i$.
\item[(2)]  The intersection of the shifted map $\wt{w}$  in the above at $z$ is transversal.
\end{itemize}
\qed
\end{definition}
We shall take $\Xi$ later on as above, but assume that $(S,j,\bar{M},\bar{D})$ is a stable Riemann surface, where 
$\bar{M}=M\cup \Gamma^-_0\cup\Gamma^+\cup\Xi$ and $\bar{D}=D\cup \{(z,b_i(z))\ |\ z\in\Gamma^+_{i-1},\ i\in \{1,...,k\}\}.
$ Now we can give the uniformizer construction.

\newpage
\part*{Lecture 9}
\section{Uniformizers and Transition Germs}
We shall describe a very useful variant of the uniformizer and transition construction.  When we introduced the notion of a local uniformizer construction
$F:{\mathcal C}\rightarrow \text{SET}$ we assumed that $({\mathcal C},{\mathcal T})$ was a given GCT. However, very often in applications, the starting point
is just a groupoidal category. These, in fact, have a natural metrizable topology, but constructing it would already require some of the arguments necessarily arising 
in the local uniformizer construction. So it seems to make sense to construct the topology at the same time as the local uniformizers. Since the topology already 
occurs in the definition of a local uniformizers and the topology is determined by all the uniformizers there is something like a `chicken or egg problem'.
For that reason one replaces the notion of a local uniformizer by that of a uniformizer, where local refers to the compatibility with the topology (which we do not have).
A uniformizer has the same properties as a local uniformizer, but we do not require that $|\Psi|$ is a local homeomorphism.
 Of course, the construction $\bm{M}$ has to be replaced by one, say $\mathscr{F}$,  which gives a topology and a transition construction, albeit in a more tricky way.

 \subsection{Abstract Uniformizer Construction}
We first define what we understand by a uniformizer construction.
\begin{definition}
Let ${\mathcal C}$ be a groupoidal category. A {\bf uniformizer} at an object $c$ with automorphism group $G$ 
is a functor $\Psi:G\ltimes O\rightarrow {\mathcal C}$ with the following properties.
\begin{itemize}
\item[(1)] $G\ltimes O$ is the translation groupoid associated to a M-polyfold $O$ equipped with an action of $G$ by sc-diffeomorphisms.
\item[(2)] There exists $\bar{o}\in O$ with $\Psi(\bar{o})=c$.
\item[(3)] $\Psi$ is injective on objects.
\item[(4)] $\Psi$ is full and faithful.
\end{itemize}
We shall call $\Psi$ a tame uniformizer provided $O$ is tame. 
\qed
\end{definition}
The constructions which are important for us are the uniformizer constructions. 
\begin{definition}
A \textbf{uniformizer construction} is a functor $F:{\mathcal C}\rightarrow \text{SET}$ which associates to an object $c$ a set of uniformizers.
If $F(c)$ for every object $c$ only contains tame uniformizers, we shall call $F$ a \textbf{tame uniformizer construction}.
\qed
\end{definition}
As in the case of local uniformizers we can consider the transition sets $\bm{M}(\Psi,\Psi')$. The second important construction is what we call a
transition germ construction $\mathscr{F}$. As we shall see a uniformizer construction $F$ together with a transition germ construction 
$\mathscr{F}$ produces a natural topology ${\mathcal T}$ for the orbit space $|{\mathcal C}|$ and M-polyfold structures for the transition sets 
$\bm{M}(\Psi,\Psi')$. In any given application one needs to verify that the natural topology ${\mathcal T}$ is metrizable. In fact it is not difficult
to come up with examples where the natural topology would not even by Hausdorff.
 Assuming that that the natural topology passes
the ``metrizability test'' the functor  $F$ will become a local uniformizer construction for 
the GCT $({\mathcal C},{\mathcal T})$ and the construction of the M-polyfold structure for $\bm{M}(\Psi,\Psi')$ will be a transition construction.
Hence $(F,\mathscr{F})$ for the groupoidal category ${\mathcal C}$ implies $(F,\bm{M})$ for the GCT $({\mathcal C},{\mathcal T})$ provided 
${\mathcal T}$ is metrizable. We discuss the transition germ construction next.

\subsection{Abstract Transition Germs  Construction}
We assume that at  this point we have the 
uniformizer construction $F:{\mathcal C}\rightarrow \text{SET}$. 
Associated to this uniformizer construction we can build the transition sets
$\bm{M}(\Psi,\Psi')$.   The new type of construction is as follows and is denoted by $\mathscr{F}$ and we assume that we are given 
a groupoidal category ${\mathcal C}$ and have a uniformizer construction $F$ for it.
\begin{definition}
Let $F$ be a uniformizer construction. A transition germ construction $\mathscr{F}$ associates for given $\Psi\in F(c)$ and $\Psi'\in F(c')$
to $h=(o,\Phi,o')\in \bm{M}(\Psi,\Psi')$ a germ of map $F_h: {\mathcal O}(O,o)\rightarrow (\bm{M}(\Psi,\Psi'),h)$ with the following properties,
where $f_h:=t\circ F_h$.
\begin{itemize}
\item[(A)] \textbf{Diffeomorphism Property:} The germ $f_h:{\mathcal O}(O,o)\rightarrow{\mathcal O}(O',o')$ is a local sc-diffeomorphism
and $s(F_h(q))=q$ for $q$ near $o$. If $\Psi=\Psi'$ and $h=(o,\Psi(g,o),g\ast o)$ then $F_h(q)=(q,\Psi(g,q),g\ast q)$ for $q$ near $o$ so that $f_h(q)=g\ast q$.
\item[(B)] \textbf{Stability Property:} $F_{F_h(q)}(p)=F_h(p)$ for $q$ near $o=s(h)$ and $p$ near $q$.
\item[(C)] \textbf{Identity Property:}  $F_{u(o)}(q)=u(q)$ for $q$ near $o$.
\item[(D)] \textbf{Inversion Property:} $F_{\iota(h)}(f_h(q))=\iota(F_h(q))$ for $q$ near $o=s(h)$. Here $\iota(p,\Phi,o')=(o',\Phi^{-1},o)$.
\item[(E)] \textbf{Multiplication Property:} If $s(h')=t(h)$ then $f_{h'}\circ f_h(q)=f_{m(h',h)}(q)$ for $q$ near $o=s(h)$, and $m(F_{h'}(f_h(q)),F_h(q))=F_{m(h',h)}(q)$ for $q$ near $o=s(h)$.
\item[(F)] \textbf{M-Hausdorff Property:}  For different $h_1,h_2\in \bm{M}(\Psi,\Psi')$ with $o=s(h_1)=s(h_2)$ the images under $F_{h_1}$ and 
$F_{h_2}$ of small neighborhoods are disjoint.
\end{itemize}\qed
\end{definition}
As already previously stated, it  is a general fact that the constructions $(F,\mathscr{F})$ define a natural topology ${\mathcal T}$ on $|{\mathcal C}|$
for which the $|\Psi|$ are homeomorphisms onto open subsets, and they define M-polyfold structures on the $\bm{M}(\Psi,\Psi')$.
If ${\mathcal T}$ is metrizable, a fact which has to be proved in any given context, then $(F,\mathscr{F})$ implies a construction $(F,\bm{M})$
for the GCT $ ({\mathcal C},{\mathcal T})$, see \cite{HWZ2017} and \cite{FH-book}.   
\begin{example}
Here is an example how a non-metrizable topology may arise.
Consider the groupoidal category ${\mathcal C}$ with objects being the points in $\{0,1\}\times {\mathbb R}$
and the morphisms besides the identities being the pairs $((0,t),(1,t)):(0,t)\rightarrow (1,t)$ and similarly $((1,t),(0,t))$
for $t<0$. If we equip $\{0,1\}\times{ \mathbb R}$ with the obvious metrizable topology the orbit space obtains a non-Hausdorff topology. One can give a uniformizer and transition germ construction for ${\mathcal C}$ which will yield 
this topology.
\qed
\end{example}

This will apply in the case of the category of stable maps ${\mathcal S}$ and we shall start the associated discussion later on in the present lecture.
 In order to digest the definition of $\mathscr{F}$
one should note that the basic ingredients are that given a morphism $\Phi:\Psi(o)\rightarrow \Psi'(o')$ (defining $h=(o,\Phi,o')$)
there exists an associated sc-diffeomorphism $f_h:(U(o),o)\rightarrow (U(o'),o')$ and a family $q\rightarrow \Phi^h_q$ for $q\in U(o)$ 
so that $F_h(q)=(q,\Phi^h_q,f_h(q))$. The latter gives a notion being able to say that $F_h(p)$ is close to $F_h(q)$ 
if $p\in U(o)$ close to $q\in U(o)$.  The stability conditions then say $F_{F_h(q)}$ for $q\in U(o)$ has the form
$F_{F_h(q)}(p)=(p,F_h(p),f_h(p))$ for $p$ close to $q$. The other properties are self-evident.

\subsection{Preparation for the SFT Uniformizer Construction}
Next we begin the uniformizer construction $F:{\mathcal S}\rightarrow \text{SET}$. It requires some preparation.
We are given the closed odd-dimensional manifold $Q$ equipped with a non-degenerate stable Hamiltonian structure,  i.e. $(Q,\lambda,\omega)$.
Fixing a compatible $J$ we get the spectral gap map $\delta_J: {\mathcal P}^{\ast}\rightarrow (0,2\pi]$ and pick for the periodic orbits weight sequences resulting in $\delta$.
We can define the category of stable maps ${\mathcal S}^{3,\delta_0}(Q,\lambda,\omega)$.  First we shall describe the construction of uniformizers.
We start with an object $\alpha=(\alpha_0,\widehat{b}_1,...,\widehat{b}_k,\alpha_k)$ and underlying $\sigma=(\sigma_0,b_1,...,b_k,\sigma_k)$ having isotropy group $G$.
We fix a stabilization set $\Xi$ which is invariant under $G$ with associated constraints ${\mathcal H}$ and disjoint from nodes and punctures, anchor sets, and 
a small disk structure $\bm{D}$ so that the union of disks associated to $D_i$ and the $D_{(i-1,i)}$ are invariant and do not contain marked points and are mutually disjoint.  We also require that $\bar{\sigma}= (S,j,\bar{M},\bar{D})$ is a stable Riemann surface, where $\bar{M}=M\cup \Gamma^-_0\cup\Gamma^+_k\cup \Xi$ and $\bar{D}=D\cup \{\{z,b_i(z)\}\ |\ z\in \Gamma^+_{i-1},\ i\in \{1,...,k\}\}$. 
 \begin{figure}[htp]
\begin{center}
\includegraphics[width=10.0cm]{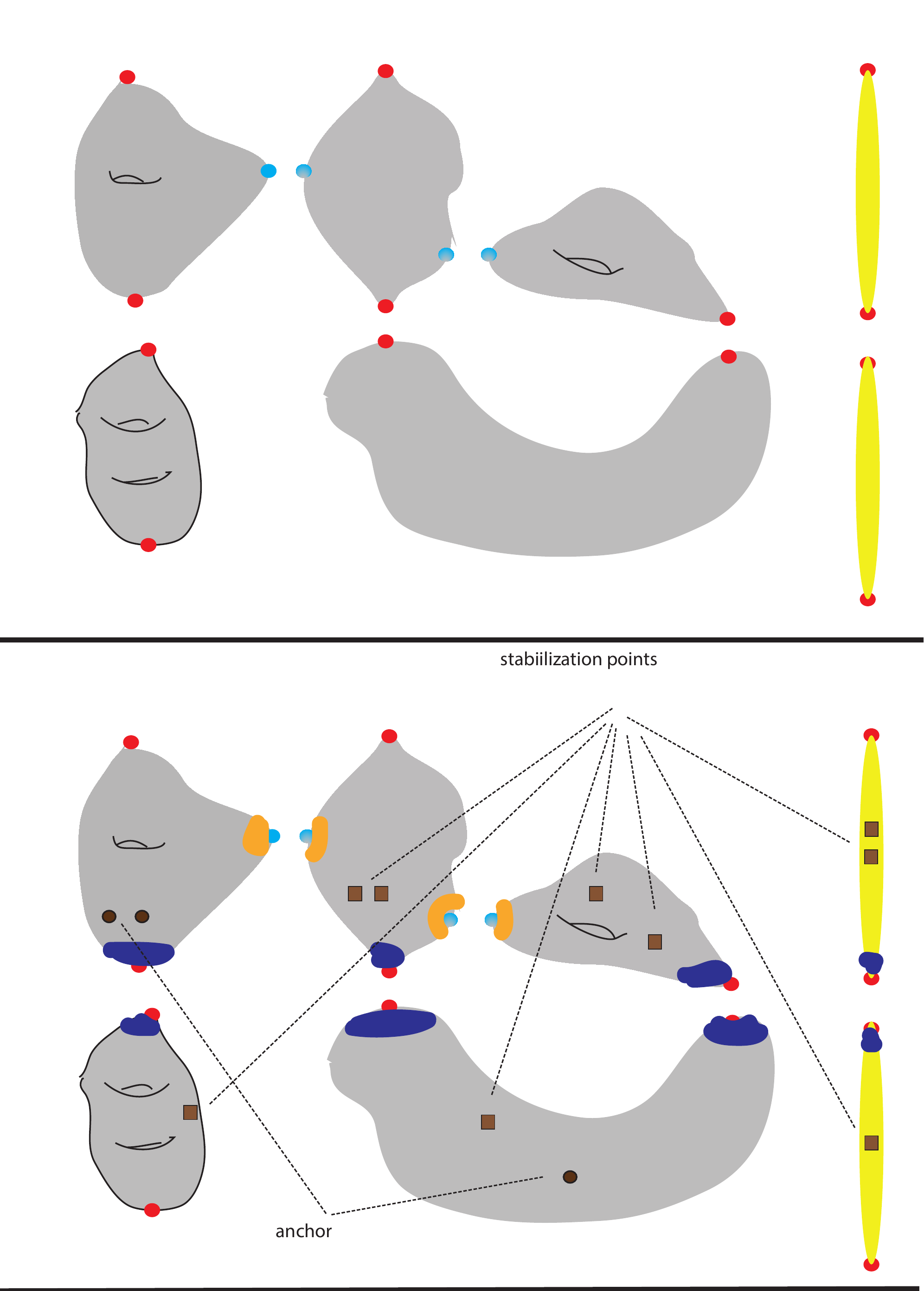}
\end{center}
\caption{Associated to the orbits of stabilization points under $G$ we have transversal constraints}
\end{figure}

The data which we have is then
\begin{itemize}
\item[-] $\alpha$, $\sigma$, and a small disk structure $\bm{D}$.
\item[-] stabilization set $\Xi$  and transversal constraints ${\mathcal H}={(\widetilde{H}_{[z]})}_{[z]}$
\item[-] anchor set $\rup$.
\end{itemize}
\vspace{0.2cm}
\noindent We can build the M-polyfold $Z^3_{\bm{\sigma},{\tiny\rup},{\mathcal H},\varphi}({\mathbb R}\times Q,\bar{\digamma})$ which has a distinguished element $\overline{\widetilde{u}}$ coming from $\alpha$. The M-polyfold structure is induced from the ambient M-polyfold $Z^3_{\bm{\sigma},{\tiny\rup},\varphi}({\mathbb R}\times Q,\bar{\digamma})$ of which our space is a sub-M-polyfold, in  fact a tame one.  The sc-smooth embedding 
\begin{eqnarray}
Z^3_{\bm{\sigma},{\tiny\rup},{\mathcal H}, \varphi}({\mathbb R}\times Q,\bar{\digamma})\xrightarrow{\text{incl}}
Z^3_{\bm{\sigma},{\tiny\rup},\varphi}({\mathbb R}\times Q,\bar{\digamma})
\end{eqnarray}
will be important to us.
We have an action of $G$ on  $(S,j,\bar{M},\bar{D})$ which preserves the floors. Of course $\bar{\sigma}$ has a larger finite automorphism group denoted by  $G^\ast$.
Then  $G\subset G^\ast$ and we have seen in the discussion of the DM-theory that  can take particular deformation $V\ni v\rightarrow j(v)$ of $j$ such that
$$
G^\ast\ltimes (V\times {\mathbb B}_{\bar{D}})\rightarrow {\mathcal R}: (v,\mathfrak{a})\rightarrow (S_{\mathfrak{a}},j(v)_{\mathfrak{a}},\bar{M}_{\mathfrak{a}},D_{\mathfrak{a}})
$$
defines a good unformizer when restricted to $G^\ast\ltimes O^{\ast}$, where $O^{\ast}\subset V\times {\mathbb B}_{\bar{D}}$ is a suitable open $G^{\ast}$-invariant neighborhood
of $(0,0)$. Recall that there exists a uniformizer construction for ${\mathcal R}$.  It is important  that we have the splitting, which was described in the discussion of DM-theory.
\begin{eqnarray}\label{DM-COMPP}
&H^1(\bar{\sigma})\equiv H^1(\bar{\sigma}_0^{tc})\oplus H^1(\bar{\sigma}_0^{ntc})\oplus..\oplus H^1(\bar{\sigma}_{k}^{tc})\oplus H^1(\bar{\sigma}_{k}^{ntc}):&\\
&v\equiv(v_0^{tc},v_0^{ntc},..., v_{k}^{tc},v_{k}^{ntc}).&\nonumber
\end{eqnarray}
The following piece of data is important: The good uniformizer around $\bar{\sigma}$ coming from $\alpha$.

\begin{eqnarray}
& \boxed{G^{\ast}\ltimes O^{\ast}\rightarrow {\mathcal R}:(v,\mathfrak{a})\rightarrow (S_{\mathfrak{a}},j(v)_{\mathfrak{a}},\bar{M}_{\mathfrak{a}},\bar{D}_{\mathfrak{a}})}.&
\end{eqnarray}

\subsection{Pre-Uniformizer}
The first step in the uniformizer construction is the construction of pre-uniformizers. 
We shall construct what we shall call pre-uniformizers. Namely given $\alpha$ and having carried out the preparations described above 
we define
$$
\text{\textbf{pre}}\Psi: G\ltimes (V\times Z^3_{\bm{\sigma},{\tiny\rup},{\mathcal H},\varphi}({\mathbb R}\times Q,\bar{\digamma}))\rightarrow {\mathcal S}
$$
as follows. The element ${\widetilde{u}}\in  Z^3_{\bm{\sigma},{\tiny\rup},{\mathcal H},\varphi}({\mathbb R}\times Q,\bar{\digamma})$ has underlying domain
$S_{\widetilde{\mathfrak{a}}}$  and gluing parameter $\widetilde{\mathfrak{a}}$. We set 
$$
\textbf{pre}\Psi(v,\widetilde{u}) = (S_{\widetilde{\mathfrak{a}}},j(v)_{\widetilde{\mathfrak{a}}}, M_{\widetilde{\mathfrak{a}}},D_{\widetilde{\mathfrak{a}}},\Gamma_{\mathfrak{a}}, [\widetilde{u}]).
$$
This is short-hand for the following: The element $\wt{u}$ decomposes as 
$$
\wt{u}=(\wt{u}_0,\wh{b}_{i_1},...,\wh{b}_{i_\ell},\wt{u}_{i_\ell}),
$$
where $\wt{u}_e$ is defined on the punctures $(\Gamma_{\wt{\mathfrak{a}},e}^-,S_{\wt{\mathfrak{a}}}^e,j(v)^e_{\wt{\mathfrak{a}}},M_{\wt{\mathfrak{a}},e},D_{\wt{\mathfrak{a}},e},[\wt{u}_e],\Gamma^+_{\wt{\mathfrak{a}},e})$.
The uniformizers then will be obtained by restricting pre-uniformizers to suitable subsets.  This will happen in the next lecture.
The basic fact is that the knowledge of ${\mathcal S}$ allows us to formulate a \textbf{Recipe}, i.e. a general rule, 
to characterize for a constructed $\textbf{pre}\Psi$ neighborhoods $G$-invariant neighborhoods $O$ of $(0,\overline{\wt{u}})$ so that 
$\Psi: G\ltimes O\rightarrow {\mathcal S}$ obtained as $\textbf{pre}\Psi|G\ltimes O$ (many different choices of $O$ are possible for 
a given $\textbf{pre}\Psi$), so that the $\Psi$ has very specific properties.  This will define a functor $F:{\mathcal S}\rightarrow \text{SET}$
by associating to $\alpha$ the set of $\Psi$ obtained by constructing a set of all possible pre-uniformizers following the recipe 
for pre-uniformizer constructions, and then using these to take all allowable restrictions according to the recipe we have to define.
The recipe has the feature that for given $\Psi\in F(\alpha)$ and $\Psi'\in F(\alpha')$ one has enough properties to carry the transition germ construction $\mathscr{F}$.
\newpage
\part{From Local to Global}
We shall study the global relationships between the local pieces.
\part*{Lecture 10}
\section{Uniformizers and Transition Germs for SFT}
We  shall describe the construction of a polyfold structure on ${\mathcal S}$.
\subsection{Background for the  Uniformizer Construction}
We have outlined the construction of a pre-uniformizers at an object $\alpha$.
The unformizer construction will define a criterion for picking, for a given \textbf{pre}$\Psi$ at $\alpha$ with special element $\bar{o}=(0,\overline{\wt{u}})$,
a $G$-invariant neighborhood $O$ so that the restriction to $G\ltimes O$, say
$$
\Psi:G\ltimes O\rightarrow {\mathcal S}
$$
 will have desirable properties. The collection of all such $\Psi:G\ltimes O\rightarrow {\mathcal S}$
obtained by taking suitable restrictions of pre-uniformizers will define $F(\alpha)$.  Clearly we want that $\Psi:G\ltimes O\rightarrow {\mathcal S}$ 
is 
\begin{itemize}
\item \textbf{full} and \textbf{faithful} (to reflect the structure of ${\mathcal S}$).
\end{itemize}
We also would like to have 
\begin{itemize}
\item \textbf{injectivity} of $\Psi$ on objects (distinguishing particular full subcategories of ${\mathcal S}$).
\end{itemize}
   \begin{figure}[h]
\begin{center}
\includegraphics[width=8.5cm]{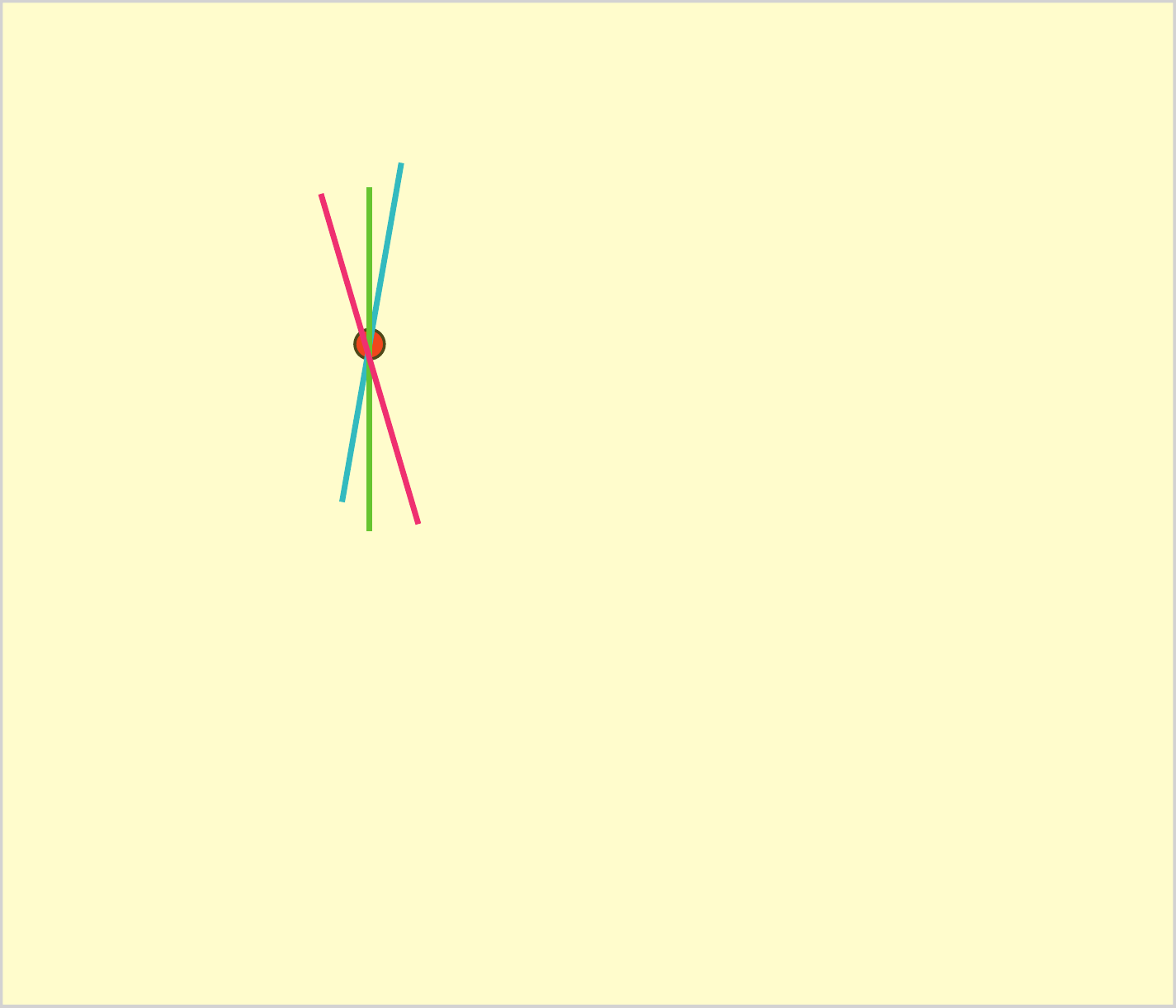}
\end{center}
\caption{The images of three different $\Psi: G\ltimes O\rightarrow {\mathcal S}$ associated to $\alpha$.}
\end{figure}

It is clear that DM-theory for the underlying 
$$
\Psi^{\ast}:G^{\ast}\ltimes O^{\ast}\rightarrow {\mathcal R}.
$$ 
will play a role which require that 
\begin{itemize}
\item for $(v,\wt{u})\in O$ we must have that $(v,\mathfrak{a}(\wt{u}))\in O^{\ast}$. 
\end{itemize}
The first basic result is the following.
\begin{prop}\label{important!}
Given an object $\alpha$ and a pre-uniformizer \textbf{pre}$\Psi$ 
there exists an open $G$-invariant neighborhood $O$  of $\bar{o}$ in $V\times Z^3_{\bm{\sigma},{\tiny\rup},{\mathcal H}, \varphi}({\mathbb R}\times Q,\bar{\digamma})$ with the following properties.
\begin{itemize}
\item[(1)] If $(v,\wt{u})\in O$ then $(v,\mathfrak{a}(\wt{u}))\in O^{\ast}$.
\item[(2)] The restriction of \textbf{pre}$\Psi$  denoted by $\Psi:G\ltimes O\rightarrow {\mathcal S}$ is a fully faithful functor and injective on objects.
\end{itemize}
\qed
\end{prop}
The only more involved step is the fullness of $\Psi$ for a suitable $O$. Taking for $\alpha$ the collection $F(\alpha)$ consisting of the $\Psi$ which are obtained 
as restrictions of pre-uniformizers which satisfy (1) and (2) is \textbf{not yet a uniformizer construction}, but a suitable subset of every $F(\alpha)$ will be. We need to show that for a suitable choice of $O$ a third, very important  condition can be  satisfied. 
This  condition is  a kind of transversality condition 
and it  is not so easy to guess, and also is not necessary for $\alpha$ which are not too complicated, so it might be easily overlooked.
In the following we shall refer to the image $\Psi(O)$ as a \textbf{slice}.
  \begin{figure}[h]
\begin{center}
\includegraphics[width=5.2cm]{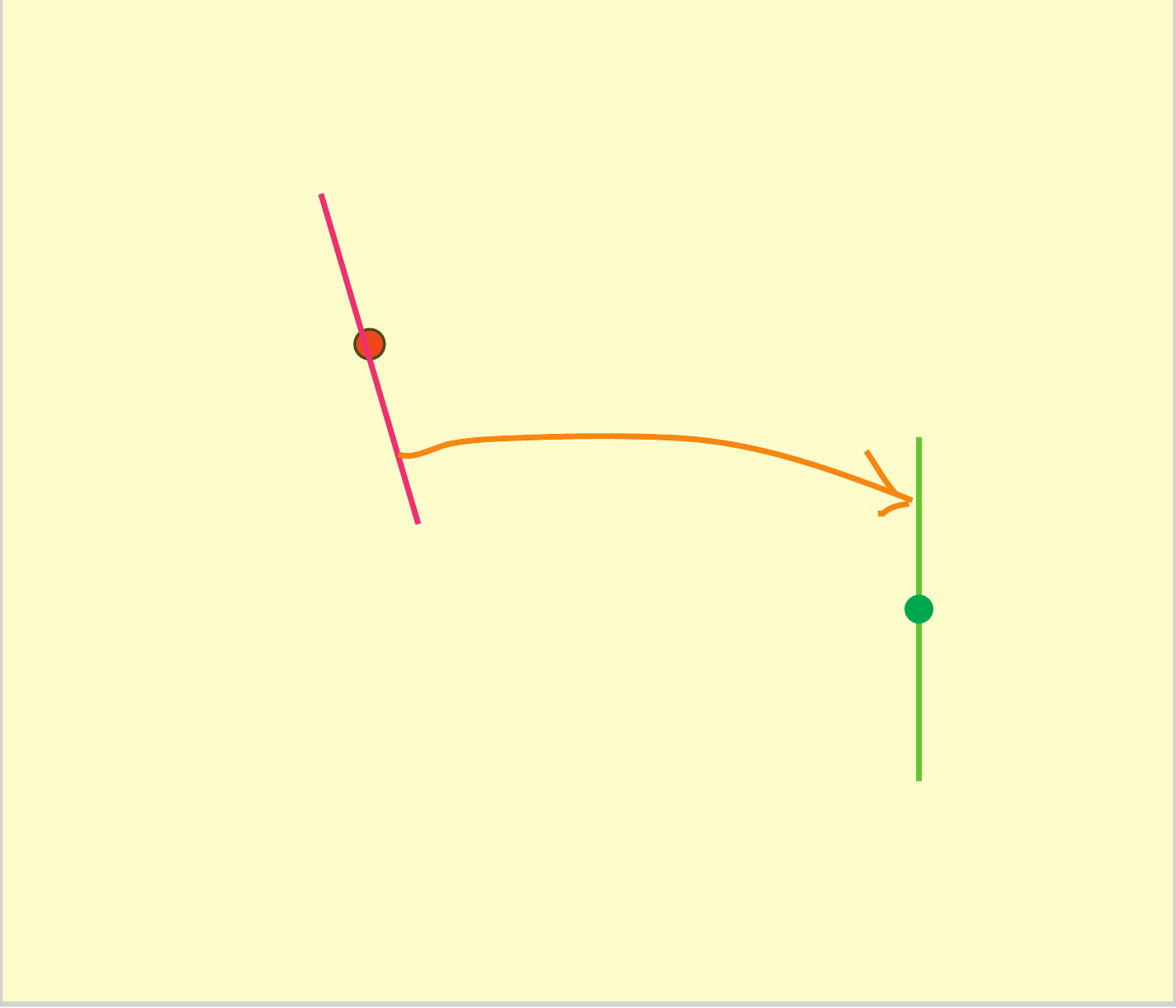}
\end{center}
\caption{The transition construction. One needs a criterion, given an arrow from one $\alpha$-slice to a $\alpha'$-slice, how to distinguish
a family of arrows starting nearby, so that the output varies sc-smoothly. Of course, if the two slices are identical, the families should come from $G$. Moreover, there should be an associativity for three slices and obvious other properties.}\label{FIG200}
\end{figure}

The problem highlighted in the Figure \ref{FIG200} leads to  a transversality question which will be discussed next. A priori such a transversality
question involves obviously the properties of $\Psi$ and $\Psi'$ at the same time. For a uniformizer construction we are forced to give a recipe
defining for a given $\alpha$ the set $F(\alpha)$, i.e. to define properties for its members $\Psi\in F(\alpha)$ without reference 
to elements $\Psi'\in F(\alpha')$.  With other words it is better true that a transversality condition exists which can be formulated 
for every $\alpha$ so that the construction of $F$ achieves the goal formulated in the caption of Figure \ref{FIG200}.

\subsection{A Transversality Question}
In the previous subsection we have learned that one can obtain certain properties for $\Psi:G\ltimes O\rightarrow {\mathcal S}$, when one restricts a pre-uniformizer
to a smaller $G$-invariant neighborhood of the distinguished element. We shall consider two such $\Psi$ and $\Psi'$, i.e. 
restrictions of pre-uniformizers to $G\ltimes O$ and $G'\ltimes O'$, where $O$ and $O'$ are such that the conclusion of Proposition \ref{important!}
holds and motivate why an additional property is needed.
We shall refer to them  for the following discussion as uniformizers. 
We   assume we are given two objects $\alpha$ and $\alpha'$ and associated uniformizers $\Psi$ and $\Psi'$ and consider an element $h\in \bm{M}(\Psi,\Psi')$.
That means we are given $o\in O$, $o'\in O'$ and $\Phi:\Psi(o)\rightarrow \Psi'(o')$. 
We need conditions so that we can give a recipe for the  local germ $F_h$ which, of course, is based on exhibiting a locally unique deformation of $\Phi^h$ 
$$
q\rightarrow \Phi^h_q, \ \ \text{for $q$ close to $o$}.
$$
All these arrows are supposed to start at $\Psi(q)$ and to end in the slice associated to $\Psi'$, so that we can define 
$f_h:({\mathcal O}(O),o)\rightarrow (O',o')$ by   $F_h(q)=(q,\Phi^h_q, f_h(q))$. Of course,  $f_h$ is determined by the recipe 
defining ${\Phi}^h$ and one needs to show that $f_h$  is a germ of sc-diffeomorphism. 
Of course, not surprisingly the criterion for determining such a unique choice is some kind of transversality and, as already pointed out, this transversality condition should \textbf{not!} be a condition on the pair, but a condition on the individual $\Psi$,
so that any such pair has favorable  properties.

For the following note that we shall denote a global gluing parameter by $\mathfrak{a}$ rather than $\wt{\mathfrak{a}}$. (Recall
that our convention was that $\wt{\mathfrak{a}}$ is the $D$-part of $\wt{\mathfrak{a}}$, but this will be ignored for the moment!).
We shall over-line fixed reference data and usually study small variations of the data, so that for example $\mathfrak{a}$ is a variation 
of $\overline{\mathfrak{a}}$.

Assume that $o =(\bar{v},\bar{\wt{u}})$, $o'=(\bar{v}',\bar{\wt{u}}')$ with underlying total gluing parameters denoted by  $\bar{\mathfrak{a}}$ and $\bar{\mathfrak{a}}'$.
The morphism $\Phi$ is represented by a biholomorphic map 
$$
\bar{\phi}:\sigma_{(\bar{{\mathfrak{a}}},\bar{v})}\rightarrow \sigma'_{(\bar{{\mathfrak{a}}}',\bar{v}')},
$$ 
which has to preserve the usual data,
namely sends marked points in $M_{\bar{\mathfrak{a}}}$ to marked points in $M'_{\bar{\mathfrak{a}}'}$, nodal pairs in $D_{\bar{\mathfrak{a}}}$ to nodal pairs in $D'_{\bar{\mathfrak{a}}'}$ and punctures 
in $\Gamma_{\bar{\mathfrak{a}}}$ to those in $\Gamma'_{\bar{\mathfrak{a}}'}$.
 If we consider $\bar{\wt{u}}':=\bar{\wt{u}}\circ \bar{\phi}^{-1}$ we know that adjusting the different floors by using the ${\mathbb R}$-action 
we obtain the map $[\bar{\wt{u}}']$, which satisfies the anchor constraints (not general the virtual ones).  This map $[\bar{\wt{u}}']$ belongs to $Z^3_{\bm{\sigma}',{\tiny\rup'},{\mathcal H}',\varphi}({\mathbb R}\times Q,\bar{\digamma}')$. Of course, $(v',[\bar{\wt{u}}'])\in O'$. For every $i'\in \{0,...,k'\}$ with $z'\in \Xi'_{i'}$ it holds 
$$
(-\text{av}_{{\tiny\rup'_{i'}}}(\bar{\wt{u}}'))\ast \bar{\wt{u}}'(z')\in \wt{H}'_{[z']},\ z'\in \Xi'_{i'}
$$
and varying $z'$ the intersection is transversal.  We note that the left-hand side is ${\mathbb R}$-invariant and therefore we need not to work with $[\bar{\wt{u}}']$.
We rewrite the above as 
$$
(-\text{av}_{{\tiny\rup'_{i'}}}(\bar{\wt{u}}\circ\bar{\phi}^{-1}))\ast \left(\bar{\wt{u}}\circ \bar{\phi}^{-1}(z')\right)\in \wt{H}'_{[z']},\ z'\in \Xi'_{i'}
$$
In order to construct $F_h$ we need to find a map $(v,\wt{u})\rightarrow \phi_{(v,\wt{u})}$, with sufficiently smooth properties,  defined for $(v,\wt{u})$ near $(\bar{v},\bar{\wt{u}})$ so that 
$$
(-\text{av}_{{\tiny\rup'_{i'}}}(\wt{u}\circ\phi_{(v,\wt{u})}^{-1}))\ast \left(\wt{u}\circ \phi_{(v,\wt{u})}^{-1}(z')\right)\in \wt{H}'_{[z']},\ z'\in \Xi'_{i'}
$$
and the intersection is transversal with respect to a variation of $z'$.  If $z'$ belongs to $\Xi'^{ntc}$ the 
 adjustment by the anchor average is not needed since the associated constraint $\wt{H}_{[z']}$ is ${\mathbb R}$-invariant, but the left-hand side will only be a sc-smooth function of input if 
 this adjustment is made due to vanishing gluing parameters.  The unique solvability of this equation will follow from  an implicit function theorem provided the appropriate hypotheses hold. We shall discuss this in the next subsection in more detail.
 Here we only note the following. 
 
 Denote by $\Xi^{\ast}$ the preimage of $\Xi'$ under $\bar{\phi}$. The set $\Xi^{\ast}$ is a subset 
 of  the Riemann surface associated to the parameters $(\bar{v},\bar{\mathfrak{a}})$ and with these two parameters fixed we can consider deformations
 $(\bar{v},\bar{\mathfrak{a}},\bm{x})$ where $\bm{x}$ maps a point $z^{\ast}\in \Xi^{\ast}$ to a nearby point on the same surface.
 We can do the same for nearby parameters  $(v,\mathfrak{a})$ and consider $(v,\mathfrak{a},\bm{x})$ where $\bm{x}$ maps points in $\Xi^{\ast}$ to points 
 in the surface associated to $(v,\mathfrak{a})$. Using the \underline{universal property} of $(v',\mathfrak{a}')\rightarrow \bar{\sigma}'_{(v',\mathfrak{a}')}$
there exists a uniquely defined deformation $(v,\mathfrak{a},\bm{x})\rightarrow \phi_{(v,\mathfrak{a},\bm{x})}$ of $\bar{\phi}$ which maps
$$
(S_{\mathfrak{a}},j(v)_{\mathfrak{a}},{(\Gamma^-_0\cup M\cup \bm{x}(\Xi^{\ast})\cup\Gamma^+_k)}_{\mathfrak{a}}, \bar{D}_{\mathfrak{a}})\xrightarrow{\phi_{(v,\mathfrak{a},\bm{x})}} (S_{\mathfrak{a}'}',j'(v')_{\mathfrak{a}'},\bar{M}'_{\mathfrak{a}'},\bar{D}'_{\mathfrak{a}'}),
$$
where $(v',\mathfrak{a}')$ is a smooth function of $(v,\mathfrak{a},\bm{x})$. If $(v,\wt{u})$ is near $(\bar{v},\bar{\wt{u}})$, which also means that $(v,\mathfrak{a}(\wt{u}))$, is near to
$ (\bar{v},\bar{\mathfrak{a}})$, one \underline{needs!} to show that for $\bm{x}$ close to $\bm{\bar{x}}$ defined by $\bm{\bar{x}}(z^{\ast})=z^{\ast}$ for $z^{\ast}\in \Xi^{\ast}$ it holds that 
there exists a unique $\bm{x}=\bm{x}(v,\mathfrak{a})$, $\mathfrak{a}=\mathfrak{a}(\wt{u})$, with 
$$
\left((-\text{av}_{{\tiny\rup}_i'}(\wt{u}\circ \phi_{(v,\mathfrak{a},\bm{x})}^{-1}))\ast \left(\wt{u}\circ \phi_{(v,\mathfrak{a},\bm{x})}^{-1}\right)\right)(z')\in \wt{H}_{[z']}\ \text{for}\ z'\in\Xi'_i,
$$
which also would be smoothly depending on $(v,\mathfrak{a})$.
Of course, then the germ, where $[.]$ means ${\mathbb R}$-adjustments to satisfy the effective anchor constraints,
\begin{eqnarray}
(v,\wt{u})\rightarrow (v'(v,\wt{u}),[\wt{u}\circ  \phi_{(v,\mathfrak{a},\bm{x}(v,\wt{u}))}^{-1}])
\end{eqnarray}
is $f_h$. At this point one could only claim that $f_h$ is sc-smooth using results about domain transformations, see \cite{FH-book}. It is not clear  that $f_h$ is a local sc-diffeomorphism. This  would follow by interchanging the roles of $\Psi$ and $\Psi'$ if the same discussion
would hold. 

As it turns out,  the fact that $f_h$ is sc-smooth
only depends on an additional property of $\Psi'$ besides the properties already required from $\Psi$ and $\Psi'$. If $\Psi$ also satisfies
such an additional  property (to be stated),  interchanging the roles of $\Psi$ and $\Psi'$ will imply that $f_h$ is in fact a local sc-diffeomorphism.
The germ $F_h$ would be defined by $F_h(v,\wt{u})=((v,\wt{u}),\Phi^h_{(v,\wt{u})},f_h(v,\wt{u}))$, where 
$\Phi^h_{(v,\wt{u})}$ is the morphism  associated to the biholomorphic map $\phi_{(v,\mathfrak{a},\bm{x}(v,\wt{u}))}^{-1}$.  

As we mentioned before the property that $f_h$ is sc-smooth will only depend on an additional requirement on $\Psi'$. The basic reason is the following.
The biholomorphic map $\phi_{(v,\mathfrak{a},\bm{x})}$ can be decomposed as follows.  Take for $(v,\mathfrak{a})$ near $(\bar{v},\bar{\mathfrak{a}})$ a smooth section
$\bm{x}^{\ast}_{(v,\mathfrak{a})}$ with $\bm{x}^{\ast}_{(\bar{v},\bar{\mathfrak{a}})}(z^\ast)=z^{\ast}$. We mean by this 
that $\bm{x}^{\ast}_{(v,\mathfrak{a})}$ belongs to the surface $S_{\mathfrak{a}}$ equipped with $j(v)_{\mathfrak{a}}$ and varies
smoothly as a function of $(v,\mathfrak{a})$.
Associated to this we obtain by the universal property
$$
\psi_{(v,\mathfrak{a})}:=\phi_{(v,\mathfrak{a},\bm{x}^{\ast}_{(v,\mathfrak{a})})},
$$
which maps, preserving the obvious other data, from the surface associated to $(v,\mathfrak{a})$ to the surface associated to $(v'(v,\mathfrak{a}),\mathfrak{a}'(v,\mathfrak{a}))$.
Given $(v,\mathfrak{a},\bm{x})$ we can use $\psi_{(v,\mathfrak{a})}$ to map this data to some $\bm{y}'$ via 
$$
(v,\mathfrak{a},\bm{x})\rightarrow (v',\mathfrak{a}',\psi_{(v,\mathfrak{a})}\circ\bm{x}), \ \text{were $\bm{y}'= \psi_{(v,\mathfrak{a})}\circ\bm{x}$}
$$
Note that the choice of $\bm{x}$, for fixed $(v,\mathfrak{a})$, does not affect $(v',\mathfrak{a}')$. 
The right-hand side is now data on a surface associated to $\Psi'$. We see that for fixed $(v,\mathfrak{a})$ we have a local diffeomorphism between
deformations of $\Xi^{\ast}$ and deformations of $\Xi'$. Given a small open neighborhood of $(\bar{v},\bar{\wt{u}})$ in $O$ we can map 
it via 
\begin{eqnarray}\label{omgh}
(v,\wt{u})\rightarrow (v',[\wt{u}\circ \psi^{-1}_{(v,\mathfrak{a}(\wt{u}))}])
\end{eqnarray}
 into and open neighborhood of $(\bar{v}',\bar{\wt{u}}')$ in $Z^3_{\bm{\sigma}',{\tiny\rup}',\varphi}({\mathbb R}\times Q,\bar{\digamma}')$. Observe that the image of $(\bar{v},\bar{\mathfrak{a}})$
satisfies the constraints associated to ${\mathcal H}'$, but the images of the other elements in the small neighborhood usually do not!

  Next we consider $(v',\mathfrak{a}',\bm{y}')$ and using the \underline{universal property} there exists $\psi_{(v',\mathfrak{a}',\bm{y}')}$
$$
(S'_{\mathfrak{a}'},j'(v')_{\mathfrak{a}'},{({\Gamma'}^-_0\cup M'\cup\bm{y}'(\Xi^{\ast})\cup{\Gamma'}^+_{k'})}_{\mathfrak{a}'},\bar{D}_{\mathfrak{a}'})\xrightarrow{\psi_{(v',\mathfrak{a}',\bm{y}')}} (S'_{\mathfrak{a}''},j'(v'')_{\mathfrak{a}''}, \bar{M}'_{\mathfrak{a}''},\bar{D}'_{\mathfrak{a}''}),
$$
where $(v'',\mathfrak{a}'')$ depend smoothly on the input $(v',\mathfrak{a}',\bm{y}')$.
We note the important fact that  
$$
\phi_{(v,\mathfrak{a},\bm{x})} = \psi_{(v',\mathfrak{a}',\bm{x}\circ \psi_{(v,\mathfrak{a})})}\circ \psi_{(v,\mathfrak{a})}.
$$
This implies the possibility that by picking $(v',\mathfrak{a}',\bm{y}')$ properly we can adjust the image of the map
in (\ref{omgh}) via $ \psi_{(v',\mathfrak{a}',\bm{y}')}$ to satisfy the constraints. The appropriate choice of $\bm{y}'$ then defines a choice of $\bm{x}$.
From this discussion it follows that we only need to map the before-mentioned small open neighborhood in $Z^3_{\bm{\sigma}',{\tiny\rup}',\varphi}({\mathbb R}\times Q,\bar{\digamma}')$
by a suitable choice of $\psi_{(v',\mathfrak{a}',\bm{y}'(v',\mathfrak{a}'))}$ into $Z^3_{\bm{\sigma}',{\tiny\rup}',{\mathcal H}',\varphi}({\mathbb R}\times Q,\bar{\digamma}')$.
Of course, the map $(v',\mathfrak{a}')\rightarrow \bm{y}'(v',\mathfrak{a}')$ has to be found by an implicit function theorem, but the entire procedure only depends on $\Psi'$. This is being discussed in more detail in the next subsection in terms of $\Psi$ to simplify notation,
i.e. getting rid of the primes.

\subsection{The Transversality  Condition}
As we have seen the transversality condition is a property which can be formulated for a single $\Psi$. Rather than with $\Psi'$ we shall work with $\Psi$ to formulate it.
We also consider for simplicity the inverses of the maps  considered in the previous subsection.
The transversality  condition depends on two ingredients. The first one is the universal property from DM-theory.
We started with a stable map $\alpha$, fixed a stabilization set $\Xi$ and a small disk structure $\bm{D}$.
From this we obtain $\sigma=(S,j,\bar{M},\bar{D})$ , where $\bar{M}= M\sqcup\Xi\sqcup\Gamma_0^-\sqcup\Gamma^+_k$
and $\bar{D}= D\sqcup \{ \{z,b_i(z)\}\ |\ z\in\Gamma^+_i,\ i\in \{0,....,k-1\}\}$.  Then taking a good deformation 
$\mathfrak{j}$ with previously described properties we obtain the good uniformizer 
$$
\Psi^{\ast}:G^{\ast}\ltimes O^{\ast}\rightarrow {\mathcal R}: (\mathfrak{a},v)\rightarrow (S_{\mathfrak{a}},j(v)_{\mathfrak{a}},\bar{M}_{\mathfrak{a}},\bar{D}_{\mathfrak{a}}).
$$
\subsubsection{Stabilization Deformation}
We consider the good uniformizer for ${\mathcal R}$ and use its universal property, recalling that the points in $\Xi$ were artificially added via the transversal constraint construction, and that adding $\Xi$ stabilized the Riemann surface.
For given $(\mathfrak{a}_0,v_0)\in O^{\ast}$ the universal property guarantees a $G$-invariant open neighborhood $U(\Sigma)$ of the form 
$$
U:=U(\Sigma)=\coprod_{z\in\Xi} U(z).
$$
We denote by $U^{\Sigma}$ the smooth manifold of maps $\bm{x}:\Xi\rightarrow U$ with $\bm{x}(z)\in U(z)$. For such a $\bm{x}$ we define
$$
\bar{M}(\bm{x}):=\Gamma^+_0\sqcup M\sqcup \bm{x}(\Sigma)\sqcup \Gamma^+_k.
$$
By the universal property there exists for $\bm{x}$ near $\bm{\bar{x}}$ defined by $\bm{\bar{x}}(z)=z$ a uniquely determined biholomorphic map near the identity
$$
\psi_{(\mathfrak{a}_0,v_0,\bm{x})}: (S_{\mathfrak{a}},j(v)_{\mathfrak{a}},\bar{M}_{\mathfrak{a}},\bar{D}_{\mathfrak{a}})\rightarrow (S_{\mathfrak{a}_0},j(v_0)_{\mathfrak{a}_0},\bar{M}(\bm{x})_{\mathfrak{a}_0},\bar{D}_{\mathfrak{a}_0}),
$$
where $(\mathfrak{a},v)$ is a smooth map of $(\mathfrak{a}_0,v_0,\bm{x})$.
\begin{definition}
We shall refer to the data 
$$
(\mathfrak{a}_0,v_0,\bm{x})\rightarrow \psi_{(\mathfrak{a}_0,v_0,\bm{x})}\ \ \text{and}\ \ (\mathfrak{a}(\mathfrak{a}_0,v_0,\bm{x}),v(\mathfrak{a}_0,v_0,\bm{x}))
$$
the \textbf{stabilization deformation}.
\qed
\end{definition}
\begin{tcolorbox}[ams align]\label{Q-oemgaX2XX}
\text{Stabilization Deformation}\nonumber\\
\boxed{(\mathfrak{a}_0,v_0,\bm{x})}   \xrightarrow{\psi\ sd} \boxed{(\mathfrak{a}(\mathfrak{a}_0,v_0,\bm{x}),v(\mathfrak{a}_0,v_0,\bm{x}),\overline{\bm{x}})}\nonumber \\
(\mathfrak{a}_0,v_0,\bm{x})\rightarrow \psi_{(\mathfrak{a}_0,v_0,\bm{x})}\nonumber\\
\psi_{(\mathfrak{a}_0,v_0,\bm{x})}: (S_{\mathfrak{a}},j(v)_{\mathfrak{a}},\bar{M}_{\mathfrak{a}}, \bar{D}_{\mathfrak{a}})\rightarrow (S_{\mathfrak{a}_0},j(v_0)_{\mathfrak{a}_0},
\bar{M}(\bm{x})_{\mathfrak{a}_0},\bar{D}_{\mathfrak{a}_0})
\nonumber
\end{tcolorbox}
\begin{figure}[htp]
\begin{center}
\includegraphics[width=12.0cm]{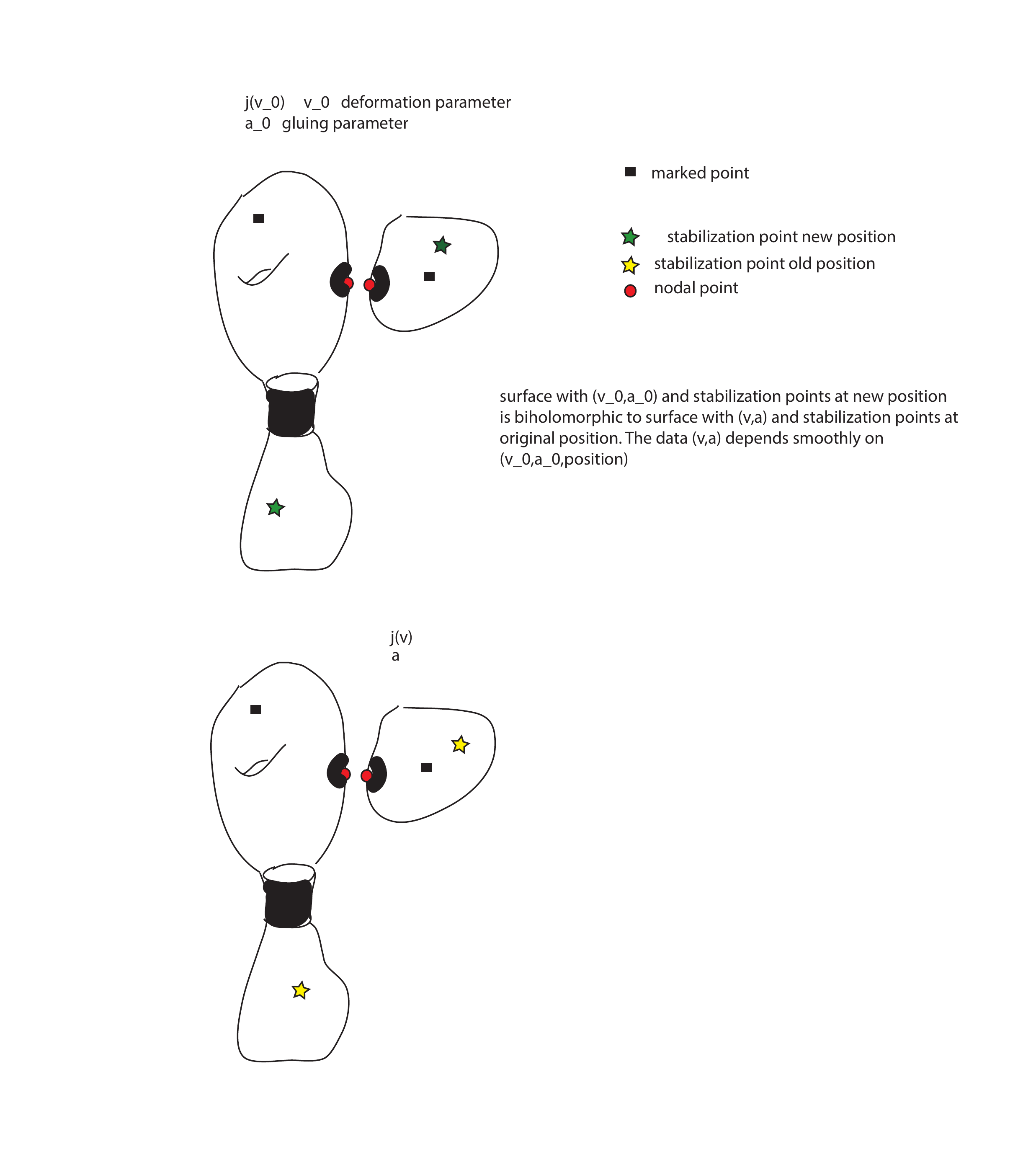}
\end{center}
\caption{Stabilization Deformation}\label{FIG 820}
\end{figure}
If we restrict $O^{\ast}$ to some smaller $G$-invariant neighborhood and $U(\Sigma)$ to a suitable $U^{\ast}$ we can obtain some 
uniformity of the stabilization deformation with respect to the input $(\mathfrak{a}_0,v_0,\bm{x})$. More precisely we obtain.
\begin{prop}
There exists a 
$G$-invariant open neighborhood  $O^{\ast\ast}$ of $(0,0)$ contained in $O^{\ast}$ and a sufficiently small open $G$-invariant neighborhood
$U^{\ast}(\Xi)\subset U(\Xi)$, where $U^\ast$ is again the disjoint union of disk-like $U^\ast(z)$,  such that there exists a smooth map (a uniform version of the sd-transformation)
\begin{eqnarray}
&O^{\ast\ast}\times U^\ast(\Xi)\rightarrow O^\ast&\\
&({\mathfrak{a}},v,\bm{x})\rightarrow ({\mathfrak{a}}',v'):=({\mathfrak{a}}'({\mathfrak{a}},v,\bm{x}),v'({\mathfrak{a}},v,\bm{x}))&\nonumber
\end{eqnarray}   
and a uniquely determined sc-smooth family of biholomorphic $\psi_{({\mathfrak{a}},v,\bm{x})}$
with $\psi_{(0,0,\overline{\bm{x}})}=Id$  
$$
\psi_{({\mathfrak{a}},v,\bm{x})}: (S_{{\mathfrak{a}}'},j(v')_{{\mathfrak{a}}'},\bar{M}_{{\mathfrak{a}}'},\bar{D}_{{\mathfrak{a}}'})\rightarrow (S_{{\mathfrak{a}}},j(v)_{{\mathfrak{a}}},\bar{M}({\bm{x}})_{{\mathfrak{a}}},\bar{D}_{{\mathfrak{a}}})
$$
so that in addition the points in $M\cup\Gamma^+_k\cup\Gamma^-_0\subset S_{{\mathfrak{a}}} $ 
and $M\cup\Gamma^+_k\cup\Gamma^-_0\subset S_{{\mathfrak{a}}'} $ are point-wise fixed
and  $z\in\Xi$ is mapped to  $\bm{x}(z)$. In addition the points of unglued nodal pairs, or the points of unglued ordered interface nodal pairs
are being point-wise preserved.
\qed
\end{prop}
\subsubsection{A Very Special Retraction}
We have the inclusion 
$$
Z^3_{\bm{\sigma},{\tiny\rup},{\mathcal H},\varphi}({\mathbb R}\times Q,\bar{\digamma})\rightarrow Z^3_{\bm{\sigma},{\tiny\rup},\varphi}({\mathbb R}\times Q,\bar{\digamma})
$$
 as sub-M-polyfold and coming from $\alpha$ there is the special element $\wt{u}^0$ belonging to $Z^3_{\bm{\sigma},{\tiny\rup},{\mathcal H},\varphi}({\mathbb R}\times Q,\bar{\digamma})$.
 Recall that ${\mathcal V}\subset H^1(\sigma)$ is the open neighborhood of $0$ occurring in the deformation $v\rightarrow j(v)$. We define $\bar{q}_0:=(0,\wt{u}^0)$ and shall 
 denote for a given map $\wt{w}$ by $[\wt{w}]$ the map obtained by making suitable ${\mathbb R}$-shifts so that the anchor averages vanish.
 Before we state the result we describe the idea. The element $\bar{q}_0=(0,\wt{u}^0)$ has the property that for $\wt{u}^0$ its anchor averages vanish and 
 the transversal constraints are satisfied in the sense
 $$
 (-\text{av}_{{\tiny\rup}_i}(\wt{u}^0))\ast \wt{u}^0(z)\in \wt{H}_{[z]}\ \text{for}\ z\in\Xi_i
 $$
 and the intersections are transversal. If we take an element $(v,\wt{u})$ near $(0,\wt{u}^0)$ in ${\mathcal V}\times Z^3_{\bm{\sigma},{\tiny\rup},\varphi}({\mathbb R}\times Q,\bar{\digamma})$
 the anchor averages are still vanishing, but the transversal constraints are not satisfied.  Recall that the constraints associated to $[z]$, where $z$ does not lie on a trivial cylinder component
 are ${\mathbb R}$-invariant, which, however is \underline{not!} the case if they lie on trivial cylinder components. 
 However, by a quite subtle implicit function theorem there exists  $(v',\mathfrak{a}')$,  a deformation $\bm{x}$ of $\Xi$ all depending on $(v,\wt{u})$, where we recall that $\mathfrak{a}=\mathfrak{a}(\wt{u})$\
 so that the associated stabilization deformation 
 $$
 \psi: (S_{\mathfrak{a}'},j(v')_{\mathfrak{a}'},\bar{M}_{\mathfrak{a}'},\bar{D}_{\mathfrak{a}'})\rightarrow (S_{\mathfrak{a}},j(v)_{\mathfrak{a}},\bar{M}(\bm{x})_{\mathfrak{a}},\bar{D}_{\mathfrak{a}})
 $$
 has the property that $\wt{w}'= \wt{w}\circ \psi$ satisfies the transversal constraints and adding suitable constants to the floor we obtain $[\wt{w}']$ so that $(v',[\wt{w}'])$ belong to 
 ${\mathcal V}\times Z^3_{\bm{\sigma},{\tiny\rup},{\mathcal H},\varphi}({\mathbb R}\times Q,\bar{\digamma})$.
Again,  the difficulty in the whole argument arises from the fact that the transversal constraints over trivial cylinder domains are \underline{not!} ${\mathbb R}$-invariant.

  \begin{thm}\label{good-sd-nei}
 There exists a $G$-invariant open neighborhood $\bm{O}$ of $\bar{q}_0$ 
 with following properties where $\bm{O}_{\mathcal H}:= \bm{O}\cap ({\mathcal V}\times Z^3_{\bm{\sigma},{\tiny\rup},{\mathcal H},\varphi}({\mathbb R}\times Q,{\bar{\digamma}}))$
 \begin{itemize}
 \item[(1)] We have a well-defined sc-smooth retraction $\rho:\bm{O}\rightarrow \bm{O}$
 with $\rho(\bm{O})=\bm{O}_{\mathcal H}$.
 \item[(2)]  The retraction in (1) has the following form. Given $q=(v,\wt{w})\in \bm{O}$
 the image $\rho(v,\wt{w})$ with ${\mathfrak{a}}$ being the underlying total gluing parameter
 is $(v',[\wt{w}'])$ with underlying total gluing parameter ${\mathfrak{a}}'$, where 
 $$
 \wt{w}' = \wt{w}\circ \psi_{({\mathfrak{a}}({q}),v({q}),\bm{x}({\mathfrak{a}}({q}),v({q}))},
 $$
 where $\bm{x}({\mathfrak{a}}({q}),v({q}))$ is determined by an implicit functions theorem and as a function of $q$
 is sc-smooth.
 \end{itemize}
 \end{thm}
As a consequence we can make the following crucial definition.
\begin{definition}\label{DEFR29.2.7xX}
Assume we are given a stable map $\alpha$ and have fixed the auxiliary structures $\bm{D},\rup,\Xi, {\mathcal H}$ and $\mathfrak{j}$ producing the 
DM-uniformizer $\Psi^\ast: G^\ast\ltimes O^{\ast}\rightarrow {\mathcal R}$ at the M-data $\bar{\sigma}$ and the sd-deformation 
$$
O^{\ast\ast}\times U^\ast(\Xi)\rightarrow O^{\ast}.
$$
Denote by $\bar{q}_0$ the element $(0,q_0)\in {\mathcal V}\times Z_{\bm{\sigma},{\tiny\rup},{\mathcal H},\varphi}({\mathbb R}\times Q,\bar{\digamma})$, where $q_0$ is  the element 
associated to $\alpha$. 
Then a  {\bf large sd-retraction neighborhood}\index{large  sd-retraction neighborhood} is a $G$-invariant open neighborhood $\bm{O}=\bm{O}(\bar{q}_0)$ in ${\mathcal V}\times Z_{\bm{\sigma},{\tiny\rup},\varphi}({\mathbb R}\times Q,\bar{\digamma})$ as guaranteed
by Theorem \ref{good-sd-nei}. In particular there exists an sc-smooth retraction $\rho:\bm{O}\rightarrow \bm{O}$
with image $ \bm{O}_{\mathcal H}$. This retraction has the form
$$
\rho(v,\wt{w}) = (v'({\mathfrak{a}}(\wt{w}),v),\wt{w}\circ \psi_{({\mathfrak{a}}(v,\wt{w}),v,\bm{x}({\mathfrak{a}}(\wt{w}),v))}).
$$
associated to the stabilization deformation. 
\qed
\end{definition}

\subsection{The  Uniformizer and Transition Germs}
Let  $\alpha$ be an object and $\textbf{pre}\Psi$ a pre-uniformizer at $\alpha$. 
\begin{definition}
A \textbf{good open neighborhood} $O$ of the special element $\bar{o}$  representing $\alpha$ is a $G$-invariant  open neighborhood
of $\bar{o}$ in $V\times Z^3_{\bm{\sigma},{\tiny\rup},{\mathcal H},\varphi}({\mathbb R}\times Q,\bar{\digamma})$ having the following properties 
\begin{itemize}
\item[(1)] For $(v,\wt{u})\in O$ we have that $(v,\mathfrak{a}(\wt{u}))$ belong to $O^{\ast\ast}$ so that the sd-deformation $O^{\ast\ast}\times U^{\ast}(\Sigma)\rightarrow O^{\ast}$ is defined,
where $\Psi^{\ast}: G^{\ast}\ltimes O^{\ast}\rightarrow {\mathcal R}$ is a good uniformizer for ${\mathcal R}$.
\item[(2)] The restriction of the pre-uniformizer to $O$, say $\Psi$ is injective on objects and otherwise full and faithful.
\item[(3)] $O=\bm{O}_{\mathcal H}$, where $\bm{O}$ is a large sd-retraction neighborhood in the sense of Definition \ref{DEFR29.2.7xX}.
\end{itemize}
\qed
\end{definition}
We define $F(\alpha)$ to be the set of all uniformizers obtained from pre-uniformizers by restricting to a good open neighborhood. 
This gives a uniformizer construction. Then everything is place and from the discussion in this section we obtain a transition germ construction.
Of course, it takes some work to verify the required properties.

\newpage

\part*{Lecture 11}
\section{Strong Polyfold Bundle Structure and Fredholm Theory}
We have concentrated on the polyfold construction for ${\mathcal S}$.  The construction of strong bundle structures for functors $\mu:{\mathcal S}\rightarrow \textbf{Ban}$ can be carried out 
in a similar fashion and we allow ourselves to be brief.  
After having fixed $J$ we have the functor 
$\mu_J:{\mathcal S}\rightarrow \textbf{Ban}$ which associates to an object $\alpha$ the $T({\mathbb R}\times Q)$-valued $(0,1)$-forms of class $(2,\delta_0)$ along the stable map.

   \begin{figure}[h]
\begin{center}
\includegraphics[width=6.5cm]{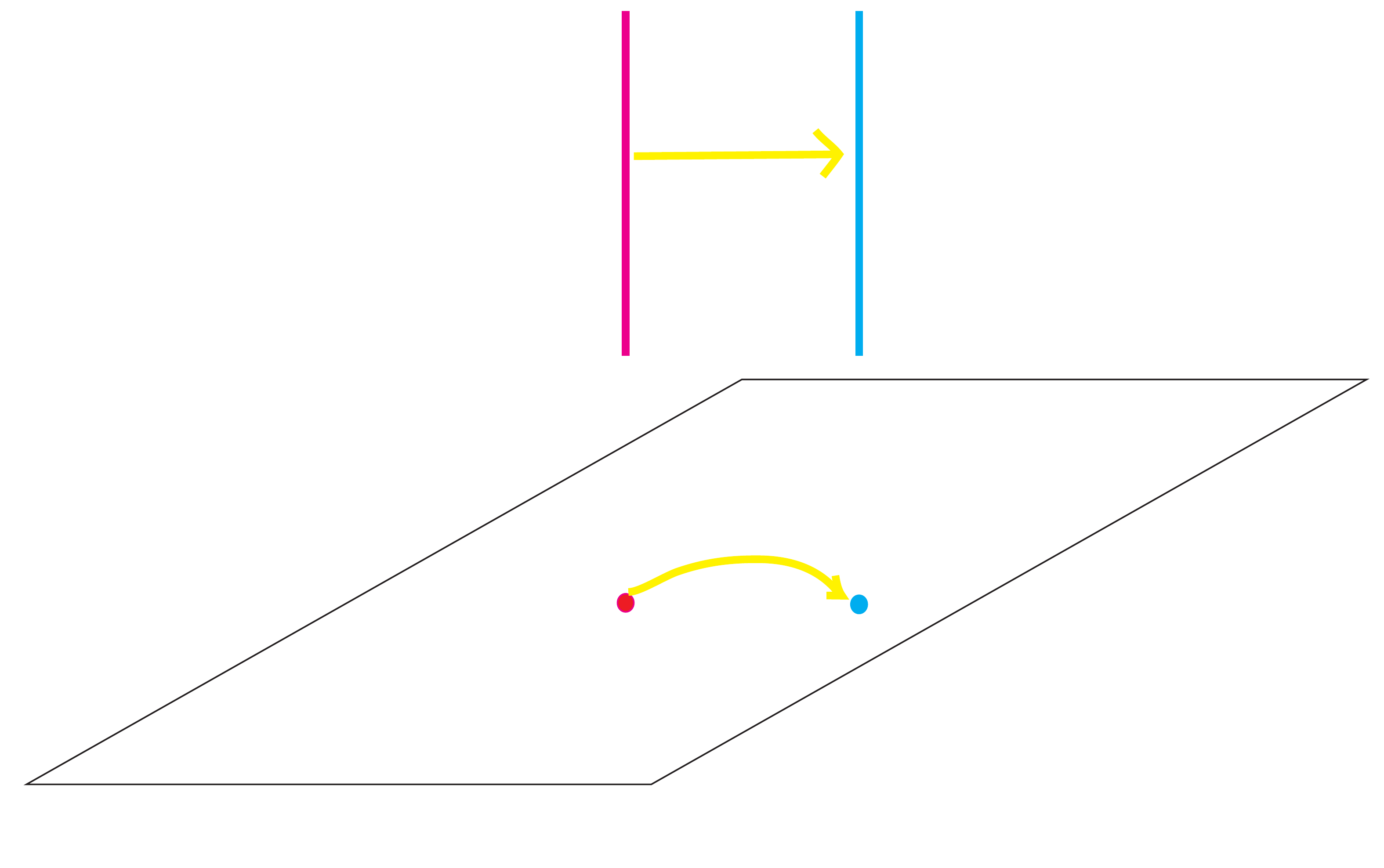}
\end{center}\caption{$\mu_J:{\mathcal S}\rightarrow \textbf{Ban}$.}
\end{figure}

The idea is to construct a strong bundle $K\rightarrow O$ and a lift $\bar{\Psi}$ for $\Psi\in F(\alpha)$ fitting into the commutative diagram
$$
\begin{CD}
G\ltimes K @>\bar{\Psi} >>  {\mathcal E}_J\\
@VVV @VVV\\
G\ltimes O @>\Psi>>   {\mathcal S}.
\end{CD}
$$
  \begin{figure}[h]
\begin{center}
\includegraphics[width=6.5cm]{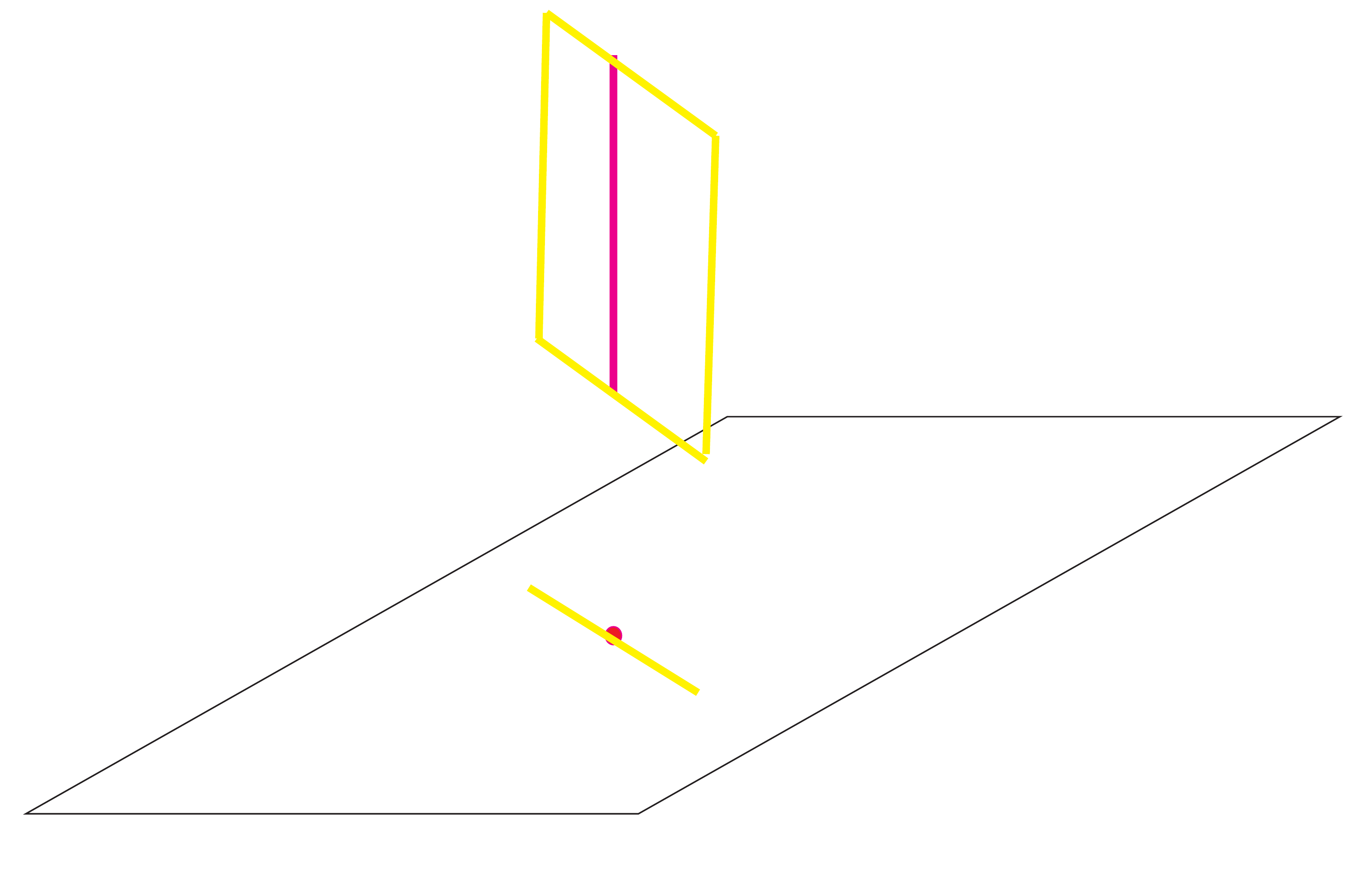}
\end{center}\caption{A $\bar{\Psi}$-slice}
\end{figure}

We shall mention the parts we need to also be able to introduce a special class of sc$^+$-sections, which will be used to define the special sc$^+$-multisection 
functors used for the perturbation theory.
\subsection{Remarks on Construction Functors for Strong Bundles}
We described the construction of the $\Psi$ in great detail and have used a variety of tools  to do so. In particular we showed that we have a construction functor 
$$
({\mathbb R}^N,\bar{\digamma})\rightarrow Z^3_{\bm{\sigma},{\tiny\rup},\varphi}({\mathbb R}\times {\mathbb R}^N,\bar{\digamma}).
$$
 This allowed
the extension to manifolds by a general method resulting in $Z^3_{\bm{\sigma},{\tiny\rup},\varphi}({\mathbb R}\times Q,\bar{\digamma})$.
 Finally having a construction for a manifold $Q$ one can introduce  transversal constraints and obtain a suitable sub-M-polyfold 
$$
Z^3_{\bm{\sigma},{\tiny\rup},{\mathcal H},\varphi}({\mathbb R}\times Q,\bar{\digamma})\subset Z^3_{\bm{\sigma},{\tiny\rup},\varphi}({\mathbb R}\times Q,\bar{\digamma}).
$$
In order to obtain a  construction  for strong complex bundles we again can use the idea of a construction functor. More precisely, we have a category whose objects are
$({\mathbb R}^N\times {\mathbb K}^L,\bar{\digamma})$, where ${\mathbb K}$ is either ${\mathbb R}$ or ${\mathbb C}$. We view ${\mathbb R}^N\times {\mathbb K}^L\rightarrow {\mathbb R}^N$
as the obvious trivial vector bundle and $({\mathbb R}^N,\bar{\digamma})$ is a previously considered object, namely a collection of weighted periodic orbits in ${\mathbb R}^N$.
The morphisms are smooth maps $A:{\mathbb R}^N\times {\mathbb K}^L\rightarrow {\mathbb R}^{N'}\times {\mathbb K}^{L'}$ of the form $(m,\ell)\rightarrow (f(m),A(m)(\ell))$,
where $A(m)$ is ${\mathbb K}$-linear. Moreover, $f:({\mathbb R}^N,\bar{\digamma})\rightarrow ({\mathbb R}^{N'},\bar{\digamma}')$ is a morphism in the obvious category, which was previously introduced.
The new construction functor is build as follows. 
\begin{definition}\label{DEFX11.1}
We  define the sc-Hilbert space $H^{2,\delta}_{\bar{\sigma}}({\mathbb C}^L)$ to  consist of all continuous maps 
$\eta$, which associate to $z\in S$ a complex anti-linear map $\eta(z):(T_zS,j)\rightarrow {\mathbb C}^L$ so that the following holds. 
\begin{itemize}
\item[(1)] $\eta(z)=0$ for $z\in |\bar{D}|\cup\Gamma^-_0\cup\Gamma^+_k$.
\item[(2)] Away from points in $|\bar{D}|\cup\Gamma^-_0\cup\Gamma^+_k$ the map $z\rightarrow \eta(z)$ is of class $H^2_{loc}$.
\item[(3)] For every $x\in |\bar{D}|\cup\Gamma^-_0\cup\Gamma^+_k$ taking positive holomorphic polar coordinates around $x$
the map $(s,t)\rightarrow \eta\circ \frac{\partial\sigma_{\wh{x}}^+}{\partial s}$ belongs to $H^{2,\delta_0}({\mathbb R}^+\times S^1,{\mathbb C}^N)$.
\end{itemize}
The sc-structure is given by defining level $m$ as regularity $(2+m,\delta_m)$.  
\qed
\end{definition}After fixing a small disk structure $\bm{D}$ we can define 
$X^{2,\delta_0}_{\bm{\sigma},\varphi,0}({\mathbb C}^L)$ to consist of all maps $\eta$ on the different glued surfaces so that 
$\eta(z)$ is complex anti-linear, is away from nodes of class $H^2_{loc}$ and has at nodal points or punctures  with respect to holomorphic polar coordinates
the $(2,\delta_0)$ behavior.
This set can be defined easily by the imprinting method
 $$
 \oplus:{\mathbb B}_{\bar{D}}\times H^{2,\delta}_{\bar{\sigma}}({\mathbb C}^L)\rightarrow X^{2,\delta_0}_{\bm{\sigma},\varphi,0}({\mathbb C}^L).
$$
In this case the extraction of gluing parameters is submersive. 
We know also know that the map
$$
\bar{a}:Z_{\bm{\sigma},{\tiny\rup},\varphi}({\mathbb R}\times {\mathbb R}^N,\bar{\digamma}) \rightarrow {\mathbb B}_{\bar{D}},
$$
extracting the global gluing parameter, is sc-smooth.
 We can consider the pull-back diagram
$$
\begin{CD} @.   X^{2,\delta}_{\bm{\sigma},\varphi,0}({\mathbb C}^L)\\
@.   @VVV\\
Z^3_{\bm{\sigma},{\tiny\rup},\varphi}({\mathbb R}\times {\mathbb R}^N,\bar{\digamma})@>>>  {\mathbb B}_{\bar{D}}
\end{CD}
$$
We denote the pull-back functor by 
$$
\Omega^{3,2}_{\bm{\sigma},{\tiny\rup},\varphi}(({\mathbb R}\times {\mathbb R}^N\times {\mathbb C}^L,\bar{\digamma})\rightarrow Z^3_{\bm{\sigma},{\tiny\rup},\varphi}({\mathbb R}\times {\mathbb R}^N,\bar{\digamma}),
$$
which is a strong bundle construction functor.  By using a previous ideas we can use the embedding method to define for the complex vector bundle $({\mathbb R}\times TQ,\wt{J})\rightarrow Q$. 
the strong bundle 
$$
\Omega^{3,2}_{\bm{\sigma},{\tiny\rup},\varphi}(({\mathbb R}\times ({\mathbb R}\times TQ),\bar{\digamma})\rightarrow Z^3_{\bm{\sigma},{\tiny\rup},\varphi}({\mathbb R}\times Q,\bar{\digamma}).
$$
Here ${\mathbb R}\times  ({\mathbb R}\times TQ)\rightarrow {\mathbb R}\times Q$ is the complex vector bundle with fiber over $(a,q)$ to consist
of all $(a,(h,b))$ with $h\in {\mathbb R}$ and $b\in T_qQ$. With other words we take the pull-back of $({\mathbb R}\times TQ,\wt{J})\rightarrow Q$ by the projection ${\mathbb R}\times Q\rightarrow Q$.
Taking transversal constraints for the basis we obtain the pull-back diagram
$$
\begin{CD} @.   \Omega^{3,2}_{\bm{\sigma},{\tiny\rup},\varphi}({\mathbb R}\times ({\mathbb R}\times TQ),\bar{\digamma})\\
@.   @VVV\\
Z^3_{\bm{\sigma},{\tiny\rup},{\mathcal H}, \varphi}({\mathbb R}\times Q,\bar{\digamma})@>>>  Z^3_{\bm{\sigma},{\tiny\rup}, \varphi}({\mathbb R}\times Q,\bar{\digamma}).
\end{CD}
$$
We denote this pull-back bundle by
\begin{eqnarray}\label{eqn18}
\boxed{\Omega^{3,2}_{\bm{\sigma},{\tiny\rup},{\mathcal H},\varphi}(({\mathbb R}\times ({\mathbb R}\times TQ),\bar{\digamma})\rightarrow Z^3_{\bm{\sigma},{\tiny\rup},{\mathcal H},\varphi}({\mathbb R}\times Q,\bar{\digamma})}
\end{eqnarray}
\subsection{Incorporating $\mathfrak{j}$}
In the pre-uniformizer construction we also fixed a deformation $\mathfrak{j}$, $v\rightarrow j(v),\ v\in V$,  of $j$ after having fixed a small disk structure $\bm{D}$.
In (\ref{eqn18}), so far, the construction only uses $j$. We define for $v\in V$ and a gluing parameter $\mathfrak{a}$ the map
$$
\chi_{(v,\mathfrak{a})} =\frac{1}{2}\cdot [\text{Id} - j\circ j(v)]_{\mathfrak{a}}: (TS_{\mathfrak{a}},j(v)_{\mathfrak{a}})\rightarrow (TS_{\mathfrak{a}},j_{\mathfrak{a}}).
$$
Given $(\wt{u},\wt{\eta})\in \Omega^{3,2}_{\bm{\sigma},{\tiny\rup},\varphi}({\mathbb R}\times ({\mathbb R}\times TQ),\bar{\digamma})$ we see that 
$\wt{\eta}\circ \chi_{(v,\mathfrak{a})}$ is point-wise a complex anti-linear map if the domain is equipped with $j(v)_{\mathfrak{a}}$. 
Of course, here we have that $\mathfrak{a}=\mathfrak{a}(\wt{u})$ is the underlying gluing parameter. 

Given the uniformizer $\Psi\in F(\alpha)$ as previously constructed, say $\Psi: G\ltimes O\rightarrow {\mathcal S}$
we define the strong bundle $p:K\rightarrow O$, i.e. $K$,  by first considering the pull-back of (\ref{eqn18}) by 
\begin{eqnarray}\label{DISP19}
V\times Z^3_{\bm{\sigma},{\tiny\rup},{\mathcal H},\varphi}({\mathbb R}\times Q,\bar{\digamma})\rightarrow Z^3_{\bm{\sigma},{\tiny\rup},{\mathcal H},\varphi}({\mathbb R}\times Q,\bar{\digamma})
\end{eqnarray}
and then restricting the result to $O$. Finally we define the lift $\bar{\Psi}$ of $\Psi$ by 
$$
\bar{\Psi}(v,\wt{u},\wt{\eta}) = (\Psi(v,\wt{u}), \wt{\eta}\circ \chi_{(v,\mathfrak{a}(\wt{u}))}).
$$
The lift of the transition germs for ${\mathcal S}$ to ${\mathcal E}$ is straight forward. As a result we obtain
\begin{thm}
The functor which associates to an object $\alpha$ the collection $\bar{F}(\alpha)$ of lifts of uniformizers 
together with the lift of the transition germ construction $\bar{\mathscr{F}}$ defines metrizable topologies on $|{\mathcal S}|$ and $|{\mathcal E}|$
and viewing the categories as GCT's a strong bundle structure for $P:{\mathcal E}_J\rightarrow {\mathcal S}$. 
\qed
\end{thm}

\subsection{The sc-Fredholm Functor $\bar{\partial}_{\widetilde{J}}$}
Having the strong bundle construction $\bar{F}:{\mathcal S}\rightarrow \textbf{Ban}$ we have for an object $\alpha$ and $\bar{\Psi}\in \bar{F}(\alpha)$
the commutative diagram
$$
\begin{CD}
G\ltimes K @>\bar{\Psi}>> {\mathcal E}_J\\
@V\bar{\partial}_{\wt{J},\bar{\Psi}}VV @V P VV\\
G\ltimes O @>\Psi >>  {\mathcal S}
\end{CD}
$$
where $\bar{\partial}_{\wt{J},\bar{\Psi}}$ is the local representative.  
\begin{thm}
For every $\bar{\Psi}$ the local representative is sc-Fredholm, i.e. by definition $\bar{\partial}_{\wt{J}}$ is a sc-Fredholm functor.
\qed
\end{thm}
One should mention  that in \cite{FH-book} a pre-Fredholm theory has been developed which is some kind of modular theory
which guarantees the Fredholm property as a consequence of smaller pieces of analysis, which makes the Fredholm theory 
rather straight forward and most importantly the smaller pieces of analysis can be recycled for new constructions. We shall not discuss
these ideas here and refer the reader for a detailed discussion to \cite{FH-book}.

\subsection{Reflexive Local Compactness Property}
We know that for every $a\in \pi_0(Z)$ the intersection $a\cap \overline{\mathcal M}_J$ is compact and that 
$\bar{\partial}_{\wt{J}}$ is a sc-Fredholm functor. We begin by stating facts which follow from standard
polyfold theory. However, it should be pointed out that the category of stable maps has additional features
which have to be exploited in order to construct SFT. First we discuss the standard parts of the polyfold theory,
and the special features will, as we shall see, add to the fine structure.

For the perturbation theory the so-called reflexive local compactness 
property will be very important. An auxiliary norm $N:{\mathcal E}_{J}\rightarrow [0,\infty]$ is a functor 
so that for a local strong bundle uniformizer $\bar{\Psi}$ we have that $N\circ\bar{\Psi}:K\rightarrow [0,\infty]$ is an auxiliary norm.
The $(0,1)$- fibers of $K\rightarrow O$ are Hilbert spaces and therefore reflexive.   As shown in \cite{HWZ3} there is a well-defined notion 
of \textbf{mixed convergence}  for sequences in $K$  which in local coordinates corresponds to convergence in $O$ on level $0$
and in the fiber to weak convergence in the $(0,1)$-fiber. For a  comprehensive treatment see  \cite{HWZ2017}. We write 
$k_i\stackrel{m}{\rightharpoonup} k$ if $(k_i)$ converges in this sense. The notion also descends to orbit spaces.

  There is a particular important class of auxiliary norm 
called {\bf reflexive auxiliary norms}, which have the additional property that for a sequence $(k_i)\subset K$ 
with $p(k_i)\rightarrow x$ in $O$ and $\text{liminf}_{i\rightarrow \infty} N\circ \bar{\Psi}(k_i)<\infty$, there exists a subsequence 
such that $k_i\stackrel{m}{\rightharpoonup} k$ for some $k\in p^{-1}(x)$ in the $(0,1)$-fiber with $N\circ \bar{\Psi}(k)\leq \text{liminf}\ N\circ\bar{\Psi}(k_i)$.
If we pass to orbit space $|{\mathcal E}|$ it still makes sense to talk about mixed convergence and $N$ defines a map
$n$ which restricted to $|{\mathcal E}_{(0,1)}|$ is continuous and verifies the obvious version of the mixed convergence requirement.
The final goal in this lecture is to state important local results for $f:=|\bar{\partial}_{\wt{J}}|$. Recall the abbreviation $Z:=|{\mathcal S}|$.
\begin{thm}
For every reflexive auxiliary norm $N$ with associated $n:|{\mathcal E}|\rightarrow [0,\infty]$ the following holds. Given a point $z\in Z$ there exists 
an open neighborhood $U(z)$ with the property that $\text{cl}_Z(\{y\in U(z)\ |\ n(f(y))\leq 1 \})$ is compact. 
\qed
\end{thm}
The other result we need is the following.
\begin{prop}\label{prop11.5}
Assume that $f(z)\in (1,\infty]$.  Then there exists an open neighborhood $U(z)$ such that 
$f(y)>1$ for all $y\in U(z)$.
\end{prop}
\begin{proof}
Arguing indirectly we find a sequence $(z_k)$ converging to $z$ in $Z$ and 
$$
\text{liminif}\ n(f(z_k))\leq 1.
$$
We deduce that without loss of generality we may assume that $f(z_k)\stackrel{m}{\rightharpoonup} k\in |{\mathcal E}_{(0,1)}|$, which also implies $f(z_k)\rightarrow k$ in $|{\mathcal E}_{(0,0)}|$. Hence $1<n(f(z))=n(k)\leq 1$ giving a contradiction.
\end{proof}
\newpage

%\newpage
\part*{Lecture 12}
\section{Accommodation of Special Features}

\subsection{A Strong ssc-Bundle}
Recall that $\sigma$ and $\bm{\sigma}$ have a floor structure.  We introduced the open subset ${\mathcal O}$ of $[0,1)^k\times Z^3_{\sigma,{\tiny\rup}}({\mathbb R}\times {\mathbb R}^N,\bar{\digamma})$
in Definition \ref{DEFX8.2}. Recall the sc-Hilbert space $H^{2,\delta}_{\bar{\sigma}}({\mathbb C}^L)$ introduced in Definition \ref{DEFX11.1}.
  Recall that $D$ denotes all nodal pairs occurring on floors and 
 ${\mathbb B}_D$ the manifold of associated gluing parameters.  Then 
 $$
 \left({\mathbb B}_D\times {\mathcal O}\right)\triangleleft  H^{2,\delta}_{\bar{\sigma}}({\mathbb C}^L)\rightarrow {\mathcal O}
 $$
  is a strong ssc-bundle. We recall that there is a sc-smooth map associating the $(\mathfrak{a},(r_1,...,r_k,\wt{u}))$ the global admissible gluing parameter $\wt{\mathfrak{a}}=\wt{\mathfrak{a}}(\mathfrak{a},(r_1,...,r_k,\wt{u}))$, see Definition \ref{DEFX8.3.1}.  Since we started with a stable map $\alpha$ in ${\mathcal S}$ we can distinguish between different types of domain components. 
  \begin{definition}
  Let $\alpha$ be a stable map in ${\mathcal S}$ and $\sigma$ the underlying Riemann surface with floor structure. A domain component $C$ of $\sigma$ is a \textbf{trivial cylinder component}
  provided it harbors a trivial cylinder as part of $\alpha$. All other components are called \textbf{nontrivial components}.
  \qed
  \end{definition}

\subsection{Adapted Auxiliary Norm}
Recall the strong bundle 
$$
\Omega^{3,2}_{\bm{\sigma},{\tiny\rup},\varphi}(({\mathbb R}\times ({\mathbb R}\times {\mathbb R}^N\times 
{\mathbb C}^L,\bar{\digamma})\rightarrow Z^3_{\bm{\sigma},{\tiny\rup},\varphi}({\mathbb R}\times {\mathbb R}^N,\bar{\digamma}) 
$$
which comes from the pull-back diagram
$$
\begin{CD}
@.       X^{2,\delta}_{\bm{\sigma},\varphi,0}({\mathbb C}^L)\\
@.     @VVV\\
Z^3_{\bm{\sigma},{\tiny\rup},\varphi}({\mathbb R}\times {\mathbb R}^N\times {\mathbb C}^L,\bar{\digamma})@>>>{\mathbb B}_{\bar{D}}
\end{CD}
$$
The horizontal arrow has its image in the admissible gluing parameters.
We shall introduce a  map $\wh{N}:  X^{2,\delta}_{\bm{\sigma},\varphi,0}({\mathbb C}^L)\rightarrow [0,\infty]$
with suitable properties. Each domain component has a a floor number coming from the original data.
Given the gluing parameter $\wt{\mathfrak{a}}$ we have the associated domain $\sigma_{\wt{\mathfrak{a}}}$.
\begin{itemize}
\item On the core region we define the weight  function $w\equiv 1$. 
\item On an unglued trivial cylinder segment define $w\equiv \infty$.
\item On an unglued disk associated to a nodal point we extend $w$ by $w=e^{\delta_1 s}$ using positive 
holomorphic polar coordinates. 
\item On a glued disk pair associated to a nodal pair biholomorphic to $[0,R]\times S^1$
we set $w(s,t) = \text{min} \{ e^{\delta_1 s},e^{\delta_1 (R-s)}\}$
\item For a puncture associated to $\Gamma^-_0\cup\Gamma^+_k$ we take the weight
$e^{\delta_1 s}$ for positive holomorphic polar coordinates.
\item For an unglued puncture pair where each of the punctures does not belong 
to a trivial cylinder domain we take on the disks the weight $e^{\delta_1 s}$.
\item  For a glued puncture pair where each of the punctures does not belong 
to a trivial cylinder domain we take on the glued disks which is biholomorphic to  $[0,R]\times S^1$ the weight $\text{min} \{e^{\delta_1 s},e^{\delta_1 (R-s)}\}$.
\end{itemize}
The remaining cases consist of looking at maximal chains of glued trivial cylinder components and we can distinguish
between four cases.

 \begin{figure}[htp]
\begin{center}
\includegraphics[width=9.5cm]{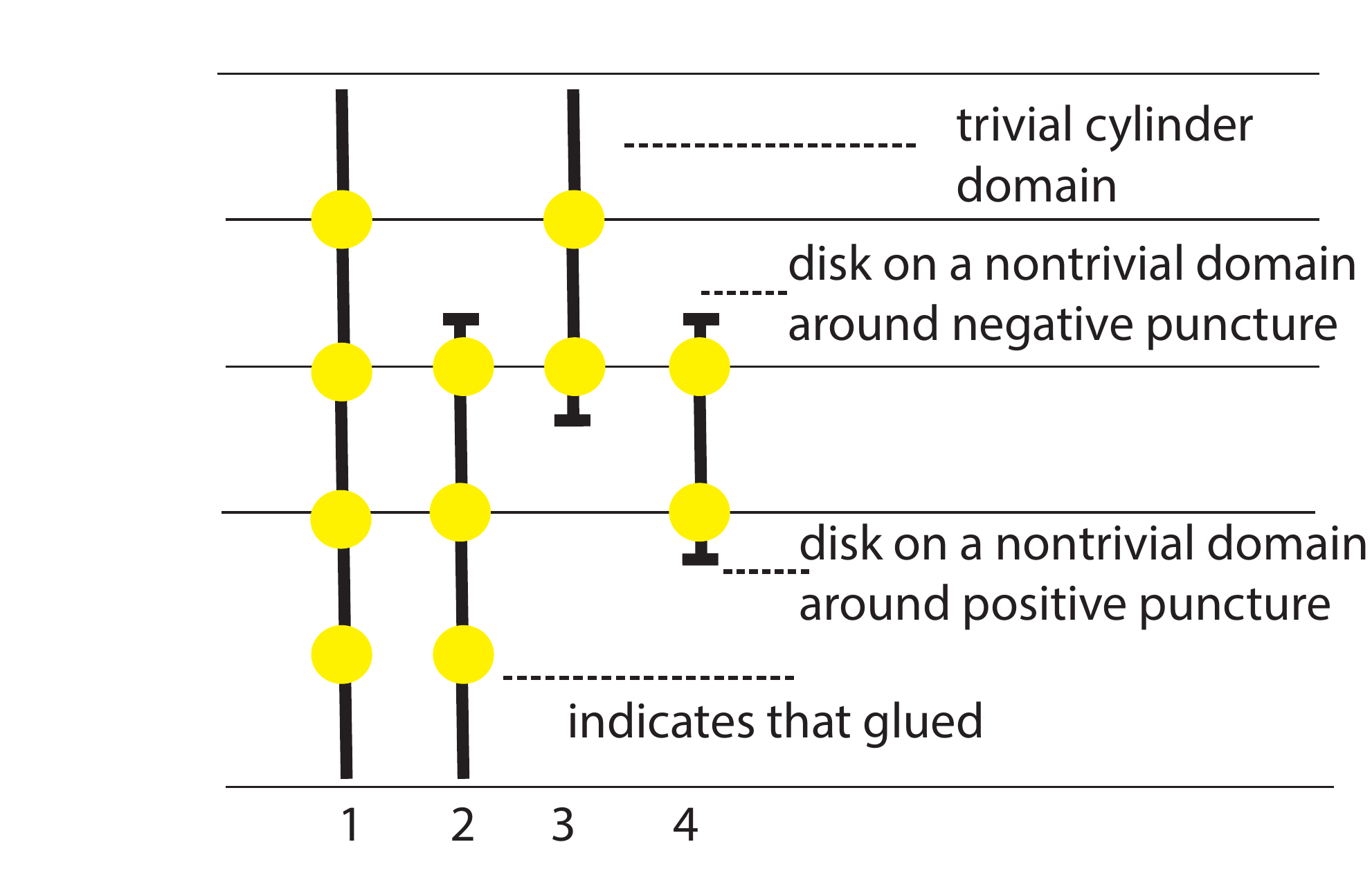}
\end{center}
\caption{Configurations involving trivial cylinder domains}\label{FIG XXX67}
\end{figure}
\begin{itemize}
\item In case (1) we have a finite sequence of glued trivial cylinder components, which, of course, itself is a new trivial cylinder component. We take the weight $w\equiv \infty$.
\item In case (2) we have at the top a disk around a negative puncture and otherwise glued  trivial cylinder components. We just extend the weight $w=e^{\delta_1|s|}$ on the disk with respect to negative  holomorphic polar coordinates.
\item In case (3) similarly a
disk around positive puncture followed by glued trivial cylinder components. Here we extend $w=e^{\delta_1 s}$.
\item  Finally in case (4) two disks at a positive and negative puncture respectively and 
at least one glued trivial cylinder component in between. In this case we have two extensions 
of the weight functions associated to negative (positive) holomorphic polar coordinates on the disk, say
$e^{-\delta_1 s'}$  (negative $s\in [-R,0]$) and $e^{\delta_1 s}$ (positive $s\in [0,R]$). With the relation 
$R+s'=s$ we define the weight function 
$$
\min\ \{ e^{\delta_1(R-s)}, e^{\delta_1 s}\}
$$
\end{itemize}
These choices define for fixed $\wt{\mathfrak{a}}$ a weight function $w_{\wt{\mathfrak{a}}}$.
Then we consider $\wt{h}$  on $S_{\wt{\mathfrak{a}}}$  and use the weighted $H^3$-norm to obtain
$$
|\wt{h}|^2_{S_{\wt{\mathfrak{a}}}}.
$$
This expressions define $\wh{N}: \Omega^{3,2}_{\bm{\sigma},{\tiny\rup},\varphi}\rightarrow [0,\infty]$.
\begin{eqnarray}\label{q17}
\wh{N}(((\mathfrak{a},(r_1,...,r_k,\wt{u})),\wt{h}).
\end{eqnarray}
We note that if $(r_1,...,r_k)=0$ and $\wh{N}(((\mathfrak{a},(r_1,...,r_k,\wt{u})),\wt{h}) <\infty$, then $\wt{h}$ necessarily vanishes on trivial cylinder components.
\begin{definition}
We shall call $\wh{N}$ a \textbf{penalizing adapted auxiliary norm}, and for short a \textbf{ps-norm}.
\qed
\end{definition}
Compared to the usual definition of auxiliary norm it can take on non-zero vectors on the $(0,1)$-fiber the value $\infty$. 
The definition of $\wh{N}$ depends on choices, however, when we construct $\wh{N}'$ making different choices there exists a constant $c>0$ such that $c\cdot \wh{N}\leq \wh{N}'\leq \frac{1}{c}\cdot \wh{N}$. This is the local picture.
Invoking the embedding method and strong bundle uniformizers we can define $\wh{N}:{\mathcal E}\rightarrow [0,\infty]$.  If we obtain $\wh{N}$ and $\wh{N}'$ this way there will be a continuous functor ($f\circ\Psi$ is continuous)
$f:{\mathcal S}\rightarrow (0,\infty)$ satisfying 
\begin{eqnarray}\label{DISPLAY}
f\cdot \wh{N}\leq \wh{N}'\leq \frac{1}{f}\cdot \wh{N}.
\end{eqnarray}
With other words we obtain a class of model ps-norms which show the compatibility 
as in (\ref{DISPLAY}) for certain $f$. We can then consider abstract ps-norms with the obvious expected 
properties and which can be sandwiched between model ps-norms. 
\begin{definition}
A (general) \textbf{ps-norm} is a functor $\wh{N}:{\mathcal E}\rightarrow [0,\infty]$ 
having the following properties.
 \begin{itemize}
 \item[(0)] For a suitable continuous functor $f:{\mathcal S}\rightarrow(0,\infty)$ and a model ps-norm $\wh{N}'$
 it holds $f\cdot \wh{N}'\leq\wh{N}\leq \frac{1}{f}\cdot \wh{N}'$.
 \item[(1)] $\wh{N}:{\mathcal E}\rightarrow [0,\infty]$ is a functor and for a given object $\alpha$
 the subset of ${\mathcal E}$ with $\wh{N}(\alpha,e)<\infty$ is a vector space and the restriction of $\wh{N}$ to it is a complete norm.
 \item[(2)] If $k_j:=|(\alpha_j,e_j)|\stackrel{m}{\rightharpoonup} k:=|(\alpha,e)|$ then $\wh{N}(\alpha,e)\leq \text{liminf}\ \wh{N}(\alpha_j,e_j)$.
 \item[(3)] If $\wh{n}(k_j)$ is bounded and the underlying $z_j$ converges then there exists a mixed convergent 
 subsequence.
 \item[(4)] $\wh{N}(\alpha,e)=\infty$ if there exists a nontrivial cylinder component on which $e$ is nonzero.
 \item[(5)] $\wh{n}:|{\mathcal E}_{(0,1)}|\rightarrow [0,\infty]$ is continuous, where $[0,\infty]$ is equipped with the topology of the 1-point compactification.
 \end{itemize}
\qed
\end{definition}
The ps-norms can be constructed by a continuous partitions of unity using that $Z$ is metrizable.

\subsection{Special  Sc$^+$-Section}
Having the ps-norms at hand we can introduce the sc$^+$-sections we are interested in. 
 We go back to the model
where we defined $\wh{N}$ locally, see (\ref{q17}). Consider a sc$^+$-section $\mathsf{f}$ 
of $\Omega^{3,1}_{\bm{\sigma},{\tiny\rup},{\mathcal H},\varphi}\rightarrow Z^3_{\bm{\sigma},{\tiny\rup},{\mathcal H},\varphi}$.  We say that $\mathsf{f}$ \textbf{vanishes strongly  near special points}, i.e.  punctures and nodal points provided 
for every $\bar{\wt{u}}$ there exists an open neighborhood $U$ of the set of punctures in $\Gamma^-_0\cup\Gamma^+_k$,
the unglued nodal pairs and the unglued interface puncture pairs 
 defining in an obvious sense open subsets 
of the glued surfaces denoted by $U_{\wt{\mathfrak{a}}}$ for $\wt{\mathfrak{a}}=\wt{\mathfrak{a}}(\wt{u})$
and $\wt{u}$ near $\bar{\wt{u}}$ such that for $\wt{u}$ near $\bar{\wt{u}}$ the element $\mathsf{f}(\wt{u})$
vanishes on the just constructed neighborhoods.

\begin{definition}
Consider an sc$^+$-section $\mathsf{f}$  of our strong bundle defined over an open subset  $U$ of the base space $ {\mathbb B}_D\times {\mathcal O}$. 
We say $\mathsf{f}$ is \textbf{special} provided the following holds.
\begin{itemize}
\item[(1)] $N\circ \mathsf{f}(\mathfrak{a},(r_1,...,r_k,\wt{u})) <\infty$.
\item[(2)] $\mathsf{f}$ vanishes strongly near special points.
\end{itemize}
\qed
\end{definition}
The definition of special sc$^+$-sections does not depend on the choices involved by the previously stated facts.
\subsection{Special sc$^+$-Multisection Functors}
We have introduced the notion of a special sc$^+$-section.
Just using such sc$^+$-sections we can defined define associated sc$^+$-multisection functors and 
which we shall call special sc$^+$-multisection.

\begin{definition}
A functor $\Lambda:{\mathcal E}\rightarrow {\mathbb Q}^+\cap [0,1]$ is called a \textbf{special sc$^+$-multisection functor}
provided for an object $\alpha$ and $\bar{\Psi}\in \bar{F}(\alpha)$ there exists an open $G$-invariant neighborhood $U$ of $\bar{o}\in O$ 
and a finite set of special sc$^+$-sections ${(s_i)}_{i\in I}$ of $p:K\rightarrow O$ defined over $U$, together with an action $G$ on $I$ 
satisfying 
$$
s_{g(i)}(g\ast o) = g\ast s_i(o)\ \text{for}\ o\in U,\ g\in G
$$
such that for $k\in K$ with $p(k)\in U$ it holds
$$
\Lambda\circ \bar{\Psi}(k) =\frac{1}{\sharp I} \cdot \sharp\{i\in I\ |\ s_i(p(k))=k\}.
$$
\qed
\end{definition}
We note that this only has to be tested by a set of $\bar{\Psi}$ so that the associated $U_{\bar{\Psi}}$ have the property that $|\Psi(U_{\bar{\Psi}})|$ cover $Z$. 
These sc$^+$-multisections are easy to construct when we have sc-smooth partitions of unity, which happens to be the case in our application.
For this fix any smooth object $\alpha$ and take $\bar{\Psi}\in \bar{F}(\alpha)$, say $\bar{\Psi}:G\ltimes K\rightarrow {\mathcal E}$.
Pick an invariant open neighborhood $U$ of $\bar{o}$ such that $\text{cl}_Z(|\Psi(U)|)\subset |\Psi(O)|$. Then for given smooth vector 
$e\in K_{\bar{o}}$, which vanishes on trivial cylinder segments as well as near punctures and nodal points,   we find an special  sc$^+$-section with support in $U$ and $s(\bar{o})=e$.
Then move this around by $G$ to obtain $s_g$. Define $\Lambda$ over $\bar{\Psi}(K)$ by 
$$
\Lambda\circ \bar{\Psi}(k) =\frac{1}{\sharp G}\cdot \sharp\{g\in G\ |\ s_g(p(k))=k\}
$$
Generally if there exists a morphism $(\alpha',e')\rightarrow \bar{\Psi}(k)$ for some $k\in K$ define $\Lambda(\alpha',e'):=\Lambda\circ\bar{\Psi}(k)$. 
If no such morphism exists define $\Lambda(\alpha',e'):=\Lambda_0(\alpha',e')$.  There are several operations which are important.
\begin{itemize}
\item[(1)] $\Lambda\oplus\Lambda'(\alpha,e) = \sum_{e'+e''=e} \Lambda(\alpha,e')\cdot\Lambda'(\alpha,e')\ \ \textbf{(convolution sum)}$
\item[(2)] For a sc-smooth functor $\beta:{\mathcal S}\rightarrow {\mathbb R}$ define $\beta\odot \Lambda$ by 
$\beta(\alpha)\odot \Lambda(\alpha,e)=\Lambda_0(\alpha,e)$ if $\beta(\alpha)=0$ and otherwise by 
$\beta(\alpha)\odot \Lambda(\alpha,e)= \Lambda(\alpha,(1/\beta(\alpha))e)$.  This is called \textbf{(rescaling)}.
\end{itemize}
They also behave well with respect to proper coverings, which we need, but shall not discuss further.
The upshot of this discussion is that there are many special  sc$^+$-multisections to address all occurring transversality questions.
However, for inductive constructions we need extension theorems for sc$^+$-multisection functors defined on the boundary.
With the definition given above such extension theorems might not exist and the problems are discussed in detail in \cite{HS}.  What we need is a subclass of special sc$^+$-multisection functors which is rich enough to achieve transversality and also admits a controlled extension result.  The notions which are important are that of structured or structurable sc$^+$-multisections introduced in \cite{HWZ2017}
and that of $\bm{V}$-structured or $\bm{V}$-structurable sc$^+$-multisection functors introduced in \cite{HS}. In particular \cite{HS} contains 
a discussion of the relationships between different notions. In fact the notions of being structurable  and being $\bm{V}$-structurable are equivalent,
whereas a structure or a $\bm{V}$-structure being quite different concepts. Moreover there are enough special sc$^+$-multisections to achieve transversality.

\subsection{The Strong Topology $\wh{\mathcal T}$ on $\wh{Z}$}
The strong topology has to be seen in connection with the special type of sc$^+$-sections and multisections we are going to use.
We shall describe $\wh{\mathcal T}$ in detail in \cite{FH-book} and restrict ourselves to describe the properties of the topology
$ \wh{\mathcal T}$.
\begin{itemize}
\item[(1)]  For a parent $a\in \pi^p_0(Z)$ the sets $\{z\in a\cap \wh{Z}\ |\ d_Z(z)=0\}$ and $\{z\in a\ |\ d_Z(z)=0\}$ are the same  the topologies
$\wh{\mathcal T}$ and ${\mathcal T}$ restricted to $\{z\in a\ |\ d_Z(z)=1\}$ coincide.
\item[(2)]  Given $z\in \wh{Z}$, say $z\in a $, where $a\in \pi_0(Z)$,  there exists an open ${\mathcal T}$-neighborhood $U=U(z)\subset a$ and a 
ps-norm $\wh{N}$ defined on ${\mathcal E}_{\text{cl}_Z(U(z))}$ with the following property, where $f=|\bar{\partial}_{\wt{J}}|$ and $\wh{n}$ is induced by $\wh{N}$: Given  $(z_k)\subset \text{cl}_Z(U(z))$ with $\wh{n}\circ f(z_k)\leq 1$  there exists a convergent subsequence with respect
to ${\mathcal T}$. Every such convergent subsequence also converges with respect to $\wh{\mathcal T}$, and in particular the limit belongs to 
$\wh{Z}\cap \text{cl}_Z(U(z))$.
\end{itemize}
Property (2) essentially  characterizes $\wh{\mathcal T}$.  This topology will be discussed in detail in \cite{FH-book}.
%We refer the reader for more details about $\wh{\mathcal T}$ to Appendix \ref{hatt}.

\newpage

\part{Perturbation and Transversality Theory}
We shall describe the ingredients of a perturbation and transversality theory in detail.

%%\input{new-parts}
%\newpage
\part*{Lecture 13}
\section{Inductive Compactness Control}
The perturbation theory will proceed inductively and SFT will exhibit some of the more subtle aspects 
which one can encounter in these kind of problems.  We shall consider here the case of a 
 closed manifold $Q$ equipped with a non-degenerate contact form $\lambda$, i.e. $(Q,\lambda,d\lambda)$.
 In this case we can organize the induction rather than with respect to $d^J$ in a different way, which 
 is less demanding in its constructive aspects.\\
\begin{center}
\textbf{Standing assumption from now on:}\\  
$(Q,\lambda,d\lambda)$ is a non-degenerate contact form.
\end{center}
\vspace{0.5cm}
The consequence of this assumption is the following.
\begin{prop}
For every $a\in \pi_0(Z)$ it holds that 
$$
\bar{d}_a:=\text{\em max}\ \{ d_Z(z)\ |\ z\in a\} <\infty.
$$
\qed
\end{prop}
With other words, every connected component only has a finite number of faces.  This does not hold in general
for the stable Hamiltonian case.
\begin{definition}
The \textbf{complexity} is the map $\bar{d}:\pi_0(Z)\rightarrow \{0,1,2,...\}:a\rightarrow \bar{d}_a$ defined by 
$\bar{d}_a:= \text{max}\ \{d_Z(z)\ |\ z\in a\}$. 
\qed
\end{definition}

\subsection{Preparation}
We describe the ingredients, notions,  and procedures which are part of an inductive construction which produces tools for a quantitative control
of compactness.
\begin{definition}
Given a subset $A$ of $Z$ we denote by ${\mathcal S}_A$ the full subcategory associated to objects 
having isomorphism class in $A$. Further we define ${\mathcal E}_A$ to be the full subcategory associated to all $(\alpha,e)$ with $\alpha$ being an object in ${\mathcal S}_A$.
\qed
\end{definition}
We need the following notion.
\begin{definition}
Let $A\subset \wh{Z}$ be a closed subset for $\wh{\mathcal T}$.
A  functor $\wh{N}_A:{\mathcal E}_A\rightarrow [0,\infty]$ is said to be \textbf{auxiliary norm-like}, provided
for every object $\alpha$ with $|\alpha|\in A$ the following holds.
\begin{itemize}
\item[(1)] $\wh{N}_A(\alpha,\tau\cdot e)=|\tau|\cdot \wh{N}_A(\alpha,e)$, where we use the convention $\tau\cdot\infty=0$ for $\tau=0$ and $\tau\cdot\infty=\infty$ for $\tau>0$. 
\item[(2)] $\wh{N}_A(\alpha,e+e')\leq \wh{N}_A(\alpha,e)+ \wh{N}_A(\alpha,e')$, where we use the convention
that $c+\infty=\infty$ for $c\in [0,\infty]$.
\item[(3)] If $\wh{N}_A(\alpha,e)<\infty$ then $(\alpha,e)$ belongs to ${\mathcal E}_{(0,1)}$, i.e. is on the $(0,1)$-bi-level.
\item[(4)] For fixed $\alpha$, $|\alpha|\in A$,  $\wh{N}_A$ restricted to the vector space of all $\{(\alpha,e)\in {\mathcal E}\ |\ \wh{N}_A(\alpha,e)<\infty\}$ is a complete norm.
\end{itemize}
\qed
\end{definition}
When considering an auxiliary norm-like functor $\wh{N}_A$,  it will frequently occur that it behaves on certain 
parts of ${\mathcal E}_A$ as a ps-norm, i.e. having some continuity properties with respect to $|\alpha|$.
 The appropriate notion capturing this behavior is given as follows.
\begin{definition}
Let $A$ be a closed subset of $(\wh{Z},\wh{\mathcal T})$.  Assume that $\wh{N}_A:{\mathcal E}_A\rightarrow [0,\infty]$ is an auxiliary norm-like functor.
We say that $\wh{N}_A$ is \textbf{ps-like over} $A$ provided $\wh{N}_A: {\mathcal E}_A\rightarrow [0,\infty]$ is a ps-norm.
\qed
\end{definition}
We shall constructed inductively \^open (i.e. open with respect to $\wh{\mathcal T}$) neighborhoods $\wh{U}$,  $\overline{\mathcal M}_J\subset \wh{U}\subset \wh{Z}$,
with suitable properties so that with $A=\text{cl}_{\wh{Z}}(\wh{U})$ we can also construct a ps-norm
over $A$, i.e. $\wh{N}_A:{\mathcal E}_A\rightarrow [0,\infty]$, again with suitable properties.

\begin{definition}
Assume that $a\in \pi_0(Z)$. We say a pair $(\wh{U}_a,\wh{N}_a)$ \textbf{controls compactness}  (of $\bar{\partial}_{\wt{J}}$) provided the following holds.
\begin{itemize}
\item[(1)] $\wh{U}_a\subset a$ is an \^open neighborhood
of $a\cap \overline{\mathcal M}_J$ in $\wh{Z}$,  and $\wh{N}_a:{\mathcal E}_{\text{cl}_{\wh{Z}}(\wh{U}_a)}\rightarrow [0,\infty]$ a ps-norm.   
\item[(2)] If $\wh{N}_a\circ\bar{\partial}_{\wt{J}}(\alpha)\leq 1$ and $|\alpha|\in \text{cl}_{\wh{Z}}(\wh{U}_a)$ then $|\alpha|\in \wh{U}_{a}$.
\item[(3)] The closure of all $|\alpha|\in \wh{U}_a$ in $\wh{Z}$ with $\wh{N}_a\circ\bar{\partial}_{\wt{J}}(\alpha)\leq 1$ is compact in $\wh{Z}$. 
\item[(4)] If $(z_k)$ is a sequence in $\text{cl}_{\wh{Z}}(\wh{U}_a)$ with the property
that $\text{liminf}_{k\rightarrow \infty} \wh{N}_a\circ \bar{\partial}_{\wt{J}}(\alpha_k)\leq 1$
where $|\alpha_k|=z_k$ then $(h_k)$, $h_k:=|\bar{\partial}_{\wt{J}}(\alpha_k)|$ has a mixed convergent subsequence $h_k\stackrel{m}{\rightharpoonup} h$
and $|\wh{N}_a|(h)\leq \text{liminf}\ |\wh{N}_a|(h_k)$. Moreover the underlying $(z_k)$ of this subsequence converges in $\wh{\mathcal T}$.
\end{itemize}
\qed
\end{definition}
The inductive procedure is done with respect to the elements in $\pi^{\leq \ell}_0(Z)$, $\ell\in \{0,1,2,...\}$. Here 
$$
\pi^{\leq \ell}_0(Z)=\{a\in \pi_0(Z)\ |\ \bar{d}_a\leq \ell\}.
$$
We construct \^open sets $\wh{U}_a\in \wh{\mathcal T}$ and ps-norms  $\wh{N}_a:{\mathcal E}_{\text{cl}(\wh{U}_a)}\rightarrow [0,\infty]$ for $a\in \pi_0(Z)$ with $\bar{d}_a\leq \ell$, so that $\wh{N}_a$ is ps-like over $\text{cl}_{\wh{Z}}(\wh{U}_a)$
and $(\wh{U}_a,\wh{N}_a)$ controls compactness of $\bar{\partial}_{\wt{J}}$.  The inductive step then uses the construction for $\leq \ell$
to extend it to data $a$ for $\bar{d}_a= \ell+1$ by keeping the data for $\leq \ell$ and adding the new data for $a$ with $\bar{d}_a=\ell+1$.
Choices during the constructions have to be made for $a\in \pi^p_0(Z)$ with $\bar{d}_a=\ell+1$.
Then the constructions for $a\in \pi^d_0(Z)\cup\pi^u_0(Z)$ with $\bar{d}_a=\ell+1 $  are canonically.  We need some additional input for carrying out the inductive constructions.

\subsection{Compactness and Extension Results}
We shall describe two features which are important  in inductive steps which occur later on when we 
construct an open neighborhood $\wh{U}$ of $\overline{\mathcal M}_J$ with suitable properties,
and a ps-norm $\wh{N}:{\mathcal N}_{\text{cl}_{\wh{Z}}(\wh{U})}\rightarrow [0,\infty]$ such that $(\wh{U},\wh{N})$ controls compactness.
and moreover  when we construct compatible special sc$^+$-multisection functors.
Proposition \ref{prop11.5} and some considerations about $\bar{\partial}_{\wt{J}}$ lead to the following version, where we note that a ps-norm defined on ${\mathcal E}_A$,
for a closed subset $Z$ of $Z$ can always be extended to ${\mathcal E}$. 
\begin{thm}
For every ps-norm $\wh{N}$ with associated $\wh{n}:|{\mathcal E}|\rightarrow [0,\infty]$ the following holds. Given a point $z\in \wh{Z}$ there exists 
an \^open neighborhood $\wh{U}(z)$ with the property that $\text{cl}_{\wh{Z}}(\{y\in \wh{U}(z)\ |\ \wh{n}(f(y))\leq 1 \})$ is compact
with respect to $\wh{\mathcal T}$.
\qed
\end{thm}
Suppose we are given a parent class $a\in \pi_0(Z)$ with $\bar{d}_a=\text{max}\ \{d_Z(z)\ |\ z\in a\} =0$. We know that $a\cap \overline{\mathcal M}_J$ is compact.
We can fix a ps-norm $\wh{N}_a:{\mathcal E}_a\rightarrow [0,\infty]$. Employing the previous theorem and the compactness of $a\cap \overline{\mathcal M}_J$ in $(\wh{Z},{\mathcal T})$
we can find finitely many points $z_1,..,z_j$ and $\wh{U}(z_i)$ such that 
$$
\wh{U}:= \bigcup_{i=1}^j \wh{U}(z_i) \supset a\cap \overline{\mathcal M}_J
$$
is an \^open covering and for every $i\in \{1,..,j\}$ it holds that $\text{cl}_{\wh{Z}}(\{y\in \wh{U}(z_i)\ |\ \wh{n}(f(y))\leq 1 \})$ is compact
with respect to $\wh{\mathcal T}$. Then the same holds for $\wh{U}$, i.e. 
$$
\text{cl}_{\wh{Z}}(\{y\in \wh{U}\ |\ \wh{n}(f(y))\leq 1 \})
\ \ \text{
is compact.}
$$
We observe that $c:=\text{inf} \{\wh{n}(f(z))\ |\ z\in \partial \wh{U}\}>0$. Indeed otherwise we find $(z_\ell)\subset \partial \wh{U}$ such that 
$\wh{n}(f(z_{\ell}))\rightarrow 0$ and after taking a subsequence we may assume that $z_\ell\rightarrow z\in \partial\wh{U}$
and that $f(z_{\ell})\stackrel{m}{\rightharpoonup}\xi$, where $\xi$ is the class of the zero vector above $z$, i.e. $f(z)=0$ which gives a contradiction
since $\partial\wh{U} \cap \overline{\mathcal M}_J=\emptyset$. Take an \^open neighborhood $\wh{V}$ of $a\cap \overline{\mathcal M}_J$ with
$\text{cl}(\wh{V})\subset \wh{U}$. We can take a continuous function $\sigma:\text{cl}_{\wh{Z}}(\wh{U})\rightarrow [1, \infty)$ which on $\wh{V}$ takes the value $1$ and on 
$\partial \wh{U}$ a value greater than $2/c$. The $\sigma$ defines a continuous functor $\sigma:{\mathcal S}_{\text{cl}_{\wh{Z}}(\wh{U})}\rightarrow [1,\infty)$
and we can define a new reflexive auxiliary norm by $\wh{N}':=\sigma\cdot \wh{N}$ over $\text{cl}(\wh{U})$. Then 
$\wh{n}'\circ f(z)\leq 1$ for some $z\in \partial \wh{U}$ implies $1\geq \wh{n}'(f(z))= \sigma(z)\cdot \wh{n}(f(z))\geq \sigma(z)\cdot c> 2$ giving a contradiction,
i.e. the elements satisfying $\wh{n}'(f(z))\leq 1$ and $z\in \text{cl}_{\wh{Z}}(\wh{U})$ belong to $\wh{U}$.
Hence we have shown the following.
\begin{thm}
Given a parent class $a\in \pi^p_0(Z)$ with $\bar{d}_a=0$ there exists an \^open neighborhood $\wh{U}$ of $a\cap \overline{\mathcal M}_J$  and a ps-norm defined over $\text{cl}_{\wh{Z}}(\wh{U})$
such that $(\wh{U}_a,\wh{N}_a)$ controls compactness.
\qed
\end{thm}

We need an extension result along the same line. 
\begin{thm}
Assume that $a\in \pi^p_0(Z)$ and $\wh{U}^{\partial}_a$ is an \^open neighborhood  of $\partial a\cap \overline{\mathcal M}_J$ in $\partial a$
and $\wh{N}$ is a ps-norm defined over $\text{cl}_{\wh{Z}}(\wh{U}^{\partial}_a)$ such that $(\wh{U}^{\partial}_a,\wh{N}^{\partial}_a)$ controls compactness.
Then there exists an \^open neighborhood $\wh{U}_a$ of $a\cap \overline{\mathcal M}_J$ and a ps-norm $\wh{N}_a$ defined over $\text{cl}_{\wh{Z}}(\wh{U}_a)$ so that 
following holds.
\begin{itemize}
\item[(1)] The restriction of $(\wh{U}_a,\wh{N}_a)$ to the boundary is $(\wh{U}_a^\partial,\wh{N}_a^\partial)$.
\item[(2)] $(\wh{U}_a,\wh{N}_a)$ controls compactness.
\end{itemize}
\qed
\end{thm}

\subsection{Main Assertion}
In this subsection we shall  state the basic result about compactness control.

\begin{thm}\label{thm13.2}
For every $a\in \pi_0(Z)$ there exists $\wh{U}_a\in \wh{\mathcal T}$, which contains $a\cap \overline{\mathcal M}_J$,
and a ps-norm $\wh{N}_a:{\mathcal E}_{\text{cl}_{\wh{Z}}(\wh{U}_a)}\rightarrow [0,\infty]$ so that the following property {\em\textbf{(P$_{\ell}$)} }holds for every $\ell\geq 0$, where we define $\wh{U}^{\ell}$, $\wh{N}^{\ell}$ and
${\mathcal E}^{\ell}:={\mathcal E}_{a^{\ell}}$, $a^{\ell}=\bigcup_{\bar{d}_a\leq \ell} a$, as follows.
\begin{eqnarray*}
&\wh{U}^{\ell}=\bigcup_{a\in \pi_0(Z),\ \bar{d}_a\leq \ell} \wh{U}_a&\\
&\wh{N}^{\ell}:{\mathcal E}_{\text{cl}_{\wh{Z}}(\wh{U}^{\ell})}\rightarrow [0,\infty]\ \text{with}\ \wh{N}^{\ell}|{\mathcal E}_{\text{cl}_{\wh{Z}}(\wh{U}_a)}= \wh{N}_a,\ \bar{d}_a\leq \ell.&
\end{eqnarray*}

\noindent  {\em\textbf{(P$_{\ell}$)}} For every $a\in \pi_0^{\leq \ell}(Z)$:
\begin{itemize}
\item[(1)] $(\wh{U}_a,\wh{N}_a)$ controls compactness.
\item[(2)] For $a\in \pi_0(Z)$ where $\bar{d}_a\leq \ell$ the following holds.
Namely for  $\alpha=(\alpha_0,\wh{b}_1,..,\wh{b}_k,\alpha_k)$ in ${\mathcal S}_{a\cap\wh{Z}}$
the statement $|\alpha|\in \wh{U}^{\ell}$is equivalent to the statement $|\alpha_{i,\mathfrak{c}}|\in \wh{U}^{\ell}$ for
$i\in \{0,...,k\}$ and $\mathfrak{c}\in \pi_0^{\text{ntriv}}(\bar{S}_i)$.
\item[(3)] For 
$$
(\alpha,e)=((\alpha_0,e_0),\wh{b}_1,..,\wh{b}_k,(\alpha_k,e_k))
$$
with $|\alpha|\in \text{cl}_{\wh{Z}}(\wh{U}_a)$ the following holds.
If for some $i\in \{0,...,k\}$ the map $e_i$ is non-zero on some trivial cylinder segment it holds that 
$\wh{N}^{\ell}(\alpha,e)=\infty$,  otherwise 
$$
\wh{N}^{\ell}(\alpha,e) =\text{\em max}_{i=0}^k \ \text{\em max}_{\mathfrak{c}\in \pi^{\text{ntriv}}_0(\bar{S}_i)}\ \wh{N}^{\ell}(\alpha_{i,\mathfrak{c}},e_{i,\mathfrak{c}_i}).
$$
\end{itemize}
\qed
\end{thm}
We see that the statement suggests an inductive construction with respect to 
the elements $a\in \pi^{\leq \ell}_0(Z)$.

\subsection{Inductive Construction}

We use the complexity  $\bar{d}:\pi_0(Z)\rightarrow \{0,1,...\}$ for the induction, meaning that 
we construct in the $\ell$-th step (starting with $\ell=0$) $(\wh{U}_a,\wh{N}_a)$  for $a\in \pi_0^{\leq\ell}(Z)=\{a\in\pi_0(Z)\ |\ \bar{d}_a\leq \ell\}$ using data associated to  $\pi^{\leq \ell-1}_0(Z)$ and new choices.\\

\begin{center}
$\bm{[\ell=0]}$
\end{center}

\noindent  $\bm{\ell=0:p)}$ We consider $a\in \pi^p_0(Z)$ with $\bar{d}_a=0$. 
We find a ps-norm $\wh{N}$ over ${\mathcal S}_a$ and 
a strong \^open  neighborhood $\wh{U}_a$ of $a\cap \overline{\mathcal M}_J$ such that 
$(\wh{U}_a,\wh{N}_a)$ controls compactness. Here $\wh{N}_a$ is the restriction
of $\wh{N}$ to $\text{cl}_{\wh{Z}}(\wh{U}_a)$.
We note that in this case $\wh{U}_a\in {\mathcal T}$ as well and 
$\wh{N}_a$ is taken over $a$ as a reflexive auxiliary norm.\\ 

\noindent  $\bm{\ell=0:d)}$ We consider  $a\in \pi^d_0(Z)$ with $\bar{d}_a=0$ and parent $\bar{a}$. We take an \^open subset  $\wh{U}_a$ of $a\cap\overline{\mathcal M}_J$ with parent $\wh{U}_{\bar{a}}$, and $\wh{N}_a$ with parent $\wh{N}_{\bar{a}}$. Then $(\wh{U}_a,\wh{N}_a)$ controls compactness.\\

\noindent  $\bm{\ell=0:u)}$ We note that $\pi^u_{0}(Z)$ does not contain an element with $\bar{d}_a=0$.\\

We define $\wh{U}^0=\bigcup_{\{a\in \pi_0(Z)\ |\ \bar{d}_a=0\}}\wh{U}_a$ and obtain $\wh{N}^0$
so that $(\wh{U}^0,\wh{N}^0)$ has the obvious properties.
By construction \textbf{(P$_0$)} holds.  \\
 
 We do also the case $\ell=1$ before giving the general step. Two new complications enter. The first is that we need to extend neighborhoods on the boundary 
 together with a ps-norm to the interior so that compactness is controlled. The second issue is concerned with union classes. This is the point which makes the use of special 
 sc$^+$-sections necessary.
  \begin{center}
 $$
\bm{[ \ell =1]}
 $$
 \end{center}
 In this case we have boundaries but no corners.  The boundary faces come with covering functors 
 $$
 {\mathcal S}_{\theta}\rightarrow \left({\mathcal S}\times_{\mathfrak{P}}{\mathcal S}\right)_{c(\theta)}:(\alpha,\wh{b},\alpha')\rightarrow (\alpha,b,\alpha').
 $$
 The class $c(\theta)$ determines classes $a',a''$ with 
$1=\bar{d}_a\geq \bar{d}_{a'}+\bar{d}_{a''}+1$, where $\theta\subset a$, so that in particular 
$d_{a'}=d_{a''}=0$. .\\

\noindent$\bm{\ell=1:p)}$
We take the parent class $a$ with $\bar{d}_a=1$.
Since  $\wh{U}_{a'}$ and $\wh{U}_{a''}$ are already  given by the step $\ell=0$
one can define $\wh{U}^{\partial}_{\theta}$ to consist of all $(\alpha',\wh{b},\alpha'')$ with $|\alpha'|\in \wh{U}_{a'}$ and 
$|\alpha''|\in \wh{U}_{a''}$. The elements in the closure of $\wh{U}^{\partial}_{\theta}$ consists  of all $|\alpha|$
with $|\alpha'|$ and $|\alpha''|$ belonging to the respective closures.  The union of all $\wh{U}^{\partial}_{\theta}$
is by definition $\wh{U}^{\partial}_a$, i.e.
$$
\wh{U}^{\partial}_a = \bigcup_{\theta\in \text{face}_a} U^{\partial}_{\theta}
$$
We define $\wh{N}_a^{\partial}$ over the closure of $\wh{U}^{\partial}$ as follows.
Over the relevant parts of the  face $\theta\subset a$ 
$$
\wh{N}_a^{\partial }((\alpha',e'),\wh{b},(\alpha'',e''))=\text{max}\ \{\wh{N}_{a'}(\alpha',e'),\wh{N}_{a''}(\alpha'',e'')\}.
$$
It is evident that $(\wh{U}^{\partial}_a,\wh{N}^{\partial}_a)$ controls compactness.  Using the extension result
we can extend $\wh{N}^{\partial}_a$ to a ps-norm
$\wh{N}_a$ and an \^open neighborhood of $a\cap \overline{\mathcal M}_J$
such that $(\wh{U}_a,\wh{N}_a)$ controls compactness. \\

\noindent$\bm{\ell=1:p)} $ The extension of the data to a descendent is obvious, i.e. as in the case $\ell=0$.\\

\noindent$\bm{\ell=1:u)}$ This time we also have union classes and start with a union parent. Let $a$ be the union parent.
Since $\bar{d}_a=1$ it follows that for an object $\alpha$ with $d(\alpha)=0$ and $|\alpha|\in a$ it holds
$$
\sharp \pi^{\text{ntriv}}_0(\bar{S})=\sharp \pi_0(\bar{S})  =2.
$$
The following Figure \ref{FIG 199}  illustrates this.
 \begin{figure}[h]
\begin{center}
\includegraphics[width=5.5cm]{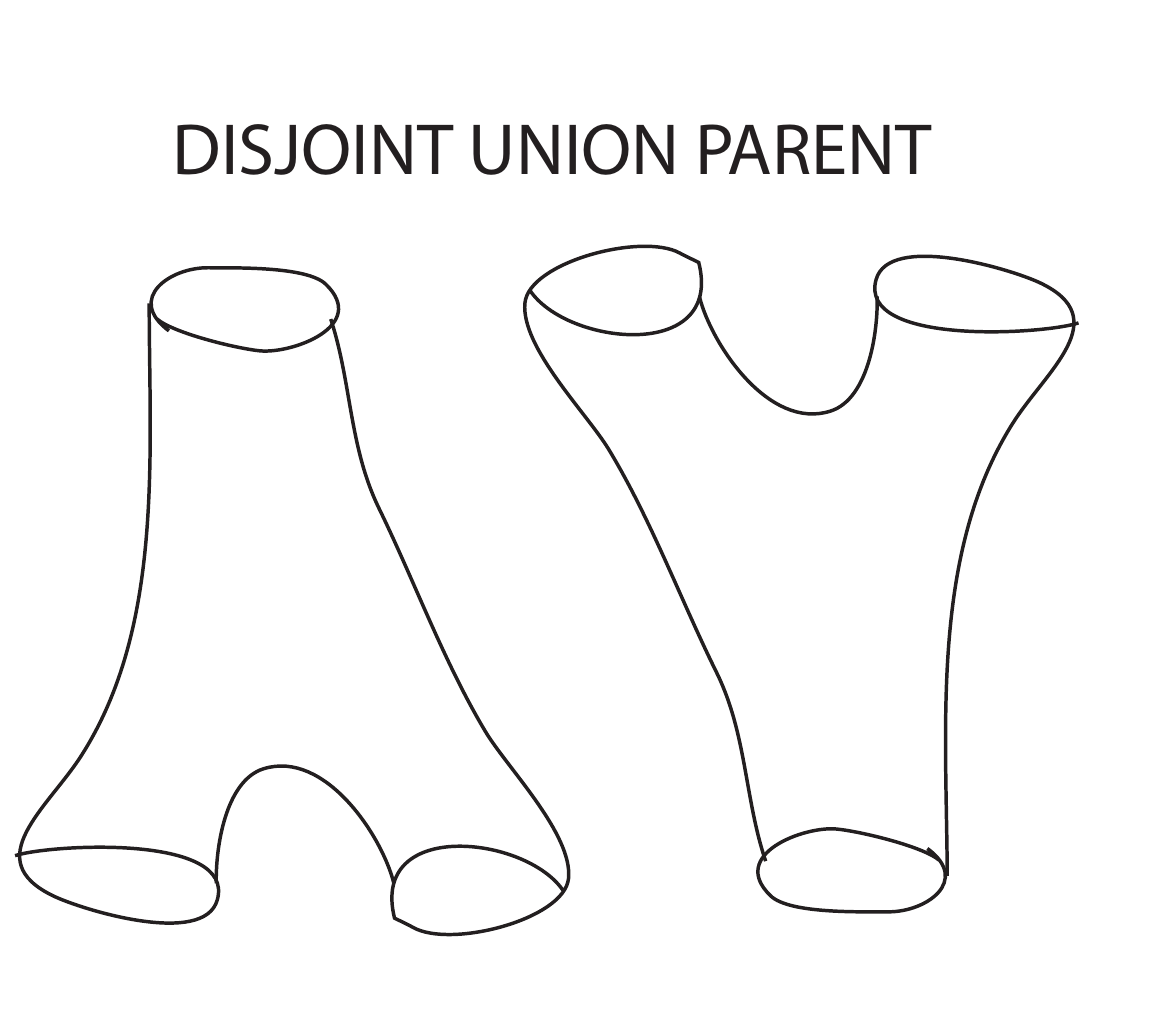}
\end{center}
\caption{An element of union parent type of height $1$.}\label{FIG 199}
\end{figure}

Configurations representing elements in $\partial a$  look as in the following Figure \ref{FIG 200}.
 \begin{figure}[h]
\begin{center}
\includegraphics[width=6.5cm]{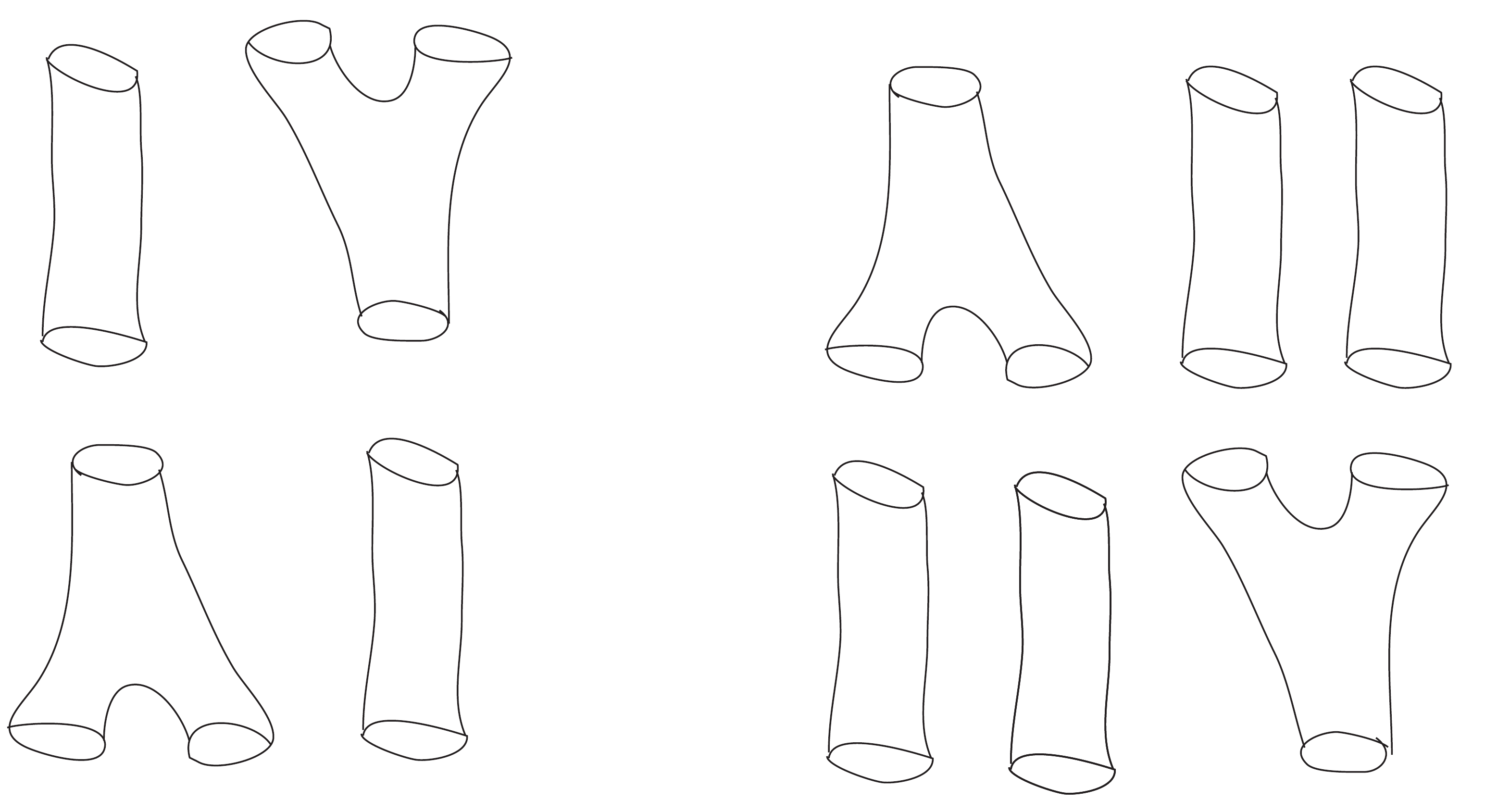}
\end{center}
\caption{Occurring configurations in the boundary. The cylinders are trivial cylinders, but not necessarily $\wt{J}$-holomorphic.
However, in the intersection $\partial a\cap \wh{Z}$ such configurations only have $\wt{J}$-holomorphic cylinders.
}\label{FIG 200}
\end{figure}
We note that the data from $\ell=0$ determines not only data for $\partial a$ but also for $\dot{a}=a\setminus\partial a$.
For the boundary we can define as before $\wh{N}^{\partial}_a$ and $\wh{U}^{\partial}_a$. 
The natural definition over $\dot{a}$ is $\wh{U}_{\dot{a}}$ which consists 
of elements $|\alpha|\in \wh{Z}$ so that the associated two $\alpha',\alpha''$ satisfy
$|\alpha'|,|\alpha''|\in \wh{U}^0$. Further
$$
\wh{N}_{\dot{a}}(\alpha,e)= \text{max}\ \{N_{a_{\mathfrak{c}}}(\alpha_{\mathfrak{c}},e_{\mathfrak{c}})
\ |\ \mathfrak{c}\in \pi_0(\bar{S}) (=\pi_0^{\text{ntriv}}(\bar{S}))\}
$$
We define $\wh{U}_a= \wh{U}_a^{\partial} \cup\wh{U}_{\dot{a}}$.
With these two natural definitions we obtain a functor $\wh{N}_a:{\mathcal E}_{\text{cl}_{\wh{Z}}(\wh{U}_a)}\rightarrow [0,\infty]$ which is auxiliary norm-like. 
\begin{prop}
$\wh{U}_a$ is \^open, contains $\overline{\mathcal M}_J$, $\wh{N}_a$ is a ps-norm and $(\wh{U}_a,\wh{N}_a)$ controls compactness.
\qed
\end{prop}
The fact that $\wh{N}_a$ is a ps-norm would generally not be true for an open neighborhood of $a\cap \overline{\mathcal M}_J$ in $Z$.
The reason is that it would in general not have the required continuity properties.\\

Next follows the general argument.

 \begin{center}
 $$
\bm{[ \ell \Longrightarrow \ell+1]}
 $$
 \end{center}

Assuming \textbf{(P$_{\ell}$)} we shall show that choices can be made so that \textbf{(P$_{\ell+1}$)} holds.
So by assumption we have $\wh{N}_a$ and $\wh{U}_a$ given for all $a\in \pi_0^{\leq \ell}(Z)$
where $(\wh{U}_a,\wh{N}_a)$ controls compactness and $\wh{N}^{\ell}$ over $\text{cl}_{\wh{Z}}(\wh{U}_a)$  is
a ps-norm. Further some additional properties as previously listed hold.
In order to prepare for the argument we carry out some preparations. The data associated to $\ell$ determines
certain data for all $a$ with $\bar{d}_a=\ell+1$ which has to be derived first.
 Important ingredients are the covering functors 
 $$
 {\mathcal S}_{\theta}\rightarrow \left({\mathcal S}\times_{\mathfrak{P}}{\mathcal S}\right)_{c(\theta)}:(\alpha,\wh{b},\alpha')\rightarrow (\alpha,b,\alpha').
 $$
 The class $c(\theta)$ determines classes $a',a''$ with 
$\bar{d}_a= \bar{d}_{a'}+\bar{d}_{a''}+1$, where $\theta\subset a$. If $\wh{U}_{a'}$ and $\wh{U}_{a''}$ are given 
one can define $\wh{U}^{\partial}_{\theta}$ to consist of all $(\alpha',\wh{b},\alpha'')$ with $|\alpha'|\in \wh{U}^{\ell}$ and 
$|\alpha''|\in \wh{U}^{\ell}$.
Assume that $\theta,\theta'\in \text{face}_a$ are different and  intersect.
Then we have two functors
$$
{\mathcal S}_{\theta\cap\theta'}\rightarrow {\mathcal S}\times_{\mathfrak{P}}{\mathcal S}\times_{\mathfrak{P}}{\mathcal S}
$$
given by 
\begin{eqnarray}
&(\alpha,\wh{b},\alpha',\wh{b}',\alpha'')\rightarrow ((\alpha,\wh{b},\alpha'),b',\alpha'')\rightarrow (\alpha,b,\alpha',b',\alpha'')&\\
&(\alpha,\wh{b},\alpha',\wh{b}',\alpha'')\rightarrow (\alpha,{b},(\alpha',\wh{b}',\alpha''))\rightarrow (\alpha,b,\alpha',b',\alpha'')&\nonumber
\end{eqnarray}
This associativity has some important consequences. For example given $\wh{U}_a$, $\wh{U}_{a'}$, and $\wh{U}_{a''}$
it will follow that 
$$
\wh{U}^{\partial}_{\theta}\cap \theta' = \wh{U}^{\partial}_{\theta'}\cap \theta.
$$
We have the following consequence.
\begin{prop}
The set $\wh{U}^{\partial}_a$ defined for some $a\in \pi_0(Z)$  with $\bar{d}_a=\ell+1$ by 
$$
\wh{U}^{\partial}_a=\bigcup_{\theta\in \text{face}_a} \wh{U}^{\partial}_{\theta}
$$
is \^open in $ \wh{Z}\cap \partial a $ and $\wh{N}^{\partial}_a$ is a ps-auxiliary norm over $\text{cl}(\wh{U}^{\partial}_a)$ on the boundary 
and $(\wh{U}^{\partial}_a,\wh{N}^{\partial}_a)$ controls compactness on the boundary.
\qed
\end{prop}
Now we are in the position to discuss the different cases.\\

\noindent {\textbf{Parent:}}
We can extend  for $a\in \pi^p_0(Z)$ with $\bar{d}_a=\ell+1$  the data $(\wh{U}^{\partial},\wh{N}^{\partial}_a)$
to a pair $(\wh{U}_a,\wh{N}_a)$ controlling compactness and restricting on the boundary to the given one. \\

\noindent\textbf{Descendant:}  For a descendent $a$ with $\bar{d}_a=\ell+1$  let $\bar{a}$  be the associated parent.
The data $(\wh{U}_a,\wh{N}_a)$  is obtained from $(\wh{U}_a,\wh{N}_a)$ in the previously, i.e. $\ell$-case, discussed way. \\

\noindent{\textbf{Union:}  Let us consider an union parent class $a$ with $\bar{d}_a=\ell+1$. If $a$ is the union of $a_1,..,a_e$
we must have the identity
$$
\ell+2 = \bar{d}_a +1 =\sum_{i=1}^e (\bar{d}_{a_i}+1)\ \ \text{and}\ \  \sharp\pi^{\text{ntriv}}_0(\bar{S}) =\sharp \pi_0(\bar{S})\geq 2.
$$
The second inequality is by definition and the equality can be seen as follows and generalizing the contents of the Figure  \ref{FIG 111122}.
Indeed for each $i\in \{1,...,e\}$ we can take a configuration with with degeneracy being maximal, i.e. $\bar{d}_{a_i}$ and from which 
we can extract $\bar{d}_{a_i}+1$ many parent pieces. Lining all this pieces up suitably by using trivial cylinder segments 
we can construct an element $\alpha$ representing $a$ satisfying 
\begin{eqnarray}\label{equ*lity}
d(\alpha)+1= \sum_{i=1}^e (\bar{d}_{a_i}+1).
\end{eqnarray}
 It is easy to see
that for this specific element $d(\alpha)=\bar{d}_{a}$ because otherwise we obtain a contradiction to the value of the expression
on the right-hand side of (\ref{equ*lity}).

We  construct new data in terms of data in lower grading. Assume that $a$  has underlying parent classes $a_1,...,a_{e}$.
It holds $\bar{d}_a+1 =\sum_{i=1}^{e} (\bar{d}_{a_i}+1)$. We define with $\dot{a}=a\setminus\partial a$
the \^open subset $\wh{U}_{\dot{a}}$ of $\dot{a}$ to consists of all $|\alpha|\in \wh{Z}$ with 
the underlying $|\alpha_{\mathfrak{c}}|\in \wh{U}^{\ell}$ for $\mathfrak{c}\in \pi^{\text{ntriv}}_0(Z)$ and we 
 define  $\wh{U}^{\partial}_a$ as the union over all faces $\theta\in \text{face}_a$ of the sets $\wh{U}^{\partial}_{\theta}$
which are obtained by using the associated proper covering functors. Finally  we define 
$$
\wh{U}_a:=\wh{U}^{\partial}_a\cup \wh{U}_{\dot{a}}.
$$
\begin{lem}
The set $\wh{U}$ is \^open and contains $a\cap \overline{\mathcal M}_J$.
\qed
\end{lem}
Using $\wh{N}^{\ell}$ we can define $\wh{N}_a$ on $\text{cl}_{\wh{Z}}(\wh{U}_a)$ by 
$\wh{N}_a(\alpha,e)=\infty$ if $e$ is nonzero on a trivial cylinder component and otherwise by 
$$
\wh{N}_a(\alpha,e) =\text{max}_{i=0}^k \text{max}_{\mathfrak{c}\in \pi_0^{\text{ntriv}}(\bar{S}_i)}\ \wh{N}^{\ell}(\alpha_{i,\mathfrak{c}},e_{i,\mathfrak{c}}).
$$
\begin{prop}
$\wh{N}_a:{\mathcal E}_{\text{cl}(\wh{U}_a)}\rightarrow [0,\infty]$ is a ps-norm and $(\wh{U}_a,\wh{N}_a)$ controls compactness.
\qed
\end{prop}
If $a$ is a union descendent and $\bar{a}$ the underlying parent there is the standard extension $(\wh{U}_a,\wh{N}_a)$ using the data $(\wh{U}_{\bar{a}},\wh{N}_{\bar{a}})$\\

Finally we set $\wh{U}^{\ell+1}$ as the union of $\wh{U}^{\ell}$ and all $\wh{U}_a$ with $\bar{d}_a=\ell+1$. Further we define $\wh{N}^{\ell+1}$ in the obvious way.
The induction is complete. At this point we have constructed a system of \^open neighborhoods $\wh{U}_a$, $a\in \pi_0(Z)$
and ps-norms $\wh{N}_a$ defined on ${\mathcal E}_{\text{cl}_{\wh{Z}}(\wh{U})}$ so that for every $a$
the pair $(\wh{U}_a,\wh{N}_a)$ controls in a quantitative way compactness. At this point we are ready to begin with the perturbation and transversality theory.

\newpage

\part*{Lecture 14}
\section{Perturbation Theory}
The standard question of extending a multisection from a boundary with corners to the interior did not get the attention it deserved. 
In fact the problem is subtle. In \cite{HWZ2017} an abstract  method in the polyfold framework is given and the method is further refined in \cite{HS} for applications in inductive procedures.
The difficulty of extending multi\-sections was also observed in \cite{FOOO2017II} but seem not to have 
been adequately addressed previously, see \cite{FOOO,FOOO1}, specifically page 479, but also the introduction concerning \cite{Fuk}. More details about the general difficulties are given in Jake Solomons lecture, \cite{Solomon}\\

In this section we shall carry out the perturbation based on special  sc$^+$-multi\-sections,
which we have introduced previously. Their basic feature boils down to a growth condition related to $\wt{J}$-holomorphic cylinder 
segments,  and which also vanish near nodal points and punctures in a suitable way.
\subsection{Main Perturbation Result}
Recall from the previous section $(\wh{U},\wh{N})$, where $\wh{U}\in\wh{\mathcal T}$ is an \^open neighborhood
of $\overline{\mathcal M}$ in $(\wh{Z},\wh{\mathcal T})$ and $\wh{N}:{\mathcal E}_{\text{cl}_{\wh{Z}}(\wh{U})}\rightarrow [0,\infty]$
is a ps-norm. Further, when
$\wh{U}_a$ defines the intersection of $a$ with $\wh{U}$ we have the property that $(\wh{U}_a,\wh{N}_a)$ controls compactness of $\bar{\partial}_{\wt{J}}$. 
\begin{thm}
Let $(\wh{U},\wh{N})$ control compactness. Given $\varepsilon\in (0,1) $  there exists a special sc$^+$-multisection functor 
$$
\Lambda:{\mathcal E}_{\wh{U}}\rightarrow [0,1]\cap {\mathbb Q}
$$
which has the  following properties.
\begin{itemize}
\item[(1)] $\Lambda$  satisfies 
$$
\Lambda(\alpha,e)=\left(\Lambda\circ(\text{Id}-\Pi)(\alpha,e)\right)\cdot\left( \sum_{i=0}^{d(\alpha)}\sum_{\mathfrak{c}\in \pi^{\text{ntriv}}_0(\bar{S}_i)} \Lambda(\alpha_{i,\mathfrak{c}},e_{i,\mathfrak{c}})\right).
$$
\item[(2)] $\wh{N}(\Lambda)(\alpha) <\varepsilon$ and consequently  for every $a\in \pi_0(Z)$
the set $a\cap |\text{supp}(\Lambda\circ\bar{\partial}_{\wt{J}})|$ is a compact subset contained in $a\cap \wh{U}$.
\item[(3)] $(\Lambda,\bar{\partial}_{\wt{J}})$ are in general position over $\wh{U}$ and 
$\Theta:{\mathcal S}_{\wh{U}}\rightarrow [0,1]\cap {\mathbb Q}$ is a weighted tame branched orbifold so that 
$|\text{supp}(\Theta)| $ intersected with each $a$ is compact.
\end{itemize}
\qed
\end{thm}
\subsection{Extension Result}
As already mentioned there is more to the extension results than it seems and the relevant reference is \cite{HS}.
The essential message from \cite{HS} is that for any reasonable class of sc$^+$-sections one can define a class
of sc$^+$-multisections, let us call them ``good'' for the moment. Being ``good"  is invariant under standard operations 
like $\Lambda\oplus\Lambda$, $\beta\odot\Lambda$, pull-backs by proper covering maps and 
moreover a good sc$^+$-section on the boundary has a good extension.  Further, for inductive proofs, very often the following occurs, which we also have seen in our application. 
In the inductive step one constructs, using previous data, new data on the faces, and it is important that if faces intersect the data
on these intersections coincides. This should, of course, be also true for the overhead, i.e. the ``goodness" whatever that means in a given context. Therefore, in general there has to be some localization of the notion of being good to the boundary for example. 
The realization of such a good system in \cite{HS} has all these properties. In general there might be many different realizations 
of good systems of multisections. 

The issues we just raised are, of course, very important, 
but unfortunately require  a larger amount of time be explained properly.  We refer to the lecture by Jake Solomon and the upcoming 
\cite{HS}.  Everything we describe now can be carried out this way, but to be a complete proof it requires to carry some overhead through the induction.  

We require the reader to be familiar with the usual finite-dimensional transversality theory as well as parameterized versions,
see for example \cite{Hirsch} or \cite{Abraham-Robbin}.
The basic fact about perturbations by multisections based on a class of sections  is that what ever can be achieved in the case without symmetries by a perturbation 
using the given class, can be achieved in the case of symmetries by a multisection. This is the guiding principle. 

\subsection{Induction}
By the previous discussion we have a pair $(\wh{U},\wh{N})$ controlling compactness and this data satisfies certain compatibility conditions. 
The perturbation is constructed by induction with respect to $\bar{d}_a$.  For the given $\varepsilon\in (0,1)$
we pick a sequence $0<\varepsilon_0<\varepsilon_1<...<\varepsilon_i<\varepsilon_{i+1}<... < \varepsilon$. 
Define 
$$
a^{\ell} =\bigcup_{a\in\pi_0(Z),\ \bar{d}_a\leq \ell} a
$$
and 
$$
\wh{U}^{\ell} :=\wh{U}\cap a^{\ell}\ \text{and}\ \  \wh{N}^{\ell} := \wh{N}|{\mathcal E}_{\text{cl}_{\wh{Z}}(\wh{U}^{\ell})}.
$$
\vspace{0.5cm}

\noindent\textbf{Induction Statement}
Assume there exists  for some $\ell\in \{0,1,2,..\}$ a special sc$^+$-multisection functor $\Lambda^{\ell}:{\mathcal E}_{\wh{U}^{\ell}}\rightarrow [0,1]\cap {\mathbb Q}^+$  having the following properties 
\begin{itemize}
\item[(1)] $(\Lambda^{\ell},\bar{\partial}_{\wt{J}})$ is in general position over $\wh{U}^{\ell}$.
\item[(2)] The coarse moduli space $|\text{supp}(\Theta^{\ell})|$ associated to $\Theta^{\ell}:\wh{U}^{\ell}\rightarrow [0,1]\cap {\mathbb Q}$, $\Theta^{\ell}=\Lambda^{\ell}\circ\bar{\partial}_{\wt{J}}$, intersected with every $a\in \pi_0(Z)$, $\bar{d}_a\leq \ell$, 
is compact.
\item[(3)] $\wh{N}^{\ell}(\Lambda^{\ell})(\alpha)<\varepsilon_{\ell}$ for $|\alpha|\in \wh{U}^{\ell}$.
\item[(4)] $\Lambda^{\ell}(\alpha,e)=\left(\Lambda_0\circ(\text{id}-\Pi)(\alpha,e)\right)\cdot \left( \sum_{i=0}^{d(\alpha)}\sum_{\mathfrak{c}\in \pi^{\text{ntriv}}_0(\bar{S}_i)} \Lambda^{\ell}(\alpha_{i,\mathfrak{c}},e_{i,\mathfrak{c}})\right).$
\end{itemize}
Then there exists $\Lambda^{\ell+1}:{\mathcal E}_{\wh{U}^{\ell+1}}\rightarrow [0,1]\cap {\mathbb Q}$ which satisfies
\begin{itemize}
\item[(0)] $\Lambda^{\ell+1} |{\mathcal E}_{\wh{U}^{\ell}}=\Lambda^{\ell}$.
\item[(1)] $(\Lambda^{\ell+1},\bar{\partial}_{\wt{J}})$ is in general position over $\wh{U}^{\ell+1}$.
\item[(2)] The coarse moduli space associated to $\Theta^{\ell+1}:\wh{U}^{\ell}\rightarrow [0,1]\cap {\mathbb Q}$, $\Theta^{\ell+1}=\Lambda^{\ell+1}\circ\bar{\partial}_{\wt{J}}$, when intersected  with $a\in \pi_0(Z)$,  $\bar{d}_a \leq \ell+1$
is compact.
\item[(3)] $\wh{N}^{\ell+1}(\Lambda^{\ell+1})(\alpha)<\varepsilon_{\ell+1}$ for $|\alpha|\in \wh{U}^{\ell+1}$.
\item[(4)] $\Lambda^{\ell+1}(\alpha,e)=\left(\Lambda_0\circ(\text{Id}-\Pi)(\alpha,e)\right)\cdot \left( \sum_{i=0}^{d(\alpha)}\sum_{\mathfrak{c}\in \pi^{\text{ntriv}}_0(\bar{S}_i)} \Lambda^{\ell+1}(\alpha_{i,\mathfrak{c}},e_{i,\mathfrak{c}})\right).$
\end{itemize}
\vspace{0.5cm}

\begin{center}
$\bm{[\ell=0]}$
\end{center}
 We pick a parent class $a\in \pi_0(Z)$ and take $\Lambda_a$ defined over $\wh{U}_a$ such that $\wh{N}(\Lambda_a)(\alpha)<\varepsilon_0$ for 
$|\alpha|\in \wh{U}_a$ and so that $(\bar{\partial}_{\wt{J}},\Lambda_a)$ is in general position over $\wh{U}_a$. Then the perturbed moduli space is a compact branched weighted orbifold. 

Next pick a descendent class $a\in \pi^d_0(Z)$ with $\bar{d}_a=0$ and parent $\bar{a}$. We define $\Lambda_a$ over $\wh{U}_{a}$ by 
$$
\Lambda_a(\alpha,e) =\left(\Lambda_0\circ (\text{Id}-\Pi)(\alpha,e)\right)\cdot \Lambda_{\bar{a}}(\alpha_{\mathfrak{c}},e_{\mathfrak{c}}).
$$
Here $\mathfrak{c}$ is the unique nontrivial component. We note that $\wh{N}_a(\Lambda_a)(\alpha)<\varepsilon_0$. 

There are no disjoint union classes.  

The data defines a $\Lambda^0:{\mathcal E}_{\wh{U}^0}\rightarrow [0,1]\cap {\mathbb Q}$ which is in general position to $\bar{\partial}_{\wt{J}}$. Further the solution set 
is compact without boundary and defines a tame branched weighted orbifold. It holds that
$$
\wh{N}(\Lambda^0)(\alpha)<\varepsilon\ \ \text{for}\ \ \alpha\ \text{in}\ {\mathcal S}_{\wh{U}^0}.
$$
\vspace{0.5cm}
\begin{center}
$\bm{[\ell\Longrightarrow \ell+1]}$
\end{center}
Assume $\Lambda^{\ell}$ has been constructed with the previously described properties.\\

\noindent\textbf{parent:} Pick a parent class $a\in \pi_0(Z)$ with $\bar{d}_a=\ell+1$. For every face $\theta\in \text{face}_a$ we can use the proper covering functor and pull-back data
from the fibered product which only involves data  from the $\ell$-case. It  defines $\Lambda^{\ell+1,\partial}$ on ${\mathcal E}_{\wh{U}^{\ell+1}_{\partial}}$.
Then $\wh{N}^{\ell+1}(\Lambda^{\ell+1,\partial})(\alpha)<\varepsilon_{\ell}$ for $\alpha$ with $|\alpha|\in \wh{U}^{\ell+1}_{\partial}$.
Then an extension result gives $\Lambda_a$ defined on $\wh{U}_a$ with $\wh{N}^{\ell+1}(\Lambda_a)(\alpha)<\varepsilon_{\ell+1}$ and $|\alpha|\in \wh{U}_a$.\\

\noindent\textbf{descendant:}  This is automatic from the previous choices.\\

\noindent\textbf{union:}  Also automatic.\\

Finally we define $\Lambda^{\ell+1}:{\mathcal E}_{\wh{U}^{\ell+1}}\rightarrow [0,1]\cap {\mathbb Q}$ by setting it equal to $\Lambda^{\ell}$ on ${\mathcal E}_{\wh{U}^{\ell}}$
and for $a\in \pi_0(Z)$ with $\bar{d}_a=\ell+1$ we define it as $\Lambda_a$.
The induction is complete and this proves the existence of a general position $\Lambda:{\mathcal E}_{\wh{U}}\rightarrow [0,1]\cap {\mathbb Q}$.

\subsection{Conclusion}
We summarize where we stand at this point. Starting with a pair $(\wh{U},\Lambda)$ controlling compactness
we have constructed a special sc$^+$-multisection $\Lambda:{\mathcal E}_{\wh{U}}\rightarrow [0,1]\cap {\mathbb Q}$ 
satisfying $\wh{N}(\Lambda)(\alpha)<\varepsilon<1$ for all objects $\alpha$ satisfying $|\alpha|\in \wh{U}$.  Moreover
$(\Lambda,\bar{\partial}_{\wt{J}})$ is in general position over $\wh{U}$.  As a consequence we obtain the moduli category
$\text{supp}(\Theta)$
which consists of all objects in ${\mathcal S}_{\wh{U}}$ with $\Theta(\alpha) :=\Lambda\circ\bar{\partial}_{\wt{J}}(\alpha)>0$.
For these properties it follows that the associated orbit space $\mathfrak{M}:= |\text{supp}(\Theta)|$, the coarse moduli spaces,
has the property that for every $a\in \pi_0(Z)$ the intersection $\mathfrak{M}_a:=a\cap \mathfrak{M}$ is compact. Further
$\Theta$ is of tame branched weighted orbifold type. Let us define $\Theta_a$ as the restriction of $\Theta$ to ${\mathcal S}_a$.
Since $\Lambda$ has the property
$$
\Lambda(\alpha,e)=\left(\Lambda\circ(\text{Id}-\Pi)(\alpha,e)\right)\cdot\left( \sum_{i=0}^{d(\alpha)}\sum_{\mathfrak{c}\in \pi^{\text{ntriv}}_0(\bar{S}_i)} \Lambda(\alpha_{i,\mathfrak{c}},e_{i,\mathfrak{c}})\right).
$$
there are many relationships between the different $\Theta_a$.

Denote by $Z'$ the orbit space of the full subcategory associated to all stable maps without marked points.
Clearly $Z'$ is the union of all $a$ coming from a suitable subset $\pi'$ of $\pi_0(Z)$. We also have a decomposition 
of $\pi'$ into parent, descendant, and union classes. We can associate to each $a\in \pi'$ an integer $\text{ind}(a')$
which is the Fredholm index of $\bar{\partial}_{\wt{J}}$ over $a$.
We observe the following theorem.
\begin{prop}
For  a union class $a'\in \pi'$ with $\mathfrak{M}_a\neq \emptyset$ it holds $\text{ind}(a)\geq 1$. 
\qed
\end{prop}

\newpage
\part*{Lecture 15}
\section{Orientations}
A good source for the  orientation question is \cite{Wendl}, which follows general ideas as given in 
\cite{BM,EGH,FloerH1}.   For the discussion we use the set up  of the orientation questions for the sc-Fredholm problem for $\wt{\mathcal E}\rightarrow \wt{\mathcal S}$ associated to $\bar{\partial}_{\wt{J}}$, which can be viewed as a lift of the problem ${\mathcal E}_J\rightarrow {\mathcal S}$. Recall  the definition of $\wt{\mathcal S}$ from Subsection \ref{covering}.

\begin{remark}
(1) We also point out that there is an alternative way where we use the functor  $\Theta:=\Lambda\circ \bar{\partial}_{\wt{J}}$
$$
\Theta:{\mathcal S}_{\wh{U}}\rightarrow [0,1]\cap {\mathbb Q},
$$
which is of tame branched weighted orbifold type,
and by adding asymptotic marked and numberings of top and bottom punctures construct a finite-to-one covering of it on which we deal with the orientation questions. \\

\noindent(2) Another possible way is to start right from the beginning with $\wt{\mathcal E}\rightarrow \wt{\mathcal S}$ and use the functorial actions of rotating markers and 
renumbering which are defined on the respective subcategories. Then the perturbation theory has to be compatible with these actions. 
\qed
\end{remark}
\subsection{Lift to ${\wt{\mathcal E}\rightarrow \wt{\mathcal S}}$}
Recall that $\wt{\mathcal E}$ is the pull-back of ${\mathcal E}\rightarrow {\mathcal S}$
by the forgetful functor, which is a proper covering functor $\wt{\mathcal S}\rightarrow {\mathcal S}$.
Passing to orbit space we obtain $|\wt{\mathcal S}|\rightarrow |{\mathcal S}|$ and denote the preimage
of $\wh{U}$ by $\wt{U}$. Then $\wt{\mathcal S}_{\wt{U}}$ corresponds to ${\mathcal S}_{\wh{U}}$ under the proper covering functor.
We also pull-back the previously constructed $\Lambda$  to obtain
$$
\wt{\Lambda}: \wt{\mathcal E}_{\wt{U}}\rightarrow [0,1]\cap {\mathbb Q}^+.
$$
  This sc$^+$-multisection together with the CR-operator is in general position and the weighted
moduli category $\wt{\Lambda}\circ \bar{\partial}_{\wt{J}}$ is a proper covering of $\Lambda\circ \bar{\partial}_{\wt{J}}$.
The new set-up has some advantages concerning the orientation question.  
\begin{remark}
One should be able to construct the SFT potential directly for the original set-up.  In the context of coarse moduli spaces there is a discussion related to this point in \cite{Wendl}.
\qed
\end{remark}
\subsection{Linearisations}
For the orientation question the following considerations are important.
Assume that $\wt{\alpha}'$ and $\wt{\alpha}''$ are two smooth objects in $\wt{\mathcal S}$ 
which lie over the same object $\alpha$ in ${\mathcal S}$ (via the proper covering).  
We have three associated linearization spaces, namely 
$$
\text{Lin}(\bar{\partial}_{\wt{J}},\wt{\alpha}'),\ \text{Lin}(\bar{\partial}_{\wt{J}},\wt{\alpha}''),\ \text{Lin}(\bar{\partial}_{\wt{J}},{\alpha}).
$$
There are natural bijections between these spaces
$$
\text{Lin}(\bar{\partial}_{\wt{J}},\wt{\alpha}')\rightarrow \text{Lin}(\bar{\partial}_{\wt{J}},{\alpha})\leftarrow \text{Lin}(\bar{\partial}_{\wt{J}},\wt{\alpha}'')
$$
so that  a given $L:T_{\alpha}{\mathcal S}\rightarrow {\mathcal E}_{\alpha}$ corresponds 
to a $L': T_{\wt{\alpha}'}{\wt{\mathcal S}}\rightarrow {\mathcal E}_{\alpha}$ and
a $L'': T_{\wt{\alpha}''}{\wt{\mathcal S}}\rightarrow {\mathcal E}_{\alpha}$. Note that the targets are all the same. 
We can use this to push forward an orientation of $L'$ to an orientation of $L$ and then to $L''$.
Hence we can compare orientations under the change of numbering of positive or negative punctures,
the change of numbering of marked points, as well as rotation of asymptotic markers.
A so-called coherent orientation would give orientations to the linearization spaces $\text{Lin}(\bar{\partial}_{\wt{J}},\wt{\alpha})$ satisfying some rules. However, there are some subtleties associated to so-called bad orbits. We shall discuss some of the issues in the next subsection.

\subsection{Conley-Zehnder Index and Parity}
We start with $([\gamma],T_0,1)$, which is a prime periodic orbit, i.e. $k=1$,  so that $T_0$ is the minimal period.   Then $\dot{\gamma}= T_0\cdot R(\gamma)$ and with $x=\gamma(0)$
we obtain the linearized return map $A:\xi_x\rightarrow \xi_x$ associated to $([\gamma],T_0,1)$.  The linearized return map 
associated to $([\gamma],k\cdot T_0,k)$ is $A^k$. The non-degeneracy assumption implies that $1\not\in \sigma(A^m)$ for all $m\geq 1$.
In particular $\pm 1\not\in \sigma(A)$. Since $A$ is symplectic the real eigenvalues appear in pairs $\tau,\tau^{-1}\subset {\mathbb R}\setminus\{1,-1\}$. 
We can count the number $e$  of eigenvalues in $(-1,0)$  with multiplicity.
\begin{definition}
A periodic orbit $([\gamma],T,k)$ is \textbf{bad} provided $k$ is even and $e$ associated to $([\gamma],T/k,1)$ is odd. 
\qed
\end{definition}
Denote by $c_1:=c_1(\xi)$ the first Chern number of the contact structure $\xi$ associated to $\lambda$. Then the Conley-Zehnder index is defined in 
${\mathbb Z}/ 2c_1$ and we have the parity relation
$$
(-1)^{CZ([\gamma],T,k)+n+1} =  \text{sign}(\text{det}(\text{Id}-A_{([\gamma],T,k)}))
$$
\begin{definition}
Given a periodic orbit $([\gamma],T,k)$, we define a number in ${\mathbb Z}_2=\{0,1\}$ called \textbf{parity}
by $\text{parity}([\gamma],T,k])= \text{CZ}([\gamma],T,k)+n-3$ mod $2$.
\qed
\end{definition}
\subsection{Orientation Bundle}
The notion of coherent orientation is well-known, see \cite{FloerH1,BM,Wendl}. We shall not go into precise details but describe 
how it looks in the current formalism.
We consider now $\wt{\mathcal E}\rightarrow \wt{\mathcal S}$. 
For every smooth object $\wt{\alpha}$ there exists the contractible space of linearizations 
of $\text{Lin}(\bar{\partial}_{\wt{J}}, {\alpha})$ which is a convex set of linear sc-Fredholm operators
$$
L: T_{\wt{\alpha}}\wt{\mathcal S}\rightarrow \wt{\mathcal E}_{\wt{\alpha}}
$$
Over the convex set of Fredholm operators we have the associated determinant bundle with  two possible 
orientations. As a consequence we have the following gadget, where $\wt{S}^{\infty}$ is the subcategory of smooth objects
$$
{\wt{\text{Or}}}\rightarrow \wt{\mathcal S}^{\infty}.
$$
We can view $\wt{\text{Or}}$ as the category with objects $(\mathfrak{o},\wt{\alpha})$, where $\mathfrak{o}$
is an orientation of the family of sc-Fredholm operators in  $\text{Lin}(\bar{\partial}_{\wt{J}}, {\alpha})$.  A morphism 
${\Phi}$ defines 
$$
(\mathfrak{o},\wt{\Phi}): (\mathfrak{o},s(\wt{\Phi}))\rightarrow (\wt{\Phi}_{\ast}\mathfrak{o},t(\wt{\Phi})).
$$
\begin{prop}\label{prop333}
Given a smooth $\wt{\alpha}$ its isotropy group $\wt{G}$ acts on $\wt{Or}_{\wt{\alpha}}$ in an orientation preserving way.
\qed
\end{prop}
Next we introduce the notion that two smooth objects are related by a path and similarly if they are equipped with orientations.
\begin{definition}
We say two smooth objects $\wt{\alpha}',\wt{\alpha}''$ are \textbf{related by a path} if there exists a third object $\wt{\alpha}$ and $\Psi\in F(\wt{\alpha})$, say $\Psi:G\ltimes O\rightarrow \wt{\mathcal S}$
and a sc-smooth path $\gamma:[0,1]\rightarrow O$ with $\Psi(\gamma(0))=\wt{\alpha}'$ and $\Psi(\gamma(1))=\wt{\alpha}''$.
If $(\mathfrak{o}',\wt{\alpha}')$ and $(\mathfrak{o}'',\wt{\alpha}'')$ are given, we say that they are \textbf{related by a path}, provided $\wt{\alpha}'$ and $\wt{\alpha}''$ are related by a path
and the prolongation of the orientation $\mathfrak{o}'$ along $\gamma$ gives $\mathfrak{o}''$.  
\qed
\end{definition}
The Cauchy-Riemann section functor  $\bar{\partial}_{\bar{\Psi}}$  defines for $\bar{\Psi}$   an orientation bundle over $O_{\infty}$, so that the the prolongation of the orientation along a path is well-defined, see \cite{HWZ2017} for the precise argument.

In view of Proposition \ref{prop333} we can pass to orbit space and we obtain a ${\mathbb Z}_2$-bundle over the metrizable space $\wt{Z}_\infty$ say
$$
\wt{\mathsf{O}}\rightarrow \wt{Z}_\infty,
$$
where 
$\wt{Z}:=|\wt{\mathcal S}|$ and $\wt{Z}_{\infty}$ is the metrizable space of smooth points.  We shall call it
\textbf{orientation bundle}. If $\wt{a}\in \pi_0(\wt{Z})$ we can restrict the orientation bundle to $\wt{a}$ and denote this restriction by 
$\wt{\mathsf{O}}_{\wt{a}}$.
\begin{thm}
For every class $\wt{a}\in \pi_0(\wt{Z})$ the restricted orientation bundle $\wt{\mathsf{O}}_{\wt{a}}$ is orientable.
\qed
\end{thm}
This is a good start, but still not enough for constructing SFT. In fact, what we need is a system of orientations which is compatible
in a precise sense. For this we need some preparation.
Given a smooth object $\wt{\alpha}$ in the interior of a face of $\wt{\mathcal S}$ we can consider 
the space of linearizations of $\bar{\partial}_{\wt{J}}$ and take a representative $L$.  
\begin{itemize}
\item[(1)]   Then  the restriction $L'$ to the tangent space of the face at $\wt{\alpha}$ can be related to a specific product of linearizations via the proper covering map (pick a convention here).
\end{itemize}
Assume for simplicity $L'$ is surjective (In the general case there is  a formula). Then 
\begin{itemize}
\item[(2)]    The kernel of $L$ will have a vector $h$  which is outward pointing. We write the orientation of $L$  in terms of $h$ as a first vector  followed
by those in the kernel of $L'$ tensored by the cokernel.
\end{itemize}
This allows to use the concatenation structure to define a notion of something like a coherent orientation, see \cite{Wendl,BM,FloerH1}.
It is particularly important to understand the rules which have to be implemented, when renumbering periodic orbits as
well as rotation asymptotic markers.  This is discussed in detail in \cite{Wendl}.
\subsection{Remark about Invariants}
At this point we would be able to introduce the Hamiltonian $H$ introduced in \cite{EGH}
and prove its property $H\circ H=0$ (if using closed forms) or more generally $dH+H\circ H=0$. 
One can follow \cite{EGH}.   We intend to expand this lecture note to a graduate text, which would 
cover this.  Specialists know at this point know what to do and beginning  graduate students find details  in \cite{Wendl} or \cite{EGH}. 
For example if $Q$ is simply connected and the first Chern class of the contact structure vanishes
we can take for every periodic orbit $([\gamma],T,k)$ a cap and compute the CZ-index with respect to an associated trivialization. Every periodic orbit will have a well-defined CZ-index, 
since under our assumption the  result is independent of the choices.   Given a stable map  we can glue in the caps 
and we obtain second homology class $A\in H_2(Q,{\mathbb Z})$.  
Consider  a parent class $a\in \pi_0^p(\wt{Z})$ and 
take a non-nodal representative and glue in the caps.  This allows to associate to $a$ a second 
homology class $A$, the genus of the underlying domain, the positive asymptotic periodic orbits and the negative 
asymptotic periodic orbits.  We have the evaluations functors at the marked points which can be used to  pull-back differential forms on $Q$ to sc-smooth differential forms on $\wt{\mathcal S}$, see \cite{HWZ2017} for the underlying theory. Similarly we have a forgetful functor
into DM-space associating to a stable map  the stable part of the underlying domain.  We can define 
so-called correlators by integrating certain expressions over moduli spaces 
having no bad orbits. Bundling these expressions suitably, see \cite{EGH}, we obtain the so-called Hamiltonian.
The orientation properties imply $[H,H]=0$ for a suitably defined $[.,.]$.

\newpage
\part*{Lecture 16}
\section{Homotopy}
We go back to ${\mathcal E}\rightarrow {\mathcal S}$ and investigate the relationship between $(\wh{U},\wh{N},\Lambda)$
and $(\wh{U}', \wh{N}',\Lambda')$, where $\Lambda$ and $\Lambda'$ are transversal to $\bar{\partial}_{\wt{J}}$.  
Define $\wh{U}''=\wh{U}'\cap \wh{U}'$ and equip with $\wh{N}'' =\text{max}\ \{ \wh{N},\wh{N}''\}$. 
We can find a transversal $(\wh{U},\wh{N},\Lambda'')$ such that the solutions associated to $(\Lambda'',\bar{\partial}_{\wt{J}})$
in $\wt{U}$ belong to $\wt{U''}$. We can do the same for $(\wh{U}', \wh{N}',\Lambda')$. The upshot is that 
we may assume without loss of generality that $\wt{N}=\wt{N}'$ and $\wh{U}=\wh{U}'$.

\subsection{Set-Up}
We consider the category ${\mathcal S}_{[0,1]}=[0,1]\times {\mathcal S}$ equipped with the pull-back bundle still denoted by ${\mathcal E}$. We denote the transversal perturbations 
by $\Lambda_0$ and $\Lambda_1$.  Compactness is not really the issue since it is guaranteed provided the perturbations are small enough. The set controlling compactness
will be $[0,1]\times \wh{U}$ and the ps-norm is $\wh{N}$ on ${\mathcal E}_{(t,\alpha)}\equiv {\mathcal E}_{\alpha}$ provided
$|\alpha|\in \text{cl}_{\wh{Z}}(\wh{U})$.
We again proceed inductively using $\overline{d}$ and the parent, descendent, and union classification of classes in $\pi_0(Z)$.

We have a projection $\overline{t}:[0,1]\times {\mathcal S}\rightarrow [0,1]$. Whatever, the extension is of $\Lambda_0\sqcup \Lambda_1$ it will be transversal to the boundaries
$(\{0\}\times {\mathcal S}) \sqcup (\{1\}\times {\mathcal S})$.  Recall that as part of the inductive procedure 
already in the second step we consider boundary faces and pull back data associated to some 
$(a',a'')$ which we dealt with in the previous step.   This time we have to deal with with families 
$$
{(\Lambda_{a',t})}_{t\in I}\ \ \text{and}\ \ {(\Lambda_{a'',t})}_{t\in I}
$$
 and we have to take a fibered product with respect to $t$.  This, of course, works
if $\bar{t}$ is a submersion on at least one of the moduli spaces. Not surprisingly this can generally not be expected.
One can get away with somewhat less assumptions than that, but it is still not feasible in our case. 
However, there is a simple trick to achieve this which can be adapted from  Kuranishi framework, see \cite{FOOO2010,FOOO2010b,FOOO2013,FOOO2015I} (in the last reference particularly Section 7).  The issue in these references
is the same since it involves fibered product constructions.  

 \begin{remark}
 The idea is a logical consequence of understanding the problems which went into 
 the decision to use  multisections when dealing with problems having local
symmetries.  Symmetries generally  obstruct transversality. To deal with this, rather than considering $\Lambda_0\circ\bar{\partial}_{\wt{J}}$
($\Lambda_0$ supported on the zero-section),
we considered a suitable small perturbation $\Lambda\circ\bar{\partial}_{\wt{J}}$, where we recall $\sum_{(\alpha,e)}\Lambda(\alpha,e)=1$
for every fixed $\alpha$.  We  broke the symmetry by a perturbation, but then  took a symmetric family of problems
locally by using a symmetric family of perturbations. Each problem in the finite family is allocated an appropriate weight, or alternatively 
we `average' over the problems.  One introduces a notion when two such local families  coincide locally and consequently obtains 
a notion of a global problem.  Usually homological information over ${\mathbb Q}$ will be preserved.

 When we deal with the homotopy problem the issue arises 
from  the diagonal in $I\times I$ due to the $\bar{t}$-projection.  
\begin{figure}[h]
\begin{center}
\includegraphics[width=6.5cm]{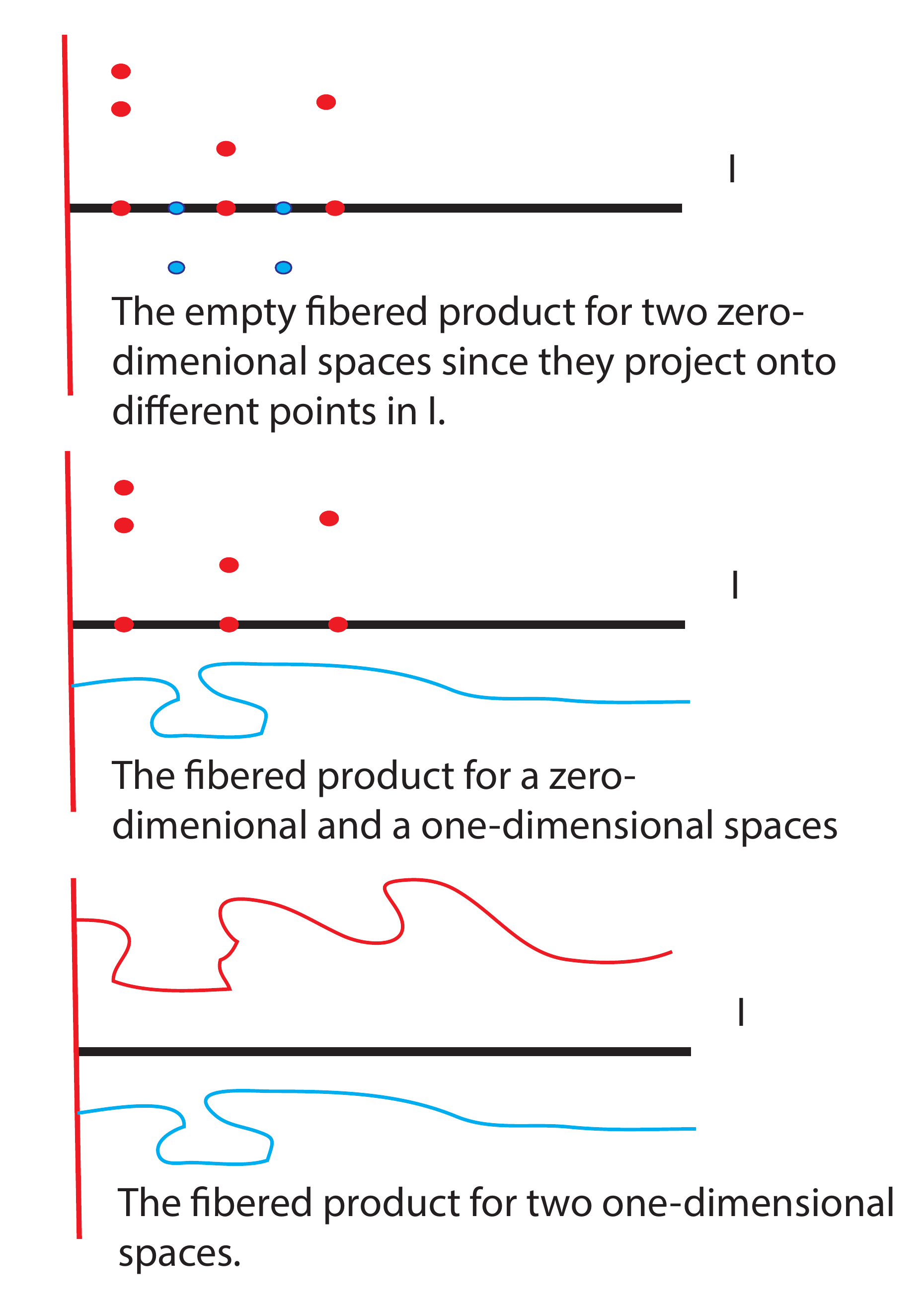}
\end{center}
\caption{Some of the fibered product issues.}\label{FIG X1}
\end{figure} 

\noindent Namely we have to consider boundaries
obtained by concatenation of descendants of the same parent, which are essentially perturbed 
in the same way!  Similarly there is a problem when considering disjoint unions!  
A model for the issues is a taking a fibered product of two identical situations.  In this case the fibered 
product construction fails in general.  
 Indeed, we need that if we have an element in the fibered product then at least one of the $t$-projections 
is submersive near the appropriate points. In a first naive attempt
 we can  try a family of homotopies, say 100 sections in sufficiently 
general position producing 100 different solution spaces with maps to $[0,1]$.
When we take a fibered product of the problem with itself,  but use different pairs of homotopies
in our $100\times 100$ collection of pairs we obtain  10000 homotopies of which  $ 100$ have the diagonal problem. Hence, roughly speaking 9900 are good compared to 100 bad ones.  Of course,  taking the limit in the number of 
such homotopies to infinity  the percentage of bad problems with respect to the overall problems goes potentially  to 
$0$.  Of course, still the bad problems could spoil the algebraic averages and there is no proof along these lines, which gets away with 
a finite number of sections. However, the idea works for smooth families together with an averaging, provided we set things up appropriately.
 \qed
 \end{remark}

Coming back to the SFT problem, the right implementation of this idea is to take suitable smooth families
 of special sc$^+$-multisection functors. We also have a class of such sc$^+$-multisections  $\bm{\Lambda}$
for which we have extension theorems, but we shall not discuss this here. The discussion in 
\cite{HWZ2017} and in the forth-coming paper \cite{HS} can be generalized to this parameter depending context.
Here is what we do.
Define ${\mathcal I}=(-1,1)$ and let $\beta:(-1,1)\rightarrow [0,\infty)$ be a map with the following properties.
\begin{itemize}
\item[(1)] $\beta$ is   smooth with compact support 
in $(-1,1)$.
\item[(2)]  $\beta(s)-\beta(-s)=0$.
\item[(3)]  $\int_{-1}^1 \beta(s) ds=1$. 
\end{itemize}
 Then we obtain the compactly supported one form 
$\tau^1=\beta(s)\cdot ds$. We define ${\mathcal I}^N$ as the $N$-fold product and 
also we put 
$$\tau^N 
:=  \beta(s_1)\cdot.. \cdot \beta(s_N)\cdot  ds_1\cdot..\cdot ds_N.
$$
  The special  sc$^+$-multisections we are interested in, can be viewed as maps which associate to $s\in {\mathcal I}^N$ a family of special sc$^+$-multisection functors of the strong polyfold bundle, or alternatively as a sc$^+$-multisection functor over ${\mathcal I}^N\times [0,1]\times {\mathcal S}$
by pulling back ${\mathcal E}\rightarrow {\mathcal S}$ via the obvious projection.
Hence given $\bm{\Lambda}$  and  strong bundle uniformizer $\bar{\Psi}:G\ltimes K\rightarrow  {\mathcal E}$ we can write 
$$
\bm{\Lambda}\circ (\text{Id}_{{\mathcal I}^N}\times [0,1]\times \bar{\Psi}) (s,t,k) = \frac{1}{\sharp I}\cdot \sharp\{ i\in I\ |\ s_i(s,t,p(k))=k\}
$$
for $s\in {\mathcal I}^N$ and all $k\in K$ with $p(k)$ near $\bar{o}$. In order to have extension theorems we need
to add some additional conditions which we shall suppress and which are similar to those in the non-parameterized case. There is an obvious operation to extend the parameter space given $\bm{\Lambda}$, namely we just multiply the domain with some additional ${\mathcal I}^K$
with parameters, which  however are ineffective.  We shall refer to this as the \textbf{trivial enrichment}.

A quick example without symmetry.  Assume that $p:E\rightarrow M$ is a smooth oriented vector bundle over a connected  smooth oriented manifold 
without boundary, so that $\dim(M)= \dim(E_m)$ and  that $a:M\rightarrow I$ is a  submersion.  Denoted by $\Omega$ the Thom form. 
 Then 
$$
\int_M f^{\ast}\Omega = e(p),
$$
where $e(p)$ is the Euler number. If $f$ is transversal to the zero section $f^{-1}(0)$ consists  of finitely many points and $s: f^{-1}(0)\rightarrow I$
is nowhere a submersion. 
Again starting with the given $f$  in the previous discussion, we can defined   $\bar{f}:{\mathcal I}^N\times M\rightarrow E$ 
by
$$
\bar{f}(s,m)=f(m).
$$
Taking a sufficiently large $N$ we find a small perturbation $\wt{f}$ of $\bar{f}$ such that
\begin{itemize}
\item[(1)]   $\wt{f}^{-1}(0)$ is a smooth manifold of dimension $N$.
\item[(2)]  $a:\wt{f}^{-1}(0)\rightarrow [0,1]$ is a submersion.
\end{itemize} 
We do not(!) claim that the map in (2) is surjective or proper.  However, we note that 
$$
\wt{f}^{-1}(0)\rightarrow {\mathcal I}^N
$$
is proper.   Then we consider the compactly supported form 
$$
a^{\ast}\tau^N \wedge \wt{f}^{\ast}\Omega
$$
 on ${\mathcal I}^N\times M$ of degree $N+\dim{M}$.
We note that 
$$
\int_{{\mathcal I}^N\times M} a^{\ast}\tau^N \wedge \wt{f}^{\ast}\Omega =e(p)
$$
This is some kind of averaging.  Hence by passing from $f$ to $\wt{f}$ we did not lose the Euler class information,
but gained the additional property that $\wt{f}^{-1}(0)\rightarrow I$ is a submersion, which could be used for further constructions.
More precisely in our SFT-case, before we actually carry out  the integration over ${\mathcal I}^N$ (i.e. the averaging), we can make further constructions, for example  fibered product constructions and only at the very end we average using the accumulated averaging 
parameters.

\subsection{Consequences of Compactness Control}
Recall that we work with $(\wh{U},\wh{N})$ controlling compactness. Assume that we are given 
a small perturbation $\bm{\Lambda}$ in the sense that  $\wh{N}(\bm{\Lambda})(s,t,\alpha)\leq \delta <1$.
We denote for $a\in \pi_0(Z)$ by ${\mathcal I}_a$ the finite product ${\mathcal I}^{n_a}$.
We assume that  $\bm{\Lambda}$ for $a\in \pi_0(Z)$ defines
$$
\bm{\Lambda}_a:{\mathcal I}_a\times [0,1]\times {\mathcal E}_{\wh{U}_a}\rightarrow [0,1]\cap {\mathbb Q}
$$
by  $(s,t,(\alpha,e))\rightarrow \bm{\Lambda}(s,t,(\alpha,e))$. Next we consider
$$
\bm{\Theta}_a:{\mathcal I}_a\times [0,1]\times {\mathcal S}_{\wh{U}_a}\rightarrow [0,1]\cap {\mathbb Q}: (s,t,\alpha)\rightarrow \bm{\Lambda}(s,t,\bar{\partial}_{\wt{J}}(\alpha)).
$$
Consider the moduli category which is generated by the objects $(s,t,\alpha)$ on which $\bm{\Theta}_a$ is positive 
and take its orbit space denoted by  $\mathfrak{M}_a$. Then $\mathfrak{M}_a\rightarrow {\mathcal I}_a$ is a proper map, i.e.
the preimage of a compact subset of ${\mathcal I}_a$ is compact.
If we take the dimension ${\mathcal I}_a$ large enough we can achieve that 
the $t$ projection into $[0,1]$ is a submersion (not necessarily onto).  

\subsection{The Perturbation Scheme}
All the perturbations are assumed to be  small which can be measured by $(\wh{U},\wh{N})$. 
The key is to  take enough parameters. The organization of the perturbation scheme is the same as the one  already used
and previously discussed.\\

\noindent{$\bm{\ell=0:}$} We start with $a\in \pi_0^p(Z)$ and $\bar{d}_a=0$.  We can find $\bm{\Lambda}_a$ with parameter 
set ${\mathcal I}_a:={\mathcal I}^{n_a}$ 
extending  the corresponding trivial enrichment of $\Lambda_{0,a}$ and $\Lambda_{1,a}$. We can do this in such a way that 
$(\bm{\Lambda},\bar{\partial}_{\wt{J}})$ is in general positive and the projection onto $[0,1]$ is a submersion.
This perturbation extends trivially to descendants and there are no union classes. \\

\noindent{$\bm{\ell=1:}$} We start again with a parent $a\in \pi^p_0(Z)$ and assume $\bar{d}_a=1$.
Pick a face $\theta$ and recall that two different faces are disjoint. The face $\theta$ corresponds
to $(a',a'')$. The relevant set $\wh{U}^{\partial}_{\theta}$ consists of all $|(\alpha',\wh{b},\alpha'')|$ 
such that $|\alpha'|\in \wh{U}_{a'}$ and $|\alpha'|\in \wh{U}_{a''}$ and using  the proper covering functor we   define 
$$
\bm{\Lambda}_{a}:{\mathcal I}_{a'}\times{\mathcal I}_{a''}\times[0,1]\times {\mathcal E}_{\wh{U}^{\partial}_{\theta}}\rightarrow {\mathbb Q}\cap [0,1]
$$
by
$$
\bm{\Lambda}_a(s',s'',t,((\alpha',e'),\wh{b},(\alpha'',e'')))= \bm{\Lambda}_{a'}(s',t,(\alpha',e'))\cdot\bm{\Lambda}_{a''}(s'',t,(\alpha'',e'')).
$$
The moduli subcategory of ${\mathcal S}_{\theta}$ associated to $(\bm{\Lambda},\bar{\partial}_{\wt{J}})$
is precisely a fibered product with respect to the submersive $t$-projections. We do this for every face 
and then extend and perhaps we need to use a trivial enrichment over the boundary. Of course, this extension 
has to be done carefully in order to obtain transversality, the submersion property, and preservation of compactness.
Then we extend to descendants in a trivial way.  The disjoint union is given by the standard formula and automatically has the desired
properties. \\

\noindent$\bm{[\ell\Longrightarrow \ell+1:]}$  The construction proceeds in the obvious way. \\

\subsection{Concluding Remarks}
  Finally we can lift everything to $\wt{\mathcal S}$
and introduce orientations. Then everything proceeds as expected. The discussion in the (generally non-existing)
regular case used by authors in \cite{EGH}, see also \cite{Wendl},  to explain the philosophy of SFT will carry through  in an averaged sense.
The integrals or correlators have a factor of a pullback of  $\tau^{n_a} $ in them.

We add some more detail and we assume for simplicity that  we work with the full subcategory of stable maps without marked points.
The associated orbit space is an open subset $Z'$ of $Z$ which is the union of suitable elements $a\in \pi'\subset  \pi_0(Z)$. Denote by
$\wt{Z}'$ the lift of $Z'$ and by $\wt{\pi}'$ the lift of $\pi'$.
  We shall call a class $\wt{a}\in \wt{\pi}'$ with even Fredholm index an \textbf{odd class} (this seems somewhat odd, but 
recall that we have divided out by the ${\mathbb R}$-action) and a class with odd Fredholm index an \textbf{even} class.

Suppose the orientations have been worked out.  
For a parent class $\wt{a}\in \wt{\pi}'$ with $\text{ind}(\wt{a})=0$ (the Fredholm index on $\wt{a}$) we denote by $H_i(\wt{a})$ for $0,1$
the count of solutions (The coarse moduli space consists of isolated points with rational multiplicities.).  Hence we obtain maps $H_i: \wt{\pi}_i\rightarrow {\mathbb Q}$.  We are only interested in the restriction of $H_i$ to classes $\wt{a}$ which are parent classes
and which in addition do not have bad orbits at the positive and negative punctures. Denote these restrictions by $\wt{H}_i$
and the domain by $\Pi$, i.e. $\wt{H}_i:\Pi\rightarrow {\mathbb R}$.

 Using the lift of our transversal 
perturbation we can take for fixed $t\in [0,1]$, using the fact that the $t$-projection is submersive,
 the moduli category consisting of all objects $(s,t,\alpha)$, where $\bm{\Theta}_a(s,t,\alpha):= \bm{\Lambda}_a(s,t,\bar{\partial}_{\wt{J}}(\alpha))>0$
and we define for $\wt{a}\in \Pi$
$$
\wt{H}_t(\wt{a}):=\oint_{\bm{\Theta}_{\wt{a},t}} p^{\ast}\tau^{n_{\wt{a}}},
$$
see \cite{HWZ2017} for the definition. This is essentially a sophisticated integral over $\mathfrak{M}_{\wt{a},t}$, i.e. 
the coarse moduli space above $t$, which uses some of the overhead to define it.
Hence 
$$
\wt{H}_t:\Pi\rightarrow {\mathbb R}.
$$
Since we have a ${\mathbb Z}_2$-grading by even and odd elements in $\Pi$ we can write 
a map on $\Pi$ as a sum of an even and an odd map.  A odd (even) map vanishes on even (odd) elements.

Then, for fixed $\wt{a}\in \Pi$,  $t\rightarrow H_t(\wt{a})$ is a smooth map interpolating between $H_i(\wt{a})$
for $i\in \{0,1\}$ (for this we need that the (parametrized) moduli spaces are in general position).  It is important that the latter satisfies a linear differential equation.  To explain this,  we note that 
one can define for $\wt{a}$ with $\text{ind}(\wt{a})=-1$ a family $t\rightarrow L_t(\wt{a})$. Then we can also define 
a super commutator $[.,.]$ which is defined as follows 
$$
[g,f]_{\wt{a}} = \sum_{\theta\in \text{face}_{\wt{a}}} g(\wt{a}')\cdot f(\wt{a''}) \pm f(\wt{a}')\cdot g(\wt{a}'')
$$
where $(a',a'')$ is associated to $\theta$. The signs are picked in such a way that $[.,.]$ is graded commutative.
Then we obtain the flow 
$$
\frac{d}{dt} \Phi_t =[L_t,\Phi_t]
$$
and it has the important property that $H_t = (\Phi_t)_\ast H_0$.  SFT (in this case without marked points) is then a suitable representation 
of this data, where we recall that associated $\wt{a}$ we have its genus as well as the ingoing and outgoing periodic orbits 
and a relative 2nd homology class in $Q$. (One can extend  this procedure  to the case when we have marked points.)
In the non-marked case $L_t$ is obtained from classes with $\text{ind}(\wt{a})=-1$.

Finally a remark about the way $L_t$ is defined in a simple model. Assume that $M$ is an oriented compact manifold with smooth boundary with corners and $E\rightarrow M$ an oriented vector bundle with $\dim(E_m)-1 =\dim(M)$. 
Consider the projection $p:[0,1]\times M \rightarrow M$ and take the pull-back of the bundle $E\rightarrow M$.
We consider ${\mathcal I}^N \times [0,1]\times M\rightarrow E: (s,t,m)\rightarrow f(s,t,m)$ tranversal to the zero section
so that $W:=f^{-1}(0)$ is a smooth manifold so that in addition the $t$-projection 
is submersive and for every $t\in [0,1]$ the fiber $W_t$ of $W\rightarrow [0,1]$ is in general position to the boundary
in ${\mathcal I}^N\times [0,1]\times M$.  Then the $t$-fiber $W_t$ is a smooth manifold of dimension $N$ and the projection $W_t\rightarrow {\mathcal I}^N$ is
proper. The pull-back of the $1$-form  $dt$ on $[0,1]$ to $W$ defines a point-wise non-zero one-form $\sigma$ on $W$.
We can take a vectorfield $X$ on $W$, so that $\sigma(X)\equiv 1$.  We define $\Sigma:=i_X((p|W)^{\ast}\tau^N)$ which is 
$(N-1)$-form on $W$. This form does not depend on the choice of $X$. Now we can integrate $\Sigma$ over every $W_t$
to obtain a function $L_t$. In the SFT-case this has to be done in the branched case and to be successful it requires some properties
from the perturbation $\bm{\Lambda}$, which are also used to make sure that one has an extension result.

\newpage

\part*{Appendices and References}

In this appendix we collect known facts and introduce notation and notions which are used throughout this text.

\section{Structures Associated to Riemann Surfaces}\label{Sec16}
The main purpose of this appendix is to recall basic facts about Riemann surfaces and most importantly fix notions
and notation which will be used throughout this lecture note. a
\subsection{Basic Notions}\label{APP5.1}
We recall some facts feeding into the DM-theory and refer the reader to \cite{HWZ5} for more detail, particularly with respect to the modified version using the exponential gluing profile.
For the latter there are also important details in \cite{HWZ8.7}.

\subsubsection*{Disk Pairs}
We are interested in compact disk-like Riemann surfaces $D_x$ with smooth boundary and interior point $x$, which we shall write as $(D_x,x)$. We refer to $x$ as a {\bf nodal point}.  An {\bf un-ordered nodal disk pair}
${\mathcal D}$ has the form $(D_x\sqcup D_y,\{x,y\})$, where $D_x$ and $D_y$ are as just described. An {\bf ordered nodal disk pair}
has the form $(D_x\sqcup D_y,(x,y))$. The ordered pair $(x,y)$ is called an {\bf ordered nodal pair} and $\{x,y\}$ is called an {\bf un-ordered nodal pair}.
In the case $(x,y)$ we shall refer to $x$ as the {\bf lower nodal point} and $y$ as the {\bf upper nodal point}.
Given $(D_x,x)$, a {\bf decoration}  $\wh{x}$ of the nodal point $x$ is a an oriented real line $\wh{x}\subset T_xD_x$.
The circle 
 $S^1={\mathbb R}/{\mathbb Z}$ acts naturally on the tangent spaces using their complex structures and therefore
 it acts also on the possible decorations for $x$ by 
\begin{eqnarray}
 (\theta,\wh{x})\rightarrow \theta\ast \wh{x}:=e^{2\pi i\theta}\cdot \wh{x}.
\end{eqnarray}
 Next we consider unordered pairs $\{\wh{x},\wh{y}\}$ which we call a {\bf decorated unordered nodal pair} or a {\bf decoration}
 of the nodal pair $\{x,y\}$.
 We declare $\{\wh{x},\wh{y}\}$  to be {\bf equivalent} to
$\{\theta\ast \wh{x},\theta^{-1}\ast \wh{y}\}$ where $\theta\in S^1$. The symbol $[\wh{x},\wh{y}]$
denotes the equivalence class associated to $\{\wh{x},\wh{y}\}$
\begin{eqnarray}
[\wh{x},\wh{y}] =\left\{ \{\theta\cdot\wh{x},\theta^{-1}\cdot \wh{y}\}\ |\ \theta\in S^1\right\}.
\end{eqnarray}
We call $[\wh{x},\wh{y}]$ a  {\bf natural angle} or {\bf argument} associated to $\{x,y\}$.
We denote by ${\mathbb S}_{\{x,y\}}$ the collection of all $[\wh{x},\wh{y}]$  associated to $\{x,y\}$ 
and call it  the {\bf set of  arguments or angles} associated to $\{x,y\}$.
Denote by ${\mathbb S}^1$ the standard unit circle in ${\mathbb C}$.
Fixing $z=[\wh{x}_0,\wh{y}_0]$ the map 
\begin{eqnarray}
\text{ar}_z: {\mathbb S}_{\{x,y\}}\rightarrow {\mathbb S}^1: [\theta\ast \wh{x}_0,\theta'\ast \wh{y}_0]\rightarrow e^{2\pi i(\theta+\theta')}
\end{eqnarray}
 is a bijection and
any two such maps, say $\text{ar}_z$ and $\text{ar}_{z'}$  have a transition map $\text{ar}_{z'}\circ \text{ar}_z^{-1}$ which is a rotation on ${\mathbb S}^1$.
Hence ${\mathbb S}_{\{x,y\}}$ has a natural smooth structure. It also has a natural orientation by requiring that $\text{ar}_z$ is orientation preserving,
where ${\mathbb S}^1$ is equipped with the orientation as a boundary of the unit disk.

Consider formal expressions $r\cdot [\wh{x},\wh{y}]$, where $r\in [0,1/4)$ (the choice of 1/4 has no deeper meaning other
than that certain constructions need a bound on the choice of $r$ and in our case $1/4$ is always a good bound). We say that $r\cdot[\wh{x},\wh{y}]=r'\cdot [\wh{x}',\wh{y}']$ provided either  $r=r'=0$,  or $r=r'>0$ and $[\wh{x},\wh{y}]=[\wh{x}',\wh{y}']$. In the following we shall call $r\cdot [\wh{x},\wh{y}]$ a {\bf natural gluing parameter} associated  to $\{x,y\}$. The collection ${\mathbb B}_{\{x,y\}}$  of these formal gluing parameters $r\cdot [\wh{x},\wh{y}]$ has a natural one-dimensional holomorphic 
manifold structure, so that fixing any $\{\wh{x}_0,\wh{y}_0\}$ the map
$$
r\cdot [\theta\cdot \wh{x}_0,\theta'\cdot \wh{y}_0]\rightarrow r\cdot e^{2\pi i(\theta +\theta')}
$$
 onto the standard
open disk in ${\mathbb R}$ of radius $1/4$ is a biholomorphic map.

When we deal with a finite number of disk pairs we can take their unordered or ordered 
nodal pair as an index set. We shall for example write $D$ for the whole collection of all occurring $\{x,y\}$
and we shall write ${\bm{D}}$ for the collection, i.e.
\begin{eqnarray}
{\bm{D}}=\left\{ (D_x\sqcup D_y,\{x,y\})\ |\ \{x,y\}\in D\right \}.
\end{eqnarray}
Sometimes, always clear from the context, we also view $\bm{D}$ as defining the disjoint union of all $D_x$, where $x$ varies
over $|D|=\cup_{\{x,y\}\in D} \{x,y\}$, together with the collection $D$ of nodal pairs
\begin{eqnarray}
\bm{D}=  \left(\coprod_{z\in |D|} D_z,\ D\right).
\end{eqnarray}
This is a specific compact nodal Riemann surface with smooth boundary.
Associated to every $\{x,y\}$ we have the set of natural gluing parameter ${\mathbb B}_{\{x,y\}}$
and we shall write ${\mathbb B}_D$ for the set of total gluing parameters, which are maps
$\mathfrak{a}$ associating to $\{x,y\}$ an element $a_{\{x,y\}}\in {\mathbb B}_{\{x,y\}}$. We can view ${\mathbb B}_D$ as sections of  a bundle over the finite set $D$, namely
$$
\coprod_{\{x,y\}\in D}{\mathbb B}_{\{x,y\}}   \rightarrow D: a_{\{x,y\}}\rightarrow \{x,y\}.
$$
 From this viewpoint a natural gluing parameter is a section.

\subsubsection*{Gluing Disks}

Consider an un-ordered nodal disk pair ${\mathcal D}:=(D_x\sqcup D_y,\{x,y\})$, consisting  of disk-like Riemann surfaces  $D_x$ and $D_y$, with smooth boundaries
containing the interior points $x$ and $y$, respectively, so that $(D_x,x)$ and $(D_y,y)$ are biholomorphic 
to $({\mathbb D},0)$, where ${\mathbb D}\subset {\mathbb C}$ is the closed unit disk. These biholomorphic maps are not unique but any two of them differ by a rotation which is biholomorphic. Denote for $0\in {\mathbb D}$
by $\wh{0}$ the standard decoration given by ${\mathbb R}\subset T_0{\mathbb D}$ with the standard orientation of the real numbers. If $\wh{x}$ is a decoration of $x$ there exists a unique biholomorphic map
$$
h_{\wh{x}}:(D_x,\wh{x})\rightarrow ({\mathbb D},\wh{0}).
$$
In the following we need the {\bf exponential gluing profile} $\varphi:(0,1]\rightarrow [0,\infty)$ defined by 
$$
\varphi(r)=e^{\frac{1}{r}}-e.
$$
\begin{definition}
Given an unordered disk pair $(D_x\sqcup D_y,\{x,y\})$ and a non-zero gluing parameter $a_{\{x,y\}}=r\cdot [\wh{x},\wh{y}]$
define  the set $Z_{a_{\{x,y\}}}$ by
\begin{eqnarray}\label{KOPTY20.1}
Z_{a_{\{x,y\}}}&=&\left\{ \{z,z'\}\ |\  z\in D_x,\ z'\in D_y,\ \right.\\
&&\left.\phantom{XXXX} h_{\wh{x}}(z)\cdot h_{\wh{y}}(z') = e^{-2\pi \varphi(r)}\right\}.\nonumber
\end{eqnarray}
Here  $\{\wh{x},\wh{y}\}$  is a representative of $[\wh{x},\wh{y}]$, but the 
 definition of the set does not depend on its choice.  If the gluing parameter vanishes, i.e. if $a_{\{x,y\}}=0$
 we define $Z_0=(D_x\sqcup D_y,\{x,y\})$.   We note that $Z_a=Z_b$ if and only if $a=b$, and in fact $Z_a\cap Z_b\neq \emptyset$ if and only if $a=b$.
 $Z_{a_{\{x,y\}}}$ is said to be obtained from $(D_x\sqcup D_y,\{x,y\})$ by gluing with gluing parameter
 $a_{\{x,y\}}$.
 \qed
 \end{definition}
 \begin{remark}
 We use this special gluing profile $\varphi$ in order to have compatibility with the sc-Freholm theory.
We obtain the classical Deligne-Mumford theory when we use the gluing profile $r\rightarrow -\frac{1}{2\pi}\cdot \ln(r)$.
 \qed
 \end{remark}
 Given a non-zero gluing parameter $a_{\{x,y\}} =r_{\{x,y\}}\cdot [\wh{x},\wh{y}]$ put $R=\varphi(r_{\{x,y\}})$
 and define the closed annuli $A_x(R)\subset D_x$ and $A_y(R)\subset D_y$ of modulus $2\pi R$  by 
\begin{eqnarray*}
 A_x(R)&=&\left\{ z\in D_x\setminus\{x\}\ |\ |h_{\wh{x}}(z)|\geq e^{-2\pi R}\right\}\\
  A_y(R)&=&\left\{ z'\in D_y\setminus\{y\}\ |\ |h_{\wh{y}}(z')|\geq e^{-2\pi R}\right\}.
\end{eqnarray*}
 The set $Z_{a_{\{x,y\}}}$ defined in (\ref{KOPTY20.1}) for non-zero gluing parameter  has a natural holomorphic manifold structure making it biholomorphic to a closed annulus of modulus $2\pi \cdot \varphi(r_{\{x,y\}})$ so that in addition the maps
\begin{eqnarray}\label{EQNA1.1}
 A_x(R) \xleftarrow{\pi_x^{a_{\{x,y\}}}} Z_{a_{\{x,y\}}}\xrightarrow{\pi_y^{a_{\{x,y\}}} }A_y(R)
\end{eqnarray}
 defined by $\pi_x(z,z')=z$ and $\pi_y(z,z')=z'$  are biholomorphic. Hence
 \begin{lem}
 $Z_{a_{\{x,y\}}}$ has a natural structure as a Riemann surface.\qed
 \end{lem}
 Assume that $a_{\{x,y\}}$ and $a'_{\{x,y\}}$ are two nonzero gluing parameters in ${\mathbb B}_{\{x,y\}}$ with the same modulus. We abbreviate
 $$
 R:=\varphi(|a_{\{x,y\}}|)=\varphi(a_{\{x',y'\}}|).
 $$
 In this case we obtain two copies of the diagram (\ref{EQNA1.1}), say, with $a=a_{\{x,y\}}$ and $a'=a'_{\{x,y\}}$
\begin{eqnarray}
&A_x(R) \xleftarrow{\pi^a_x} Z_{a}\xrightarrow {\pi^a_y} A_y(R)&\\
&A_x(R) \xleftarrow{\pi^{a'}_x} Z_{a'}\xrightarrow {\pi^{a'}_y} A_y(R).&\nonumber
\end{eqnarray}
We can compare the following two maps $A_x(R)\rightarrow A_y(R)$
$$
\pi_y^a\circ {(\pi^a_x)}^{-1}\ \ \text{and}\ \ \pi_y^{a'}\circ {(\pi^{a'}_x)}^{-1}.
$$
Given $(D_x,x)$ there is a well-define notion of a {\bf rotation} by $\theta\in S^1$. Namely take any biholomorphic map
$h: (D_x,x)\rightarrow ({\mathbb D},0)$ and define $R^x_\theta(z) = h^{-1}( e^{2\pi i\theta}\cdot h(z))$. This definition does not depend on the choice of $h$, and it follows immediately that the following identity holds.
\begin{eqnarray}
h_{\wh{x}} \circ  R^x_{\theta}  =   h_{e^{2\pi i\theta}\cdot \wh{x}}  = e^{2\pi i\theta}\cdot h_{\wh{x}}.
 \end{eqnarray}
We obtain the following lemma, which can be verified by a straight forward calculation.
 \begin{lem}
 Writing $a=|a|\cdot [\wh{x},\wh{y}]$ and $a'=|a'|\cdot [\wh{x}, e^{2\pi i\theta}\cdot \wh{y}]$, where  $|a|=|a'|\neq 0$, it holds that 
 $$
 \pi_y^a\circ {(\pi^a_x)}^{-1} = R_{\theta}^y  \circ\pi_y^{a'}\circ {(\pi^{a'}_x)}^{-1}.
$$
\end{lem}
\begin{proof}
By definition $Z_a=\{\{z,z'\}\ |\ h_{\wh{x}}(z)\cdot h_{\wh{y}}(z')=e^{-2\pi \varphi(r)}\}$ and with $a'=|a| \cdot [\wh{x},e^{2\pi i\theta}\cdot \wh{y}]$ we see that
\begin{eqnarray}
Z_{a'} &= &\{\{z,z'\}\ |\ h_{\wh{x}}(z)\cdot e^{2\pi i\theta}\cdot h_{\wh{y}}(z')= e^{-2\pi \varphi(r)}\}\\
&=& \{\{z,z'\}\ |\ h_{\wh{x}}(z) \cdot h_{\wh{y}}(R_{\theta}(z'))= e^{-2\pi \varphi(r)}\}.\nonumber\\
&=& \{\{z,R_{-\theta}(z')\}\ |\ h_{\wh{x}}(z) \cdot h_{\wh{y}}(z')= e^{-2\pi \varphi(r)}\}.\nonumber
\end{eqnarray}
From this it follows that if  $  \pi_y^a\circ {(\pi^a_x)}^{-1} (z)= z'$ then  $\pi_y^{a'}\circ {(\pi^{a'}_x)}^{-1}(z) = R_{-\theta}(z')$, and consequently 
$$
R_{\theta}\circ \pi_y^{a'}\circ {(\pi^{a'}_x)}^{-1}(z) =z' =  \pi_y^a\circ {(\pi^a_x)}^{-1} (z).
$$
\end{proof}

 \begin{definition}
If $a=a_{\{x,y\}}\neq 0$ with $R=\varphi(|a_{\{x,y\}}|)$ we denote by $M^p_{{\mathcal D},a}$ the collection of all $\{z,z'\}\in Z_a$
 such that $-\frac{1}{2\pi}\cdot \ln(|h_{\wh{x}}(z)|\in (R/2 -p/2,R/2+p/2)$. Note that this is only well-defined if $|a_{\{x,y\}}|$ is sufficiently small given $p$. 
 We call $M^p_{{\mathcal D},a}$ the {\bf middle annulus of width} $2p$ of $Z_a$.
 \qed
 \end{definition}
 For example with $a\in {\mathbb B}_{\{x,y\}}$ this is well-defined as long as $0<p<25$.
 When we are given a finite family of unordered disk pairs ${{\bm{D}}}$ and a gluing parameter
 $\mathfrak{a}\in {\mathbb B}_D$,  we denote by ${{\bm{Z}}}_{\mathfrak{a}}$ the disjoint union of all
$Z^{\{x,y\}}_{a_{\{x,y\}}}$. In the case that $a_{\{x,y\}}=0$ we have that $Z^{\{x,y\}}_{a_{\{x,y\}}}=(D_x\sqcup D_y,\{x,y\})$ and consequently, if $\mathfrak{a}\equiv 0$
we see that ${{\bm{Z}}}_{\mathfrak{a}}=\bm{D}$.
In the case of  ordered disk pairs we use a similar formalism. It will be clear from the context in which situation we are.
\subsubsection*{Holomorphic Polar Coordinates} 
Given $(D_x,x)$  let $\wh{x}$ be a decoration.
 Take the associated $h_{\wh{x}}:(D_x,x)\rightarrow ({\mathbb D},0)$ satisfying $Th_x(\wh{x})=\wh{0}$.
 We introduce the biholomorphic maps 
 $$
 \sigma_{\wh{x}}^+:[0,\infty)\times S^1\rightarrow D_x\setminus\{x\} :(s,t)\rightarrow h_{\wh{x}}^{-1}\left(e^{-2\pi (s+it)}\right)
 $$
 and
 $$
 \sigma^-_{\wh{x}}:(-\infty,0]\times \rightarrow D_x\setminus\{x\}:(s',t')\rightarrow h_{\wh{x}}^{-1}\left( e^{2\pi (s'+it')}\right).
 $$
 We shall call $\sigma_{\wh{x}}^\pm$ {\bf positive and negative holomorphic polar coordinates} on $D_x$ around $x$ 
 associated to the decoration $\wh{x}$.
 
If $(D_x\sqcup D_y,\{x,y\})$ is nodal disk pair and $a\in {\mathbb B}_{\{x,y\}}$ we obtain through gluing the space
 $Z_a$. With $a=|a|\cdot [\wh{x},\wh{y}]$  fix  a representative  $\{\wh{x},\wh{y}\}$.  We have the special biholomorphic maps,
 where $R=\varphi(|a|)$
\begin{eqnarray}\label{ERT31}
&&\\
&\sigma_{\wh{x}}^{+,a}:  [0,R]\times S^1\rightarrow Z_a: (s,t) \rightarrow \{\sigma^+_{\wh{x}}(s,t), \sigma^-_{\wh{y}}(s-R,t)\}&\nonumber\\
&\sigma_{\wh{y}}^{a,-}:[-R,0]\times S^1\rightarrow Z_a: (s',t')\rightarrow \{\sigma^+_{\wh{x}}(s'+R,t'),\sigma^-_{\wh{y}}(s',t')\}&\nonumber
\end{eqnarray}
 There are also maps $\sigma_{\wh{x}}^{-,a}$ and $\sigma_{\wh{y}}^{+,a}$ obtained by interchanging the roles of 
 $x$ and $y$.  In the case of an ordered disk pair the maps in (\ref{ERT31}) are the relevant ones, i.e. we take positive holomorphic polar coordinates
 for the lower disk and negative one for the upper disk.
 
As part of the constructions we shall consider maps $u:Z_a\rightarrow {\mathbb R}^N$ and  sometimes need to evaluate the average 
 over the loop in the middle. For this we can pick a nodal point in $\{x,y\}$, say $x$,  and take $\sigma^{+,a}_{\wh{x}}$ and calculate with $R=\varphi(|a|)$
 $$
 \int_{S^1} u\circ\sigma^{+,a}_{\wh{x}}(R/2,t) dt.
 $$
 The integral does not depend on the choice of $x$ or $y$ in $\{x,y\}$, and after the choice of $x$, it does not depend
 on the decoration $\wh{x}$.  We call the integral
 the {\bf middle loop average}.
 We call any of the maps $t\rightarrow \sigma^{\pm,a}_{\wh{x}}(\pm R/2,t)$ or $t\rightarrow \sigma^{\pm,a}_{\wh{y}}(\pm R/2,t)$  a  {\bf middle-loop map}. 
 If $a=0$ and $u$ is defined on the disk pair, being continuous over the nodal value (i.e. $u(x)=u(y)$),  then we can define the associated 
 middle loop average as $u(x)$ or $u(y)$, which are the same. 

A related concept is that of an {\bf $a$-loop}.  Assume we have a map defined on $D_x$ or the punctured $D_x\setminus\{x\}$.
Pick positive holomorphic polar coordinates centered at $x$, i.e.
$$
\sigma^+:{\mathbb R}^+\times S^1\rightarrow D_x\setminus\{x\}.
$$
Assume that $u:D_x\rightarrow {\mathbb R}^N$ is a continuous map and $a\in {\mathbb B}$ a nonzero gluing parameter.
We define an {\bf $a$-loop} as the map
$$
S^1\rightarrow D_x\setminus\{x\}:  t\rightarrow \sigma^+_a(t):= \sigma^+(R/2,t)
$$
where $R=\varphi(|a|)$. There is a whole $S^1$-family of $a$-loops. However the integral
$$
\int_{S^1}  u\circ\sigma_a^+(t)\cdot dt
$$
does not depend on the choice of the specific $a$-loop.

\subsection{Riemann Surfaces}\label{APPENDIX-RS}
After some preparation we shall  describe the category of stable Riemann surfaces and introduce auxiliary structures
for the DM-theory.

\subsubsection{Nodal Riemann Surfaces}
We shall consider tuples $\alpha=(S,j,M,D)$ consisting of a compact  Riemann surface $(S,j)$ without boundary,
but possibly disconnected, where  $M$ is a finite (unordered) subset of $S$ called {\bf (unordered) marked points}, and $D$ is a finite collection  of unordered pairs $\{x,y\}$, where $x,y\in S$ are different points. We require that $D$ has the property that $\{x,y\}\cap \{x',y'\}\neq \emptyset$ implies that
$\{x,y\}=\{x',y'\}$. We shall write $|D|$ for the union of all the $\{x,y\}$, i.e.
$$
|D| =\bigcup_{\{x,y\}\in D} \{x,y\},
$$
 and require that $|D|\cap M=\emptyset$. We refer to the elements $\{x,y\}$ as {\bf nodal pairs} and to $x$ and $y$ as {\bf nodal points}.  We can view the tuples as objects of a category.
For the following discussion we assume $M$ and the elements of $D$ to be unordered.
A morphism $\Phi:\alpha\rightarrow \alpha'$ is given by a tuple $\Phi=(\alpha,\phi,\alpha')$,
where $\phi:(S,j)\rightarrow (S',j')$ is a biholomorphic map such that $\phi(M)=M'$  and $\phi_\ast(D)=D'$,
where
$$
\phi_\ast (D) =\{\{\phi(x),\phi(y)\}\ |\ \{x,y\}\in D\}.
$$
We denote the category with objects $(S,j,M,D)$ and morphisms $\Phi$ by $\bar{\mathcal R}$.
 \begin{remark}
We shall sometimes consider modifications, namely we may allow $M$ to be ordered and in this case referred to as the set of {\bf ordered}
marked points. We also sometimes allow some of the nodal pairs to be ordered, i.e. the object are $(x,y)$ rather than $\{x,y\}$ and we refer
to an {\bf ordered} nodal pair. 
\qed
\end{remark}
We shall discuss later on in more detail objects in $\bar{\mathcal R}$ together with a finite group 
$G$ acting on it by biholomorphic maps preserving the additional structure. 

\subsubsection{The Category of Stable Riemann Surfaces}\label{RRRR---}
Denote by ${\mathcal R}$ the  full subcategory  of $\bar{\mathcal R}$ associated to objects, which satisfy 
an additional condition. Namely we impose the {\bf stability condition} (\ref{STABBB}) that for every connected component $C$ of $S$ its genus
$g(C)$ and the number $\sharp C:= C\cap (M\cup |D|)$ satisfies
\begin{eqnarray}\label{STABBB}
2g(C) + \sharp C\geq 3.
\end{eqnarray}
 From the classical Deligne-Mumford theory, see \cite{DM,RS} and also \cite{HWZ-DM}, it follows that
${\mathcal R}$ is what we shall call a  {\bf groupoidal category}. Namely every morphism is an isomorphism, between two objects are at most finitely morphisms (a consequence of the stability condition),
and the collection of isomorphism classes $|{\mathcal R}|$ is a set. It is also an important fact that $|{\mathcal R}|$  has a natural metrizable topology for which the connected components are compact. We shall call ${\mathcal R}$ the {\bf category of stable Riemann surfaces} with unordered marked points and nodal pairs.

\subsubsection{Glued Riemann Surfaces} 
Let $\alpha$ be an object in $\bar{\mathcal R}$ and denote by $G$ a finite 
group  acting by automorphisms of $\alpha$.   
\begin{definition}
A pair $(\alpha,G)$, where $\alpha$ is an object in $\bar{\mathcal R}$ and $G$ a finite 
group acting by automorphisms will be called a {\bf Riemann surface with a finite group action}.
\qed
\end{definition}
\begin{remark}
Note that the biholomorphic automorphism group might be infinite. However, $G$ utilizes only a finite part of the existing 
symmetries. An obvious example is the Riemann sphere, where we can take a finite subgroup of the rational transformations.
\qed
\end{remark}
Assume $(\alpha,G)$ is given, where $\alpha$ is an object in $\bar{\mathcal R}$.
Define ${\mathbb B}_\alpha$
by
$$
{\mathbb B}_\alpha =\prod_{\{x,y\}\in D} {\mathbb B}_{\{x,y\}},
$$
which as a product of one-dimensional complex manifolds is a complex manifold.  There is a natural projection
$\pi:{\mathbb B}_\alpha\rightarrow D$, and 
a {\bf gluing parameter} for an object $\alpha$ is a section $\mathfrak{a}$ of $\pi$, i.e. it  associates to $\{x,y\}\in D$
a symbol $a_{\{x,y\}}\in {\mathbb B}_{\{x,y\}}$
$$
\mathfrak{a}: D\rightarrow {\mathbb B}_\alpha: \{x,y\}\rightarrow a_{\{x,y\}}.
$$
The natural action of $G$ on $D$ by $g\ast \{x,y\}=\{g(x),g(y)\}$ lifts to a {\bf natural action}
of $G$ on the complex manifold of natural gluing parameters
$$
G\times {\mathbb B}_\alpha\rightarrow {\mathbb B}_\alpha,
$$
by $g\ast \mathfrak{a}=\mathfrak{b}$, where, writing $a_{\{x,y\}}=r_{\{x,y\}}\cdot [\wh{x},\wh{y}]$ we have
$$
b_{\{g(x),g(y)\}} = r_{\{x,y\}}\cdot [(Tg)\wh{x},(Tg) \wh{y}]\ \ \text{for}\ \ \{x,y\}\in D.
$$
We fix for every $z\in |D|$
a closed disk-like neighborhood $D_z$ with smooth boundary and $z$ an interior point
so that the union of these $D_z$ is invariant under $G$. We also require that $M\cap D_z=\emptyset$
for all $z\in |D|$. This choice gives for every $\{x,y\}\in D$ an unordered nodal disk pair 
${\mathcal D}_{\{x,y\}}=(D_x\sqcup D_y,\{x,y\})$. 
\begin{definition}
The collection ${{\bm{D}}}$ of these disk pairs,
having the properties stated above is called a {\bf small disk structure} for $\alpha$. 
\end{definition}
In case we have an ordered nodal pair $(x,y)$ we obtain an ordered disk pair written as $(D_x\sqcup D_y,(x,y))$.
We shall refer to $D_x$ as the {\bf lower disk} and $D_y$ as the {\bf upper disk}. We usually would also assume that
the action of $G$ would map an ordered nodal pair to an ordered nodal pair and also preserve the ordering. 

Given an object $\alpha=(S,j,M,D)$ in $\bar{\mathcal R}$ and a small disk structure 
${{\bm{D}}}$ it is convenient to note that given ${\bm{D}}$ we can recover $D$.  Hence, we introduce the
objects $(S,j,M,{\bm{D}})$ which are compact Riemann surfaces with small disk structure 
as well as $((S,j,M,{\bm{D}}),G)$, which is $((S,j,M,D),G)$ equipped with a small disk structure 
so that the union of the disks is invariant.

Assume that $\alpha=(S,j,M,D)$ is a nodal stable Riemann surface with unordered marked points and nodal points and $G$ is a finite group acting on $\alpha$ as previously described.  Fix a small disk structure ${\bm{D}}$, which for every $\{x,y\}$ gives us an unordered nodal disk pair $(D_x\sqcup D_y,\{x,y\})$.  Hence we consider $((S,j,M,{\bm{D}}),G)$.  Given a gluing parameter $\mathfrak{a}$ for $\alpha$ we obtain the  $a_{\{x,y\}}$ and obtain by disk-gluing
 $$
 Z^{\{x,y\}}_\mathfrak{a} := Z_{a_{\{x,y\}}},
 $$
 which is obtained from $(D_x\sqcup D_y,\{x,y\})$ by gluing with $a_{\{x,y\}}$. If $a_{\{x,y\}}=0$ we recover the nodal
 disk pair. Remove from $S$ for every nonzero $a_{\{x,y\}}$ the complement in $D_x\sqcup D_y$ 
 of $A_x(R)\sqcup A_y(R)$, where $R=\varphi(|a_{\{x,y\}}|)$. Here $A_x=\{z\in D_x\ |\ z=\sigma^+_{\wh{x}}(s,t),\ s\in [0,R]\}$
 and similarly for $A_y$.
 We define a new surface $S_{\mathfrak{a}}$, using for every $\{x,y\}$ with non-zero gluing parameter
   (\ref{EQNA1.1}),  the holomorphic   equivalence relation on $A_x(R)\sqcup A_y(R)$
   identifying $z\in A_x(R)$ with $z'\in A_y(R)$ provided $\{z,z'\}\in Z_\mathfrak{a}^{\{x,y\}}$.
 If the gluing parameter for some $\{x,y\}$ vanishes we do not do anything.
Having carried out this for every nodal pair  we obtain a nodal Riemann surface  surface $S_\mathfrak{a}$
with associated almost complex structure $j_\mathfrak{a}$.  We denote by $D_\mathfrak{a}$ the collection of all
 $\{x,y\}\in D$ with $a_{\{x,y\}}=0$. We shall write $M_\mathfrak{a}$ for the set $M$ viewed as a subset of $S_\mathfrak{a}$. Finally we set
 $$
 \alpha_\mathfrak{a}=(S_\mathfrak{a},j_\mathfrak{a},M_\mathfrak{a},D_\mathfrak{a}).
 $$
\begin{definition}
 We call $\alpha_\mathfrak{a}$ the stable Riemann surface {\bf obtained from $\alpha$ by gluing with $\mathfrak{a}$}.
 \qed
 \end{definition}
For $g\in G$ one easily verifies that the construction of $S_\mathfrak{a}$ allows to construct in a natural way a biholomorphic map
 $g_\mathfrak{a}:\alpha_\mathfrak{a}\rightarrow \alpha_{g\ast\mathfrak{a}}$. Given $h,g\in G$ we have the functorial properties
 $$
 h_{g\ast \mathfrak{a}}\circ g_\mathfrak{a} = (h\circ g)_{\mathfrak{a}}
 $$
 and $1_\mathfrak{a} = Id_{\alpha_\mathfrak{a}}$.

\subsection{Riemann Surface Buildings}
The building blocks are tuples $(\Gamma^-,S,j,D,\Gamma^+)$, where $(S,j)$ is a not necessarily connected compact Riemann surface, $D$ is a a set of nodal pairs,
and $\Gamma^\pm$ are a finite set of so-called positive and negative punctures.  The sets $|D|$, $\Gamma^+$, and $\Gamma^-$ are mutually disjoint. Let us denote such an object
by $\alpha$.  We allow finite groups $G$ acting on such a $\alpha$ by biholomorphic maps, where we impose the restriction that $G$ preserves the sets $\Gamma^+$, $\Gamma^-$,
and the set of nodal pairs $D$. A small disk structure ${\bm{D}}$ for $\alpha$ consists of a small disk structure associated to the nodal pairs in $D$ so that the union of the 
disks is invariant under $G$. The disks are assumed to be mutually disjoint and not to contain the points in $\Gamma^\pm$.  
Denoting by ${\mathbb B}_{\alpha}$ the set of natural gluing parameters we obtain through gluing $\alpha_{\mathfrak{a}}=(\Gamma^-,S_\mathfrak{a},j_\mathfrak{a},D_\mathfrak{a},\Gamma^+) $, where we identify $\Gamma^\pm$ naturally as a subset of $S_\mathfrak{a}$.  We shall call $\alpha_\mathfrak{a}$ a ${\bm{D}}$-descendent of $\alpha$.
It is convenient to consider the smooth manifold of gluing parameters ${\mathbb B}_D$ together with the $G$-action as a translation groupoid
$G\ltimes {\mathbb B}_D$.  We can also consider the groupoid whose objects are the glued $\alpha_\mathfrak{a}$ and the morphisms 
are the $(\alpha_{\mathfrak{a}},g_{\mathfrak{a}},\alpha_{g\ast\mathfrak{a}})$.  Obviously the two groupoids are isomorphic via $\mathfrak{a}\rightarrow \alpha_\mathfrak{a}$
and $(g,\mathfrak{a})\rightarrow (\alpha_\mathfrak{a},g_\mathfrak{a},\alpha_{g\ast\mathfrak{a}})$
We also note that 
$G$ defines actions $G\times \Gamma^\pm\rightarrow \Gamma^\pm$.  Denote by $G\ltimes \Gamma^\pm$ the associated translation groupoids, for which we have the  equivariant
diagram of inclusions
$$
\Gamma^+ \rightarrow \alpha_{\mathfrak{a}}\leftarrow \Gamma^-
$$
We generalize this now as follows.
We first consider tuples 
$$
(\alpha_0,b_1,...,b_{k-1},\alpha_{k-1}),
$$
 where $\alpha_i=(\Gamma^-_i,S_i,j_i,D_i,\Gamma^+_i)$ is as just described and $b_i:\Gamma^+_{i-1}\rightarrow \Gamma^-_{i}$
is a bijection.  We assume $G$ is a finite group acting acting on each  $\alpha_0,...,\alpha_{k-1}$ by biholomorphic maps  as previously described. Moreover, 
$G$ defines actions on the $\Gamma^\pm_i$ and  we assume that these actions are such that every $b_i:\Gamma^+_{i-1}\rightarrow \Gamma^-_i$ is equivariant.
We fix for every $(z,b_i(z))$ an ordered disk pair ${\mathcal D}_{(z,b_i(z))}=(D_z\sqcup D_{b_i(z)}, (z,b_i(z)))$ and assume that the union of these 
disks is invariant under $G$. Of course, the disks are mutually disjoint and do not intersect the floor disks.
An interface gluing parameter for the $(i-1,i)$-interface, $i=1,..,k$, is a map $\mathfrak{a}_{i-1,i}$,  which assigns to $(z,b_i(z))$ a gluing parameter ${\mathfrak{a}}_{i-1,i}(z)$,
having the additional property that either all its components are zero or all of its components are non-zero.
A total gluing parameter is given
as $(2k+1)$-tuple  $(\mathfrak{a}_0,{\mathfrak{a}}_{0,1},\mathfrak{a}_1,...,{\mathfrak{a}}_{k-1,k},\mathfrak{a}_{k})$. Denote by $1\leq i_1<..<i_\ell\leq k$
the ordered sequence of indices such that ${\mathfrak{a}}_{i-1,i}=0$.  We say that we have nontrivial interfaces $(i_1-1,i_1)$,..,$(i_{\ell}-1,i_{\ell})$.  Given $\alpha$ and the small disk structure, applying the total gluing parameter, we obtain $\alpha_\mathfrak{a}$ which again is a Riemann surface building.
The collection of all  such Riemann surface buildings we shall refer to as the set of ${\bm{D}}$-descendents of $\alpha$.  For example 
if $\mathfrak{a}$ has the nontrivial interfaces $1\leq i_1<..<i_\ell\leq k$ let us introduce $i_0=0$ and $i_{\ell+1}=k+1$. Then define 
$$
S_e= \coprod_{i\in \{i_e,i_{e+1}-1\}} S_i,
$$
and
$$
D_e =\coprod_{i\in \{i_e,i_{e+1}-1\}} D_i \coprod_{i\in \{i_e+1, i_{e+1}-1\}} D_{i-1,i}.
$$
  As punctures we take $\Gamma_{i_e}^-$ and $\Gamma_{i_{e+1}-1}^+$ so that we obtain
$(\Gamma_{i_e}^-,S_e,j_e,D_e,\Gamma^+_{i_{e+1}-1})$. The gluing parameters $(\mathfrak{a}_{i_e},\mathfrak{a}_{i_e,i_e+1},...,\mathfrak{a}_{i_{e+1}-1})$ allow us
to glue this surface and obtain $\alpha_{\mathfrak{a},e}$.
Together with the $b_{i_e}$ we obtain the Riemann surface building
$$
(\alpha_{\mathfrak{a},0},b_{i_1},\alpha_{\mathfrak{a},1},b_{i_2},...,b_{i_\ell},\alpha_{\mathfrak{a},\ell}).
$$
\section{Stable Hamiltonian Structures and  Periodic Orbits}\label{APP11}
We recall the notion of a stable Hamiltonian structure and derive useful results which are needed in this text  and
quite well-known.
\subsection{Stable Hamiltonian Structures}

One of the important objects is that of a stable Hamiltonian structure. A detailed study of these structures can be found in 
\cite{CV1}.
\subsubsection{Basic Definition}
We begin with the definition of a stable Hamiltonian structure.
\begin{definition} \label{28.1QQ}
Let $Q$ be a closed odd-dimensional manifold of dimension $\dim(Q)=2n-1$. A stable Hamiltonian structure on $Q$ is given by a pair
$(\lambda,\omega)$, where $\omega$ is a closed two-form of maximal rank on $Q$ and $\lambda$ a one-form such that
\begin{itemize}
\item[(1)] $\lambda\wedge \omega^{n-1}$ is a volume-form.
\item[(2)] The vector field $R$, called the {\bf Reeb vector field},  defined by
$$
i_R\lambda =1\ \ \hbox{and}\ \  {i_R} \omega=0
$$
 satisfies
$$
L_R\lambda =0.
$$
\end{itemize}
\qed
\end{definition}
The latter condition implies by Cartan's formula
$$
0=L_R\lambda= di_R\lambda +i_Rd\lambda = i_Rd\lambda.
$$
Since $R$ spans the kernel of $\omega$ this implies
\begin{eqnarray}\label{EQP28}
\text{ker}(\omega)\subset \text{ker}(d\lambda).
\end{eqnarray}
 The fact that $i_R\omega=0$ implies again by the Cartan formula that $L_R\omega=0$.
\begin{remark}
The standard example for a stable Hamiltonian structure is $(\lambda,d\lambda)$, where $\lambda$ is a contact form on $Q$.
\qed
\end{remark}
Stable Hamiltonian structures are an interesting object to study, see \cite{CV1}.  It is important to have such structures when studying pseudoholomorphic curves,
see \cite{EGH}. These structures allow to control certain area-based energy functionals, which is important in obtaining a priori estimates.

Associated to a stable Hamiltonian structure $(\lambda,\omega)$ on $Q$
we have the distribution $\xi=\ker(\lambda)$ and the natural splitting
of the tangent bundle
$$
TQ={\mathbb R}R\oplus \xi.
$$
We observe that the line bundle ${\mathbb R}R$ has a distinguished section $R$ and $\xi$ is in a natural way a symplectic vector bundle with symplectic structure being  $\omega|\xi\oplus \xi$. Let us observe that the flow $\phi_t$ associated to $R$ maps a vector in $\xi_x$ to a vector in $\xi_{\phi_t(x)}$
\begin{eqnarray}
T\phi_t:\xi\rightarrow \xi_{\phi_t(x)}.
\end{eqnarray}
\subsubsection{Symplectic Forms Associated to $(Q,\lambda,\omega)$}
We discuss stable Hamiltonian manifolds $(Q,\lambda,\omega)$ in somewhat more detail.
Denote by $p:{\mathbb R}\times Q \rightarrow Q $ the obvious projection and given a smooth map $\phi:{\mathbb R}\rightarrow {\mathbb R}$ 
we denote by $\wh{\phi}$ the map ${\mathbb R}\times Q\rightarrow {\mathbb R}$ defined by $\wh{\phi}(s,q)=\phi(s)$.
Given $(Q,\lambda,\omega)$ and $\phi$ we  can consider the two-form
$\Omega_{\phi}$  on ${\mathbb R}\times Q$ defined by
$$
\Omega_{\phi} = p^\ast \omega + d(\wh{\phi}\cdot p^\ast\lambda)
$$
which we shall write sloppily as $\omega+d(\phi\lambda)$. 
We observe that $\Omega_{\phi}(s,q)= \omega_q +\phi(s) \cdot d\lambda_q + \phi'(s) ds\wedge \lambda$. 
If $|\phi(s)|$ is small enough we see that $\omega_q+\phi(s) d\lambda_q$ as a two-form on $Q$ is maximally non-degenerate
and $\text{ker}(\omega+\phi(s) d\lambda)=\text{ker}(\omega)$ in view of (\ref{EQP28}).
The maximal non-degeneracy implies that $\omega+\phi(s) d\lambda$ restricted to 
$\{0\}\times\xi_q \subset T_{(s,q)}({\mathbb R}\times Q)$ is non-degenerate, i.e. a symplectic form.
 If $\phi$ satisfies this smallness condition 
and in addition $\phi'(s)>0$ for all $s$ then $\Omega_{\phi}$ is a symplectic form.
Hence we have obtain.
\begin{lem}\label{LEMMA-epsilon}
Given a smooth manifold $Q$ equipped with a stable Hamiltonian structure $(Q,\lambda,\omega)$ 
there exists $\varepsilon>0$ such that for every smooth $\phi:{\mathbb R}\rightarrow [-\varepsilon,\varepsilon]$ 
with $\phi'(s)>0$ for all $s\in {\mathbb R}$ the two-form $\Omega_{\phi}$ is symplectic.
\qed
\end{lem}
In view of this lemma we make the following definition.
\begin{definition}\label{DEFQQ00}
We denote for $\varepsilon>0$ by $\Sigma_{\varepsilon}$ the set of all smooth maps $\phi:{\mathbb R}\rightarrow [-\varepsilon,\varepsilon]
$ satisfying $\phi'(s)\geq 0$ for all $s\in {\mathbb R}$.\qed
\end{definition}
In the case of $Q$ equipped with a contact form we obtain the stable Hamiltonian structure $(Q,\lambda,d\lambda)$.
In this case we can take any smooth map $\phi:{\mathbb R}\rightarrow [-1,\infty)$ and assuming that $\phi'(s)>0$ 
it follows that
$$
\Omega_{\phi}= (1+\phi)d\lambda + \phi'(s)ds\wedge \lambda= d((1+\phi)\lambda)
$$
is symplectic.  A possible example is the map $\phi(s)=e^s-1$ which gives $\Omega_{\phi} = d(e^t\lambda)$ which is the usual symplectization form. We note that in the general case of stable Hamiltonian structures the upper bound on $\phi$ is important to get symplectic forms.

\subsubsection{Compatible Almost Complex Structures}
Starting with a stable Hamiltonian structure $(Q,\lambda,\omega)$ we take the manifold ${\mathbb R}\times Q$
and fix $\varepsilon>0$ with the properties guaranteed by Lemma \ref{LEMMA-epsilon}.
We consider the set of all 2-forms $\Omega_{\phi}$ on ${\mathbb R}\times Q$ with $\phi\in\Sigma_{\varepsilon}$.
Then this collection is invariant under the ${\mathbb R}$-action on ${\mathbb R}\times Q$ via addition on the first factor.
With $\xi=\ker(\lambda)$ we obtain the symplectic vector bundle $(\xi,\omega)\rightarrow Q$ and fix a complex structure 
for this vector bundle, i.e. a smooth fiber-preserving map $J:\xi\rightarrow \xi$ with the following two properties
\begin{itemize}
\item[(1)] $J^2=-Id$.
\item[(2)] $g_J:\xi\oplus\xi\rightarrow {\mathbb R}$ defined by $g_J(q)(h,k)=\omega_q(h,J(q)k)$ is fiber-wise a positive definite inner product.
\end{itemize}
With $R$ being the Reeb vectorfield we define a smooth ${\mathbb R}$-invariant  almost complex structure $\wt{J}$ for ${\mathbb R}\times Q$ by
\begin{eqnarray}\label{REDF28.3}
\wt{J}(a,q)(h,kR(q)+\Delta)= (-k, hR(q)+J(q)\Delta),
\end{eqnarray}
where $h,k\in {\mathbb R}$ and $\Delta\in\xi$.  Consider for $\phi\in\Sigma_{\varepsilon}$ the fiber-wise bilinear form
$$
\Omega_{\phi}\circ (Id\oplus \wt{J}).
$$
We compute with a vector $(h,kR(q)+\Delta)\in T_{(s,q)}({\mathbb R}\times Q)$
\begin{eqnarray}
&&\Omega_{\phi}\circ (Id\oplus \wt{J})((h,kR(q)+\Delta),(h,kR(q)+\Delta))\\
&=&(\omega+\phi(s)d\lambda)(\Delta,J(q)\Delta) + \phi'(s)(h^2 +k^2)\nonumber\\
&\geq & 0.\nonumber
\end{eqnarray}
Since $\omega+\phi(s)d\lambda$ is non-degenerate on $\xi$ we see that in case $\phi'(s)>0$ the expression is a positive definite
quadratic form.  
\begin{lem}
Given a smooth manifold with stable Hamiltonian structure $(Q,\lambda,\omega)$ and a ${\mathbb R}$-invariant 
almost complex structure $\wt{J}$ as described in (\ref{REDF28.3}) pick an admissible $\varepsilon>0$ as in Definition 
\ref{DEFQQ00}.  Then for every $\phi\in\Sigma_{\varepsilon}$ the (fiber-wise) symmetric quadratic form 
$\mathsf{Q}^{\phi}$ defined with $(h,kR(q)+\Delta)\in T_{(s,q)}({\mathbb R}\times Q)$ by 
$$
\mathsf{Q}^{\phi}_{(s,q)}(h,kR(q)+\Delta):= \Omega_{\phi}((h,kR(q)+\Delta,\wt{J}(s,q)(h,kR(q)+\Delta))
$$
satisfies $\mathsf{Q}^{\phi}_{(s,q)}\geq 0$.  If $\phi\in\Sigma_{\varepsilon}$ has the property $\phi'(s)>0$ for all $s\in {\mathbb R}$
it satisfies $\mathsf{Q}_{(s,q)}^{\phi}>0$.
\qed
\end{lem}

\subsection{Periodic Orbits}
We start with the(abstract) notion of a periodic orbit in a smooth manifold.
\begin{definition}
A  \textbf{periodic orbit} in a smooth manifold
$Q$ is a tuple $([\gamma],T,k)$, where
\begin{itemize}
\item[(1)]  $\gamma:S^1\rightarrow Q$ is a smooth embedding and $[\gamma]$ is the equivalence class of $\gamma$, where 
$\gamma\sim \gamma'$ provided $\gamma(t)=\gamma'(t+c)$ for all $t\in S^1$, where $c\in S^1$ is a suiable constant.
\item[(2)]  $T$ is a positive real number and $k\geq 1$ an integer.
\end{itemize}
Here $T$ is called the \textbf{period} and $k$ is the \textbf{covering number}. The fraction $T_0:= T/k$ is called the \textbf{minimal period}.
\qed
\end{definition}
It is an important fact that a given Hamiltonian structure $(\lambda,\omega)$ on an odd-dimensional manifold $Q$
produces automatically a set of periodic orbits. We shall explain this next, where we give a formulation compatible with the general definition.
\begin{definition}
Assume $Q$ is equipped with $(\lambda,\omega)$ and $R$ is the associated Reeb vector field. A periodic orbit for $(V,\lambda,\omega)$ is a tuple
$([\gamma],T,k)$ with $k$ being a nonzero positive integer, $T$ a positive number, $\gamma:S^1\rightarrow Q$ a smooth embedding 
and $[\gamma]$ the set of reparameterisations of $\gamma$, i.e. $t\rightarrow \gamma(t+c)$, where $c\in S^1$,
such that the following property holds
$$
\frac{d\gamma}{dt}(t) = \frac{T}{k}\cdot R(\gamma(t)).
$$
We call $T$ the {\bf period} and $k$ the {\bf covering number}, and $T/k$ the {\bf minimal period}.
\qed
\end{definition} 
\begin{remark}
The way to think about a periodic orbit for $(Q,\lambda,\omega)$ is as follows.  Take the Reeb vector field $R$ 
and solve $\dot{x}=R(x)$. Assume we have a periodic orbit $(x,T)$, i.e. there exists a $T>0$ such that $x(0)=x(T)$.
Then there exists an integer $k\geq 1$ such that $x(t)\neq x(0)$ for $0<t<T/k$, $x(0)=x(T/k)$, and $T/k$ is called the minimal period.
We can define  the embedding $\gamma:{\mathbb R}/{\mathbb Z}\rightarrow Q$ by 
$$
\gamma(t) = x(tT/k).
$$
Then $\frac{d\gamma}{dt}(t) = (T/k)\cdot R(\gamma(t))$. If we take another point on $x({\mathbb R})$ and solve the differential equation
with this as starting point we obtain a map $y$, which again can be viewed as a $T$-periodic solution $(y,T)$. Applying the same procedure 
we obtain another element in $[\gamma]$. Hence our notation keeps track of the period,  the minimal period, and the set $\gamma(S^1)$ with
a preferred class of parameterizations.
\qed
\end{remark}

\begin{definition}
Consider a periodic orbit $([\gamma],T,k)$ associated to $(Q,\lambda,\omega)$.
We say that $([\gamma],T,k)$ is {\bf non-degenerate} if the symplectic map $T\phi_{T}(p):\xi_p\rightarrow \xi_p$ for some fixed 
$p\in \gamma(S^1)$  does not have $1$ in its spectrum. The definition does not depend on the choice of $p$. If for a Hamiltonian structure $(\lambda,\omega)$ all periodic orbits are non-degenerate we say that $(\lambda,\omega)$ is a {\bf non-degenerate stable Hamiltonian structure}.
\qed
\end{definition}

\begin{remark}
It is always possible to perturb a contact form in $C^\infty$, by keeping the associated contact structure, so that   the new form is non-degenerate.
The situation for stable Hamiltonian structure is more subtle, see \cite{CFP,CV1}. It seems to be  possible to always perturb them to a Morse-Bott situation.
Our discussion of polyfold structures can be generalized to this case, but we shall not do it here and concentrate on the non-degenerate case.
\qed
\end{remark}

For the sc-Fredholm Theory  we are interested in the situation where the periodic orbits come from a 
non-degenerate stable Hamiltonian structure $(\lambda,\omega)$ on $Q$. 
\begin{definition}
Given a smooth compact manifold equipped with a non-degenerate stable Hamiltonian structure $(Q,\lambda,\omega)$ 
we denote by ${\mathcal P}(Q,\lambda,\omega)$ the collection of all
periodic orbits $([\gamma],T,k)$. We denote by  ${\mathcal P}^\ast(Q,\lambda,\omega)$ the union
$$
{\mathcal P}^\ast(Q,\lambda,\omega) := {\mathcal P}(Q,\lambda,\omega)\bigcup \{\emptyset\}.
$$
\qed
\end{definition}
For the sc-Fredholm theory it will be important  to associate to the elements in ${\mathcal P}^\ast(Q,\-\lambda,\omega) $
weight sequences. However, care has to be taken that these choices are compatible with spectral gaps coming from
a certain class of self-adjoint operators which occur naturally after a choice of almost complex structure compatible 
with $(\lambda,\omega)$ has been made. We shall discuss this at the end of the next subsection, after introducing
the before-mentioned class of self-adjoint operators, called asymptotic operators.

\subsection{Special Coordinates and Asymptotic Operators}
Consider  $(Q,\lambda,\omega)$,  a smooth manifold with a stable Hamiltonian structure. With $\xi:=\text{ker}(\lambda)$
we equip the symplectic vector bundle $(\xi,\omega)\rightarrow Q$ with a compatible complex structure $J:\xi\rightarrow \xi$, $J^2=-Id$,
so that $\omega\circ (Id\oplus J)$ equips each fiber with a positive definite inner product.  We equip  $Q$ with the Riemannian metric
\begin{eqnarray}
g_J:= \lambda\otimes \lambda +\omega\circ(Id\oplus J).
\end{eqnarray}
  As we already explained before, the data  $(\lambda,\omega,J)$ will determine a ${\mathbb R}$-invariant 
almost complex structure $\wt{J}$ on ${\mathbb R}\times Q$, see (\ref{REDF28.3}), and this structure  will be important in
the sc-Fredholm theory.  However, not every $\wt{J}$ will work for a pseudoholomorphic curve theory
in ${\mathbb R}\times Q$.  It will be important that the underlying $(Q,\lambda,\omega)$ is non-degenerate and that weight sequences
are picked appropriately.  One can express this in various ways. The choice here is to pick suitable special coordinates and to bring 
the study of a periodic orbit into the context of a special model. Of course, it will be important to verify that 
different choices of special coordinates lead to the same conclusion.
\subsubsection{Special Coordinates}
We are interested in the geometry near
a given periodic orbit 
$$
([\gamma],T,k).
$$
The idea is to transfer the general problem to the model case $Q_0:=S^1\times {\mathbb R}^{2n-2}$
with periodic orbit $([\gamma_0],T,k)$, where $\gamma_0(t)=(t,0)$.  We equip  $Q_0$ with $\lambda_0=dt$ and $\omega_0=\sum_{i=1}^{n-1} dx_i\wedge dy_i$, which is a degenerate structure.
We denote by  $J_0$  the standard structure on ${\mathbb R}^{2n-2}={\mathbb R}^2\oplus..\oplus {\mathbb R}^2$, where on the ${\mathbb R}^2$-factors 
$(1,0)$ is mapped to $(0,1)$. We define $\xi^0=\text{ker}(dt)$ and as inner product $dt\otimes dt + \langle.,.\rangle_{{\mathbb R}^{2n-2}}$.  Note that 
$\omega_0\circ (Id\oplus J_0)=\langle.,.\rangle_{{\mathbb R}^{2n-2}}$.  
What will be of interest to us are structures 
$(\lambda',\omega',J')$ defined near $\gamma_0(S^1)$, which coincide with $(\lambda_0,\omega_0,J_0)$ on $\gamma_0(S^1)$.
This data $(\lambda',\omega',J')$ will be obtained as the push forward of a restriction of $(\lambda,\omega,J)$ on $Q$ to a small 
open neighborhood of $\gamma(S^1)$ by a special choice of coordinates.
\begin{definition}
Let $(Q,\lambda,\omega)$ be a smooth manifold with a stable Hamiltonian structure.  Consider a periodic orbit
$([\gamma],T,k)$.
A {\bf special coordinate transformation} is a smooth diffeomorphism  $\phi:U(\gamma(S^1))\rightarrow U(\gamma_0(S^1))$ which has the following properties.
\begin{itemize}
\item[(1)] There exists a representative $\gamma^{\phi}\in[\gamma]$ such that $\phi\circ\gamma^{\phi}(t)=\gamma_0(t)$.
\item[(2)] For every $t\in S^1$ the tangent map 
$$
(T\phi)(\gamma^{\phi}(t)):T_{\gamma^{\phi}(t)}Q\rightarrow T_{\gamma_0(t)}Q_0
$$
induces  a linear isomorphism $\wh{\phi}_t:\xi_{\gamma^{\phi}(t)}\rightarrow \xi^0_{\gamma_0(t)}$
which is complex linear and isometric for the distinguished inner products.
\end{itemize}
\qed
\end{definition}
The following is left as an exercise.
\begin{lem}
Given $(Q,\lambda,\omega,J)$ and a periodic orbit $([\gamma],T,k)$ there exist for given representative $\gamma^{\phi}$
suitable open neighborhoods
$U(\gamma(S^1))$ and $U(\gamma_0(S^1))$ and a special coordinate transformation $\phi:U(\gamma(S^1))\rightarrow
U(\gamma_0(S^1))$ with the property $\phi\circ \gamma^{\phi}(t)=\gamma_0(t)$.
\qed
\end{lem}
\subsubsection{Asymptotic Operator}\label{ASYMOperator}
We start with $(Q,\lambda,\omega)$ and an associated periodic orbit $([\gamma],T_0,k_0)$
Consider the map which associates to an element $y$  in  $C^1(S^1,Q)$ the loop $\frac{d}{dt}y - T_0R(y)$.
The latter can be viewed as a $C^0$-section of $y^\ast TQ$. The section vanishes at $x=\gamma$ and consequently 
has a linearization $L_{\gamma}$ at every representative $\gamma$ of $[\gamma]$. 

Having fixed $\gamma:=\gamma^{\phi}\in [\gamma]$ take a special coordinate 
$\phi$ such that $\phi(\gamma(t))=\gamma_0(t)$. 
With this  choice of $\phi$ and given a loop $y$ near $\gamma$ we obtain a loop $\phi\circ y$ near $\gamma_0$.
We consider the following which associates to a smooth loop $z$ near $\gamma_0$ the smooth loop $\eta=\eta(z)$ in ${\mathbb R}^{2n-1}$
defined by
\begin{eqnarray}
t\rightarrow \text{pr}_2\circ T\phi\left(\frac{d}{dt}(\phi^{-1}(z(t)))- T_0\cdot R(\phi^{-1}(z(t)))\right).
\end{eqnarray}
We note that we can rewrite this as
$$
t\rightarrow \frac{d}{dt}z(t) - T_0\cdot [pr_2\circ T\phi \circ R\circ \phi^{-1}(z(t))]
$$
We differentiate this expression at $z=\gamma_0$ in the direction $h$ to obtain $L^{\phi}(h)$, which is a loop $S^1\rightarrow {\mathbb R}\times {\mathbb R}^{2n-2}$.
With the linear map 
$$
\wh{B}^{\phi}(t):{\mathbb R}\times {\mathbb R}^{2n-2}\rightarrow {\mathbb R}\times {\mathbb R}^{2n-2}
$$
 being obtained by differentiating $y\rightarrow T_0\cdot [pr_2\circ T\phi\circ  R\circ \phi^{-1}(y)$ at $\gamma_0(t)$ we see that
$L^{\phi}$ has the form 
$$
L^{\phi}(h)=\frac{d}{dt}h - \wh{B}(t)h.
$$
\begin{lem}
The following holds true.
\begin{itemize}
\item[(1)]  The map $\wh{B}^{\phi}$ has the form 
$$
\wh{B}^{\phi}(t)(h_1,\Delta)=(0,B^{\phi}(t)\Delta),
$$
where $h=(h_1,\Delta)\in {\mathbb R}\times {\mathbb R}^{2n-2}$.
\item[(2)] $-J_0B^{\phi}(t)$ is symmetric for the standard structure on ${\mathbb R}^{2n-2}$.
\item[(3)] The unbounded operator defined in $L^2(S^1,{\mathbb R}^{2n-2})$ with domain $H^1(S^1,{\mathbb R}^{2n-2})$
by $h\rightarrow -J_0[\frac{d}{dt}h -B^{\phi}(t)h]$ is self-adjoint and has a compact resolvent.
\end{itemize}
\end{lem}
\begin{proof}
These are known results and we refer for a discussion  to \cite{BEHWZ,Dragnev,Wendl}.
\end{proof}
With the above discussion in mind we make the following definition.
\begin{definition}
Given $(Q,\lambda,\omega,J)$ and a periodic orbit $([\gamma],T_0,k_0)$ we denote for given special coordinates 
by $\bm{L}^{\phi}$ the linear unbounded  self-adjoint operator
\begin{eqnarray}
&\bm{L}^{\phi}: L^2(S^1,{\mathbb R}^{2n-2})\supset H^1(S^1,{\mathbb R}^{2n-2})\rightarrow L^2(S^1,{\mathbb R}^{2n-2}):&\\
&h\rightarrow 
-J_0[\frac{d}{dt}h -B^{\phi}(t)h].&\nonumber
\end{eqnarray}
\qed
\end{definition}
It is important to know the relationship between $\bm{L}^{\phi}$ and $\bm{L}^{\psi}$ given two different choices 
of special coordinates.   Given an ordered pair $(\phi,\psi)$ we have that 
$$
\phi\circ \gamma^{\phi}(t)=\gamma_0(t)=\psi\circ\gamma^{\psi}(t),\ t\in S^1.
$$
There exists a well-define element $c=c(\phi,\psi)\in S^1$, called {\bf phase},  such that $\gamma^{\psi}(t)=\gamma^{\phi}(t+c)$ for $t\in S^1$.
Next we study a pair $(\phi,\psi)$ with phase $c=c(\phi,\psi)$,  We consider the transition map
$\sigma:=\psi\circ \phi^{-1}$ which is defined on an open neighborhood of $\gamma_0(S^1)$ in $S^1\times {\mathbb R}^{2n-2}$.
\begin{lem}
The following identity holds  for a pair $(\phi,\psi)$ with phase $c$
$$
\psi\circ\phi^{-1}(\gamma_0(t+c))=\gamma_0(t)
$$
\end{lem}
\begin{proof}
We compute
$$
 \gamma_0(t) =\psi\circ \gamma^{\psi}(t)=\psi\circ \gamma^{\phi}(t+c) = \psi\circ\phi^{-1} \circ\gamma_0(t+c).
$$
\end{proof}
Abbreviate $\sigma=\psi\circ\phi^{-1}$.  From the properties of $\phi$ and $\psi$ we know that for $(t,z)\in S^1\times {\mathbb R}^{2n-2}$
it holds that $\sigma(t+c,0)=(t,0)$ and moreover $D\sigma(t,0)(\{0\}\times {\mathbb R}^{2n-2})\subset \{0\}\times {\mathbb R}^{2n-2}$.
Further the induced map ${\mathbb R}^{2n-2}\rightarrow {\mathbb R}^{2n-2}$ is unitary for using the complex structure coming from $J_0$.
Denoting this map by $U_{(\phi,\psi)}(t)$ we obtain a loop of unitary matrices. We note that 
$$
D\sigma(t,0)(h,k) =(h,U(t)k),\ (h,k)\in {\mathbb R}\times {\mathbb R}^{2n-2}.
$$
Unitary here means that the operators commute with $J_0$ and are isometric.
Given the unitary loop  
$U_{(\phi,\psi)}$ we shall write $U_{(\phi,\psi)}^{-1}$  for the point-wise inverted loop.
\begin{lem}
For given $(\phi,\psi,\sigma)$
we have the identity 
\begin{eqnarray}\label{Gldfg}
U_{\phi,\sigma}(t) = U_{\psi,\sigma}(t-c(\phi,\psi))\circ U_{\phi,\psi}(t).
\end{eqnarray}
\end{lem}
\begin{proof}
We shall write $\sigma\circ\phi^{-1} =(\sigma\circ\psi^{-1})\circ (\psi\circ\phi^{-1})$ and 
 recall  that $\psi\circ\phi^{-1}(\gamma_0(t)=\gamma_0(t-c(\phi,\psi))$.  Differentiating the first expression along $\gamma_0(t)$
and taking the ${\mathbb R}^{2n-2}$-part we obtain at $\gamma_0(t)$
$$
U_{(\phi,\sigma)}(t) = U_{(\psi,\sigma)}(t-c(\phi,\psi))\circ U_{(\phi,\psi)}(t).
$$
\end{proof}
Finally we show the following.
\begin{prop}
For given $(Q,\lambda,\omega,J)$ and a choice of periodic orbit $([\gamma],T_0,k_0)$ consider 
for associated special coordinates $\phi,\psi$ the asymptotic operators $\bm{L}^{\phi}$ and $\bm{L}^{\psi}$.
Then the following equality holds
$$
\bm{L}^{\psi} = \mathsf{U}_{(\phi,\psi)}\circ \bm{L}^{\phi}\circ \mathsf{U}^{-1}_{(\phi,\psi)}.
$$
\end{prop}
\begin{proof}
We have the following two expressions  from which $\bm{L}^{\phi}$ and $\bm{L}^{\psi}$ are being derived.
\begin{eqnarray}
& z\rightarrow \frac{d}{dt} z -T_0\cdot [\text{pr}_2\circ T\phi\circ R\circ\phi^{-1}(z)]&\\
& y\rightarrow \frac{d}{dt} y -T_0\cdot [\text{pr}_2\circ T\psi\circ R\circ\psi^{-1}(y)]&\nonumber
\end{eqnarray}
Let us define $\sigma=\psi\circ\phi^{-1}$. We note that $\sigma\circ\gamma_0(t) = \gamma_0(t-c)$
where $c=c(\phi,\psi)$. We define the map $z\rightarrow y$ by $y(t-c)=\sigma\circ z(t)$.

We compute
\begin{eqnarray*}
 &&\left(\frac{d}{dt} y -T_0\cdot [\text{pr}_2\circ T\psi\circ R\circ\psi^{-1}(y)]\right)(t-c)\\
&=& \frac{d}{dt} y(t-c) -T_0\cdot [\text{pr}_2\circ T\psi\circ R\circ\psi^{-1}(y(t-c))]\\
&=& \frac{d}{dt} (\sigma\circ z(t)) - T_0 \cdot  [\text{pr}_2\circ T\sigma\circ T\phi\circ R\circ\phi^{-1}(z(t))]\\
&=& D\sigma(z(t))[ \frac{d}{dt}z(t) -T_0 \cdot T\phi\circ R\circ\phi^{-1}(z(t))].
\end{eqnarray*}
Differentiating the relationship $y(t-c)=\sigma\circ z(t)$ between the input loop $z$
and the output loop $y$ with respect to $z$ at $\gamma_0$ in the direction of $h$ gives
$$
k(t-c) = D\sigma(t,0)h(t).
$$
Hence we obtain
\begin{eqnarray*}
&&  \frac{d}{dt}k(t-c) - \wh{B}^{\psi}(t-c) k(t-c)\\
&=& D\sigma(t,0)[\frac{d}{dt}h(t) - \wh{B}^{\phi}(t) h(t)]
\end{eqnarray*}
Specializing we obtain from this for $h\in H^1(S^1,{\mathbb R}^{2n-2})$ also the relation ship
$$
\frac{d}{dt}k(t-c) -B^{\psi}(t-c)k(t-c)= U_{(\phi,\psi)}(t) [\frac{d}{dt}h(t) - B^{\phi}(t) h(t)]
$$
This means that
$$
(\frac{d}{dt} -B^{\psi})\mathsf{U}_{(\phi,\psi)}h = \mathsf{U}_{(\phi,\psi)} (\frac{d}{dt}h- B^{\phi}h),
$$
and after multiplying by $-J_0$ we obtain the desired result.
\end{proof}
As a consequence of the previous discussion we can define the $J$-spectral interval around $0$ 
associated to a periodic orbit $([\gamma],T,k)$ associated to $(Q,\lambda,\omega)$ and $J$.
\begin{definition}\label{SPECTrum}
Let $(Q,\lambda,\omega)$ be a closed manifold equipped with a stable Hamiltonian structure 
and $J$ an admissible complex multiplication for $\text{ker}(\lambda)\rightarrow Q$. 
The $J$-{\bf spectral interval} associated to a periodic orbit $\bm{\gamma}=([\gamma],T_0,k_0)$
is the largest interval $(a,b)\subset {\mathbb R}$,  $a\leq 0\leq b$, such that $\sigma(\bm{L}^{\phi})\cap (a,b)=\emptyset$.
Here $\sigma{\bm{L}}^{\phi})$ is the spectrum associated to this self-adjoint operator. 
We also define  $\sigma(\bm{\gamma},J):= \sigma(\bm{L}^{\phi})$, which, of course, does not depend 
on $\phi$, and call it the $J$-{\bf spectrum} of $\bm{\gamma}$.
\qed
\end{definition}
\begin{remark}
Since for $\phi,\psi$ the associated operators are unitarily conjugated the $J$-interval does not depend 
on the choice of special coordinate.  In the case that $\bm{\gamma}$ is non-degenerate the spectral interval will be 
nonempty, containing $0$ in the interior.
\qed
\end{remark}
As just mentioned the non-degeneracy assumption implies that $0\not\in \sigma (\bm{\gamma},J)$.
Since the operator $\bm{L}^{\phi}$  has a compact resolvent we have a spectral gap around
$0$. We call a positive number $\delta$ associated to $\bm{\gamma}=([\gamma],T,k)$  admissible, provided
$$
\sigma(\bm{\gamma},J)\cap [-\delta,\delta]=\emptyset.
$$
Note that the admissibility of $\delta$ depends on the original choice of $J$, which most of the time is fixed
from the beginning. If we want to stress the dependence on $J$ we call $\delta$ admissible for $(\bm{\gamma},J)$.
There are several notions and results associated to periodic orbits, which are frequently used in constructions.
\begin{definition}\label{DEFSpectrum}
Given a non-degenerate stable Hamiltonian structure $(\lambda,\omega)$ on the closed
manifold $	Q$ and a compatible $J$ we call a map
$$
\delta_o:{\mathcal P^\ast}\rightarrow (0,\infty)
$$
a {\bf weight selector} associated to $J$ provided for every $\bm{\gamma}=([\gamma],T,k) \in {\mathcal P}$,  the  number $\delta(\bm{\gamma},J)$ associated to $(\bm{\gamma},J)$ is admissible and bounded strictly by $2\pi$.  In addition we require that $\delta(\emptyset)\in (0,2\pi)$. 
A {\bf weight sequence} $(\delta_i)$ for $(\lambda,\omega,J)$ is a sequence
of weight functions $\delta_m$ so that for every periodic orbit $\gamma$ or $\gamma=\emptyset $ we have
$$
0<\delta_0(\gamma)<\delta_1(\gamma)<....
$$
\qed
\end{definition}

\subsection{Conley-Zehnder and Maslov Index}\label{CZMaslov}
In the (Fredholm) index theory for the CR-oper\-ator the Conley-Zehnder index plays an important role. We follow 
\cite{Dragnev}, which is based on \cite{FloerH1,FloerH2,HWZ-Em,SZ}. We view ${\mathbb R}^2$ as a symplectic vector space with coordinates 
$(x,y)$ and symplectic form $dx\wedge dy$.  Then ${\mathbb R}^{2n}$ is identified with the direct sum ${\mathbb R}^2\oplus..\oplus {\mathbb R}^2$,
coordinates $(x_1,y_1,...,x_n,y_n)$, and symplectic form $\omega=\sum_{i=1}^n dx_i\wedge dy_i$. 
\subsubsection{Conley-Zehnder Index} Denoting by $\text{Sp}(n)$ the group of linear symplectic maps 
${\mathbb R}^{2n}\rightarrow {\mathbb R}^{2n}$ we consider the space of continuous arcs $\Phi:[0,1]\rightarrow \text{Sp}(n)$ starting at the identity $\text{Id}_{2n}$ 
at $t=0$ and ending at $\Phi(1)$ which is a symplectic map not having $1$ in the spectrum. We denote by $ \Sigma(n)$ the maps 
$\alpha:[0,1]\rightarrow \text{Sp}(n)$ starting and ending at $\text{Id}_{2n}$. The map
$$
\text{G}(n)\times\Sigma(n)\rightarrow \Sigma(n): (\alpha,\Phi)\rightarrow \alpha\cdot\Phi, \ (\alpha\cdot\Phi)(t)=\alpha(t)\circ\Phi(t)
$$
is well-defined. We also have the inversion map
$$
\Sigma(n)\rightarrow \Sigma(n): \Phi\rightarrow \Phi^{-1},\ \Phi^{-1}(t):= (\Phi(t))^{-1}.
$$
Finally there is the obvious map
$$
\Sigma(n)\times\Sigma(m)\rightarrow\Sigma(n+m): (\Phi,\Psi)\rightarrow \Phi\oplus\Psi.
$$
A classical map $\mu^n_M:\text{G}(n)\rightarrow {\mathbb Z}$ is the Maslov index which is characterized by the following theorem.
\begin{thm}
The maps $\mu^n_M$ for $n\in \{1,....\}$ are characterized by the following requirements.
\begin{itemize}
\item[(1)] Two loops  $\alpha_1,\alpha_2\in \text{G}(n)$ are homotopic in $\text{G}(n)$ if and only if $\mu^n_M(\alpha_1)=\mu^n_M(\alpha_2)$.
\item[(2)] The map induced on $\pi_1(\text{Sp}(n),Id)$ by $\mu^n_M$ is a group isomorphism to ${\mathbb Z}$, i.e in this particular case equivalently
$$
\mu^n_M(\alpha_1\cdot \alpha_2) =\mu^n_M(\alpha_1)+\mu^n_M(\alpha_1).
$$
Here $(\alpha_1\cdot \alpha_2)(t)=\alpha_1(t)\circ\alpha_2(t)$.
\item[(3)] It holds 
$$
\mu^M_n\left(\left[t\rightarrow\left(e^{2\pi it}Id_2 \oplus Id_{2n-2}\right)\right]\right)=1.
$$
\end{itemize}
\qed
\end{thm}
Having characterized the Maslov index we state the main result about the Conley-Zehnder index.
The Conley-Zehnder index refers to a family of maps $\mu_{CZ}^n:\Sigma(n)\rightarrow {\mathbb Z}$, $n\in \{1,2,..\}$.

\begin{thm}
There exists a unique family $\mu_{CZ}^n: \Sigma(n)\rightarrow {\mathbb Z}$ for $n\in\{1,2,..\}$
characterized by the following properties:
\begin{itemize}
\item[(1)] Homotopic maps in $\text{G}(n)$ have the same index $\mu_{CZ}^n$.
\item[(2)] For $\alpha\in \text{G}(n)$ and $\Phi\in \Sigma(n)$ the identity
$$
\mu_{CZ}^n(\alpha\cdot\Phi) =\mu_{CZ}^n(\Phi) + 2\cdot \mu_M^n(\alpha).
$$
\item[(3)] $\mu_{CZ}^n(\Phi^{-1})+\mu_{CZ}^n(\Phi)=0$.
\item[(4)] $\mu_{CZ}^1(\gamma)=1$, where $\gamma(t)=e^{\pi it}Id_{{\mathbb R}^2}$.
\item[(5)] $\mu^{n+m}_{CZ}(\Phi\oplus \Psi)=\mu^n_{CZ}(\Phi) +\mu^m_{CZ}(\Psi)$.
\end{itemize}
\qed
\end{thm}

%-----------------------------------------------------------------------------
% End of biblio.tex
%-----------------------------------------------------------------------------

\printindex

\begin{thebibliography}{999}
\bibitem{Abbas} C. Abbas, An introduction to compactness results in symplectic field theory. Springer, Heidelberg, 2014. viii+252 pp.
\bibitem{Abraham-Robbin} R. Abraham and J. Robbin, 
Transversal mappings and flows.
An appendix by Al Kelley W. A. Benjamin, Inc., New York-Amsterdam 1967 x+161 pp.
\bibitem{AH} C. Abbas and H. Hofer, Holomorphic Curves and Global Questions in Contact Geometry, preprint 2018.
\bibitem{Adams} R. Adams,  Sobolev spaces. Pure and Applied Mathematics, Vol. 65. Academic Press, New York-London, 1975. xviii+268 pp. 
\bibitem{BH} E. Bao and K. Honda, Semi-global Kuranishi charts and the definition of contact homology, arXiv:1512.00580 .
\bibitem{BM} F.~Bourgeois and K.~ Mohnke, Coherent orientations in symplectic field theory. Math. Z. 248 (2004), no. 1, 123-146. 
\bibitem{BEHWZ} F.~Bourgeois, Y.~Eliashberg, H.~Hofer,
K.~Wysocki and E.~Zehnder,  \textit{Compactness Results in Symplectic Field
Theory},  Geometry and Topology, Vol. 7, 2003, pp.799-888.
\bibitem{Cartan} H. ~Cartan, Sur les
r\'etractions d'une vari\'et\'e, C. R. Acad.Sc. Paris, t. 303, Serie I, no 14, 1986, p. 715.
\bibitem{CV1} K. Cieliebak and E. Volkov, \textit{First steps in stable Hamiltonian topology},  J. Eur. Math. Soc. (JEMS) 17 (2015), no. 2, 321--404.
\bibitem{CFP}  K.  Cieliebak, U. Frauenfelder, G. Paternain, Stability is not open. Ann. Inst. Fourier (Grenoble) 60 (2010), no. 7, 2449--2459 (2011)
	\bibitem{Dragnev} D.~Dragnev, Fredholm theory and transversality for noncompact pseudoholomorphic maps in symplectizations, Comm. Pure Appl. Math. 57 (2004), no. 6, 726-763.
\bibitem{DM} P.~Deligne and D.~Mumford,
\textit{The irreducibility of the space of curves of given genus}, 
Inst. Hautes \'Etudes Sci. Publ. Math., No. 36 (1969), 75-109.
\bibitem{Ebin} D. Ebin, The manifold of Riemannian metrics. 1970 Global Analysis (Proc. Sympos. Pure Math., Vol. XV, Berkeley, Calif., 1968) pp. 11Ð40 Amer. Math. Soc., Providence, R.I.
\bibitem{EGH} Y. Eliashberg, A. Givental\ and\ H. Hofer, \textit{Introduction to Symplectic Field Theory},
 Geom. Funct. Anal. {\bf 2000}, Special Volume, Part II, 560--673.
\bibitem{El} H. Eliasson, \textit{Geometry of manifolds of maps}, J. Differential Geometry \textbf{ 1} (1967), 169--194.
\bibitem{FFGW} O. Fabert, J. W. Fish, R. Golovko, and K. Wehrheim, Polydolds: A First and Second Look,  EMS Surv. Math. Sci. 3 (2016), no. 2, 131--208.
\bibitem{Fi} B. Filippenko, Polyfold Regularization of Constrained Moduli Spaces, arXiv:1807.00386.
\bibitem{FWZ} B. Filippenko, Z. Zhou and K. Wehrheim, Counterexamples in Scale Calculus, arXiv:1807.02591.

\bibitem{Floer1} A. Floer, The unregularized gradient flow of the symplectic action. Comm. Pure Appl. Math. 41
(1988), no. 6, 775--813.
\bibitem{FloerH1} A. Floer and H. Hofer, Coherent orientations for periodic problems in symplectic geometry.
Math. Z. 212 (1993), no. 1, 13--38.
\bibitem{FloerH2} A. Floer and H. Hofer, Symplectic homology. I. Open sets in Cn. Math. Z. 215 (1994), no. 1,
37--88.
\bibitem{FH-primer} J. Fish and H. Hofer, Polyfold and  SFT Notes I:
A Primer on Polyfolds and Construction Tools,  arxiv:1806.07025v2
\bibitem{FH-local-local} J. Fish and H. Hofer, Polyfold and  SFT Notes II: Local-Local Constructions, arxiv:1808.04939
 \bibitem{FH-book} J. W. Fish and H. Hofer, Polyfold Constructions: Tools, Techniques, and Functors, book in preparation.
\bibitem{FH-Lecturenote} J. W. Fish and H. Hofer, Lectures on Polyfold Constructions and Symplectic Field Theory, in preparation.
\bibitem{Fry} R. Fry and S. McManus, Smooth Bump Functions and the Geometry of Banach Spaces, Expositiones Mathematicae 20 (2002), 143--183.
\bibitem{FO} K. Fukaya\ and\ K. Ono,
    \textit{Arnold conjecture and Gromov--Witten invariants.}
    Topology, vol. 38,  No 5, 1999.pp. 933-1048.
      \bibitem{FOOO} K. Fukaya, Y.-G. Oh, H. Ohta and K. Ono,
\textit{Lagrangian intersection Floer theory: anomaly and obstruction}. Part I. AMS/IP Studies in Advanced Mathematics, 46.1.  American Mathematical Society, Providence, RI; International Press, Somerville, MA, 2009. xii+396 pp.
\bibitem{FOOO1} K. Fukaya, Y.-G. Oh, H. Ohta and K. Ono, \textit{Lagrangian intersection Floer theory: anomaly and obstruction}. Part II. AMS/IP Studies in Advanced Mathematics, 46.2. American Mathematical Society, Providence, RI; International Press, Somerville, MA, 2009,  397--805.
\bibitem{FOOO2010} K. Fukaya, Y. G. Oh, H. Ohta and K. Ono, Lagrangian Floer theory on compact toric
manifolds I, Duke Math. J. 151 (2010), no. 1, 23-174.
\bibitem{FOOO2010b}  K. Fukaya, Y. G. Oh, H. Ohta and K. Ono, Lagrangian Floer theory on compact toric
manifolds II: bulk deformations, Selecta Math. (N.S.) 17 (2011), no. 3, 609-711.
\bibitem{FOOO2013} K. Fukaya, Y.-G. Oh, H. Ohta and K. Ono, Technical details on Kuranishi structure
and virtual fundamental chain, arXiv:1209.4410.
\bibitem{FOOO2015I} K. Fukaya, Y.-G. Oh, H. Ohta and K. Ono, Kuranishi structure, Pseudo-holomorphic curve, and virtual fundamental chain, Part 1, arXiv:1503.07631.
\bibitem{FOOO2017II} K. Fukaya, Y.-G. Oh, H. Ohta and K. Ono, Kuranishi structure, Pseudo-holomorphic curve, and virtual fundamental chain: Part 2, arxiv:1704.01848
\bibitem{Fuk} K. Fukaya, private communication (emails) on multisections, July 27-August 7, 2018
\bibitem{GTWZ} C. Godefry, S. Troyanski, J.H.M. Whitfield, and V. Zizler, Smoothness in weakly compactly generated Banach spaces, J. Func. Anal., 52 (1983), 344--352.
\bibitem{G} M.~Gromov, \textit{Pseudoholomorphic Curves in
Symplectic Geometry},  Invent. Math. vol. 82 (1985), 307-347.
\bibitem{Hirsch} M. Hirsch, Differential topology. Graduate Texts in Mathematics, No. 33. Springer-Verlag, New York-Heidelberg, 1976. x+221 pp.
\bibitem{H0} H. Hofer, \textit{Pseudoholomorphic curves in symplectisations with applications to the Weinstein conjecture in dimension three}, Invent. Math. vol. 114 (1993), 515--563.
\bibitem{H2014} H. Hofer, Polyfolds and Fredholm Theory, Lectures on Geometry, Clay Lecture Notes Series, edited by N. Woodhouse, Oxford University Press  2017,  87-156.
\bibitem{Hofer} H. Hofer,  \textit{A General Fredholm Theory and
Applications}, Current Developments in Mathematics, edited by D.
Jerison, B. Mazur, T. Mrowka, W. Schmid, R. Stanley, and S. T. Yau,
International Press, 2006.
\bibitem{HWZ-DM} H. Hofer, K. Wysocki and E. Zehnder,  \textit{Deligne--Mumford type Spaces with a View Towards Symplectic Field Theory},  draft.
\bibitem{HS} H. Hofer and J. Solomon, Remarks on Inductive Perturbation Theory via Multisections, in preparation.
\bibitem{HWZ} H. Hofer, K. Wysocki\ and\ E. Zehnder,
\textit{Properties of pseudoholomorphic curves in symplectisations. I. Asymptotics.},   Ann. Inst. H. Poincar\'e Anal. Non Lin\' eaire  vol. 13  (1996),  no. 3, 337--379. 
\bibitem{HWZ99} H. Hofer, K. Wysocki and E. Zehnder,  Properties of pseudoholomorphic curves in symplectizations. III. Fredholm theory. Topics in nonlinear analysis, 381--475, Progr. Nonlinear Differential Equations Appl., 35, Birkhäuser, Basel, 1999. 
\bibitem{HWZ2} H. Hofer, K. Wysocki\ and\ E. Zehnder,
\textit{A General Fredholm Theory {I}: A Splicing-Based Differential
Geometry}, JEMS, vol. 9, issue 4 (2007),  841--876.
\bibitem{HWZ3} H. Hofer, K. Wysocki\ and\ E. Zehnder,
\textit{A General Fredholm Theory {II}: Implicit Function Theorems,}
 Geom. Funct. Anal.  \textbf{19}  (2009),  no. 1, 206--293.
\bibitem{HWZ3.5} H. Hofer, K. Wysocki\ and\ E. Zehnder,
\textit{A General Fredholm Theory {III}: Fredholm Functors and Polyfolds,}
 Geom. Topol.  \textbf{13}  (2009),  no. 4, 2279--2387.
\bibitem{HWZ7.1} H.~Hofer, K.~Wysocki and E.~Zehnder,
\textit{Finite Energy Cylinders with Small Area,}
Ergodic Theory Dynam. Systems,  \textbf{22} (2002), No. 5, 1451-1486.
\bibitem{HWZ8.7} H. Hofer, K. Wysocki\ and\ E. Zehnder, \textit{Sc-Smoothness, Retractions and New Models for Smooth Spaces}, 
Discrete and Continuous Dynamical Sytems, Vl 28 (No 2), (2010), 665--788.
\bibitem{HWZ5} H. Hofer, K. Wysocki\ and\ E. Zehnder,
\textit{Applications of Polyfold Theory {I}: The Polyfolds of Gromov-Witten Theory},  Mem. Amer. Math. Soc. 248 (2017), no. 1179, v+218 pp. 
\bibitem{HWZ2017} H. Hofer, K. Wysocki, and E. Zehnder, Polyfold  and Fredholm Theory, Reference Book, preprint on arXiv, 1707.08941
\bibitem{HWZ-Em} H. Hofer, K. Wysocki, and E. Zehnder, Properties of pseudo-holomorphic curves in symplectizations.
II. Embedding controls and algebraic invariants. Geom. Funct. Anal. 5 (1995), no. 2,
270--328.
\bibitem{HZ} H. Hofer and E. Zehnder, Symplectic invariants and Hamiltonian dynamics. Birkh\"auser Advanced Texts: Basler Lehrb\"ucher,  Birkh\"auser Verlag, Basel, 1994. xiv+341 pp.
\bibitem{Hummel} Ch. Hummel,  Gromov's compactness theorem for pseudo-holomorphic curves. Progress in Mathematics, 151. Birkhäuser Verlag, Basel, 1997. viii+131 pp.
\bibitem{Hu} M.~ Hutchings, Michael An index inequality for embedded pseudoholomorphic curves in symplectizations. J. Eur. Math. Soc. (JEMS) 4 (2002), no. 4, 313-361.
\bibitem{HuN} M.~Hutchings and J.~Nelson, Automatic transversality for contact homology I: Regularity, Abh. Math. Semin. Univ. Hambg. 85 (2015), no. 2, 125-179. 
\bibitem{HuN2} M.~Hutchings and J.~Nelson, Cylindrical contact homology for dynamically convex contact forms in three dimensions, J. Symp. Geom. 14, No, 4, 983-1012, 2016. 
\bibitem{HuT1} M.~ Hutchings and C.~Taubes,
Gluing pseudoholomorphic curves along branched covered cylinders. I. 
J. Symplectic Geom. 5 (2007), no. 1, 43-137. 
\bibitem{HuT2}  M.~ Hutchings and C.~Taubes, Gluing pseudoholomorphic curves along branched covered cylinders. II. J. Symplectic Geom. 7 (2009), no. 1, 29-133.
\bibitem{IS} S. Ishikawa, Construction of general symplectic  field theory, arXiv:1807.09455.
\bibitem{Kon} M. Kontsevich,  Enumeration of rational curves via torus action, in ''The
Moduli Space of Curves" (R. Dijkgraaf, C. Faber and G. van der Geer,
eds.) Birkhauser (1995), 335-568.
\bibitem{JEM} M. Jemison, Polyfolds of Lagrangian Floer Theory in All Genera, in preparation.
\bibitem{Lego} Lego, https://www.lego.com/en-us/legal/legal-notice/fair-play
\bibitem{LT} Jun Li and G. Tian, \textit{Virtual moduli cycles and {G}romov-{W}itten invariants of general symplectic manifolds}, Topics in symplectic $4$-manifolds (Irvine, CA, 1996), First Int. Press Lect. Ser., I, Int. Press, Cambridge, (1998), 47--83.
\bibitem{Lock}  J. Lockhart and McOwen, \textit{Elliptic differential operators on non-compact manifolds}, Ann.
Scuola Norm. Sup. Pisa Cl. Sci., \textbf{4} (1985), 409-447.
\bibitem{MCDUFF} D. McDuff, Groupoids, branched manifolds and multisections. J. Symplectic Geom. 4 (2006), no. 3, 259-315.
\bibitem{McW1} D. McDuff and Katrin Wehrheim,  The topology of Kuranishi atlases,   Proc. Lond. Math. Soc. (3) 115 (2017), no. 2, 221--292. 
\bibitem{McW2} D. McDuff and K. Wehrheim,  The fundamental class of smooth Kuranishi atlases with trivial isotropy,  J. Topol. Anal. 10 (2018), no. 1, 71-243.
\bibitem{McW3} D. McDuff and K. Wehrheim,  Smooth Kuranishi atlases with isotropy,  Geom. Topol. 21 (2017), no. 5, 2725--2809.
\bibitem{Omori} H. Omori, Infinite dimensional Lie transformation groups. Lecture Notes in Mathematics, Vol. 427. Springer-Verlag, Berlin-New York, 1974. xi+149 pp.
\bibitem{Pardon1} J.~Pardon,  An algebraic approach to virtual fundamental cycles on moduli spaces of pseudo-holomorphic curves, Geometry \& Topology 20 (2016), 779--1034.
\bibitem{Pardon2} J.~Pardon, Contact homology and virtual fundamental cycles, arXiv: 1508.03873.
\bibitem{RS}  J. Robbin and D. Salamon, A construction of the Deligne-Mumford orbifold. J. Eur. Math. Soc. (JEMS) 8 (2006), no. 4, 611--699. 
\bibitem{RS-err} J. Robbin and D. Salamon, Corrigendum: A construction of the
Deligne-Mumford orbifold, J. Eur. Math. Soc. 9, 901-905.
\bibitem{RS2} J. Robbin and D. Salamon, The spectral flow and the Maslov index. Bull. London Math. Soc. 27 (1995), no. 1, 1-33. 
\bibitem{Siefring1} R. Siefring, Relative asymptotic behavior of pseudoholomorphic half-cylinders. Comm. Pure Appl. Math. 61 (2008), no. 12, 1631--1684. 
\bibitem{Siefring2} R. Siefring, Intersection theory of punctured pseudoholomorphic curves. Geom. Topol. 15 (2011), no. 4, 2351--2457.
\bibitem{SZ}  D. Salamon and E. Zehnder,  Morse theory for periodic solutions of Hamiltonian systems and the
Maslov index. Comm. Pure Appl. Math. 45 (1992), no. 10, 1303--1360.
\bibitem{Smale} S. Smale,  An infinite dimensional version of Sard's theorem. Amer. J. Math. 87 1965 861-866.
\bibitem{Solomon} J. Solomon,  Inductive extension of multisections, SFT 9 talk, August 2018.
   \bibitem{Tr} H. Triebel, {\it Interpolation theory, function spaces, differential operators}, North-Holland, Amsterdam, 1978.
   \bibitem{Wehrheim} K. Wehrheim, Fredholm Notions in Scale Calculus and Hamiltonian Floer Theory, arxiv:1209.4040v2
   \bibitem{Wendl} Ch. Wendl, Lectures on Symplectic Field Theory, 	arXiv:1612.01009.
   \bibitem{Wendl2} Ch,. Wendl, Transversality and super-rigidity for multiply covered holomorphic curves, arxiv1609.09867v3
   \bibitem{Yang1} D. Yang,  A choice-independent theory of Kuranishi structures and the polyfold?
Kuranishi correspondence, NYU Ph.D. Thesis.
http://webusers.imj-prg.fr/~dingyu.yang/thesis.pdf.
   \bibitem{Yang2} D. Yang, The polyfold--Kuranishi correspondence I: A choice-independent theory of Kuranishi structures , arxiv:1402.7008.
   \bibitem{Yang3}  D. Yang, Virtual Harmony, arxiv:1510.06849.
   \bibitem{Yang4} D. Yang,  Category of Kuranishi structure germs and a forgetful functor from polyfold
Fredholm sections (video), http://scgp.stonybrook.edu/archives/11730 (2014).
\end{thebibliography}
\end{document}